\documentclass[12pt]{amsart}
\usepackage{graphicx}
\usepackage{amsmath}
\usepackage{amsfonts}
\usepackage{amssymb}
\usepackage{mathrsfs}
\usepackage{mathtools}
\mathtoolsset{showonlyrefs}
\usepackage{setspace}
\usepackage{datetime}
\usepackage{color,enumitem,graphicx}
\usepackage[colorlinks=true,urlcolor=blue,
citecolor=red,linkcolor=blue,linktocpage,pdfpagelabels,
bookmarksnumbered,bookmarksopen]{hyperref}
\usepackage{geometry}
\allowdisplaybreaks[4]

\numberwithin{equation}{section}
\newcommand{\R}{\mathbb{R}}
\usepackage{array}
\usepackage[expansion=false]{microtype}

\theoremstyle{plain}
\newtheorem{theorem}{Theorem}[section]
\newtheorem{lemma}[theorem]{Lemma}
\newtheorem{proposition}[theorem]{Proposition}
\newtheorem{corollary}[theorem]{Corollary}

\theoremstyle{remark}
\newtheorem{remark}[theorem]{Remark}

\newcommand{\Id}{\mathrm{Id}}
\newcommand{\tr}{\operatorname{tr}}
\newcommand{\supp}{\operatorname{supp}}
\newcommand{\dd}{\,\mathrm d}

\title[Optimal Rigidity Results for the \(k\)-Hessian Equation]{Optimal Rigidity Results for the
  \(k\)-Hessian Equation of Lane--Emden Type}
\hypersetup{pdftitle={Optimal Rigidity Results for the
  k-Hessian Equation of Lane--Emden Type},
  pdfauthor={Wei Dai; Jingze Fu; Changfeng Gui; Guolin Qin}}
\author{Wei Dai, Jingze Fu, Changfeng Gui, Guolin Qin}
\date{}

\address{School of Mathematical Sciences, Beihang University (BUAA), Beijing 100191, P. R. China, and Key Laboratory of Mathematics, Informatics and Behavioral Semantics, Ministry of Education, Beijing 100191, P. R. China}
\email{weidai@buaa.edu.cn}

\address{School of Mathematical Sciences, Beihang University (BUAA), Beijing 100191, P. R. China}
\email{fujingze@buaa.edu.cn}

\address{Department of Mathematics, University of Macau, Macau SAR, P. R. China}
\email{changfenggui@um.edu.mo}

\address{State Key Laboratory of Mathematical Sciences, Academy of
Mathematics and Systems Science, Chinese Academy of Sciences, Beijing
100190, P.R. China;
and University of Chinese Academy of Sciences, Beijing 100049,
P.R. China}
\email{qinguolin18@mails.ucas.ac.cn}

\thanks{Wei Dai is supported by the NNSF of China (No. 12571113 \& No. 12222102), the National Science and Technology Major Project (2022ZD0116401) and the Fundamental Research Funds for the Central Universities.
Changfeng Gui is supported by NSFC Key Program (Grant No.12531010), University of Macau research grants CPG2024-00016-
FST, CPG2025-00032-FST, CPG2026-00027-FST, SRG2023-00011-FST, MYRGGRG2023-00139-FST-UMDF, UMDF
Professorial Fellowship of Mathematics, Macao SAR FDCT 0003/2023/RIA1 and Macao SAR FDCT 0024/2023/RIB1.
G. Qin is supported by National Key R\&D Program of China (Grant 2025YFA1018400) and NNSF of China (Grant 12471190).}

\begin{document}

\begin{abstract}

 In this paper, we establish optimal Liouville theorems and classification
results for the \(k\)-Hessian Lane--Emden equation \[
  \sigma_k(-D^2u)=u^p\quad\text{in }\R^n,\qquad
  -D^2u\in\overline{\Gamma_k},\qquad u\geq 0, \] where \(2\leq k<\frac{n}{2}\)
and $p>0$. Let $p_- = \frac{nk}{n-2k}$ and the critical Hessian--Sobolev
exponent $p_* = \frac{(n+2)k}{n-2k}$. Phuc and Verbitsky proved nonexistence of
positive solutions for \(k<p\leq p_-\), while Ou subsequently covered the cases
\(p\in(0,k]\). We close this gap and prove the optimal Liouville
theorem for any \(p_-<p<p_*\): any nonnegative \(C^2\) entire solution must be
identically zero. We also prove the optimal Liouville theorem for nonnegative
locally bounded Hessian-measure weak solutions. This identifies the critical
exponent \(p_*\) as the sharp Liouville threshold, since radial positive
solutions exist for \(p\geq p_*\).

For the critical case \(p=p_*\), we prove the optimal classification results for all \(n>2k\): every nontrivial
nonnegative \(C^2\) entire solution is a Hessian--Sobolev bubble, without any additional assumption.  For the limiting case \(n=2k\),
we classify finite-mass solutions to the $\frac{n}{2}$-Hessian Liouville equation under a proper asymptotic
condition $u(x)\rightarrow-\infty$ as $|x|\rightarrow\infty$. In particular, we provide the fully nonlinear
counterparts of the classical Liouville and classification theorems of Gidas--Spruck, Gidas--Ni--Nirenberg, and Caffarelli--Gidas--Spruck.

\end{abstract}

\maketitle

\medskip
\noindent\textit{2020 Mathematics Subject Classification.}
35J60, 35B53, 35B33.

\smallskip
\noindent\textit{Keywords.}
\(k\)-Hessian equations, Liouville theorems, Classification results, Critical exponent, Rigidity.

\section{Introduction}

\subsection{Background and setting of the problem}

In this paper, we study nonnegative entire admissible solutions of
\begin{equation}\label{eq:1.1}
    u\geq 0,\qquad
    -D^2u\in\overline{\Gamma_k},\qquad
    \sigma_k(-D^2u)=u^p
    \quad\text{in }\R^n,
\end{equation}
where \(2\leq k<\frac{n}{2}\), $0<p\leq p_*$, and $p_*=\frac{k(n+2)}{n-2k}$ is the Hessian--Sobolev critical exponent. At the limiting dimension \(n=2k\), we also study the scale-invariant $\frac{n}{2}$-Hessian Liouville equation
\begin{equation}\label{eq:1.2}
 -D^2u\in\Gamma_k,\qquad \sigma_k(-D^2u)=e^u
 \quad\text{in }\R^{n}.
\end{equation}
For a real symmetric matrix \(A\), we define
\begin{equation}
    \Gamma_k=\{A:\sigma_j(A)>0,\ 1\leq j\leq k\},
\end{equation}
where \(\sigma_j(A)\) denotes the
\(j\)-th elementary symmetric function of the eigenvalues of \(A\) (set
\(\sigma_0=1\)). $\Gamma_k$ is the G\aa rding cone associated with the hyperbolic polynomial
\(\sigma_k\), see \cite{Garding1959}.  Its closure is
\(\overline{\Gamma_k} =\{A:\sigma_j(A)\geq0,\ 1\leq j\leq k\}\). If \(u\not\equiv0\),
Lemma~\ref{Lem:2.2} shows that \(u>0\) and
\(-D^2u\in\Gamma_k\).  The cone condition
\(-D^2u\in\Gamma_k\) is the ellipticity
hypothesis: it makes \(-u\) \(k\)-convex and the Newton tensors through
order \(k-1\) positive definite.

For the critical case $p=p_*$, the connection between \eqref{eq:1.1} and the Hessian--Sobolev inequality is direct.  For a
smooth nonnegative admissible \(u\) with sufficient decay, put
\(\mathcal E_k[u]=\int_{\R^n}u\,\sigma_k(-D^2u)\dd x\).
The critical inequality, with optimal constant \(C_{n,k}\), has the
homogeneous form
\begin{equation}\label{eq:1.3}
 \left(\int_{\R^n}u^{p_*+1}\dd x\right)^{\frac{k+1}{p_*+1}}
 \leq C_{n,k}\mathcal E_k[u].
\end{equation}
For \(\phi\in C_c^\infty(\R^n)\) and \(|t|\) small enough that
\(u+t\phi\) remains admissible, the divergence-free identity for
\(T_{k-1}(-D^2u)\) gives
\[
 \left.\frac{\dd}{\dd t}\right|_{t=0}
 \mathcal E_k[u+t\phi]
 =(k+1)\int_{\R^n}\phi\,\sigma_k(-D^2u)\dd x.
\]
Consequently, the Euler--Lagrange equation for the quotient associated
with \eqref{eq:1.3}, after normalization of the
Lagrange multiplier, is exactly \(\sigma_k(-D^2u)=u^{p_*}\).
Chou and Wang~\cite{ChouWang2001} show that the bubbles in \eqref{eq:1.7} attain equality in \eqref{eq:1.3}.

The $\sigma_k$-equation of type \eqref{eq:1.1} also arises in several distinct geometric problems.  For a graph, prescribing
\(\sigma_j(\kappa_1,\ldots,\kappa_n)\), where the \(\kappa_i\) are the
principal curvatures, gives a Weingarten curvature equation whose
coefficients depend on both \(D^2v\) and \(\nabla v\)
\cite{CNSWeingarten1988,Reilly1973,Wang2009}.  In convex geometry, if
\(h\) is the support function of a smooth strictly convex hypersurface
in \(\R^{n+1}\), the eigenvalues of
\(\nabla_{\mathbb S^n}^2h+h\,g_{\mathbb S^n}\)
are its principal radii of curvature.  The
Christoffel--Minkowski problem therefore leads to spherical Hessian
equations of the form
\(\sigma_j\bigl(\nabla_{\mathbb S^n}^2h +h\,g_{\mathbb S^n}\bigr)=f\)
\cite{GuanMa2003,GuanMaZhou2006}. There is a second, equally important application in conformal
geometry.  If
\(A_g=\frac1{n-2}\left(\operatorname{Ric}_g -\frac{R_g}{2(n-1)}g\right)\)
is the Schouten tensor and \(\widehat g=e^{-2w}g\), then
\[
 A_{\widehat g}
 =A_g+\nabla_g^2w+\dd w\otimes\dd w
   -\frac12|\nabla_gw|_g^2g.
\]
Prescribing \(\sigma_k(\widehat g^{-1}A_{\widehat g})\), and in
particular requiring it to be constant, gives the
\(\sigma_k\)-Yamabe equation.  This extends the classical Yamabe
problem, has important consequences for curvature and topology, and
has its own Liouville and singularity theory
\cite{CGY2002,HanLiTeixeira2010,Li2006,
LiLi2003,LiLi2005,ShengTrudingerWang2007,Viaclovsky2000}.

When \(k=1\), equation \eqref{eq:1.1} becomes the usual second order Lane--Emden equation
\begin{equation}\label{eq:1.4}
    -\Delta u=u^p\qquad\text{in }\R^n.
\end{equation}
The semilinear Lane--Emden equation \eqref{eq:1.4} models many phenomena in mathematical physics and astrophysics. Consider a Newtonian self-gravitating fluid of density \(\rho\),
pressure \(P\), and gravitational potential \(\Phi\), in hydrostatic
equilibrium:
\(\nabla P=-\rho\nabla\Phi, \qquad \Delta\Phi=4\pi G\rho\).
For the polytropic equation of state
\(P=K\rho^{1+1/p}\), the specific enthalpy is
\(h=(p+1)K\rho^{1/p}\).  Since
\(\nabla(h+\Phi)=0\), it follows that
\(-\Delta h =4\pi G\bigl((p+1)K\bigr)^{-p}h^p\).
For the second order Lane-Emden equation, the distinction between the subcritical,
critical and supercritical regimes is classical.  Gidas and
Spruck~\cite{GS1981} proved nonexistence of positive solutions for
\(1\leq p<\frac{n+2}{n-2}\); Caffarelli, Gidas and
Spruck~\cite{CGS1989} classified all the positive solutions for $p=\frac{n+2}{n-2}$; and the
critical bubbles are the extremal functions in the sharp Sobolev
inequality of Aubin and Talenti~\cite{Aubin1976,Talenti1976}.  Serrin
and Zou~\cite{SZ2002} developed the corresponding Cauchy--Liouville
theory. For more literature on semilinear and quasilinear Lane-Emden type elliptic equations, refer to \cite{CDL,CDQ0,CDQ1,CFR,DDGL2026,DDGL,DGHP,DLL,DQ0,DQ1,DQ2,DMMS,GG,GM,Ou,PQS,Sciunzi,SunWang,Vetois,WX} and the references therein. The \(k\)-Hessian equation \eqref{eq:1.1} is the fully nonlinear counterpart for the second order Lane-Emden equations. For the \(k\)-Hessian equation, the scaling exponent was known from the Hessian--Sobolev inequality and the radial
equation. The remaining problem was to determine whether this exponent is also the Liouville threshold without radial symmetry or assumptions at infinity.

\medskip

In this paper, we aim to prove the optimal Liouville theorems and classification results for the \(k\)-Hessian Lane--Emden equation \eqref{eq:1.1} in subcritical and critical cases, by {overcoming} the full nonlinearity of the $k$-Hessian equation, the lack of Kelvin type transforms and conformal invariance.

Phuc and Verbitsky~\cite{PV2008} proved nonexistence of positive solutions to \eqref{eq:1.1} for \(k<p\leq p_-\) by making use of advances in
potential theory and PDE due to Kilpel\"ainen and Mal\'y \cite{KM}, Trudinger and Wang \cite{TW1997,TW1999,TW2002}, and Labutin \cite{Labutin2002}. Subsequently, Ou \cite{Ou2010} covered the range
\(0<p\leq k\) by direct integration by parts, carefully chosen test
functions, and approximation.   Wang and
Lei \cite{WL2019} treated the remaining subcritical range for radial solutions. Consequently, the gap interval
\begin{equation}\label{eq:1.5}
    \frac{nk}{n-2k}=:p_-<p<p_*:=
    \frac{k(n+2)}{n-2k}
\end{equation}
remained open for nearly twenty years. Very recently, Chen, Hu and Wang \cite{CHW2025} proved nonexistence results of positive solutions of inequalities $\sigma_k\left(-\left\{\partial_j(|\nabla u|^{q-2}\partial_{i} u)\right\}\right)\geq u^{p}$. When $q=2$, their nonexistence result still {requires} $p\leq p_-$. For the critical equation \eqref{eq:1.1}, no classification of all nonnegative admissible entire solutions without symmetry or asymptotic assumptions was available, to the best of our knowledge. Fang, Ma and Wei~\cite{FMW2023} explicitly stated the larger range
\(p_-<p\leq p_*\) as open.

\begin{table}[h]
\centering
\includegraphics[width=\textwidth]{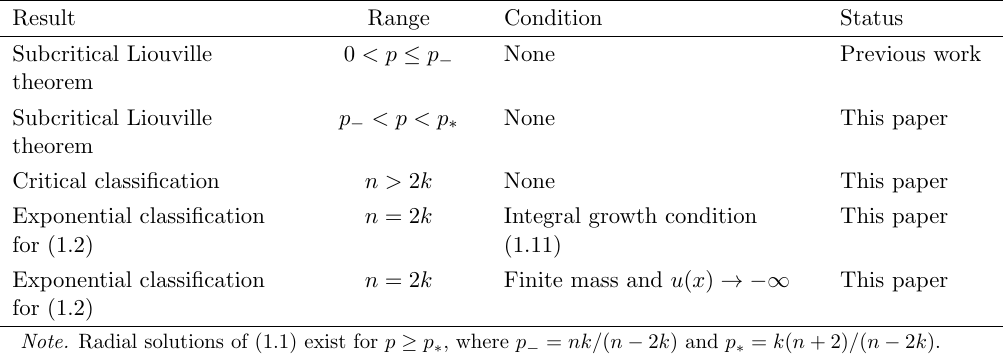}
\caption{Summary of the Liouville and classification results for (\ref{eq:1.1}) and (\ref{eq:1.2}), and comparison with known results.}
\label{Tab:1}
\end{table}

Theorems~\ref{Thm:1.1} and~\ref{Thm:1.2} close the gap in
\eqref{eq:1.5} for both \(C^2\) and locally bounded Hessian-measure solutions.
Corollary~\ref{Cor:1.4}, together with the radial existence result,
identifies \(p_*\) as the sharp Liouville threshold.
Theorem~\ref{Thm:1.6} proves the unconditional classification for
\(p=p_*\) in the full range \(n>2k\).  Theorems~\ref{Thm:1.7}
and~\ref{Thm:1.8} give the corresponding classification results for
the exponential equation when \(n=2k\).

For literature on \(k\)-Hessian theory related to our paper, refer to e.g. \cite{BaoLiLi2014,CNS1985,ChouWang2001,Guan1994,ITW2004,Labutin2002,MaQiu2019,MaZhang2022,PV2006,PV2008,PV2009,Reilly1973,Trudinger1995,
TrudingerIso1997,TrudingerWeak1997,TW1997,TW1999,TW2002,TWweak2002,TWPoincare1998,Tso1989,Tso1990,Verbitsky2015,Wang1994,Wang2009} and the references therein. For $k$-Hessian equation and its applications in geometric analysis and differential geometry, see e.g. \cite{CGM2007,GuanMa2003,GuanMaZhou2006,GuanRenWang2015} etc. For conformally invariant fully nonlinear equations, cf. \cite{CGY2003,CLL,LiLi2003,LiLi2005,Li2006}.

\subsection{Main results}
Our first result establishes the unconditional Liouville theorem in the full gap interval $(p_-,p_*)$ in \eqref{eq:1.5}.
\begin{theorem}\label{Thm:1.1}
Let \(2\leq k<\frac n2\) and
\(\frac{nk}{n-2k}<p<\frac{k(n+2)}{n-2k}\).
If \(u\in C^2(\R^n)\) is a nonnegative solution to \eqref{eq:1.1}, then $u\equiv0$ in $\R^{n}$.
\end{theorem}

\medskip

We next define all the notions used in the weak formulation.  Let
\(\Omega\subset\R^n\) be open.  An upper-semicontinuous function
\(v:\Omega\to[-\infty,\infty)\), not identically \(-\infty\) on any
connected component, is called \(k\)-convex in \(\Omega\) if, whenever
\(\phi\in C^2(\Omega)\) touches \(v\) from above at \(x\in\Omega\),
\(D^2\phi(x)\in\overline{\Gamma_k}\).
Here ``touches from above'' is understood in the standard viscosity
test-function sense~\cite{CIL1992}; the equation itself is imposed below
as an equality of Hessian measures.
We denote this class by \(\Phi^k(\Omega)\).  For
\(v\in\Phi^k(\Omega)\), the Trudinger--Wang \(k\)-Hessian measure is
denoted by \(\mu_k[v]\)~\cite{TW1997,TW1999,TW2002}.  It is the Borel
measure which agrees with the classical density for smooth functions,
\[
    \mu_k[v]=\sigma_k(D^2v)\,\dd x
    \qquad\text{when }v\in C^2(\Omega),
\]
and is weakly continuous under local convergence of \(k\)-convex
functions: if \(v_j,v\in\Phi^k(\Omega)\) and \(v_j\to v\) locally in
measure, then \(\mu_k[v_j]\rightharpoonup\mu_k[v]\) weakly on compact
subsets~\cite{TW1999}, which means that
\[
    \int_\Omega \psi\,\dd\mu_k[v_j]
    \longrightarrow
    \int_\Omega \psi\,\dd\mu_k[v]
    \qquad\text{for every }\psi\in C_c(\Omega).
\]
Thus
\(\mu_k[-u]=u^p\,\dd x\) means equality of Borel measures, and \(\dd x\)
denotes Lebesgue measure.  The notation
\(L^\infty_{\mathrm{loc}}(\Omega)\) means essential boundedness on every
relatively compact subset of \(\Omega\). The assertion
\(-u\in\Phi^k(\Omega)\) is understood for the corresponding
upper-semicontinuous representative.

By solving local Dirichlet problems with smooth right-hand sides, we
extend Theorem~\ref{Thm:1.1} to locally bounded Hessian-measure solutions.

\begin{theorem}\label{Thm:1.2}
Let \(2\leq k<\frac n2\) and
\(\frac{nk}{n-2k}<p<\frac{k(n+2)}{n-2k}\).
If $u$ is a nonnegative solution to
\begin{equation}\label{eq:1.6}
  u\in L^\infty_{\mathrm{loc}}(\R^n),\qquad
    -u\in\Phi^k(\R^n),\qquad
    \mu_k[-u]=u^p\,\dd x
    \quad\text{in }\R^n,
\end{equation}
then $u\equiv0$ in $\R^{n}$.
\end{theorem}

\begin{remark}
    The local boundedness assumption in Theorem~\ref{Thm:1.2} cannot be
    removed without another local control. Indeed,
    \cite[Remark~2.9]{PV2008} gives a singular admissible solution of
    \eqref{eq:1.1}.
\end{remark}

{
For a nontrivial weak solution in the class above, local boundedness makes
the density \(u^p\) locally bounded.  The Hessian-measure regularity theorem
\cite[Theorem~9.2 and Corollary~9.1]{Wang2009} therefore provides a locally
H\"older continuous representative.  Since \(-u\) is \(k\)-convex, it is
subharmonic; hence \(u\) is superharmonic, and the strong minimum principle
gives \(u>0\).  For \(0<p\leq p_-=nk/(n-2k)\), the equality
\(\mu_k[-u]=u^p\,\dd x\) implies the Hessian-measure inequality used in
\cite[Theorem~1.4]{Ou2010}, whose notation
\(\Phi_k=\{u:-u\in\Phi^k(\R^n)\}\) is exactly the present solution class.
This proves nonexistence on the entire previously known range; on the subrange
\(k<p\leq p_-\), the same conclusion is also
\cite[Corollary~2.8]{PV2008}.  Theorem~\ref{Thm:1.2} covers
\(p_-<p<p_*\).  For the classical assertion, Lemma~\ref{Lem:2.2} directly
gives positivity and strict admissibility for every nontrivial solution;
one may then use \cite[Theorem~1.1]{Ou2010} on \(0<p\leq p_-\) and
Theorem~\ref{Thm:1.1} on \(p_-<p<p_*\).  Consequently:}
\begin{corollary}\label{Cor:1.4}
Let \(2\leq k<n/2\), and let \(u\) be a nonnegative solution of
\(-u\in\Phi^k(\R^n)\) and
\(\mu_k[-u]=u^p\,\dd x\) in \(\R^n\). Then \(u\equiv0\) in either of
the following cases.
\begin{enumerate}[label=\rm(\roman*)]
 \item \(0<p\leq p_-\).
 \item \(p_-<p<p_*\) and \(u\in L^\infty_{\mathrm{loc}}(\R^n)\).
\end{enumerate}
In particular, every nonnegative \(C^2\) solution of \eqref{eq:1.1}
with \(0<p<p_*\) vanishes identically.
\end{corollary}

\begin{remark}
The exponent in Corollary~\ref{Cor:1.4} is sharp: Wang
and Lei~\cite{WL2019} give positive entire radial
admissible solutions for \(p>p_*\). Theorems~\ref{Thm:1.1} and
\ref{Thm:1.2} and Corollary~\ref{Cor:1.4} close the gap
\(p\in(p_-,p_*)\) and prove the optimal Liouville theorems for
\eqref{eq:1.1} and \eqref{eq:1.6}.
\end{remark}

\medskip

Next, we consider the critical case \(p=p_*\). The radial classification result in \cite{WL2019} shows that the explicit
Hessian--Sobolev extremals
\begin{equation}\label{eq:1.7}
    u(x)
    =
    \bigl(a_0+b_0|x-x_0|^2\bigr)^{-\frac{n-2k}{2k}}
\end{equation}
are the unique radial solutions.  In the literature cited here, we are not
aware of a classification of nonnegative classical entire admissible
solutions of \eqref{eq:1.8} for \(k\geq2\) and \(n>2k\) without symmetry or
an additional condition at infinity.

Our second main result establishes the unconditional classification for
\eqref{eq:1.1} with \(p=p_*\) in the full range \(n>2k\).
\begin{theorem}\label{Thm:1.6}
Let \(k\geq2\) and \(n>2k\).
Suppose \(u\in C^2(\R^n)\) satisfies
\begin{equation}\label{eq:1.8}
    u\geq0,\qquad
    -D^2u\in\overline{\Gamma_k},\qquad
    \sigma_k(-D^2u)=u^{p_*}
    \quad\text{in }\R^n.
\end{equation}
Then, either $u\equiv0$, or \(u\) is the Hessian--Sobolev bubble \eqref{eq:1.7} for some
\(x_0\in\R^n\) and \(a_0,b_0>0\) such that
\begin{equation}\label{eq:1.9}
    \binom nk a_0
    \left(\frac{n-2k}{k}b_0\right)^k=1.
\end{equation}
Conversely, every function in \eqref{eq:1.7} satisfying
\eqref{eq:1.9} solves
\eqref{eq:1.8}.
\end{theorem}

\medskip

The last main result concerns the $\frac{n}{2}$-Hessian Liouville equation \eqref{eq:1.2} in the limiting dimension \(n=2k\). In particular, when $k=1$, Chen and
Li~\cite{ChenLi1991} proved that every $C^{2}$ solution $u$ to the 2D Liouville equation
\(-\Delta u=e^u\quad\text{in }\R^2\)
with finite mass $\int_{\R^2}e^u\dd x<\infty$ is a translated and dilated logarithmic bubble. This theorem is the two-dimensional exponential counterpart of the
critical Caffarelli--Gidas--Spruck classification theorem in \cite{CGS1989}. Gui and Moradifam introduced the sphere covering inequality and applied it to investigate Moser--Trudinger inequalities, Liouville type equations and mean-field equations, see \cite{GM,GHM}. {There are} substantial classification results for Liouville type equations with exponential nonlinearities in the semilinear and quasilinear settings, see \cite{CDQ1,ChenLi1991,DDGL,DGHP,DQ1,DQ2,Esposito2018,GM,GHM,Lin,Martinazzi,WX} and the references therein.
For classification of solutions to semilinear equations on Heisenberg group or CR manifolds via Jerison-Lee identities
and invariant tensor techniques, see \cite{MO,MOW} and the references therein.

We prove the following classification result for \eqref{eq:1.2} under {an} integral growth condition instead of the finite mass condition.
\begin{theorem}\label{Thm:1.7}
Let \(k\geq2\), $n=2k$, and let \(u\in C^2(\R^{n})\) satisfy
\begin{equation}\label{eq:1.10}
    -D^2u\in\Gamma_k,
    \qquad
    \sigma_k(-D^2u)=e^u
    \quad\text{in }\R^{n}.
\end{equation}
Suppose that, for some \(C>0\) and some sequence \(R_j\to\infty\),
\begin{equation}\label{eq:1.11}
    \int_{B_{R_j}}e^{-\frac{u}{k+1}}\dd x
    \leq CR_j^{n+2}.
\end{equation}
Then
\begin{equation}\label{eq:1.12}
    u(x)=
    \log\!\left(
       2^k(k+1)^k\binom{2k}{k}\lambda^k
    \right)
    -(k+1)\log\bigl(1+\lambda|x-x_0|^2\bigr)
\end{equation}
for some \(x_0\in\R^{n}\) and \(\lambda>0\). Conversely, every function of the form \eqref{eq:1.12} is an
admissible solution of \eqref{eq:1.10}.  Moreover,
\begin{equation}\label{eq:1.13}
    \int_{\R^{n}}e^u\dd x
    =
    \frac{2^k(k+1)^k\binom{2k}{k}\pi^k}{k!}.
\end{equation}
\end{theorem}

{No} finiteness assumption on \(\int_{\R^{2k}}e^u\dd x\) is made in Theorem~\ref{Thm:1.7}. The quantitative identity for the total mass in \eqref{eq:1.13} is a conclusion.

\medskip

Under a more natural finite-mass condition with a proper asymptotic assumption \(u(x)\to-\infty\) as \(|x|\to\infty\), we can prove the following classification result.
\begin{theorem}\label{Thm:1.8}
Let \(k\geq2\), \(n=2k\), and let \(u\in C^2(\R^n)\) satisfy
\eqref{eq:1.10}. Suppose that
\begin{equation}\label{eq:1.14}
    \int_{\R^n}e^u\dd x<\infty,
    \qquad
    \lim_{|x|\to\infty}u(x)=-\infty.
\end{equation}
Then \(u\) has the form \eqref{eq:1.12} for some \(x_0\in\R^n\) and
\(\lambda>0\), and \eqref{eq:1.13} holds.
\end{theorem}

\subsection{The crucial difficulties, novelty and key ingredients in our proof}

The substantial difficulties to establish {Liouville theorems and classification results} in the range $p\in (p_-,p_*]$ are highlighted by the fact that both $p_-$ and $p_*$ can be regarded as critical {exponents} (when $k=1$, $p_-=\frac{n}{n-2}$ is the Serrin critical exponent, while $p_*$ is the well-known Sobolev critical exponent). Furthermore, since {the} $k$-Hessian equation \eqref{eq:1.1} is fully nonlinear, we have the following crucial difficulties:
\begin{itemize}
  \item Full nonlinearity of the \(k\)-Hessian equation \eqref{eq:1.1}:
  the standard moving-plane and moving-sphere arguments do not directly
  preserve both the equation and the G\aa rding cone. The vector field built
  only from \(T_{k-1}(A)\) does not provide the coercive term needed below.
  \item Absence of Kelvin type transforms: Kelvin type transforms preserve neither \(\sigma_k(-D^2u)=u^p\) nor the admissibility condition \(-D^2u\in\Gamma_k\), so we cannot gain any decay at infinity.
  \item Lack of conformal invariance: We cannot directly apply the moving-sphere
  arguments for conformally invariant fully nonlinear equations in
  \cite{CGY2003,CLL,LiLi2003,LiLi2005,Li2006}.
  \item The exponent $p_-$ is the threshold for previous arguments: {the Liouville results in \cite{PV2008,Ou2010} and the recent nonexistence result of Chen, Hu and Wang \cite{CHW2025} require $p\leq p_-$, while neither Liouville results for \(p\in(p_-,p_*)\) nor classification without symmetry assumptions for \(p=p_*\) were known.} Fang, Ma and Wei~\cite{FMW2023} explicitly stated the range \(p_-<p\leq p_*\) as open.
\end{itemize}

\medskip

In order to overcome the above crucial difficulties, we introduce the ``Crucial vector field $\mathcal{J}$+Coercivity of $\operatorname{div}\mathcal{J}$+Quantitative Newton-Maclaurin inequality for matrix+cutoff integral estimates/divergence free of $\mathcal{J}$ $\Rightarrow$ {Rigidity results}" framework, which consists {of} the following novelty and key ingredients:

\medskip

\noindent\textbf{Crucial vector field $\mathcal{J}$ and coercivity of $\operatorname{div}\mathcal{J}$:} In {the} subcritical case $p<p_*$, we need {a} decay estimate on {the} cutoff integral $\int_{\R^n}u^P\chi^\ell\dd x$ so as to derive the Liouville theorem. By Lemma \ref{Lem:3.1}, if we use the cutoff integral estimate \eqref{eq:3.4} for $\delta>0$ directly, then the decay estimate for the term $R^{-2k}
    \int_{\R^n} u^{k-\delta}\chi^{\theta-2k}\dd x$ on the right-hand side requires $p\leq p_-$. In order to break the restriction $p\leq p_-$, we prove a new cutoff integral estimate \eqref{eq:3.5} for $\delta<0$. Unfortunately, the cutoff integral term $\int_{\R^n} u^{p-\delta}\chi^\theta\dd x$ appears on the right-hand side of \eqref{eq:3.5} and hence loses control. We need to introduce some new ideas and methods. First, by testing the equation with {a} suitable function and integrating by parts, we control the cutoff integral $\int_{\R^n}u^P\chi^\ell\dd x$ by $\int_{\R^n}u^{M+p-1}|\nabla u|^2\chi^\ell\dd x$. Then, our goal is to estimate $\int_{\R^n}u^{M+p-1}|\nabla u|^2\chi^\ell\dd x$. To this end, we will introduce a crucial vector field $\mathcal{J}$ and show the coercivity of $\operatorname{div}\mathcal{J}$, namely,
    \begin{equation}\label{eq:1.15}
      \operatorname{div}\mathcal{J}\geq Cu^{M+p-1}|\nabla u|^2.
    \end{equation}
    The term \(-{\tau}u^{M-1}|\nabla u|^2T_{k-1}(A)\nabla u\) appears in the vector field. Its divergence
    contains precisely $u^{M+p-1}|\nabla u|^2$, but it also contains \(\nabla u^TL_k(A)\nabla u\). The divergence of \(-u^ML_k(A)\nabla u\) in \eqref{eq:3.38} supplies \(\tr(L_k(A)A)\). We will define
    $$\mathcal{J}:=-\tau u^{M-1}|\nabla u|^2T_{k-1}(A)\nabla u- u^ML_k(A)\nabla u,$$
    then $\operatorname{div}\mathcal{J}$ can control $u^{M+p-1}|\nabla u|^2$ pointwise, provided that \(\tr(L_k(A)A)\) {can} control \(\nabla u^TL_k(A)\nabla u\). As to the critical case $p=p_*$, inspired by {the} subcritical case, we also introduce the vector field $\mathcal{J}$. We will find that the constant $C=0$ in \eqref{eq:1.15} for the critical case, so we need to show that \(\operatorname{div}\mathcal{J}\geq0\).

\medskip

\noindent\textbf{Quantitative Newton-Maclaurin inequality for matrix:}
In order to control \(\nabla u^TL_k(A)\nabla u\) by \(\tr(L_k(A)A)\), we will prove the quantitative Newton-Maclaurin inequality for matrix (see \eqref{eq:2.12}), which is crucial in our proof. For \(A\in\Gamma_k\), define \(L_k(A)=\frac{n-k}{n}\sigma_k(A)\Id-T_k(A)\).
The classical scalar Newton--Maclaurin inequality (see Lemma \ref{Lem:2.1}) implies
\begin{equation}
    \tr(L_k(A)A)=\frac{n-k}{n}\sigma_1(A)\sigma_k(A)
    -(k+1)\sigma_{k+1}(A)\geq0.
\end{equation}
We will prove the following key matrix inequality (see \eqref{eq:2.12}) that will be used frequently throughout our paper:
\begin{equation}\label{eq:1.16}
    L_k(A)^2
    \preceq
    \frac{n-k}{n}
    \left[
       \frac{n-k}{n}\sigma_1(A)\sigma_k(A)
       -(k+1)\sigma_{k+1}(A)
    \right]T_{k-1}(A).
\end{equation}
Since \(T_{k-1}(A)>0\), inequality~\eqref{eq:1.16} is equivalent to the directional lower bound
\[
 \tr(L_k(A)A)=\frac{n-k}{n}\sigma_1(A)\sigma_k(A)
 -(k+1)\sigma_{k+1}(A)
 \geq
 \frac{n}{n-k}
 \frac{|L_k(A)\xi|^2}
      {\xi^TT_{k-1}(A)\xi}
 \qquad(\xi\ne0).
\]
Thus inequality~\eqref{eq:1.16} can be regarded as {a} quantitative Newton-Maclaurin inequality for matrix. By using the key matrix inequality \eqref{eq:1.16}, we can control \(\nabla u^TL_k(A)\nabla u\) by \(\tr(L_k(A)A)\) in the proof of Proposition \ref{Prop:3.3}, and hence obtain $\operatorname{div}\mathcal{J}\geq Cu^{M+p-1}|\nabla u|^2$. By using the key matrix inequality \eqref{eq:1.16} again to estimate the integral $\int_{\R^n}\operatorname{div}\mathcal{J}\chi^\ell\dd x$, we can derive \(\int_{\R^n}u^{M+p-1}|\nabla u|^2\chi^\ell\dd x\leq CR^{-2}\Lambda_1\). Consequently, we can apply Lemma \ref{Lem:3.1} even if $\delta<0$ to obtain \[\int_{\R^n}u^{M+p-1}|\nabla u|^2\chi^\ell\dd x
    \leq
    CR^{-2}
    \int_{\R^n}
    u^{M+p+1}\chi^{\ell-2}\dd x
    +
    CR^{-2k-2}
    \int_{\R^n}
    u^{M+k+1}\chi^{\ell-2k-2}\dd x.\]
Therefore, we can obtain the upper bound estimate for $\int_{\R^n}u^P\chi^\ell\dd x$ in Proposition \ref{Prop:3.3}. As to the critical case $p=p_*$, by using the key matrix inequality \eqref{eq:1.16}, we can also show that $\operatorname{div}\mathcal{J}\geq0$.

\medskip

\noindent\textbf{Cutoff integral estimates/divergence free of $\mathcal{J}$ $\Rightarrow$ {Liouville theorems/classification results}:}
In {the} subcritical case $p<p_*$, by the cutoff integral estimate for $\int_{\R^n}u^P\chi^\ell\dd x$ in Proposition \ref{Prop:3.3}, we can show that
\[\int_{B_R}u^P\dd x\leq C R^{-\theta}\rightarrow0, \qquad \text{as} \,\, R\rightarrow+\infty.\]
This leads to $u\equiv0$, and hence {the} Liouville theorem holds. As to the critical case $p=p_*$, we first prove that $\operatorname{div}\mathcal{J}\equiv0$.
To this end, by using the key matrix inequality \eqref{eq:1.16}, we prove that
\begin{equation*}
    |\mathcal J\cdot Y|^2
    \leq
    \frac{n-k}{n}\,
    u^{-2k/(n-2k)}|\nabla u|^2
    \operatorname{div}\mathcal J\,
    Y^TT_{k-1}(A)Y,
\end{equation*}
and
\begin{equation}
    \int_{\R^n}\eta^{\ell+2}
       \operatorname{div}\mathcal J\dd x
    \leq
    C
    \int_{\R^n}
    \eta^\ell|\nabla\eta|^2
    \sigma_{k-1}(A)|\nabla u|^2u^{-2k/(n-2k)}\dd x.
\end{equation}
By Lemma \ref{Lem:3.1} (when $n=4k$, $\delta=0$, Lemma \ref{Lem:3.1} cannot be applied, we will discuss it later), we get
\(\int_{\R^n} \sigma_{k-1}(A)|\nabla u|^2 u^{-\delta-1}\chi^{4k}\dd x \leq CR^{\mu}\),
Since $\mu<2$, it follows that
\[\int_{B_R}\operatorname{div}\mathcal J\dd x\leq C R^{-\theta}\rightarrow0, \qquad \text{as} \,\, R\rightarrow+\infty,\]
and hence $\operatorname{div}\mathcal{J}\equiv0$. Then, we prove in Lemma \ref{Lem:2.5} the characterization of equality for the Quantitative Newton-Maclaurin inequality for matrix \eqref{eq:1.16}. Then, by using the key matrix inequality \eqref{eq:1.16} and Lemma \ref{Lem:2.5}, we derive that \[\operatorname{div}\mathcal{J}\equiv0 \qquad \Rightarrow \qquad D^2\left(u^{-2k/(n-2k)}\right)=\lambda \Id,\]
which leads to the classification results.

\medskip

\noindent\textbf{The special dimension $n=4k$ for the critical case $p=p_*$:}
At \(n=4k\), \eqref{eq:5.1} implies \(M=-1\), so
\(\delta=-M-1=0\). Hence Lemma \ref{Lem:3.1} is not available for us to show $\operatorname{div}\mathcal{J}\equiv0$. In order to overcome such difficulty, our key idea is to replace the vector field $\mathcal J$ by $h^{-\varepsilon} \mathcal J$  with the weighted function $h$ properly chosen depending on $u$. Then we have
\begin{equation}
    \operatorname{div}(h^{-\varepsilon}\mathcal J)
    =h^{-\varepsilon}
    \bigl(\operatorname{div}\mathcal J
       -\varepsilon\mathcal J\cdot\nabla\log h\bigr).
\end{equation}
In order to control $h^{-\varepsilon}\operatorname{div}\mathcal J$ by $ \operatorname{div}(h^{-\varepsilon}\mathcal J)$, we need to show the following crucial estimate (see Proposition \ref{Prop:5.5}):
\begin{equation}
    |\mathcal J\cdot\nabla\log h|\leq C_k\operatorname{div}\mathcal J
    \qquad\text{in }\R^n.
\end{equation}
Therefore, we derive
\begin{equation}
    \operatorname{div}(h^{-\varepsilon}\mathcal J)
    \geq\frac12h^{-\varepsilon}\operatorname{div}\mathcal J,
\end{equation}
and hence
\begin{equation}\begin{split}
    \int\chi^{8k+2}h^{-\varepsilon}
       \operatorname{div}\mathcal J\dd x
   \leq CR^{-2}\int\chi^{8k}\sigma_{k-1}(A)|\nabla u|^2
       u^{-1-\varepsilon/k}\dd x.
\end{split}
\end{equation}
Then, by using Lemma \ref{Lem:3.1}, we get
\(\int\chi^{8k}\sigma_{k-1}(A)|\nabla u|^2 u^{-1-\varepsilon/k}\dd x\leq CR^{\mu}\),
Since $\mu<2$, it follows that
\[\int_{B_R}h^{-\varepsilon}\operatorname{div}\mathcal J\dd x\leq C R^{-\theta}\rightarrow0, \qquad \text{as} \,\, R\rightarrow+\infty,\]
and hence $\operatorname{div}\mathcal{J}\equiv0$.

\medskip

\noindent\textbf{The case \(n>4k\) for the critical case $p=p_*$:}
We first apply this framework to prove the classification result under the integral growth condition
\eqref{eq:6.3}. It remains to prove
\eqref{eq:6.3} directly from the equation. To this end, we will introduce some new tools and methods. Lemmas~\ref{Lem:6.3}, \ref{Lem:6.4}, and
\ref{Lem:6.5}
provide the new \(k\)-Hessian compactness and Newton tensor argument. The construction of the $P$-function is crucial, which is inspired by \eqref{eq:5.18}. We also developed a B\^ocher expansion for solutions with single atomic Hessian measure in
Theorem~\ref{Thm:6.7}, and the stability estimate for the Pohozaev
integral in Lemma~\ref{Lem:6.9}. Theorem~\ref{Thm:6.7}, Lemma~\ref{Lem:6.9}, and
Proposition~\ref{Prop:6.10} provide the new isolated-pole
argument. The \(k\)-Hessian compactness, Newton tensor argument, B\^ocher expansion, stability estimate for the Pohozaev
integral and isolated-pole
argument are new tools introduced by us. We also use the ideas of point selection and
the choice of first-contact radius from \cite{Zhang} in Propositions \ref{Prop:6.6}, \ref{Prop:6.13} and \ref{Prop:6.14}.

\noindent\textbf{The case \(n=2k\) for the $\frac{n}{2}$-Hessian Liouville equation \eqref{eq:1.2}:} In Section 7, we apply this framework to prove the classification result under the integral growth condition
\eqref{eq:1.11} and hence derive the classification result in Theorem \ref{Thm:1.7}. In Section 8, our goal is to prove \eqref{eq:1.11} directly from the equation under a more natural finite-mass condition with a proper asymptotic assumption $\lim\limits_{|x|\rightarrow\infty}u(x)=-\infty$. To this end, take \(v=-u/(k+1)\). The proper asymptotic assumption gives \(v(x)\to+\infty\) as $|x|\rightarrow+\infty$, so every sublevel set of \(v\) is bounded. The
finite-mass assumption implies $\int_{\R^n}\sigma_k(D^2v)\dd x<\infty$. The main difficulty is that $\int_{\R^n}\sigma_k(D^2v)\dd x<\infty$ does not directly imply the
integral growth condition \eqref{eq:1.11}. In order to overcome this difficulty, our key idea is to use the level set analysis, connect the
\(\frac n2\)-Hessian operator with the \(n\)-Laplacian through the regular level sets of \(v\) and derive a global estimate for the
\(n\)-Laplacian \(\Delta_n v=\operatorname{div}(|\nabla v|^{n-2}\nabla v)\) from the finite total \(k\)-Hessian mass. For a regular value \(t\), set \(\Sigma_t=\{v=t\}\) and
\begin{equation}
 P_j(t)=\int_{\Sigma_t}|\nabla v|^{n-2j+1}
 \sigma_{j-1}\bigl(D^2v|_{T\Sigma_t}\bigr)\dd S,
 \qquad 1\leq j\leq k.
\end{equation}
By combining the Newton-tensor identities of
Reilly~\cite[Section~2]{Reilly1973}, the inequalities in
Wang~\cite[\S2.5]{Wang2009}, the level-set inequality of Wu and
Yi~\cite[Corollary~1.3]{WuYi}, and iterating from \(j=k\) to \(j=1\), we derive
\begin{equation*}
 \sup_{t\ {\rm regular}}P_j(t)
 \leq
 C_{k,j}\left(
 \int_{\R^n}\sigma_k(D^2v)\dd x
 \right)^{\frac{n-j}{k}},
 \qquad 1\leq j\leq k.
\end{equation*}
Since $P_1(t)=\int_{\{v<t\}}\Delta_n v\dd x$, we obtain
\begin{equation*}
 \Delta_n v\geq0,\qquad
 \int_{\R^n}\Delta_n v\dd x
 \leq
 C_k\left(
 \int_{\R^n}\sigma_k(D^2v)\dd x
 \right)^{2-\frac1k}.
\end{equation*}
This estimate provides an interesting and surprising connection between the \(\frac{n}{2}\)-Hessian operator and the
\(n\)-Laplacian (see Theorem \ref{Thm:8.1}).

The logarithmic estimate of Kilpel\"ainen and
Zhong~\cite[Theorem~1.2]{KilpelainenZhong} provides the upper bound on the growth
of \(v\).  After rescaling, Theorems~\ref{Thm:8.2}
and~\ref{Thm:8.3} yield that, for some \(\beta>0\),
\begin{equation*}
 v(R_i x)-c_i\longrightarrow\beta\log|x|
 \quad\hbox{in }C^1_{\mathrm{loc}}(\R^n\setminus\{0\}),
\end{equation*}
together with
\begin{equation*}
 \int_{\R^n}\Delta_n v\dd x
 =|\mathbb S^{n-1}|\beta^{n-1},
 \qquad
 \int_{\R^n}\sigma_k(D^2v)\dd x
 =\frac{|\mathbb S^{n-1}|}{k}
 \binom{n-1}{k-1}\beta^k.
\end{equation*}
Finally, applying \eqref{eq:8.14}--\eqref{eq:8.15} on the bounded
connected components of the rescaled sublevel sets gives
\(\beta\int_{\R^n}\sigma_k(D^2v)\dd x
=2\int_{\R^n}\sigma_k(D^2v)\dd x\).
Hence \(\beta=2\). Then, \eqref{eq:8.7} implies that
\(\int_{B_R}e^{-u/(k+1)}\dd x
=\int_{B_R}e^v\dd x\leq CR^{n+2}\),
which is \eqref{eq:1.11}. By Theorem~\ref{Thm:1.7}, we derive the
finite-mass classification in Theorem \ref{Thm:1.8}.

\subsection{Notation}

We write \(B_R(x)\) for the Euclidean ball of radius \(R\) centered at
\(x\), and \(B_R=B_R(0)\).  {We write
\(\partial_j:=\partial_{x_j}\).}  The identity matrix is \(\Id\), and
\(\tr B=\sum_iB_{ii}\) denotes the matrix trace.  For a symmetric
matrix \(B\), its bilinear form is written \(X^TBY\).  The relation
\(B\preceq C\) means that \(C-B\) is positive semidefinite.
In the differential arguments, \(A\) denotes an arbitrary
matrix in \(\Gamma_k\).  The notation \(\Phi^k(\Omega)\) and
\(\mu_k[v]\) refers to the Trudinger--Wang class of \(k\)-convex
functions and its Hessian measure.

\medskip

The rest of the paper is organized as follows.  In
Section~\ref{Sec:2}, we prove
\eqref{eq:1.16} and Lemma~\ref{Lem:2.5}.  Proposition~\ref{Prop:2.7}
is used in Section~\ref{Sec:8}, and the Pohozaev identity is used in
Sections~\ref{Sec:6} and~\ref{Sec:8}.  In
Section~\ref{Sec:3}, Lemma~\ref{Lem:3.1} and
Proposition~\ref{Prop:3.3} lead to
Theorem~\ref{Thm:1.1}. Section~\ref{Sec:4} proves
Theorem~\ref{Thm:1.2}.  Section~\ref{Sec:5} proves
the divergence-rigidity part of Theorem~\ref{Thm:1.6}.
Section~\ref{Sec:6}
first proves the classification under \eqref{eq:6.3}, and then verifies
\eqref{eq:6.3} and completes the proof of Theorem~\ref{Thm:1.6}.  Section
\ref{Sec:7} proves
Theorem~\ref{Thm:1.7}.
{Section~\ref{Sec:8} proves Theorem~\ref{Thm:1.8}.}

\begin{remark}
Every nonzero classical solution considered below is strictly admissible.
Thus \(\sigma_k^{1/k}\) is locally uniformly elliptic on its Hessian
range, and the interior Evans--Krylov and Schauder estimates give
\(u\in C^\infty_{\mathrm{loc}}(\R^n)\); see
\cite{Evans1982,Krylov1982,CaffarelliCabre}.  Hence the assumption
\(u\in C^2\) is sufficient for all differential calculations below.
\end{remark}

\section{Newton tensor inequalities and integral identities}
\label{Sec:2}
In this section, we recall the Newton-tensor identities, the Newton-Maclaurin inequalities and prove a crucial matrix inequality \eqref{eq:2.12} used in
\eqref{eq:2.20}, \eqref{eq:3.42} and \eqref{eq:5.6}, and other Newton tensor inequalities.
We also prove Proposition~\ref{Prop:2.7}, used in Section~\ref{Sec:8},
and the Pohozaev identity used in Sections~\ref{Sec:6}
and~\ref{Sec:8}.

\subsection{Newton tensors and a quantitative Newton--Maclaurin inequality for matrix}

We first recall the Newton tensor identities and the Newton--Maclaurin inequalities. We then prove a crucial quantitative Newton--Maclaurin {inequality} in Lemma \ref{Lem:2.4}, which indicates that the multiple of Newton tensor and the Newton deficit {control} the square of a trace-free tensor. Its equality case will be used to prove the classification result in the critical case.

For a symmetric matrix \(A\) and \(0\leq j\leq n-1\), the
\(j\)-th Newton tensor is given by
\begin{equation}\label{eq:2.1}
    T_j(A)
    :=
    \frac{\partial\sigma_{j+1}(A)}{\partial A}
    =
    \sigma_j(A)\Id-\sigma_{j-1}(A)A+\cdots+(-1)^jA^j.
\end{equation}
The matrices \(A,T_0(A),\dots,T_{n-1}(A)\) commute, and
\begin{align}
    T_j(A)
    &=
    \sigma_j(A)\Id-T_{j-1}(A)A,
    \qquad j\geq1,
    \label{eq:2.2}\\
    \tr T_j(A)
    &=
    (n-j)\sigma_j(A),
    \label{eq:2.3}\\
    \tr\bigl(T_j(A)A\bigr)
    &=
    (j+1)\sigma_{j+1}(A).
    \label{eq:2.4}
\end{align}
{
For every \(w\in C^3(\Omega)\), the Newton tensors of its Hessian satisfy
\begin{equation}\label{eq:2.5}
    \partial_i\bigl(T_j(D^2w)\bigr)_{i\ell}=0,
    \qquad 0\leq j\leq n-1.
\end{equation}
This follows from the antisymmetry of the generalized Kronecker symbol;
see, for example, \cite[Section~2]{Reilly1973}.
}
Equations \eqref{eq:2.1}--\eqref{eq:2.4}
are the standard Newton transformation identities; see
\cite[Section~2]{Reilly1973}. If \(\lambda(A)=(\lambda_1,\dots,\lambda_n)\), we write
\(\lambda(A)|i\) for the \((n-1)\)-tuple obtained by deleting
\(\lambda_i\). We regard $\lambda(A)$ and $\lambda(A)|i$ as diagonal matrices. In an eigenbasis of \(A\),
\begin{equation}\label{eq:2.6}
    \bigl(T_j(A)\bigr)_{ii}=\sigma_j(\lambda(A)|i).
\end{equation}

\begin{lemma}
\label{Lem:2.1}
For every \(\mu=(\mu_1,\dots,\mu_n)\in\R^n\) and
\(1\leq j\leq n-1\),
\begin{equation}\label{eq:2.7}
    \sigma_j(\mu)^2
    \geq
    \frac{(j+1)(n-j+1)}{j(n-j)}
    \sigma_{j-1}(\mu)\sigma_{j+1}(\mu).
\end{equation}
Equivalently,
\begin{equation}
    \left(\frac{\sigma_j(\mu)}{\binom nj}\right)^2
    \geq
    \frac{\sigma_{j-1}(\mu)}{\binom n{j-1}}
    \frac{\sigma_{j+1}(\mu)}{\binom n{j+1}}.
\end{equation}
If \(\mu\in\Gamma_m\) for some \(1\leq m\leq n\), then
\begin{equation}\label{eq:2.8}
    \frac{\sigma_1(\mu)}n
    \geq
    \left(\frac{\sigma_2(\mu)}{\binom n2}\right)^{1/2}
    \geq\cdots\geq
    \left(\frac{\sigma_m(\mu)}{\binom nm}\right)^{1/m}.
\end{equation}
\end{lemma}

The first assertion is the Newton's coefficient inequality applied to the polynomial \(\prod_i(t+\mu_i)\); see
\cite[pp.~905--909]{Whiteley1969}.  The inequalities \eqref{eq:2.8} are the Newton--Maclaurin inequalities on
\(\Gamma_m\); see \cite[\S2.5(v)--(vi)]{Wang2009}, based on G\aa rding's hyperbolicity-cone theory~\cite{Garding1959}.

\begin{lemma}
\label{Lem:2.2}
Let \(2\leq k<n/2\) and \(p>0\), and suppose that \(u\) is a
nonnegative entire admissible solution to
\begin{equation}
    -D^2u\in\overline{\Gamma_k},\qquad
    \sigma_k(-D^2u)=u^p
    \quad\text{in }\R^n.
\end{equation}
If $u\geq0$ but $u\not\equiv0$, then $-D^2u\in\Gamma_k$ and $u>0$ in $\R^{n}$.
\end{lemma}
\begin{proof}
Since
\(-D^2u\in\overline{\Gamma_k}\), we have
\(-\Delta u=\sigma_1(-D^2u)\geq0\). Hence \(u\) is superharmonic, and
the strong minimum principle gives \(u>0\) in \(\R^n\). Moreover, the
Newton--Maclaurin inequalities in Lemma \ref{Lem:2.1}, extended to
\(\overline{\Gamma_k}\) by continuity, imply that, for \(1\leq j<k\),
\[
 \left(\frac{\sigma_j(-D^2u)}{\binom nj}\right)^{1/j}
 \geq
 \left(\frac{\sigma_k(-D^2u)}{\binom nk}\right)^{1/k}
 =
 \left(\frac{u^p}{\binom nk}\right)^{1/k}>0.
\]
Thus \(-D^2u\in\Gamma_k\) everywhere in $\R^{n}$.
\end{proof}

From Lemma \ref{Lem:2.2},  we know that, if $u$ is a solution to \eqref{eq:1.1}, then $-D^2u\in\Gamma_k$ and $u>0$ in $\R^{n}$, and hence all the strict ellipticity properties used in our paper are valid.

\begin{corollary}
\label{Cor:2.3}
Let \(A\in\Gamma_k\).  In an eigenbasis, for every \(i\),
\[
    \lambda(A)|i\in\Gamma_{k-1},
    \qquad
    \sigma_j(\lambda(A)|i)>0\quad(0\leq j\leq k-1),
\]
and hence
\(0\prec T_j(A)\qquad(0\leq j\leq k-1)\).
\end{corollary}

Corollary~\ref{Cor:2.3} is the deletion property and ellipticity
of the \(k\)-Hessian operator; see \cite[(2.4)--(2.7)]{Wang2009} and
\cite[Section~2]{CNS1985}.

For \(A\in\Gamma_k\), we define the trace free tensor
\begin{equation}\label{eq:2.9}
    L_k(A):=\frac{n-k}{n}\sigma_k(A)\Id-T_k(A).
\end{equation}
By \eqref{eq:2.3}, we have \(\tr L_k(A)=0\), and
\begin{equation}\label{eq:2.10}
\begin{split}
    \tr\bigl(L_k(A)A\bigr)
    =
    \frac{n-k}{n}\sigma_1(A)\sigma_k(A)
      -(k+1)\sigma_{k+1}(A)
    \geq0.
\end{split}
\end{equation}
The inequality in \eqref{eq:2.10} follows from the
Newton--Maclaurin inequalities in \eqref{eq:2.8}. From
\eqref{eq:2.2}, we have
\begin{equation}\label{eq:2.11}
    T_{k-1}(A)A
    =
    L_k(A)+\frac{k}{n}\sigma_k(A)\Id.
\end{equation}

We can prove the following crucial quantitative Newton--Maclaurin inequality in Lemma \ref{Lem:2.4}, which indicates that the multiple of Newton tensor and the Newton deficit {control} the square of a trace-free tensor. Its equality case will be used to prove the classification result in the critical case.
\begin{lemma}[Quantitative Newton--Maclaurin inequality for matrix]
\label{Lem:2.4}
For every \(A\in\Gamma_k\),
\begin{equation}\label{eq:2.12}
    L_k(A)^2
    \preceq
    \frac{n-k}{n}\,
    \tr\bigl(L_k(A)A\bigr)T_{k-1}(A).
\end{equation}
\end{lemma}

\begin{proof}
Diagonalize \(A\) and fix an index \(i\).  Since all matrices in
\eqref{eq:2.12} are polynomials in \(A\), it is enough to prove
the corresponding inequality in the \(i\)-th eigendirection.
Expanding the elementary symmetric functions with respect to the
\(i\)-th eigenvalue gives the following identity
\begin{align}
&\frac{n-k}{n}\,
 \tr\bigl(L_k(A)A\bigr)\sigma_{k-1}(\lambda(A)|i)
 -\left(
     \frac{n-k}{n}\sigma_k(A)-\sigma_k(\lambda(A)|i)
  \right)^2
\notag\\
={}&
 \frac{(n-k)\sigma_k(A)}{n^2}
 \Bigl(
   (n-k)\sigma_1(\lambda(A)|i)\sigma_{k-1}(\lambda(A)|i)
   -k(n-1)\sigma_k(\lambda(A)|i)
 \Bigr)
\notag\\
&+\frac1n
 \Bigl(
   k(n-k-1)\sigma_k(\lambda(A)|i)^2
   -(n-k)(k+1)
     \sigma_{k+1}(\lambda(A)|i)\sigma_{k-1}(\lambda(A)|i)
 \Bigr).
\label{eq:2.13}
\end{align}

Indeed, using the following basic identities
\[
\begin{aligned}
 \sigma_1(A)
 &=\lambda_i+\sigma_1(\lambda(A)|i),\\
 \sigma_k(A)
 &=\lambda_i\sigma_{k-1}(\lambda(A)|i)
   +\sigma_k(\lambda(A)|i),\\
 \sigma_{k+1}(A)
 &=\lambda_i\sigma_k(\lambda(A)|i)
   +\sigma_{k+1}(\lambda(A)|i),
\end{aligned}
\]
we get that the left-hand side of \eqref{eq:2.13} equals
\begin{align*}
&\frac{(n-k)^2}{n^2}\sigma_k(A)
 \Bigl[
 \sigma_1(\lambda(A)|i)\sigma_{k-1}(\lambda(A)|i)
 -\sigma_k(\lambda(A)|i)
 \Bigr]
\\
&-\frac{(n-k)(k+1)}n
 \Bigl[\sigma_k(A)-\sigma_k(\lambda(A)|i)\Bigr]
 \sigma_k(\lambda(A)|i)
\\
&-\frac{(n-k)(k+1)}n
 \sigma_{k+1}(\lambda(A)|i)
 \sigma_{k-1}(\lambda(A)|i)
\\
&+\frac{2(n-k)}n\sigma_k(A)\sigma_k(\lambda(A)|i)
 -\sigma_k(\lambda(A)|i)^2
\\
={}&\frac{(n-k)\sigma_k(A)}{n^2}
 \Bigl(
 (n-k)\sigma_1(\lambda(A)|i)\sigma_{k-1}(\lambda(A)|i)
 -k(n-1)\sigma_k(\lambda(A)|i)
 \Bigr)
\\
&+\frac1n
 \Bigl(
 k(n-k-1)\sigma_k(\lambda(A)|i)^2
 -(n-k)(k+1)\sigma_{k+1}(\lambda(A)|i)
 \sigma_{k-1}(\lambda(A)|i)
 \Bigr).
\end{align*}
This proves \eqref{eq:2.13}.

Both parentheses on the right of \eqref{eq:2.13} are
nonnegative.  Once this is proved, \(\sigma_k(A)>0\) and
\eqref{eq:2.13} give the \(i\)-th diagonal inequality in
\eqref{eq:2.12}. Since \(i\) is arbitrary, the lemma follows.

We now prove this nonnegativity.  If
\(\sigma_k(\lambda(A)|i)\leq0\), the first parenthesis in
\eqref{eq:2.13} is positive by
Corollary~\ref{Cor:2.3}.  If
\(\sigma_k(\lambda(A)|i)>0\), then \(\lambda(A)|i\in\Gamma_k\). By applying \eqref{eq:2.8} in $n-1$-dimension with $\mu=\lambda(A)|i$ and indices \(1,k-1,k\), we have the following two inequalities:
\begin{equation}\label{eq:2.14}
   \frac{\sigma_1(\lambda(A)|i)}{n-1}
    \geq
    \left(\frac{\sigma_k(\lambda(A)|i)}{\binom {n-1} k}\right)^{\frac{1}{k}}, \qquad \frac{\sigma_{k-1}(\lambda(A)|i)}{\binom {n-1} {k-1}}
    \geq
    \left(\frac{\sigma_k(\lambda(A)|i)}{\binom {n-1} k}\right)^{\frac{k-1}{k}}.
\end{equation}
Thus it follows that
\begin{equation}\label{eq:2.15}
    (n-k)\sigma_1(\lambda(A)|i)\sigma_{k-1}(\lambda(A)|i)
    \geq k(n-1)\sigma_k(\lambda(A)|i).
\end{equation}
By applying \eqref{eq:2.7} in dimension \(n-1\) with
\(j=k\), we deduce that the second parenthesis is nonnegative, that is,
\begin{equation}\label{eq:2.16}
    k(n-k-1)\sigma_k(\lambda(A)|i)^2
    \geq
    (n-k)(k+1)\sigma_{k-1}(\lambda(A)|i)
                         \sigma_{k+1}(\lambda(A)|i).
\end{equation}
Equations \eqref{eq:2.15} and \eqref{eq:2.16} prove
the required nonnegativity in \eqref{eq:2.13}. This finishes our proof of Lemma \ref{Lem:2.4}.
\end{proof}

We next characterize the equality case of \eqref{eq:2.12}, which will be used to prove the classification result in the critical case.
\begin{lemma}\label{Lem:2.5}
If equality in \eqref{eq:2.12} holds in the \(i\)-th
eigen-direction, namely,
\begin{equation}\label{eq:2.17}
\left(
    \frac{n-k}{n}\sigma_k(A)-\sigma_k(\lambda(A)|i)
\right)^2
=
\frac{n-k}{n}\,
\tr\bigl(L_k(A)A\bigr)\sigma_{k-1}(\lambda(A)|i),
\end{equation}
then all entries of \(\lambda(A)|i\) are equal.  In particular,
\begin{equation}\label{eq:2.18}
    \tr\bigl(L_k(A)A\bigr)=0
    \quad\Longrightarrow\quad
    A=
    \left(\frac{\sigma_k(A)}{\binom nk}\right)^{1/k}\Id.
\end{equation}
\end{lemma}

\begin{proof}
Equality in \eqref{eq:2.17} makes the left-hand side
of \eqref{eq:2.13} zero.  Both parentheses on the right-hand side of \eqref{eq:2.13}
are nonnegative, so both of them vanish.  The first one cannot vanish if
\(\sigma_k(\lambda(A)|i)\leq0\); hence
\(\lambda(A)|i\in\Gamma_k\), and equality must hold in the two
Newton--Maclaurin inequalities in \eqref{eq:2.14} used to prove \eqref{eq:2.15}.
Consequently, all the inequalities in \eqref{eq:2.8} in Lemma \ref{Lem:2.1} become equalities, and we have
\begin{equation}
   \frac{\sigma_1(\lambda(A)|i)}{n-1}=\left(\frac{\sigma_2(\lambda(A)|i)}{\binom {n-1} 2}\right)^{\frac{1}{2}}=\cdots=\left(\frac{\sigma_{k-1}(\lambda(A)|i)}{\binom {n-1} {k-1}}\right)^{\frac{1}{k-1}}
    =
    \left(\frac{\sigma_k(\lambda(A)|i)}{\binom {n-1} k}\right)^{\frac{1}{k}},
\end{equation}
in particular,
\begin{equation}\label{eq:2.19}
    \left(
       \frac{\sigma_1(\lambda(A)|i)}{n-1}
    \right)^2
    =
    \frac{\sigma_2(\lambda(A)|i)}{\binom{n-1}{2}}.
\end{equation}
The elementary identity
\begin{equation}
\left(
   \frac{\sigma_1(\lambda(A)|i)}{n-1}
\right)^2
-
\frac{\sigma_2(\lambda(A)|i)}{\binom{n-1}{2}}
=
\frac{1}{(n-1)^2(n-2)}
\sum_{\substack{a<b\\a,b\ne i}}(\lambda_a-\lambda_b)^2
\end{equation}
and \eqref{eq:2.19} show that all entries of the
deleted tuple \(\lambda(A)|i\) are equal.

If \(\tr(L_k(A)A)=0\), then \eqref{eq:2.12} implies
\(L_k(A)=0\).  Thus equality \eqref{eq:2.17} holds for any $1\leq i\leq n$, and hence all entries of the
deleted tuple \(\lambda(A)|i\) are equal for any $1\leq i\leq n$. Applying this with two different choices of \(i\) shows that all eigenvalues of \(A\) coincide.  Since
\(A\in\Gamma_k\), the common
eigenvalue is positive and {equals} $\left(\frac{\sigma_k(A)}{\binom nk}\right)^{1/k}$. Thus \eqref{eq:2.18} holds. This finishes our proof of Lemma \ref{Lem:2.5}.
\end{proof}

We will use the following inequality to estimate the mixed terms in
the divergence calculations.
\begin{corollary}
\label{Cor:2.6}
For every \(A\in\Gamma_k\) and \(X,Y\in\R^n\),
\begin{equation}\label{eq:2.20}
    \bigl|(L_k(A)X)^T Y\bigr|^2
    \leq
    \frac{n-k}{n}\,
    \tr\bigl(L_k(A)A\bigr)|X|^2
    Y^TT_{k-1}(A)Y.
\end{equation}
\end{corollary}

\begin{proof}
Since \(L_k(A)\) is symmetric, one has
\[
    \bigl|(L_k(A)X)^TY\bigr|^2
    =\bigl|X^TL_k(A)Y\bigr|^2
    \leq |X|^2Y^TL_k(A)^2Y.
\]
Now applying \eqref{eq:2.12} gives the desired inequality.
\end{proof}

\subsection{A Newton tensor inequality and the Pohozaev identity}
The Newton--Maclaurin inequalities also yield the following family of directional
inequalities for all Hessian orders \(1\leq j\leq k\).

\begin{proposition}\label{Prop:2.7}
Let {\(1\leq k<n\) and} \(A\in\Gamma_k\). For \(1\leq j\leq k\), set
\begin{equation}\label{eq:2.21}
 m_j=\frac{n(k-j)}{n-k}.
\end{equation}
For every unit vector \(\nu\),  \(B=A|_{\nu^\perp}\) denotes the
restriction of the symmetric bilinear form associated with \(A\) to
\(\nu^\perp\).  Then there holds
\begin{equation}\label{eq:2.22}
 (j+m_j)\sigma_j(A)-m_j\sigma_j(B)\geq0.
\end{equation}

Suppose that \(v\in C^3(\Omega)\) satisfies \(D^2v\in\Gamma_k\).  Then we have
\begin{equation}\label{eq:2.23}
 \operatorname{div}\left(
 |\nabla v|^{m_j}T_{j-1}(D^2v)\nabla v\right)
 =|\nabla v|^{m_j}
 \left[(j+m_j)\sigma_j(D^2v)-m_j\sigma_j(D^2v|_{(\nabla v)^\perp})\right]\geq0,
\end{equation}
 whenever
\(\nabla v\ne0\).  At critical points of \(v\),
\eqref{eq:2.23} is understood by continuous extension.  If \(j<k\),
then on every regular
level \(\Sigma_t=\{v=t\}\) one has
\begin{equation}\label{eq:2.24}
 \left|
 T_{j-1}\bigl(D^2v|_{(\nabla v)^\perp}\bigr)
 \nabla_{\Sigma_t}|\nabla v|
 \right|^2
 \leq
 \frac{\sigma_j\bigl(D^2v|_{(\nabla v)^\perp}\bigr)}{m_j}
 \left[
 (j+m_j)\sigma_j(D^2v)
 -m_j\sigma_j\bigl(D^2v|_{(\nabla v)^\perp}\bigr)
 \right].
\end{equation}
If \(v\) is smooth and its sublevel sets are compact,  define, for
every regular value \(t\),
\begin{equation}\label{eq:2.25}
 P_j(t)=\int_{\Sigma_t}|\nabla v|^{m_j+1}
 \sigma_{j-1}\bigl(D^2v|_{(\nabla v)^\perp}\bigr)\dd S.
\end{equation}
Then \(P_j\) admits a nondecreasing locally absolutely continuous
extension to all \(t\). Moreover, for
almost every \(t\),
\begin{equation}\label{eq:2.26}
 P_j'(t)=\int_{\Sigma_t}|\nabla v|^{m_j-1}
 \left[
 (j+m_j)\sigma_j(D^2v)
 -m_j\sigma_j\bigl(D^2v|_{(\nabla v)^\perp}\bigr)
 \right]\dd S.
\end{equation}
For \(j=k\),  one has
\begin{equation}\label{eq:2.27}
 P_k(t)=k\int_{\{v<t\}}\sigma_k(D^2v)\dd x.
\end{equation}
\end{proposition}

\begin{proof}
Without loss of generality, we diagonalize \(A\), write its eigenvalues as
\(\lambda_1,\ldots,\lambda_n\), and first take
\(\nu\) to be the \(i\)-th eigenvector.  For \(j<k\), we use \( \sigma_j(\lambda)
 =\lambda_i\sigma_{j-1}(\lambda|i)+\sigma_j(\lambda|i)\) to get
\begin{equation}\label{eq:2.28}
 (j+m_j)\sigma_j(\lambda)-m_j\sigma_j(\lambda|i)
 =(j+m_j)\lambda_i\sigma_{j-1}(\lambda|i)
   +j\sigma_j(\lambda|i).
\end{equation}

If \(\lambda_i\geq0\), Corollary~\ref{Cor:2.3} implies that the
right-hand side of \eqref{eq:2.28} is nonnegative.  Suppose now that
\(\lambda_i<0\).  By Corollary~\ref{Cor:2.3},
\(\lambda|i\in\Gamma_{k-1}\).  Since
\(
 \sigma_k(\lambda)
 =\lambda_i\sigma_{k-1}(\lambda|i)
 +\sigma_k(\lambda|i)>0,
\)
it follows that \(\sigma_k(\lambda|i)>0\), and hence
\(\lambda|i\in\Gamma_k\).  The preceding identity, together with   \eqref{eq:2.7} in dimension \(n-1\), yields
\(
 -\lambda_i
 <
 \frac{\sigma_k(\lambda|i)}{\sigma_{k-1}(\lambda|i)}
 \leq
 \frac{j(n-k)}{k(n-j)}
 \frac{\sigma_j(\lambda|i)}{\sigma_{j-1}(\lambda|i)}.
\)
Notice that
\(
 j+m_j=\frac{k(n-j)}{n-k}.
\)
This inequality and \eqref{eq:2.28} prove \eqref{eq:2.22} in every
eigendirection when \(j<k\).  When \(j=k\), \eqref{eq:2.21} implies
\(m_k=0\), so \eqref{eq:2.22} follows directly from
\(\sigma_k(A)>0\). Finally, write \(\nu=\sum_{i=1}^n\nu_i e_i\) in an eigenbasis of
\(A\).  Since
\(
 \sigma_j(B)=T_j(A)[\nu,\nu]
 =\sum_{i=1}^n\nu_i^2\sigma_j(\lambda|i),
\)
multiplying the eigendirection inequalities by \(\nu_i^2\) and
summing over \(i\) proves \eqref{eq:2.22} for every unit vector
\(\nu\).

Let \(A=D^2v\).  On \(\{|\nabla v|>0\}\), we derive from equations
\eqref{eq:2.2}, \eqref{eq:2.4}, and \eqref{eq:2.5} that
\begin{equation}\label{eq:2.29}
\begin{aligned}
 \operatorname{div}\left(|\nabla v|^{m_j}T_{j-1}(A)\nabla v\right)
 &=|\nabla v|^{m_j}\left[j\sigma_j(A)+m_j
 AT_{j-1}(A)\left[\frac{\nabla v}{|\nabla v|},
 \frac{\nabla v}{|\nabla v|}\right]\right]\\
 &=|\nabla v|^{m_j}\left[(j+m_j)\sigma_j(A)-m_j\sigma_j(D^2v|_{(\nabla v)^\perp})\right].
\end{aligned}
\end{equation}
For \(j<k\), we have
\[|\nabla v|^{m_j}\left[(j+m_j)\sigma_j(A)-m_j\sigma_j(D^2v|_{(\nabla v)^\perp})\right]=(j+m_j)|\nabla v|^{m_j}\sigma_j(A)
-m_j|\nabla v|^{m_j-2}T_j(A)[\nabla v,\nabla v].\]  Since \(\left|
 |\nabla v|^{m_j-2}T_j(A)[\nabla v,\nabla v]
 \right|
 \leq C|\nabla v|^{m_j}\) and \(m_j>0\), the right-hand side of \eqref{eq:2.29}
extends continuously by zero across \(\{\nabla v=0\}\).  Thus
\eqref{eq:2.23} follows from \eqref{eq:2.22} and \eqref{eq:2.29}.
For \(j=k\), its right-hand side is \(k\sigma_k(D^2v)\).

For \(j<k\), Corollary~\ref{Cor:2.3} and \eqref{eq:2.22} imply
\(T_j(A)^2\preceq\frac{j+m_j}{m_j}\sigma_j(A)T_j(A)\).
On a regular level \(\Sigma_t\), set
\(\nu=\nabla v/|\nabla v|\), so that
\(\nu^\perp=(\nabla v)^\perp=T\Sigma_t\).
By \eqref{eq:2.2},
\(T_j(A)\nu=\sigma_j(D^2v|_{(\nabla v)^\perp})\nu
-T_{j-1}(D^2v|_{(\nabla v)^\perp})
\nabla_{\Sigma_t}|\nabla v|\).
The two terms are orthogonal, and
\(T_j(A)[\nu,\nu]=\sigma_j(D^2v|_{(\nabla v)^\perp})\).
Evaluating the matrix inequality \(T_j(A)^2\preceq\frac{j+m_j}{m_j}\sigma_j(A)T_j(A)\) on \(\nu\) now proves
\eqref{eq:2.24}.

To apply the coarea formula below, we first show that
\(\{\nabla v=0\}\) has Lebesgue measure zero. Indeed,  Stampacchia's
theorem applied to each \(C^1\) function \(\partial_i v\) gives
{\(\nabla(\partial_i v)=0\)} almost everywhere on
\(\{\partial_i v=0\}\).  Thus \(D^2v=0\) almost everywhere on
\(\{\nabla v=0\}\), which contradicts
\(\sigma_1(D^2v)>0\) on any subset of positive measure.

On every interval \(I\) such that \(\{v<t\}\Subset\Omega\) for
\(t\in I\), extend \(P_j\) to all \(t\in I\) by setting
\(P_j(t)=\int_{\{v<t\}}\operatorname{div}
\bigl(|\nabla v|^{m_j}T_{j-1}(D^2v)\nabla v\bigr)\dd x\).
If \(t\) is a regular value, then
\(\nu=\nabla v/|\nabla v|\) on \(\Sigma_t\).  The divergence theorem
and \eqref{eq:2.6} show that this definition agrees with
\eqref{eq:2.25}.  If \(t_1<t_2\) belong to \(I\), then
\eqref{eq:2.23}, the coarea formula, and the nullity of the critical
set imply
\begin{equation}\label{eq:2.30}
\begin{aligned}
 P_j(t_2)-P_j(t_1)
 =\int_{t_1}^{t_2}\int_{\{v=s\}\cap\{|\nabla v|>0\}}
 |\nabla v|^{m_j-1}
 \Bigl[
 &(j+m_j)\sigma_j(D^2v)\\
 &-m_j\sigma_j\bigl(D^2v|_{(\nabla v)^\perp}\bigr)
 \Bigr]\dd S\dd s.
\end{aligned}
\end{equation}
The integrand is nonnegative by \eqref{eq:2.22}.  Hence \(P_j\) is
nondecreasing and locally absolutely continuous, and
\eqref{eq:2.26} follows by differentiating \eqref{eq:2.30}.
For \(j=k\), one has \(m_k=0\), so \eqref{eq:2.23} reduces to
\(\operatorname{div}(T_{k-1}(D^2v)\nabla v)
=k\sigma_k(D^2v)\), which proves \eqref{eq:2.27}.
\end{proof}

\begin{corollary}\label{Cor:2.8}
Let \(u\in C^3(\Omega)\) satisfy \(-D^2u\in\Gamma_k\), and set
\(p_0=2+\frac{n(k-1)}{n-k}\).  Then
\(-\Delta_{p_0}u
=-\operatorname{div}(|\nabla u|^{p_0-2}\nabla u)\geq0\).
Moreover, let \(v\in C^2(\Omega)\) satisfy \(D^2v\in\Gamma_k\), and
set \(V(\tau,\theta)=v(e^\tau\theta)\) whenever
\(e^\tau\theta\in\Omega\).  Then
\begin{equation}\label{eq:2.31}
 V_{\tau\tau}+\frac{n-2k}{k}V_\tau
 +\frac{n-k}{k(n-1)}\Delta_{\mathbb S^{n-1}}V\geq0.
\end{equation}
\end{corollary}

\begin{proof}
Equation \eqref{eq:2.23} with \(j=1\), \(v=-u\), and
\(m_1=p_0-2\) implies \(-\Delta_{p_0}u\geq0\).
For the second assertion, \eqref{eq:2.22} with \(j=1\) and
\(\nu=\theta\) yields
\(\Delta v+\frac{n(k-1)}{n-k}D^2v[\theta,\theta]\geq0\).
The polar-coordinate identities
\(\Delta v=e^{-2\tau}
(V_{\tau\tau}+(n-2)V_\tau+\Delta_{\mathbb S^{n-1}}V)\) and
\(D^2v[\theta,\theta]
=e^{-2\tau}(V_{\tau\tau}-V_\tau)\)
reduce this inequality to \eqref{eq:2.31}.
\end{proof}

We establish the following Pohozaev identity for later use, which is of independent interest.

\begin{lemma}\label{Lem:2.9}
Let \(D\) be a smooth bounded domain, and let \(U\) be smooth and
admissible in a neighborhood of \(\overline D\).  Set \(A=-D^2U\)
and define
\begin{equation}\label{eq:2.32}
 \mathcal Q_U=\frac1{k+1}
 \left[
 UT_k(A)x+(U-x\cdot\nabla U)T_{k-1}(A)\nabla U
 \right].
\end{equation}
Then we have
\begin{equation}\label{eq:2.33}
 \operatorname{div}\mathcal Q_U
 =\sigma_k(A)\left(
 x\cdot\nabla U+\frac{n-2k}{k+1}U
 \right),
 \quad
 \int_{\partial D}\mathcal Q_U\cdot\nu\dd S
 =\int_D\sigma_k(A)\left(
 x\cdot\nabla U+\frac{n-2k}{k+1}U
 \right)\dd x.
\end{equation}

Assume \(\overline{B_R}\subset D\).  For \(0<r<R\), define
\begin{equation}\label{eq:2.34}
 \mathscr P_U(r)=\frac1{k+1}\int_{\partial B_r}
 \left[
 rUT_k(A)[\nu,\nu]
 +(U-rU_\nu)T_{k-1}(A)[\nabla U,\nu]
 \right]\dd S.
\end{equation}

If \(\sigma_k(A)=G'(U)\) for some \(G\in C^1\), then
\begin{equation}\label{eq:2.35}
 \int_{\partial D}\left(\mathcal Q_U-G(U)x\right)\cdot\nu\dd S
 =
 \int_D\left[
 \frac{n-2k}{k+1}UG'(U)-nG(U)
 \right]\dd x
\end{equation}
and
\begin{equation}\label{eq:2.36}
 \mathscr P_U(r)-r\int_{\partial B_r}G(U)\dd S
 =
 \int_{B_r}\left[
 \frac{n-2k}{k+1}UG'(U)-nG(U)
 \right]\dd x.
\end{equation}
In particular, if \(U>0\) and
\(G(U)=\varepsilon U^{p_*+1}/(p_*+1)\), then
\begin{equation}\label{eq:2.37}
 \mathscr P_U(r)=\frac{\varepsilon r}{p_*+1}
 \int_{\partial B_r}U^{p_*+1}\dd S.
\end{equation}
Finally, if
\(U_R(x)=R^{(n-2k)/k}U(Rx)\), then
\begin{equation}\label{eq:2.38}
 \mathscr P_{U_R}(r)
 =R^{(n-2k)/k}\mathscr P_U(Rr).
\end{equation}
\end{lemma}

\begin{proof}
Equations \eqref{eq:2.2}--\eqref{eq:2.5}, applied to the Hessian
of \(-U\), imply
\(\operatorname{div}T_{k-1}(A)=\operatorname{div}T_k(A)=0\),
\(T_k(A)=\sigma_k(A)\Id-T_{k-1}(A)A\),
\(\tr\bigl(T_{k-1}(A)A\bigr)=k\sigma_k(A)\), and
\(\tr T_k(A)=(n-k)\sigma_k(A)\).
Since \(D^2U=-A\), the product rule applied to the two terms in
\eqref{eq:2.32} yields
\begin{equation}\label{eq:2.39}
\begin{aligned}
 \operatorname{div}\left(UT_k(A)x\right)
 &=x\cdot T_k(A)\nabla U+(n-k)\sigma_k(A)U,\\
 \operatorname{div}\left((U-x\cdot\nabla U)
 T_{k-1}(A)\nabla U\right)
 &=x\cdot AT_{k-1}(A)\nabla U-k\sigma_k(A)(U-x\cdot\nabla U).
\end{aligned}
\end{equation}
Adding the two identities in \eqref{eq:2.39} and using
\(T_k(A)+AT_{k-1}(A)=\sigma_k(A)\Id\) proves the pointwise identity
in \eqref{eq:2.33}.  The integral identity follows from the
divergence theorem.

If \(\sigma_k(A)=G'(U)\), then
\(G'(U)x\cdot\nabla U=x\cdot\nabla G(U)\).  Integration by parts
yields
\[
 \int_DG'(U)x\cdot\nabla U\dd x
 =
 \int_{\partial D}G(U)x\cdot\nu\dd S
 -n\int_DG(U)\dd x,
\]
which proves \eqref{eq:2.35}.  Taking \(D=B_r\) and using
\(x=r\nu\) proves \eqref{eq:2.36}.

For \(G(U)=\varepsilon U^{p_*+1}/(p_*+1)\), the integrand on the
right-hand side of \eqref{eq:2.36} vanishes by
the definition of \(p_*\).  This proves \eqref{eq:2.37}.

Finally,
\(-D^2U_R(x)=R^{n/k}A(Rx)\).  The homogeneity
\(T_j(cA)=c^jT_j(A)\) and the change of variables \(y=Rx\) in
\eqref{eq:2.34} prove \eqref{eq:2.38}.
\end{proof}

\begin{corollary}\label{Cor:2.10}
Let \(v\) be smooth with \(D^2v\in\Gamma_k\), and let \(D\) be a
bounded component of \(\{v<t\}\) such that \(\nabla v\neq0\) on \(\partial D\).
Suppose that \(\sigma_k(D^2v)=g(v)\) near \(\overline D\) with
\(g=G'\).  Define
\(\dd\omega=k^{-1}|\nabla v|\sigma_{k-1}(D^2v|_{T\partial D})\dd S\).
Then we have
\begin{equation}\label{eq:2.40}
 \omega(\partial D)=\int_Dg(v)\dd x
\end{equation}
and
\begin{equation}\label{eq:2.41}
 \int_{\partial D}(x\cdot\nabla v)\dd\omega
 =\frac{(k+1)n}{k}\left[G(t)|D|-\int_DG(v)\dd x\right]
 +\frac{n-2k}{k}\int_D(v-t)g(v)\dd x.
\end{equation}
\end{corollary}

\begin{proof}
Equation~\eqref{eq:2.27}, applied to the component \(D\), proves
\eqref{eq:2.40}.  Set \(U=-v\).  Then \(A=-D^2U=D^2v\) and
\(
 \sigma_k(A)=g(v)=g(-U)
 =\frac{d}{dU}\bigl[-G(-U)\bigr],
\)
so \eqref{eq:2.35} applies with \(G(U)\) there replaced by
\(-G(-U)\).  Moreover, \(v=t\) and
\(\nu=\nabla v/|\nabla v|\) on \(\partial D\).

Equations \eqref{eq:2.3}, \eqref{eq:2.4}, and
\eqref{eq:2.5}, together with the divergence theorem, show that the
two boundary integrals below equal
\((n-k)\int_Dg(v)\dd x\) and \(k\int_Dg(v)\dd x\), respectively.
Hence, we conclude
\begin{equation}\label{eq:2.42}
 t\int_{\partial D}T_k(D^2v)x\cdot\nu\dd S
 -t\int_{\partial D}T_{k-1}(D^2v)\nabla v\cdot\nu\dd S
 =t(n-2k)\int_Dg(v)\dd x.
\end{equation}
Moreover,
\(T_{k-1}(D^2v)[\nabla v,\nu]
=|\nabla v|\sigma_{k-1}(D^2v|_{T\partial D})\), so the term
containing \(x\cdot\nabla v\) in \eqref{eq:2.32} is exactly
\(-k\int_{\partial D}(x\cdot\nabla v)\dd\omega\) after multiplication
by \(k+1\).  Substituting this identity and
\eqref{eq:2.42} into \eqref{eq:2.35}, and
using
\(\int_{\partial D}G(t)x\cdot\nu\dd S=nG(t)|D|\), we obtain
\eqref{eq:2.41}.
\end{proof}

\section{Optimal Liouville rigidity}\label{Sec:3}

In this section, we carry out our proof of Theorem \ref{Thm:1.1}, by proving and applying the weighted cutoff integral estimate in Lemma \ref{Lem:3.1}, the crucial cutoff integral estimate in Proposition \ref{Prop:3.3} and the lower bound estimate in Lemma \ref{Lem:3.4}. From the proof of the crucial cutoff integral estimate in Proposition \ref{Prop:3.3}, we found that the vector-valued function
\begin{equation}\label{eq:3.1}
  \mathcal{J}:=-u^ML_k(A)\nabla u-\tau u^{M-1}|\nabla u|^2T_{k-1}(A)\nabla u
\end{equation}
is very important in our proofs of the optimal subcritical Liouville theorems (Theorem \ref{Thm:1.1} and \ref{Thm:1.2}). We can use its divergence to control $u^{M+p-1}|\nabla u|^2$ in the proof of Proposition \ref{Prop:3.3}. Inspired by this observation, it will also play a key role in our subsequent proofs of the critical classification theorems ({Theorems \ref{Thm:1.6} and \ref{Thm:6.2}}) in Sections 5 and 6. At the critical exponent, the nonnegative divergence in \eqref{eq:5.3} vanishes after the cutoff argument; its equality case implies that $D^2(u^{-2k/(n-2k)})$ is a scalar matrix.

\subsection{A weighted cutoff integral estimate}

In order to prove the crucial cutoff integral estimate in Proposition \ref{Prop:3.3}, we will first prove the weighted cutoff integral inequalities
\eqref{eq:3.4} and \eqref{eq:3.5} used in both subcritical and critical cutoff arguments, since the weight \(\delta\) in the cutoff inequalities
\eqref{eq:3.4} and \eqref{eq:3.5} may change signs due to the parameter choices \eqref{eq:3.27}--\eqref{eq:3.28} if we take $\delta=-M-1$ later. The case $\delta>0$ goes back to Chang, Gursky and Yang \cite{CGY2003} and Gonz\'{a}lez \cite{Gonzalez2005}, and was adapted to the pure Hessian inequality by Ou \cite{Ou2010}. The proof is included in full because the case \(-1<\delta<0\) cannot be deduced from \eqref{eq:3.4} for $\delta>0$, because the coefficients \(b_s\) in \eqref{eq:3.17} change signs.

Recall that \(A=-D^2u\in\Gamma_k\) and
\(\sigma_k(A)=u^p\). Let \(\chi=\chi_R\) be a smooth cutoff function such that
\begin{equation}\label{eq:3.2}
    0\leq\chi\leq1,\qquad
    \chi=1\ \text{on }B_R,\qquad
    \operatorname{supp}\chi\subset B_{2R},\qquad
    |\nabla\chi|\leq \frac{C}{R}.
\end{equation}
Fix an integer \(\theta>2k\). We have the following weighted cutoff integral inequalities.
\begin{lemma}\label{Lem:3.1}
Let \(\delta>-1\), \(\delta\neq0\).  For \(1\leq s\leq k\), define
\begin{equation}\label{eq:3.3}
    \Lambda_s
    :=
    \int_{\R^n}
    \sigma_{k-s}(A)|\nabla u|^{2s}
    u^{-\delta-s}\chi^\theta\dd x.
\end{equation}
Then the following conclusions hold.

\begin{enumerate}[label=\textup{(\roman*)}]
\item If \(\delta>0\), then
\begin{equation}\label{eq:3.4}
    \int_{\R^n} u^{p-\delta}\chi^\theta\dd x
    +
    \sum_{s=1}^k \Lambda_s
    \leq
    C R^{-2k}
    \int_{\R^n} u^{k-\delta}\chi^{\theta-2k}\dd x.
\end{equation}

\item If \(-1<\delta<0\), then
\begin{equation}\label{eq:3.5}
    \sum_{s=1}^k \Lambda_s
    \leq
    C\int_{\R^n} u^{p-\delta}\chi^\theta\dd x
    +
    C R^{-2k}
    \int_{\R^n} u^{k-\delta}\chi^{\theta-2k}\dd x.
\end{equation}
\end{enumerate}
The constant \(C\) may depend on \(n,k,\delta,\theta\), but is
independent of \(R\).
\end{lemma}

\begin{proof}
For \(1\leq s\leq k\), set
\begin{align}
    M_s
    &:=
    \int_{\R^n}
    (T_{k-s}(A))_{ij}{\partial_i u}\,{\partial_j u} |\nabla u|^{2(s-1)}
    u^{-\delta-s}\chi^\theta\dd x,
    \\
    E_s
    &:=
    \int_{\R^n}
    (T_{k-s}(A))_{ij}{\partial_i u}\,{\partial_j\chi} |\nabla u|^{2(s-1)}
    u^{-\delta-s+1}\chi^{\theta-1}\dd x.
    \label{eq:3.6}
\end{align}
Since
\(T_{k-s}(A) = \sigma_{k-s}(A)\Id + D^2u\,T_{k-s-1}(A)\),
the symmetry of the Newton tensors gives, for \(1\leq s\leq k-1\),
\begin{align}
    M_s
    ={}&
    \Lambda_s
    +
    \int_{\R^n}
    {\partial_i u}(T_{k-s-1}(A))_{i\ell}
    {\partial_\ell\partial_j u}\,{\partial_j u}
    |\nabla u|^{2(s-1)}u^{-\delta-s}\chi^\theta\dd x
    \notag\\
    ={}&
    \Lambda_s
    +
    \frac1{2s}
    \int_{\R^n}
    {\partial_i u}(T_{k-s-1}(A))_{i\ell}
    \partial_\ell(|\nabla u|^{2s})
    u^{-\delta-s}\chi^\theta\dd x.
    \label{eq:3.7}
\end{align}
The tensor \(T_{k-s-1}(A)\) is symmetric and divergence-free in the
distributional sense.  Moreover,
\begin{equation}\label{eq:3.8}
   (T_{k-s-1}(A))_{i\ell}{\partial_i\partial_\ell u}
    =
    -\tr\bigl(T_{k-s-1}(A)A\bigr)
    =
    -(k-s)\sigma_{k-s}(A).
\end{equation}
Since
\(|\nabla u|^{2s}{\partial_i u}
u^{-\delta-s}\chi^\theta\in C_c^1(\R^n)\), it may
be used, by smooth approximation, to test
\(\partial_\ell(T_{k-s-1}(A))_{i\ell}=0\).  Hence
\begin{align*}
    &\frac1{2s}
    \int_{\R^n}
    {\partial_i u}(T_{k-s-1}(A))_{i\ell}
    \partial_\ell(|\nabla u|^{2s})u^{-\delta-s}\chi^\theta\dd x\\
    &\quad=
    -\frac1{2s}
    \int_{\R^n} |\nabla u|^{2s}
    (T_{k-s-1}(A))_{i\ell}{\partial_i\partial_\ell u}
    u^{-\delta-s}\chi^\theta\dd x\\
    &\qquad
    -\frac1{2s}
    \int_{\R^n} |\nabla u|^{2s}
    {\partial_i u}(T_{k-s-1}(A))_{i\ell}
    \partial_\ell\bigl(u^{-\delta-s}\chi^\theta\bigr)\dd x.
\end{align*}
There is no boundary term because the test function is compactly
supported.  The derivative of the weight is
\begin{equation}\label{eq:3.9}
    \partial_\ell\bigl(u^{-\delta-s}\chi^\theta\bigr)
    =
    -(\delta+s)u^{-\delta-s-1}{\partial_\ell u}\chi^\theta
    +
    \theta u^{-\delta-s}\chi^{\theta-1}{\partial_\ell\chi}.
\end{equation}
Substituting \eqref{eq:3.9} and \eqref{eq:3.8} gives
\begin{equation}\label{eq:3.10}
    \frac1{2s}\int_{\R^n}
    {\partial_i u}(T_{k-s-1}(A))_{i\ell}
    \partial_\ell(|\nabla u|^{2s})u^{-\delta-s}\chi^\theta\dd x
    =\frac{k-s}{2s}\Lambda_s
    +\frac{\delta+s}{2s}M_{s+1}
    -\frac{\theta}{2s}E_{s+1}.
\end{equation}
Combining this with \eqref{eq:3.7} implies
\begin{equation}\label{eq:3.11}
    M_s
    =
    \frac{k+s}{2s}\Lambda_s
    +
    \frac{\delta+s}{2s}M_{s+1}
    -
    \frac{\theta}{2s}E_{s+1}.
\end{equation}

Define the coefficients
\begin{align}
    m_s
    &:=
    \frac{\delta(\delta+1)\cdots(\delta+s-1)}{2^{s-1}(s-1)!},
    \\
    b_s
    &:=
    \frac{k+s}{2^s s!}\delta(\delta+1)\cdots(\delta+s-1),
    \\
    c_1&:=\theta,\qquad
    c_s
    :=
    \frac{\theta\delta(\delta+1)\cdots(\delta+s-2)}
    {2^{s-1}(s-1)!},
    \quad 2\leq s\leq k.
\end{align}
Thus
\begin{equation}\label{eq:3.12}
    b_s=\frac{k+s}{2s}m_s,\qquad
    m_{s+1}=\frac{\delta+s}{2s}m_s,\qquad
    c_{s+1}=\frac{\theta}{2s}m_s.
\end{equation}
Multiplying \eqref{eq:3.11} by \(m_s\) and using
\eqref{eq:3.12} yields
\begin{equation}\label{eq:3.13}
    m_sM_s
    =
    m_{s+1}M_{s+1}
    +
    b_s\Lambda_s
    -
    c_{s+1}E_{s+1},
    \qquad 1\leq s\leq k-1.
\end{equation}

Using \(\tr\bigl(T_{k-1}(A)A\bigr)=k\sigma_k(A)\), the divergence-free property of
\(T_{k-1}(A)\), and \(A=-D^2u\), we obtain
\begin{equation}\label{eq:3.14}
    k\int_{\R^n}\sigma_k(A)u^{-\delta}\chi^\theta\dd x
    =
    -\int_{\R^n}
    (T_{k-1}(A))_{ij}{\partial_i\partial_j u}
    u^{-\delta}\chi^\theta\dd x
    =-\delta M_1+\theta E_1
    =-m_1M_1+c_1E_1.
\end{equation}
Since \(T_0(A)=\Id\) and \(\sigma_0=1\), we have
\begin{equation}\label{eq:3.15}
\begin{aligned}
    M_k=\Lambda_k, \qquad
    b_k
    =
    \frac{2k}{2^k k!}\delta(\delta+1)\cdots(\delta+k-1)
    =
    m_k.
\end{aligned}
\end{equation}
Summing \eqref{eq:3.13} over
\(s=1,\ldots,k-1\) and using
\eqref{eq:3.15} gives
\begin{equation}\label{eq:3.16}
    m_1M_1
    =
    \sum_{s=1}^{k}b_s\Lambda_s
    -
    \sum_{s=2}^{k}c_sE_s.
\end{equation}
Inserting \eqref{eq:3.16} into
\eqref{eq:3.14}, and recalling that \(c_1=\theta\), we obtain
\begin{equation}\label{eq:3.17}
    k\int_{\R^n}\sigma_k(A)u^{-\delta}\chi^\theta\dd x
    =
    -\sum_{s=1}^k b_s\Lambda_s
    +
    \sum_{s=1}^k c_sE_s.
\end{equation}

Because \(A\in\Gamma_k\), each \(T_{k-s}(A)\) in
\eqref{eq:3.6} is nonnegative definite.  Its operator norm is
bounded by its trace:
\[
    \lVert T_{k-s}(A)\rVert
    \leq
    \tr T_{k-s}(A)
    =
    (n-k+s)\sigma_{k-s}(A).
\]
Consequently,
\begin{align}
    |E_s|
    \leq{}&
    C R^{-1}
    \int_{\R^n}
    \sigma_{k-s}(A)
    |\nabla u|^{2s-1}
    u^{-\delta-s+1}
    \chi^{\theta-1}\dd x.
    \label{eq:3.18}
\end{align}
By applying Young's inequality with the following precise factorization of the integrand in
\eqref{eq:3.18}:
\begin{align*}
    R^{-1}\sigma_{k-s}(A)|\nabla u|^{2s-1}u^{-\delta-s+1}\chi^{\theta-1}
    =\left(\sigma_{k-s}(A)|\nabla u|^{2s}u^{-\delta-s}\chi^\theta\right)^{\frac{2s-1}{2s}}
    \left(R^{-2s}\sigma_{k-s}(A)u^{s-\delta}\chi^{\theta-2s}\right)^{\frac1{2s}},
\end{align*}
we derive that, for every \(\varepsilon>0\),
\begin{equation}\label{eq:3.19}
    |E_s|
    \leq
    \varepsilon \Lambda_s
    +
    C_\varepsilon J_s,
\end{equation}
where
\begin{equation}
    J_s
    :=
    R^{-2s}
    \int_{\R^n}
    \sigma_{k-s}(A)u^{s-\delta}
    \chi^{\theta-2s}\dd x,
    \qquad 1\leq s\leq k.
\end{equation}
In particular,
\begin{equation}
    J_k
    =
    R^{-2k}
    \int_{\R^n}
    u^{k-\delta}\chi^{\theta-2k}\dd x.
\end{equation}

It remains to reduce every \(J_s\) ($s<k$) to \(J_k\).  Since
\[
    (k-s)\sigma_{k-s}(A)
    =
    \tr\bigl(T_{k-s-1}(A)A\bigr)
    =
    -(T_{k-s-1}(A))_{ij}{\partial_i\partial_j u},
\]
testing it with the compactly supported \(C^1\) functions
\(u^{s-\delta}\chi^{\theta-2s}\) yields
\begin{equation*}
    (k-s)J_s
    =
    -R^{-2s}\int_{\R^n}
    (T_{k-s-1}(A))_{ij}{\partial_i\partial_j u}
    u^{s-\delta}\chi^{\theta-2s}\dd x
    =R^{-2s}\int_{\R^n}
    (T_{k-s-1}(A))_{ij}{\partial_i u}
    \partial_j\bigl(u^{s-\delta}\chi^{\theta-2s}\bigr)\dd x.
\end{equation*}
Moreover, noting that
\begin{equation}
    \partial_j\bigl(u^{s-\delta}\chi^{\theta-2s}\bigr)
    =
    (s-\delta)u^{s-\delta-1}{\partial_j u}\chi^{\theta-2s}
    +
    (\theta-2s)u^{s-\delta}
    \chi^{\theta-2s-1}{\partial_j\chi},
\end{equation}
we get
\begin{align}
    (k-s)J_s
    ={}&
    (s-\delta)R^{-2s}
    \int_{\R^n}
    (T_{k-s-1}(A))_{ij}{\partial_i u}\,{\partial_j u}
    u^{s-\delta-1}
    \chi^{\theta-2s}\dd x
    \notag\\
    &+
    (\theta-2s)R^{-2s}
    \int_{\R^n}
    (T_{k-s-1}(A))_{ij}{\partial_i u}\,{\partial_j\chi}
    u^{s-\delta}
    \chi^{\theta-2s-1}\dd x.
    \label{eq:3.20}
\end{align}
If \(s-\delta\leq0\), the first term on the right-hand side of
\eqref{eq:3.20} is nonpositive and may be omitted in an upper
bound.  If \(s-\delta>0\), positivity and
the trace bound for \(T_{k-s-1}(A)\) give
\[
    (s-\delta)R^{-2s}
    \int_{\R^n}
    (T_{k-s-1}(A))_{ij}{\partial_i u}\,{\partial_j u}
    u^{s-\delta-1}\chi^{\theta-2s}\dd x
    \leq
    C R^{-2s}
    \int_{\R^n}
    \sigma_{k-s-1}(A)|\nabla u|^2
    u^{s-\delta-1}\chi^{\theta-2s}\dd x.
\]
By using the following factorization
\begin{align*}
    &R^{-2s}\sigma_{k-s-1}(A)|\nabla u|^2
    u^{s-\delta-1}\chi^{\theta-2s}\\
    &\quad=
    \left(
        \sigma_{k-s-1}(A)|\nabla u|^{2s+2}
        u^{-\delta-s-1}\chi^\theta
    \right)^{\frac1{s+1}}
    \left(
        R^{-2(s+1)}\sigma_{k-s-1}(A)
        u^{s+1-\delta}\chi^{\theta-2s-2}
    \right)^{\frac{s}{s+1}},
\end{align*}
it follows from Young's inequality that
\begin{equation}\label{eq:3.21}
    (s-\delta)R^{-2s}
    \int_{\R^n}
    (T_{k-s-1}(A))_{ij}{\partial_i u}\,{\partial_j u}
    u^{s-\delta-1}\chi^{\theta-2s}\dd x
    \leq
    \varepsilon \Lambda_{s+1}
    +
    C_\varepsilon J_{s+1}.
\end{equation}

For the second term in \eqref{eq:3.20}, the trace bound and
\eqref{eq:3.2} imply
\[
    \left|
    (\theta-2s)R^{-2s}
    \int_{\R^n}
    (T_{k-s-1}(A))_{ij}{\partial_i u}\,{\partial_j\chi}
    u^{s-\delta}\chi^{\theta-2s-1}\dd x
    \right|
    \leq
    C R^{-2s-1}
    \int_{\R^n}
    \sigma_{k-s-1}(A)|\nabla u|
    u^{s-\delta}\chi^{\theta-2s-1}\dd x.
\]
Applying Young's inequality with the following factorization
\begin{align*}
    &R^{-2s-1}\sigma_{k-s-1}(A)|\nabla u|
    u^{s-\delta}\chi^{\theta-2s-1}\\
    &\quad=
    \left(
        \sigma_{k-s-1}(A)|\nabla u|^{2s+2}
        u^{-\delta-s-1}\chi^\theta
    \right)^{\frac1{2(s+1)}}
    \left(
        R^{-2(s+1)}\sigma_{k-s-1}(A)
        u^{s+1-\delta}\chi^{\theta-2s-2}
    \right)^{\frac{2s+1}{2(s+1)}},
\end{align*}
we derive
\begin{equation}\label{eq:3.22}
    \left|
    (\theta-2s)R^{-2s}
    \int_{\R^n}
    (T_{k-s-1}(A))_{ij}{\partial_i u}\,{\partial_j\chi}
    u^{s-\delta}\chi^{\theta-2s-1}\dd x
    \right|
    \leq
    \varepsilon \Lambda_{s+1}
    +
    C_\varepsilon J_{s+1}.
\end{equation}
Together with \eqref{eq:3.20},
\eqref{eq:3.21}, and \eqref{eq:3.22}, we get,
after changing \(\varepsilon\),
\begin{equation}\label{eq:3.23}
    J_s
    \leq
    \varepsilon \Lambda_{s+1}
    +
    C_\varepsilon J_{s+1},
    \qquad 1\leq s\leq k-1.
\end{equation}

Iterating \eqref{eq:3.23} from \(s=k-1\) down to \(s=1\),
and choosing the Young parameters recursively, proves
\begin{equation}\label{eq:3.24}
    J_s
    \leq
    \varepsilon\sum_{j=s+1}^k \Lambda_j
    +
    C_\varepsilon J_k,
    \qquad 1\leq s\leq k-1
\end{equation}
for every prescribed \(\varepsilon>0\).  Combining \eqref{eq:3.19} with \eqref{eq:3.24} yields
\begin{equation}\label{eq:3.25}
    |E_s|
    \leq
    \varepsilon\sum_{j=s}^k \Lambda_j
    +
    C_\varepsilon R^{-2k}
    \int_{\R^n}
    u^{k-\delta}\chi^{\theta-2k}\dd x.
\end{equation}
Applying \eqref{eq:3.25} to each \(E_s\) with its small parameter $\varepsilon$ replaced by $\frac{\varepsilon}{\sum_{j=1}^k|c_j|}$, multiplying by \(|c_s|\) and summing, we get
\begin{equation}\label{eq:3.26}
    \left|
    \sum_{s=1}^k c_sE_s
    \right|
    \leq
    \varepsilon\sum_{s=1}^k \Lambda_s
    +
    C_\varepsilon R^{-2k}
    \int_{\R^n}
    u^{k-\delta}\chi^{\theta-2k}\dd x.
\end{equation}

If \(\delta>0\), then
\(b_s>0,\qquad 1\leq s\leq k\).
Using \(\sigma_k(A)=u^p\), identity~\eqref{eq:3.17} becomes
\(k\int_{\R^n} u^{p-\delta}\chi^\theta\dd x + \sum_{s=1}^k b_s\Lambda_s = \sum_{s=1}^k c_sE_s\).
Applying \eqref{eq:3.26} and choosing
\(0<\varepsilon< \frac12\min_{1\leq s\leq k}b_s\),
we obtain
\[
    k\int_{\R^n} u^{p-\delta}\chi^\theta\dd x
    +\left(\min_{1\leq s\leq k}b_s-\varepsilon\right)
      \sum_{s=1}^k\Lambda_s
    \leq
    C_\varepsilon R^{-2k}
    \int_{\R^n}
    u^{k-\delta}\chi^{\theta-2k}\dd x.
\]
Since
\(\min_{1\leq s\leq k}b_s-\varepsilon \geq\frac12\min_{1\leq s\leq k}b_s>0\),
we have derived
\eqref{eq:3.4}.

If \(-1<\delta<0\), then
\[
    \delta<0,\qquad
    \delta+j>0\quad\text{for every integer }j\geq1.
\]
Hence
\(b_s<0,\qquad 1\leq s\leq k\).
Identity~\eqref{eq:3.17} must now be rearranged as
\(\sum_{s=1}^k |b_s|\Lambda_s = k\int_{\R^n} u^{p-\delta}\chi^\theta\dd x - \sum_{s=1}^k c_sE_s\).
Using \eqref{eq:3.26} and choosing
\(0<\varepsilon< \frac12\min_{1\leq s\leq k}|b_s|\)
gives
\[
    \left(\min_{1\leq s\leq k}|b_s|-\varepsilon\right)
    \sum_{s=1}^k\Lambda_s
    \leq
    k\int_{\R^n} u^{p-\delta}\chi^\theta\dd x
    +
    C_\varepsilon R^{-2k}
    \int_{\R^n}
    u^{k-\delta}\chi^{\theta-2k}\dd x.
\]
The coefficient on the left is at least
\(\frac12\min\limits_{1\leq s\leq k}|b_s|>0\). Thus we have derived \eqref{eq:3.5}.
\end{proof}

\begin{remark}
For \(-1<\delta<0\), \(\int_{\R^n} u^{p-\delta}\chi^\theta\dd x\) remains on
the right-hand side of \eqref{eq:3.5}.  The rearrangement uses
the equality \(\sigma_k(A)=u^p\) and does not follow from a one-sided
Hessian inequality.
\end{remark}
\subsection{\texorpdfstring{Proof of Theorem~\ref{Thm:1.1}}
{Proof of the subcritical theorem}}
\label{Subsec:3.2}

Throughout this subsection, we assume the hypotheses of
Theorem~\ref{Thm:1.1}.  Choose
\begin{equation}\label{eq:3.27}
    \frac{(n-k)p}{k(n+2)}
    <\tau<
    \frac{n-k}{n-2k},
    \qquad
    \tau\neq\frac{n-k}{2k},
\end{equation}
and set
\begin{equation}\label{eq:3.28}
    M=-\frac{2k\tau}{n-k}.
\end{equation}
The interval for $\tau$ is nonempty because \(p<p_*\).  We exclude the single
value for which \(M=-1\), as required in
Lemma~\ref{Lem:3.1} with the choice $\delta=-M-1$.  Let \(\chi=\chi_R\) be the cutoff function in
\eqref{eq:3.2}.  Constants may depend on the fixed parameters and
the cutoff profile, but not on \(R\).

\begin{proposition}
\label{Prop:3.3}
Let \(\ell\in\mathbb N\) satisfy \(\ell>2k+2\). Set $P:=M+p+\frac{p}{k}$. Then
\begin{equation}\label{eq:3.29}
    M+p>0,\qquad
    P>M+p+1>0,
\end{equation}
and
\begin{align}
    \int_{\R^n}u^P\chi^\ell\dd x
    \leq{}&
    CR^{-2}
    \int_{\R^n}u^{M+p+1}\chi^{\ell-2}\dd x+
    CR^{-2k-2}
    \int_{\R^n}u^{M+k+1}
    \chi^{\ell-2k-2}\dd x.
    \label{eq:3.30}
\end{align}
\end{proposition}

\begin{proof}
The strict upper bound for \(\tau\) gives
\(M=-\frac{2k\tau}{n-k} > -\frac{2k}{n-2k}\).
Together with \(p>\frac{nk}{n-2k}\), this yields the quantitative
lower bound
\begin{equation}
    M+p
    >
    \frac{nk-2k}{n-2k}
    =
    \frac{k(n-2)}{n-2k}
    >0, \quad\,\, P>M+p+\frac{n}{n-2k}>M+p+1.
\end{equation}
This proves \eqref{eq:3.29}. The Newton--Maclaurin inequality and the equation give
\begin{equation}\label{eq:3.31}
    -\Delta u
    =
    \sigma_1(A)
    \geq
    n\binom{n}{k}^{-1/k}\sigma_k(A)^{1/k}
    =
    n\binom{n}{k}^{-1/k}u^{p/k}.
\end{equation}
Multiply \eqref{eq:3.31} by
\(u^{M+p}\chi^\ell\) and integrate by parts. Since the test function
has compact support, we have
\begin{align}
    &n\binom{n}{k}^{-1/k}\int_{\R^n}
    u^{M+p+p/k}\chi^\ell\dd x
    \leq
    \int_{\R^n}
    (-\Delta u)u^{M+p}\chi^\ell\dd x
    \notag\\
    &\qquad =
    (M+p)\int_{\R^n}
    u^{M+p-1}|\nabla u|^2\chi^\ell\dd x
    +
    \ell\int_{\R^n}
    u^{M+p}\chi^{\ell-1}
    \nabla u\cdot\nabla\chi\dd x.
    \label{eq:3.32}
\end{align}
By {the Cauchy--Schwarz} inequality, for every \(\varepsilon>0\), we have
\begin{align}
    \ell\chi^{\ell-1}u^{M+p}
    |\nabla u||\nabla\chi|
    &=
    \ell
    \left(
        \chi^{\ell/2}u^{(M+p-1)/2}|\nabla u|
    \right)
    \left(
        \chi^{(\ell-2)/2}u^{(M+p+1)/2}
        |\nabla\chi|
    \right)
    \notag\\
    &\leq
    \varepsilon\chi^\ell u^{M+p-1}|\nabla u|^2
    +
    C_\varepsilon
    \chi^{\ell-2}u^{M+p+1}|\nabla\chi|^2.
    \label{eq:3.33}
\end{align}
Using \(|\nabla\chi|\leq C/R\), recalling
\(M+p+p/k=P\), then
\eqref{eq:3.32} and \eqref{eq:3.33} imply
\begin{equation}\label{eq:3.34}
    \int_{\R^n}u^P\chi^\ell\dd x
    \leq
    C\int_{\R^n}u^{M+p-1}|\nabla u|^2\chi^\ell\dd x
    +
    CR^{-2}
    \int_{\R^n}
    u^{M+p+1}\chi^{\ell-2}\dd x.
\end{equation}

Next, we aim to estimate the first integral on the right-hand side of \eqref{eq:3.34}:
\begin{equation}\label{eq:3.35}
    \int_{\R^n}u^{M+p-1}|\nabla u|^2\chi^\ell\dd x.
\end{equation}
To this end, we will introduce the key function
\begin{equation}\label{eq:3.36}
  \mathcal{J}:=-u^ML_k(A)\nabla u-\tau u^{M-1}|\nabla u|^2T_{k-1}(A)\nabla u,
\end{equation}
whose divergence can control $u^{M+p-1}|\nabla u|^2$. Note that the tensor
\(T_{k-1}(A)\) is divergence-free, and
\(\tr\bigl(T_{k-1}(A)A\bigr)=ku^p\).
Therefore,
\(\operatorname{div}\bigl(-T_{k-1}(A)\nabla u\bigr)=ku^p\).
Moreover, \eqref{eq:2.11} and
\(\nabla|\nabla u|^2=-2A\nabla u\) give
\[
    \nabla|\nabla u|^2\cdot
       \bigl(-T_{k-1}(A)\nabla u\bigr)
    =2\nabla u^TL_k(A)\nabla u
      +\frac{2k}{n}u^p|\nabla u|^2.
\]
Consequently, we have
\begin{align}
 &\quad \operatorname{div}\bigl(
   -u^{M-1}|\nabla u|^2T_{k-1}(A)\nabla u\bigr)
 \notag\\
 &=
 2u^{M-1}\nabla u^TL_k(A)\nabla u
 +(1-M)u^{M-2}|\nabla u|^2
       \nabla u^TT_{k-1}(A)\nabla u
 \notag\\
 &\qquad\quad+
 \frac{k(n+2)}n u^{M+p-1}|\nabla u|^2.
 \label{eq:3.37}
\end{align}
The last term in \eqref{eq:3.37} is precisely the function we want to control. However, the mixed terms in \eqref{eq:3.37} have no sign.  To control the mixed terms, we use a second divergence that produces the factor
\(\tr(L_k(A)A)\) appearing on the right-hand side of \eqref{eq:2.12}. From
\(\operatorname{div}L_k(A)=\frac{n-k}{n}\nabla(u^p)\), we get
\begin{align}
 \operatorname{div}\bigl(-u^ML_k(A)\nabla u\bigr)
 ={}u^M\tr\bigl(L_k(A)A\bigr)
 -Mu^{M-1}\nabla u^TL_k(A)\nabla u
 -\frac{n-k}{n}p\,
     u^{M+p-1}|\nabla u|^2.
 \label{eq:3.38}
\end{align}
Now, multiplying \eqref{eq:3.37} by \(\tau\) and adding
\eqref{eq:3.38}, we derive
\begin{align}
 &\quad \operatorname{div}\bigl(-u^ML_k(A)\nabla u
 -\tau u^{M-1}|\nabla u|^2T_{k-1}(A)\nabla u\bigr)
 \notag\\
 &=
 u^M\tr\bigl(L_k(A)A\bigr)
 +(2\tau-M)u^{M-1}\nabla u^TL_k(A)\nabla u
 \notag\\
 &\qquad+
 \tau(1-M)u^{M-2}|\nabla u|^2
       \nabla u^TT_{k-1}(A)\nabla u
 \notag\\
 &\qquad+
 \left(\frac{k(n+2)}n\tau-\frac{n-k}{n}p\right)
      u^{M+p-1}|\nabla u|^2.
 \label{eq:3.39}
\end{align}
The parameter choice implies
\begin{align}
 \frac{k(n+2)}n\tau-\frac{n-k}{n}p&>0,
 \label{eq:3.40}\\
 4\tau(1-M)-\frac{n-k}{n}(2\tau-M)^2
 &=4\tau\left(1-\frac{n-2k}{n-k}\tau\right)>0.
 \label{eq:3.41}
\end{align}
On the other hand, \eqref{eq:2.20} gives
\begin{equation}\label{eq:3.42}
 \bigl|\nabla u^TL_k(A)\nabla u\bigr|
 \leq
 \left(
  \frac{n-k}{n}\tr\bigl(L_k(A)A\bigr)|\nabla u|^2
  \nabla u^TT_{k-1}(A)\nabla u
 \right)^{1/2}.
\end{equation}
Thus the first three terms on the right of \eqref{eq:3.39} form a
positive definite quadratic expression in the corresponding square
roots. Consequently, by using {the Cauchy--Schwarz} inequality, we deduce from \eqref{eq:3.39} that, for some \(\kappa>0\),
\begin{align}\label{eq:3.43}
 &\quad \operatorname{div}\bigl(-u^ML_k(A)\nabla u
 -\tau u^{M-1}|\nabla u|^2T_{k-1}(A)\nabla u\bigr)\\
 &\geq \left(\frac{k(n+2)}n\tau-\frac{n-k}{n}p\right)
      u^{M+p-1}|\nabla u|^2 \nonumber\\
 &\quad +\kappa\left(u^M\tr\bigl(L_k(A)A\bigr)
 +u^{M-2}|\nabla u|^2\nabla u^TT_{k-1}(A)\nabla u\right). \nonumber
\end{align}

We now multiply \eqref{eq:3.43} by \(\chi^\ell\) and integrate by parts, and get
\begin{align}
 &\quad -\int_{\R^{n}}\ell\chi^{\ell-1}\bigl(-u^ML_k(A)\nabla u
 -\tau u^{M-1}|\nabla u|^2T_{k-1}(A)\nabla u\bigr)\cdot\nabla\chi\dd x \\
 &\geq \left(\frac{k(n+2)}n\tau-\frac{n-k}{n}p\right)
      \int_{\R^{n}}u^{M+p-1}|\nabla u|^2\chi^\ell\dd x \nonumber\\
 &\quad +\kappa\int_{\R^{n}}\left(u^M\tr\bigl(L_k(A)A\bigr)
 +u^{M-2}|\nabla u|^2\nabla u^TT_{k-1}(A)\nabla u\right)\chi^\ell\dd x. \nonumber
\end{align}
By \eqref{eq:2.12}, we have the following inequalities:
\begin{align}
 |L_k(A)\nabla u\cdot\nabla\chi|
 &\leq
 \left(
  \frac{n-k}{n}\tr\bigl(L_k(A)A\bigr)|\nabla u|^2
  \nabla\chi^TT_{k-1}(A)\nabla\chi
 \right)^{1/2},
 \label{eq:3.44}\\
 |T_{k-1}(A)\nabla u\cdot\nabla\chi|
 &\leq
 \left(
  \nabla u^TT_{k-1}(A)\nabla u\,
  \nabla\chi^TT_{k-1}(A)\nabla\chi
 \right)^{1/2}.
 \notag
\end{align}
Therefore, using Young's inequality, dropping the other nonnegative
terms and using
\(\tr T_{k-1}(A)=(n-k+1)\sigma_{k-1}(A)\), we obtain
\begin{align}
 \int_{\R^n}u^{M+p-1}|\nabla u|^2\chi^\ell\dd x
 &\leq
 C\int_{\R^n}\chi^{\ell-2}u^M|\nabla u|^2
       \nabla\chi^TT_{k-1}(A)\nabla\chi\dd x
 \notag\\
 &\leq
 CR^{-2}\int_{\R^n}
       \sigma_{k-1}(A)|\nabla u|^2u^M\chi^{\ell-2}\dd x.
 \label{eq:3.45}
\end{align}

Take \(\delta=-M-1\). The parameter choice gives \(M<0\), hence \(\delta>-1\), and
the exclusion in \eqref{eq:3.27} ensures that \(M\neq-1\), hence
\(\delta\neq0\).  In Lemma~\ref{Lem:3.1}, take
\(\theta=\ell-2>2k\).  Then the integral on the right-hand side of
\eqref{eq:3.45}
equals \(\Lambda_{1}\) in \eqref{eq:3.3}, because $-\delta-1=M$ and $\theta=\ell-2$, more precisely, \eqref{eq:3.45} bounds
\eqref{eq:3.35} by \(CR^{-2}\Lambda_{1}\). Applying \eqref{eq:3.4} when \(\delta>0\) and
\eqref{eq:3.5} when \(-1<\delta<0\), and noting that
\(p-\delta=M+p+1\) and \(k-\delta=M+k+1\), we obtain,
\begin{equation}\label{eq:3.46}
    \int_{\R^n}u^{M+p-1}|\nabla u|^2\chi^\ell\dd x
    \leq CR^{-2}\int_{\R^n}u^{M+p+1}\chi^{\ell-2}\dd x
    +CR^{-2k-2}\int_{\R^n}u^{M+k+1}\chi^{\ell-2k-2}\dd x.
\end{equation}

Combining \eqref{eq:3.34} with \eqref{eq:3.46} proves
\eqref{eq:3.30}. This completes our proof of Proposition \ref{Prop:3.3}.
\end{proof}

We will need the following lower bound estimate.
\begin{lemma}
\label{Lem:3.4}
Suppose that \(u>0\),
\(-D^2u\in\Gamma_k\), and
\(\sigma_k(-D^2u)>0\) in \(\R^n\). Let
\begin{equation}
    \rho=\frac{n-2k}{k}>0.
\end{equation}
Then, for every fixed \(R_0>0\),
\begin{equation}\label{eq:3.47}
    u(x)
    \geq
    m_0R_0^\rho |x|^{-\rho}, \quad \,\, \forall \,\, |x|\geq R_0,
    \qquad
    m_0:=\min_{|x|=R_0}u(x)>0.
\end{equation}
\end{lemma}

\begin{proof}
Fix \(S>R_0\), and define
\begin{equation}\label{eq:3.48}
    h_S(r)
    =
    m_0
    \frac{r^{-\rho}-S^{-\rho}}
         {R_0^{-\rho}-S^{-\rho}},
    \qquad
    R_0\leq r\leq S.
\end{equation}
Set
\[
    w=-u,\qquad v_S=-h_S,\qquad
    C_S=\frac{m_0}{R_0^{-\rho}-S^{-\rho}}>0.
\]
Consequently, one has
\[
    v_S'(r)=C_S\rho r^{-\rho-1},
    \qquad
    v_S''(r)=-C_S\rho(\rho+1)r^{-\rho-2}.
\]
For radial function $v_S$, the Hessian has radial eigenvalue \(v_S''\)
and tangential eigenvalue \(v_S'/r\), with tangential multiplicity
\(n-1\).  Thus the radial eigenvalue of \(D^2v_S\) and its
\(n-1\) equal tangential eigenvalues are
\begin{equation}
    \lambda_{\mathrm{rad}}
    =
    -C_S\rho(\rho+1)r^{-\rho-2},
    \qquad
    \lambda_{\mathrm{tan}}
    =
    C_S\rho r^{-\rho-2}.
\end{equation}
Writing \(\lambda=\lambda_{\mathrm{tan}}>0\), we therefore have
\(\lambda_{\mathrm{rad}}=-(\rho+1)\lambda\).  For
\(1\leq j\leq k\),
\begin{align}
    \sigma_j(D^2v_S)
    &=
    \binom{n-1}{j}\lambda^j
    -
    (\rho+1)\binom{n-1}{j-1}\lambda^j
    \notag\\
    &=
    \binom{n-1}{j-1}\lambda^j
    \left(
        \frac{n-j}{j}-(\rho+1)
    \right)
    \notag\\
    &=
    \binom{n-1}{j-1}\lambda^j
    \frac{n(k-j)}{jk}.
    \label{eq:3.49}
\end{align}
Thus
\[
    \sigma_j(D^2v_S)>0, \quad\,\, \forall \,\, 1\leq j<k,
    \qquad
    \sigma_k(D^2v_S)=0.
\]
Consequently, \(D^2v_S\in\partial\Gamma_k\), while
\[
    D^2w=-D^2u\in\Gamma_k,
    \qquad
    \sigma_k(D^2w)>0=\sigma_k(D^2v_S).
\]

On \(|x|=R_0\),
\(w=-u\leq-m_0=v_S\),
and on \(|x|=S\),
\(w=-u<0=v_S\).
If \(w-v_S\) had a positive interior maximum point $x_0\in B_{S}\setminus \overline{B_{R_{0}}}$, then
\(D^2w(x_{0})\preceq D^2v_S(x_{0})\).
Let
\[
    H=D^2v_S(x_{0})-D^2w(x_{0})\succeq0,
    \qquad
    A_t=D^2w+tH,\quad 0\leq t\leq1.
\]
As long as \(A_t\in\Gamma_k\), Corollary~\ref{Cor:2.3} gives
\begin{equation}\label{eq:3.50}
    \frac{\dd}{\dd t}\sigma_j(A_t)
    =
    \tr\bigl(T_{j-1}(A_t)H\bigr)
    \geq0,
    \qquad 1\leq j\leq k,
\end{equation}
because \(H\succeq0\).
{Indeed, if \(t_0\leq1\) satisfies \(A_{t_0}\in\partial\Gamma_k\) and $A_{t}\in \Gamma_k$ for $t\in[0,t_0)$, then \eqref{eq:3.50} on \([0,t_0)\) and continuity would {imply}
\(\sigma_j(A_{t_0})\geq\sigma_j(A_0)>0\) for every \(1\leq j\leq k\),
contradicting \(A_{t_0}\in\partial\Gamma_k\).}
Thus \(A_t\in\Gamma_k\) for any $0\leq t\leq1$.  Integrating \eqref{eq:3.50} with
\(j=k\) yields
\(\sigma_k(D^2v_S)\geq\sigma_k(D^2w)\), and therefore
that
\(0=\sigma_k(D^2v_S) \geq \sigma_k(D^2w)>0\),
a contradiction.  Hence
\(w\leq v_S \quad\text{in }B_S\setminus\overline{B_{R_0}}\),
or equivalently, \(u\geq h_S\).  Letting \(S\to\infty\) in
\eqref{eq:3.48} yields
\(u(x) \geq m_0\frac{|x|^{-\rho}}{R_0^{-\rho}} = m_0R_0^\rho|x|^{-\rho}\),
which is \eqref{eq:3.47}. This finishes our proof of Lemma \ref{Lem:3.4}.
\end{proof}

\begin{proof}[Proof of Theorem~\ref{Thm:1.1}]
Assume for contradiction that $u\geq0$ but \(u\not\equiv0\). Lemma~\ref{Lem:2.2} then implies that $u>0$ in $\R^n$.  Fix
\(\tau\) and \(M\) as in \eqref{eq:3.27}--\eqref{eq:3.28}, and recall that $P=M+p+\frac{p}{k}$.
Proposition~\ref{Prop:3.3} gives
\begin{align}
    \int_{\R^n}u^P\chi^\ell\dd x
    \leq{}&
    CR^{-2}
    \int_{\R^n}u^{M+p+1}\chi^{\ell-2}\dd x
    +
    CR^{-2k-2}
    \int_{\R^n}u^{M+k+1}\chi^{\ell-2k-2}\dd x.
    \label{eq:3.51}
\end{align}

Next, we will divide the proof into three cases.

\medskip

\noindent\emph{Case 1: \(M+k+1>0\).}
Then \(0<M+k+1<P\). Choose an integer
\(\ell\) large enough that
\begin{equation}\label{eq:3.52}
    \ell>2k+2,\qquad
    \ell\geq\frac{2P}{P-(M+p+1)},\qquad
    \ell\geq\frac{(2k+2)P}{P-(M+k+1)}.
\end{equation}

We shall use the following elementary form of Young's inequality
twice.  If \(0<\gamma<P\) and
\(m\geq\ell\gamma/P\), then
\begin{equation}\label{eq:3.53}
\begin{aligned}
    R^{-a}\chi^m u^\gamma
    \leq{}&
    \varepsilon\chi^\ell u^P
    +
    C_\varepsilon
    R^{-aP/(P-\gamma)}
    \chi^{(mP-\ell\gamma)/(P-\gamma)}.
\end{aligned}
\end{equation}

For the first scalar term in
\eqref{eq:3.51}, apply
\eqref{eq:3.53} with \(a=2\), \(\gamma=M+p+1\), and \(m=\ell-2\).
Thus
\begin{align}
    R^{-2}\chi^{\ell-2}u^{M+p+1}
    \leq{}
    \varepsilon\chi^\ell u^P
    +
    C_\varepsilon
    R^{-\frac{2P}{P-(M+p+1)}}
    \chi^{\frac{(\ell-2)P-\ell(M+p+1)}{P-(M+p+1)}}.
    \label{eq:3.54}
\end{align}
For the second scalar term, take
\[
    a=2k+2,\qquad
    \gamma=M+k+1,\qquad
    m=\ell-2k-2.
\]
Hence
\begin{align}
    R^{-2k-2}\chi^{\ell-2k-2}u^{M+k+1}
    \leq{}&
    \varepsilon\chi^\ell u^P
    +
    C_\varepsilon
    R^{-\frac{(2k+2)P}{P-(M+k+1)}}
    \chi^{\frac{(\ell-2k-2)P-\ell(M+k+1)}{P-(M+k+1)}}.
    \label{eq:3.55}
\end{align}
Both residual cutoff exponents are nonnegative. Moreover, we have
\begin{equation}\label{eq:3.56}
    \frac{2P}{P-(M+p+1)}
    =
    \frac{(2k+2)P}{P-(M+k+1)}
    =
    \frac{2kP}{p-k}.
\end{equation}
Integrate \eqref{eq:3.54} and
\eqref{eq:3.55} over \(\supp\chi\subset B_{2R}\).
Since \(0\leq\chi\leq1\), the residual cutoff powers are bounded by
\(1\), and therefore, we get
\begin{equation}\label{eq:3.57}
    \int_{\R^n}u^P\chi^\ell\dd x
    \leq
    C\varepsilon\int_{\R^n}u^P\chi^\ell\dd x
    +
    C_\varepsilon R^{-\frac{2kP}{p-k}}|B_{2R}|.
\end{equation}
Choose \(\varepsilon\) so that \(C\varepsilon<1/2\) and absorb the
first term on the right of \eqref{eq:3.57}.  Since
\(\chi=1\) on \(B_R\),
\begin{equation}
    \int_{B_R}u^P\dd x
    \leq
    CR^{n-\frac{2kP}{p-k}}.
\end{equation}
The upper bound for \(\tau\) gives $P>
 p\left(1+\frac1k\right)-\frac{2k}{n-2k}$, and hence $n-\frac{2kP}{p-k}<0$. Therefore, the
right-hand side tends to zero, as $R\rightarrow+\infty$.

\medskip
\noindent\emph{Case 2: \(M+k+1=0\).}
Choose an integer \(\ell\) so that the first two conditions in
\eqref{eq:3.52} hold. We treat the first term in \eqref{eq:3.51} exactly as
in \eqref{eq:3.54}. The second term in
\eqref{eq:3.51} now satisfies
\begin{equation}\label{eq:3.58}
    R^{-2k-2}
    \int_{\R^n}u^{M+k+1}\chi^{\ell-2k-2}\dd x\leq R^{-2k-2}
    \int_{\R^n}\chi^{\ell-2k-2}\dd x
    \leq
    CR^{n-2k-2}.
\end{equation}
Since \(M+k+1=0\), one has
\[
    P=\frac{(k+1)(p-k)}{k},
    \qquad
    \frac{2kP}{p-k}=2k+2.
\]
Consequently
\(n-2k-2=n-\frac{2kP}{p-k}<0\).
Combining \eqref{eq:3.58} with estimate on the first term in \eqref{eq:3.51} via \eqref{eq:3.54}, we finally get
\begin{equation}
    \int_{\R^n}u^P\chi^\ell\dd x
    \leq
    C\varepsilon\int_{\R^n}u^P\chi^\ell\dd x
    +
    C_\varepsilon R^{n-\frac{2kP}{p-k}},
\end{equation}
and hence
\begin{equation}
    \int_{B_R}u^P\dd x
    \leq
    CR^{n-\frac{2kP}{p-k}}\rightarrow0,  \qquad \text{as} \,\, R\rightarrow+\infty.
\end{equation}

\medskip
\noindent\emph{Case 3: \(M+k+1<0\).}
We first show that this case can occur only in the special
dimension \(n=2k+1\). Indeed, if \(n-2k\geq2\), then
\(M>-\frac{2k}{n-2k}\geq-k\),
and the first inequality is strict. Hence $M+k+1>1$,
which contradicts \(M+k+1<0\).  Thus we must have \(n-2k=1\), as claimed.

In this dimension, \(\rho=(n-2k)/k=1/k\).
Fix \(R_0>0\).  Lemma~\ref{Lem:3.4} gives
\begin{equation}\label{eq:3.59}
    u(x)\geq c_0|x|^{-\frac{1}{k}},
    \qquad \forall \,\, |x|\geq R_0
\end{equation}
for a constant \(c_0>0\).  Because \(M+k+1<0\), raising
\eqref{eq:3.59} to the power \(M+k+1\)
reverses the inequality:
\[
    u(x)^{M+k+1}
    \leq
    C|x|^{-\frac{M+k+1}{k}}
    \qquad \forall \,\, |x|\geq R_0.
\]
The integral of \(u^{M+k+1}\) over the fixed compact ball
\(\overline B_{R_0}\) is finite, since \(u\) is continuous and has
a positive minimum there.  It follows that, for \(R\geq R_0\),
\begin{align}
    R^{-2k-2}
    \int_{\R^n}
    u^{M+k+1}\chi^{\ell-2k-2}\dd x
    &\leq
    R^{-2k-2}
    \int_{B_{R_0}}u^{M+k+1}\dd x
    +
    CR^{-2k-2}
    \int_{B_{2R}\setminus B_{R_0}}
    |x|^{-\frac{M+k+1}{k}}\dd x
    \notag\\
    &\leq
    C_{R_0}R^{-2k-2}
    +
    CR^{-2k-2+n-\frac{M+k+1}{k}}
    \notag\\
    &=
    C_{R_0}R^{-2k-2}
    +
    CR^{-1-\frac{M+k+1}{k}}.
    \label{eq:3.60}
\end{align}
The upper parameter bound is now
\(\tau<\frac{n-k}{n-2k}=k+1\).
Since \(M=-2k\tau/(n-k)\) and \(n-k=k+1\), it follows that
\(M>-2k, \qquad M+k+1>1-k\).
Therefore, one has
\begin{equation}\label{eq:3.61}
    -1-\frac{M+k+1}{k}
    <
    -\frac1k
    <0.
\end{equation}
Choose an integer \(\ell\) so that the first two conditions in
\eqref{eq:3.52} hold and apply
\eqref{eq:3.54} to deal with the first term on the right-hand side of \eqref{eq:3.51}. By \eqref{eq:3.60} and \eqref{eq:3.61}, we deduce that
\begin{equation}
    \int_{B_R}u^P\dd x
    \leq
    CR^{n-\frac{2kP}{p-k}}
    +
    C_{R_0}R^{-2k-2}
    +
    CR^{-1-\frac{M+k+1}{k}}
    \longrightarrow0, \qquad \text{as} \,\, R\rightarrow+\infty.
\end{equation}

In all the above three cases, we have
\begin{equation}
    0<
    \int_{B_1}u^P\dd x
    \leq\int_{B_R}u^P\dd x\longrightarrow0, \qquad \text{as} \,\, R\rightarrow+\infty,
\end{equation}
which is absurd. As a consequence, we have concluded our proof of Theorem \ref{Thm:1.1}.
\end{proof}

\section{Optimal Liouville rigidity for \texorpdfstring{\(k\)}{k}-Hessian measure solutions}
\label{Sec:4}

In this section, we will prove Theorem~\ref{Thm:1.2}, and hence extend the optimal Liouville theorem in Theorem \ref{Thm:1.1} from $C^2$-solutions to $L^{\infty}_{\mathrm{loc}}$-weak solutions in the sense of $k$-Hessian measure. A mollification of a $k$-Hessian measure weak solution does not satisfy the equation, so the classical proof cannot be passed to the limit in this way.  We instead solve local Dirichlet problems with smooth nonlinear terms $f_j$. The difference between the nonlinear terms $f_{j}$ and the power of the approximate solution $u_j$ converges uniformly to zero.

\subsection{Local approximation}

Let \(u\not\equiv0\) satisfy the hypotheses of
Theorem~\ref{Thm:1.2} and set \(v:=-u\).  For every
\(R>0\),
\(u^p\in L^\infty(B_{4R})\subset L^q(B_{4R}), \qquad q>\frac{n}{2k}\).
By the strong minimum principle and \cite[Theorem~9.2 and Corollary~9.1]{Wang2009}, we have
\begin{equation}
    \mu_k[v]=u^p\,\dd x
    \Longrightarrow v\in C^\alpha_{\mathrm{loc}}(\R^n),
    \quad
    v\in\Phi^k(\R^n)
    \Longrightarrow \Delta v\geq0
    \Longrightarrow \Delta u\leq0
    \Longrightarrow u>0
    \quad\text{in }\R^n.
    \label{eq:4.1}
\end{equation}
We use this locally H\"older representative below. It agrees almost
everywhere with the representative in \eqref{eq:1.6}, so the equality
of Hessian measures is unchanged.

\begin{lemma}[Theorem~3.1 in \cite{TW1997}]
\label{Lem:4.1}
Let \(\Omega\subset\R^n\) be bounded and
\(v,w\in\Phi^k(\Omega)\cap C(\overline\Omega)\).  If
\[
    \mu_k[v]\geq\mu_k[w]\quad\text{in }\Omega,
    \qquad
    v\leq w\quad\text{on }\partial\Omega,
\]
then \(v\leq w\) in \(\Omega\).
\end{lemma}

\begin{lemma}
\label{Lem:4.2}
For every \(R>0\), there are \(\{u_j\}\subset C^3(\overline {B_{3R}})\) satisfying
\begin{equation*}
    u_j>0,\qquad A_j:=-D^2u_j\in\Gamma_k,\qquad
    \sigma_k(A_j)=f_j>0
    \quad\text{in }B_{3R},
\end{equation*}
and
\begin{equation}  \label{eq:4.2}
    u_j\longrightarrow u,\qquad
    r_j:=f_j-u_j^p\longrightarrow0
\end{equation}
uniformly on \(\overline {B_{3R}}\).
\end{lemma}

\begin{proof}
{
By \eqref{eq:4.1}, we have
\(0<\min_{\overline {B_{3R}}}u^p \leq\max_{\overline {B_{3R}}}u^p<\infty\).
Choose \(\{f_j\}\subset C^\infty(\overline {B_{3R}})\) and
\(\{\varphi_j\}\subset C^\infty(\partial B_{3R})\) so that
\begin{equation*}
    \|f_j-u^p\|_{L^\infty(B_{3R})}\longrightarrow0,\qquad
    \|\varphi_j-v\|_{L^\infty(\partial B_{3R})}\longrightarrow0.
\end{equation*}
Extend \(\varphi_j\) smoothly to \(\overline B_{3R}\). Adding a
sufficiently large multiple of \(|x|^2-9R^2\) gives a strict admissible
subsolution with boundary value \(\varphi_j\). Since \(B_{3R}\) is
uniformly \((k-1)\)-convex, \cite[Theorem~3.4]{Wang2009} {yields} a
solution \(v_j\) of
\begin{equation}
    v_j\in C^3(\overline {B_{3R}}),\quad
    D^2v_j\in\Gamma_k,\quad
    \sigma_k(D^2v_j)=f_j\quad\text{in }B_{3R},\quad
    v_j=\varphi_j\quad\text{on }\partial B_{3R}.
\end{equation}
Applying Lemma~\ref{Lem:4.1} directly to constant multiples and constant
translates of \(v\) yields
\begin{align*}
 &\left(\sup_{\overline {B_{3R}}}\frac{f_j}{u^p}\right)^{1/k}v
 +\inf_{\partial B_{3R}}\left[
 \varphi_j-
 \left(\sup_{\overline {B_{3R}}}\frac{f_j}{u^p}\right)^{1/k}v\right]
 \leq v_j\\
 &\hspace{18mm}\leq
 \left(\inf_{\overline {B_{3R}}}\frac{f_j}{u^p}\right)^{1/k}v
 +\sup_{\partial B_{3R}}\left[
 \varphi_j-
 \left(\inf_{\overline {B_{3R}}}\frac{f_j}{u^p}\right)^{1/k}v\right]
 \quad\text{in }\overline {B_{3R}}.
\end{align*}
The two multiplicative factors tend to \(1\), and the two boundary
corrections tend to \(0\).  Hence
\(\|v_j-v\|_{L^\infty(B_{3R})}\longrightarrow0\).
Let \(u_j=-v_j\).
Therefore, \eqref{eq:4.1} implies that for large \(j\),
\[
    \min_{\overline {B_{3R}}}u_j
    \geq\frac12\min_{\overline {B_{3R}}}u>0,
    \qquad
    \|f_j-u_j^p\|_{L^\infty(B_{3R})}\longrightarrow0.
\]
This finishes our proof of Lemma \ref{Lem:4.2}.
}
\end{proof}

\subsection{The estimate for the approximating solutions}

Fix \(M,\tau\) by \eqref{eq:3.27}--\eqref{eq:3.28}, and recall that $P=M+p+\frac{p}{k}$.

\begin{proposition}
\label{Prop:4.3}
Let \(\chi=\chi_R\) be given by \eqref{eq:3.2} and let
\(\ell\in\mathbb N\), \(\ell>2k+2\).  Then
\begin{align}
    \int_{\R^n}u^P\chi^\ell\dd x
    \leq{}
    CR^{-2}\int_{\R^n}u^{M+p+1}\chi^{\ell-2}\dd x
    +CR^{-2k-2}\int_{\R^n}
       u^{M+k+1}\chi^{\ell-2k-2}\dd x,
    \label{eq:4.3}
\end{align}
where \(C\) is independent of \(R\).
\end{proposition}

\begin{proof}
Fix \(R\), take the approximate sequence in
Lemma~\ref{Lem:4.2}, and write
\(A_j=-D^2u_j, \qquad \sigma_k(A_{j})=f_j=u_j^p+r_j\),
and similar to \eqref{eq:3.1} and \eqref{eq:3.36}, let us define the key functions
\[
    \mathcal J_j
    =
    -u_j^ML_k(A_{j})\nabla u_j
    -\tau u_j^{M-1}|\nabla u_j|^2T_{k-1}(A_{j})\nabla u_j.
\]
Then we have
\begin{align}
 \operatorname{div}\mathcal J_j
 ={}&
    u_j^M\tr\bigl(L_k(A_{j})A_{j}\bigr)
    +(2\tau-M)u_j^{M-1}\nabla u_j^TL_k(A_{j})\nabla u_j
    \notag\\
 &+\tau(1-M)u_j^{M-2}|\nabla u_j|^2
      \nabla u_j^TT_{k-1}(A_{j})\nabla u_j
    \notag\\
 &-\frac{n-k}{n}u_j^M\nabla\sigma_k(A_{j})\cdot\nabla u_j
   +\frac{k(n+2)}n\tau u_j^{M-1}
      \sigma_k(A_{j})|\nabla u_j|^2.
\end{align}
Hence
\begin{align}
 \operatorname{div}\mathcal J_j
 ={}&
    u_j^M\tr\bigl(L_k(A_{j})A_{j}\bigr)
    +(2\tau-M)u_j^{M-1}\nabla u_j^TL_k(A_{j})\nabla u_j
    \notag\\
 &+\tau(1-M)u_j^{M-2}|\nabla u_j|^2
      \nabla u_j^TT_{k-1}(A_{j})\nabla u_j
    \notag\\
 &+\left(\frac{k(n+2)}n\tau-\frac{n-k}{n}p\right)
      u_j^{M+p-1}|\nabla u_j|^2
    \notag\\
 &-\frac{n-k}{n}u_j^M\nabla r_j\cdot\nabla u_j
   +\frac{k(n+2)}n\tau u_j^{M-1}r_j|\nabla u_j|^2.
\end{align}
Arguing as in \eqref{eq:3.40},
\eqref{eq:3.41}, and \eqref{eq:3.42}, we obtain constants
\(\kappa_0,\kappa_1>0\) independent of \(j,R\) such that
\begin{align}
 \operatorname{div}\mathcal J_j
 \geq{}&
 \kappa_0\left(
      u_j^M\tr\bigl(L_k(A_{j})A_{j}\bigr)
      +u_j^{M-2}|\nabla u_j|^2
       \nabla u_j^TT_{k-1}(A_{j})\nabla u_j\right)
 \notag\\
 &+\kappa_1u_j^{M+p-1}|\nabla u_j|^2
 -\frac{n-k}{n}u_j^M\nabla r_j\cdot\nabla u_j
 +\frac{k(n+2)}n\tau u_j^{M-1}r_j|\nabla u_j|^2.
 \label{eq:4.4}
\end{align}

By \eqref{eq:4.2}, for every fixed \(a\in\R\), one has
\begin{equation}
    u_j^a\longrightarrow u^a
    \quad\text{uniformly on }\overline B_{3R}.
    \label{eq:4.5}
\end{equation}
Since \(M<0\) and \(M\neq-1\), we have
\begin{align}
    \dd\nu_j
    &:=
    \left(u_j^M\sigma_1(A_j)
    -M u_j^{M-1}|\nabla u_j|^2\right)\dd x
    \geq(-M)u_j^{M-1}|\nabla u_j|^2\,\dd x,
    \label{eq:4.6}\\
    \nu_j
    &=\operatorname{div}(-u_j^M\nabla u_j)
      =-\frac1{M+1}\Delta(u_j^{M+1})
      \quad\text{in }\mathcal D'(B_{3R}),
    \notag\\
    \nu_j&\longrightarrow
      -\frac1{M+1}\Delta(u^{M+1})
      \quad\text{in }\mathcal D'(B_{3R}), \qquad \text{as} \,\, j\rightarrow+\infty.
\end{align}
If \(K\Subset B_{3R}\) and
\(\eta\in C_c^\infty(B_{3R})\), \(\eta\geq0\) and
\(\eta\geq1\) on \(K\), then
\begin{equation}
    \nu_j(K)
    \leq\langle\nu_j,\eta\rangle
    =-\frac1{M+1}\int_{{B_{3R}}} u_j^{M+1}\Delta\eta\dd x
    \longrightarrow
    -\frac1{M+1}\int_{\R^n} u^{M+1}\Delta\eta\dd x.
\end{equation}
For every \(\zeta\in C_c^\infty(B_{2R})\), integration by parts and
\eqref{eq:4.6} give
\begin{align}
 -\int_{{B_{3R}}}\zeta u_j^M\nabla r_j\cdot\nabla u_j\dd x
 =-\int_{{B_{3R}}} r_j\zeta\,\dd\nu_j
   +\int_{{B_{3R}}} r_ju_j^M\nabla u_j\cdot\nabla\zeta\dd x,
 \end{align}
 and hence
 \begin{align}\label{eq:4.7}
 \left|-\int_{{B_{3R}}}\zeta u_j^M\nabla r_j\cdot\nabla u_j\dd x\right|
 &\leq\|r_j\|_\infty\Bigg[
 \|\zeta\|_\infty\nu_j(\operatorname{supp}\zeta)\nonumber
 \\
 &+\frac{\nu_j(\operatorname{supp}\nabla\zeta)^{1/2}}{\sqrt{-M}}
 \left(\int_{{B_{3R}}} u_j^{M+1}|\nabla\zeta|^2\dd x\right)^{1/2}
 \Bigg]=o_j(1),
 \end{align}
 and
 \begin{align}
 \left|\int_{{B_{3R}}}\zeta u_j^{M-1}r_j|\nabla u_j|^2\dd x\right|
 \leq
 \frac{\|r_j\|_\infty\|\zeta\|_\infty}{-M}
 \nu_j(\operatorname{supp}\zeta)=o_j(1).
 \label{eq:4.8}
\end{align}

Now, test \eqref{eq:4.4} with \(\chi^\ell\). {Using the estimate \eqref{eq:3.44} with $u=u_j$ and $A=A_j$, and arguing as in the derivation of \eqref{eq:3.45},} we can deduce from \eqref{eq:4.4}, \eqref{eq:4.7} and \eqref{eq:4.8} that
\begin{equation}
    \int_{{B_{3R}}} u_j^{M+p-1}|\nabla u_j|^2\chi^\ell\dd x
    \leq
    CR^{-2}\int_{{B_{3R}}}\sigma_{k-1}(A_j)|\nabla u_j|^2
       u_j^M\chi^{\ell-2}\dd x+o_j(1).
    \label{eq:4.9}
\end{equation}
Set
\[
    \delta=-M-1>-1,
    \qquad \delta\neq0,
    \qquad \theta=\ell-2>2k.
\]
Using the notation \eqref{eq:3.3}--\eqref{eq:3.6}, we can derive the following identity {in a similar way to} \eqref{eq:3.17}:
\begin{align}
    k\int_{{B_{3R}}} f_ju_j^{-\delta}\chi^\theta\dd x
    =-\sum_{s=1}^kb_s\Lambda_{s,j}+\sum_{s=1}^kc_sE_{s,j},
    \label{eq:4.10}
\end{align}
where $\Lambda_{s,j}$ and $E_{s,j}$ {are defined analogously to} $\Lambda_s$ and $E_s$ by replacing $u$ and $A$ with $u_j$ and $A_j$. Moreover, we have
\begin{align}
    \left|\int_{{B_{3R}}} r_ju_j^{-\delta}\chi^\theta\dd x\right|
    \leq C_R\|r_j\|_{L^\infty(B_{3R})}=o_j(1).
    \label{eq:4.11}
\end{align}
To record the approximation error in the Newton recursion, define
\(M_{s,j}\) by \eqref{eq:3.6} with \(u,A\) replaced by
\(u_j,A_j\).  Equations \eqref{eq:3.8}--\eqref{eq:3.10} use only
\(A_j=-D^2u_j\) and \(\partial_\ell T_{k-s-1}(A_j)_{i\ell}=0\).
They give, for \(1\leq s\leq k-1\),
\begin{equation}\label{eq:4.12}
 M_{s,j}=\frac{k+s}{2s}\Lambda_{s,j}
 +\frac{\delta+s}{2s}M_{s+1,j}
 -\frac{\theta}{2s}E_{s+1,j}.
\end{equation}
The only occurrence of the equation is the first step
\begin{equation}\label{eq:4.13}
 k\int_{{B_{3R}}}f_ju_j^{-\delta}\chi^\theta\dd x
 =-\delta M_{1,j}+\theta E_{1,j}.
\end{equation}
Multiplying \eqref{eq:4.12} by the coefficients in
\eqref{eq:3.12}, summing in \(s\), and inserting the result in
\eqref{eq:4.13} gives \eqref{eq:4.10}.  Thus the recursion contains
only \(\int r_ju_j^{-\delta}\chi^\theta\), which is \(o_j(1)\) by
\eqref{eq:4.11}.  No derivative of \(r_j\) occurs in
\eqref{eq:4.12}--\eqref{eq:4.13}.  The estimates
\eqref{eq:3.18}--{\eqref{eq:3.26}}, applied to \eqref{eq:4.10}, now give
\begin{align}\label{eq:4.14}
    \Lambda_{1,j}&=\int_{{B_{3R}}}\sigma_{k-1}(A_j)|\nabla u_j|^2
       u_j^M\chi^{\ell-2}\dd x \\
    &\leq
    C\int_{{B_{3R}}} u_j^{M+p+1}\chi^{\ell-2}\dd x
    +CR^{-2k}\int_{{B_{3R}}} u_j^{M+k+1}\chi^{\ell-2k-2}\dd x
    +o_j(1). \nonumber
\end{align}
For \(\delta>0\), the first term on the right is simply added to
\eqref{eq:3.4}.  Combining
\eqref{eq:4.9} and
\eqref{eq:4.14} yields
\begin{equation}\label{eq:4.15}
    \int_{{B_{3R}}} u_j^{M+p-1}|\nabla u_j|^2\chi^\ell\dd x
    \leq CR^{-2}\int_{{B_{3R}}} u_j^{M+p+1}\chi^{\ell-2}\dd x
    +CR^{-2k-2}\int_{{B_{3R}}} u_j^{M+k+1}\chi^{\ell-2k-2}\dd x
    +o_j(1).
\end{equation}

Uniform convergence in \eqref{eq:4.2} gives
\(f_j\geq\frac12u_j^p\) on \(\operatorname{supp}\chi\) for all
large \(j\).  Hence Newton--Maclaurin and
\eqref{eq:3.32}--\eqref{eq:3.34} {yield}
\begin{align}\label{eq:4.16}
    -\Delta u_j
    =\sigma_1(A_j)\geq c_{n,k}f_j^{1/k}
      \geq c'_{n,k}u_j^{p/k}
      \quad\text{on }\operatorname{supp}\chi.
\end{align}
Multiplying \eqref{eq:4.16} by \(u_j^{M+p}\chi^\ell\) and integrating by parts, we derive
\begin{align*}
    c'_{n,k}\int u_j^{M+p+\frac{p}{k}}\chi^\ell\dd x
    &\leq
    \int(-\Delta u_j)u_j^{M+p}\chi^\ell\dd x\\
    &=
    (M+p)\int u_j^{M+p-1}|\nabla u_j|^2\chi^\ell\dd x
    +
    \ell\int u_j^{M+p}\chi^{\ell-1}
    \nabla u_j\cdot \nabla\chi\dd x.
\end{align*}
By Young's inequality and \(|\nabla\chi|\leq C/R\), we get
\begin{align*}
    \left|\ell\int u_j^{M+p}\chi^{\ell-1}
    \nabla u_j\cdot \nabla\chi\dd x\right|
    &\leq
    C\int u_j^{M+p-1}|\nabla u_j|^2\chi^\ell\dd x+
    CR^{-2}\int u_j^{M+p+1}\chi^{\ell-2}\dd x.
\end{align*}
Combining the above two inequalities, we obtain
\begin{equation}
  \int_{{B_{3R}}} u_j^P\chi^\ell\dd x
    \leq C\int_{{B_{3R}}} u_j^{M+p-1}|\nabla u_j|^2\chi^\ell\dd x
      +CR^{-2}\int_{{B_{3R}}} u_j^{M+p+1}\chi^{\ell-2}\dd x.
    \label{eq:4.17}
\end{equation}
Equations \eqref{eq:4.15} and \eqref{eq:4.17} contain only powers of
\(u_j\). No gradient or Newton-tensor product is passed to the limit.
Equation \eqref{eq:4.5} therefore gives \eqref{eq:4.3} term by term. The
constants which survive after \(j\to\infty\) are those in the
corresponding estimates of Section~\ref{Sec:3} and are
independent of \(R\). Constants depending on the fixed \(R\) exist only
in \(o_j(1)\). This concludes our proof of Proposition \ref{Prop:4.3}.
\end{proof}

\subsection{Proof of Theorem~\ref{Thm:1.2}}

\begin{lemma}
\label{Lem:4.4}
For every \(R_0>0\), we have
\begin{equation}
    u(x)\geq
    \left(\min_{|y|=R_0}u(y)\right)
    R_0^{\frac{n-2k}{k}}|x|^{-\frac{n-2k}{k}}
    \qquad(|x|\geq R_0).
    \label{eq:4.18}
\end{equation}
\end{lemma}

\begin{proof}
Set
\[
    m_0=\min_{|y|=R_0}u(y)>0
    \qquad\mathrm{and}\qquad
    \rho=\frac{n-2k}{k}.
\]
For \(S>R_0\), let \(h_S\) be the function in \eqref{eq:3.48}.  By
\eqref{eq:3.49}, we have
\[
    -h_S\in\Phi^k(B_S\setminus\overline B_{R_0}),
    \qquad\mathrm{and}\qquad
    \mu_k[-h_S]=0.
\]
Moreover,
\[
    h_S(R_0)=m_0\leq u,
    \qquad
    h_S(S)=0\leq u,
    \qquad
    \mu_k[-u]=u^p\,\dd x\geq0=\mu_k[-h_S].
\]
Lemma~\ref{Lem:4.1} implies
\(u\geq h_S\) on \(B_S\setminus\overline B_{R_0}\).  Letting \(S\to\infty\) we obtain
\eqref{eq:4.18}.
\end{proof}

\begin{proof}[Proof of Theorem~\ref{Thm:1.2}]
Assume \(u\not\equiv0\). Then \eqref{eq:4.1} implies that $u>0$ in $\R^{n}$. Choose \(M,\tau\) as
\eqref{eq:3.27}--\eqref{eq:3.28}, and recall that $P=M+p+\frac{p}{k}$. By \eqref{eq:3.29}, we have
\begin{equation}
    M+p+1>0,
    \quad
    P-(M+p+1)=\frac{p-k}{k},
    \quad
    P-(M+k+1)=\frac{(k+1)(p-k)}{k},
    \quad
    n-\frac{2kP}{p-k}<0.
\end{equation}

We will discuss three different cases.

\smallskip

If $M+k+1>0$, applying \eqref{eq:3.54}--\eqref{eq:3.56} to
\eqref{eq:4.3}, one can obtain
\begin{equation}
    \int_{B_R}u^P\dd x
    \leq CR^{\,n-\frac{2kP}{p-k}}\longrightarrow0.
    \label{eq:4.19}
\end{equation}

If \(M+k+1=0\), choose \(\ell\) to satisfy the first two
inequalities in \eqref{eq:3.52}.  Then,
\eqref{eq:3.54},
\eqref{eq:3.58} together with
\(\frac{2kP}{p-k}=2k+2\)
imply \eqref{eq:4.19}.

If \(M+k+1<0\), choose \(\ell\) satisfying the first two
inequalities in \eqref{eq:3.52}.  Then
\[
    n-2k\geq2 \quad
    \Longrightarrow
    \quad M>-\frac{2k}{n-2k}\geq-k \quad
    \Longrightarrow \quad
    M+k+1>1,
\]
which implies \(n=2k+1\).  By
\eqref{eq:4.1} and
\eqref{eq:4.18}, for every fixed \(R_0>0\), we have
\[
    \int_{B_{R_0}}u^{M+k+1}\dd x<\infty,
    \qquad
    u(x)^{M+k+1}\leq C|x|^{-\frac{M+k+1}{k}},
    \quad  \forall \,\, |x|\geq R_0.
\]
Therefore, for \(R\geq R_0\), we have
\begin{equation}\label{eq:4.20}
    R^{-2k-2}\int u^{M+k+1}\chi^{\ell-2k-2}\dd x
    \leq C_{R_0}R^{-2k-2}
      +CR^{-2k-2+n-\frac{M+k+1}{k}}
    =C_{R_0}R^{-2k-2}+CR^{-1-\frac{M+k+1}{k}}.
\end{equation}
Since \(n=2k+1\), one has
\[
    \tau<k+1,
    \qquad
    M=-\frac{2k\tau}{k+1}>-2k,
    \qquad
    M+k+1>1-k,
    \qquad
    -1-\frac{M+k+1}{k}<-\frac1k<0.
\]
Applying \eqref{eq:3.54} to the first term on the right-hand side of \eqref{eq:4.3} and using
\eqref{eq:4.20}, we obtain
\[
    \int_{B_R}u^P\dd x
    \leq
    CR^{\,n-\frac{2kP}{p-k}}
    +C_{R_0}R^{-2k-2}
    +CR^{-1-\frac{M+k+1}{k}}
    \longrightarrow0,
    \qquad\,\,\text{as} \,\, R\rightarrow+\infty.
\]
Thus, in all the above three cases, we have
\(0<\int_{B_1}u^P\dd x \leq\int_{B_R}u^P\dd x \longrightarrow0\),
which is absurd. This concludes our proof of Theorem \ref{Thm:1.2}.
\end{proof}

\section{Unconditional classification results for
\texorpdfstring{\(2k<n\leq4k\)}{2k<n<=4k}}
\label{Sec:5}
In this section, we establish the unconditional classification results in Theorem~\ref{Thm:1.6} for $2k<n\leq 4k$.  Wang and
Lei~\cite{WL2019} classified the radial solutions, while Fang, Ma and
Wei~\cite[Section~1]{FMW2023} listed the  classification problem without assuming radial symmetry as an open   question. For all $n>2k$ with $k\geq2$, Theorem~\ref{Thm:1.6} solves this
problem completely for all $C^2$ positive entire  solutions of \eqref{eq:1.8}.

{If \(u\equiv0\), the conclusion is immediate. Henceforth, assume that \(u\not\equiv0\). Lemma~\ref{Lem:2.2} then gives \(u>0\) and \(-D^2u\in\Gamma_k\).}

The proof in Section~\ref{Sec:3} cannot be applied
directly when \(p=p_*\), since the two strict inequalities \eqref{eq:3.40} and
\eqref{eq:3.41} essentially used in
Section~\ref{Sec:3} become equalities now.  Inspired by the argument in Section~\ref{Sec:3}, we first
construct a vector field $\mathcal J$ with nonnegative divergence and prove
\eqref{eq:5.5}.  When \(2k<n<4k\), the negative power
corresponding to \(M<-1\) makes the cutoff terms tend to zero.  At
\(n=4k\), this power disappears, so we introduce a positive function
\(h\) related to the $k$-Hessian equation and consider the weighted vector field \(h^{-\varepsilon}\mathcal J\).  In both cases, the nonnegativity of \(\operatorname{div}\mathcal J\)
and the cutoff estimates force
\(\operatorname{div}\mathcal J\equiv0\).  The equality cases in
\eqref{eq:2.20} and \eqref{eq:2.12} then imply
\(D^2\!\left(u^{-2k/(n-2k)}\right)=\lambda(x)\Id\).
Differentiating this identity shows that \(\lambda\) is constant, and
hence gives the standard bubble form of \(u\).

\subsection{The   estimates for divergence}
\label{Subsec:5.1}

In the critical case \(p=p_*\), the two endpoints in
\eqref{eq:3.27} coincide, which forces
\begin{equation}\label{eq:5.1}
    \tau=\frac{n-k}{n-2k}\quad \text{and}\ M=-\frac{2k}{n-2k}.
\end{equation}
Then \(M<-1\) for \(2k<n<4k\) and \(M=-1\) for \(n=4k\).
Throughout this section, we fix
\(\delta=-M-1=\frac{4k-n}{n-2k}\).
Inspired by the identity \eqref{eq:3.39}, we define the following vector field
\begin{equation}\label{eq:5.2}
    \mathcal J
    :=
    -u^{-2k/(n-2k)}L_k(A)\nabla u
    -\frac{n-k}{n-2k}
       u^{-n/(n-2k)}|\nabla u|^2T_{k-1}(A)\nabla u.
\end{equation}

Substituting \eqref{eq:3.39}, we find
\begin{align}
    \operatorname{div}\mathcal J
    =u^{-2k/(n-2k)}\Bigl[&
    \tr\bigl(L_k(A)A\bigr)
    +\frac{2n}{n-2k}u^{-1}
       \nabla u^TL_k(A)\nabla u
    \notag\\
    &+\frac{n(n-k)}{(n-2k)^2}u^{-2}|\nabla u|^2
       \nabla u^TT_{k-1}(A)\nabla u
    \Bigr].
    \label{eq:5.3}
\end{align}
Taking \(X=Y=\nabla u\) in \eqref{eq:2.20}, we obtain
\[
    \nabla u^TL_k(A)\nabla u
    \geq
    -\left(
       \frac{n-k}{n}\,
       \tr\bigl(L_k(A)A\bigr)|\nabla u|^2
       \nabla u^TT_{k-1}(A)\nabla u
     \right)^{1/2}.
\]
Using this inequality in \eqref{eq:5.3}, we obtain
\begin{equation}\label{eq:5.4}
     \operatorname{div}\mathcal J
 \geq u^{-2k/(n-2k)}\left(
   \sqrt{\tr\bigl(L_k(A)A\bigr)}
    -\frac{\sqrt{n(n-k)}}{n-2k}\,
       u^{-1}|\nabla u|
       \bigl(\nabla u^TT_{k-1}(A)\nabla u\bigr)^{1/2}
\right)^2\geq0.
\end{equation}

\begin{proposition}
\label{Prop:5.1}
For every \(Y\in\R^n\),
\begin{equation}\label{eq:5.5}
    |\mathcal J\cdot Y|^2
    \leq
    \frac{n-k}{n}\,
    u^{-2k/(n-2k)}|\nabla u|^2
    \operatorname{div}\mathcal J\,
    Y^TT_{k-1}(A)Y.
\end{equation}
\end{proposition}

\begin{proof}
By Corollary~\ref{Cor:2.3}, we know that \(T_{k-1}(A)>0\).
The Cauchy--Schwarz inequality yields
\[
    |\mathcal J\cdot Y|^2
    \leq
    \bigl(\mathcal J^TT_{k-1}(A)^{-1}\mathcal J\bigr)
    \bigl(Y^TT_{k-1}(A)Y\bigr).
\]

It follows from \eqref{eq:2.12}
that
\[
    L_k(A)T_{k-1}(A)^{-1}L_k(A)
    \preceq
    \frac{n-k}{n}\,
    \tr\bigl(L_k(A)A\bigr)\Id.
\]
Therefore, by direct calculations, we obtain
\begin{equation}\label{eq:5.6}
\begin{aligned}
&\Bigl(
  L_k(A)\nabla u
  +\frac{n-k}{n-2k}u^{-1}|\nabla u|^2T_{k-1}(A)\nabla u
 \Bigr)^T
 T_{k-1}(A)^{-1}
\Bigl(
  L_k(A)\nabla u
  +\frac{n-k}{n-2k}u^{-1}|\nabla u|^2T_{k-1}(A)\nabla u
 \Bigr)
\\
&={}
 \nabla u^TL_k(A)T_{k-1}(A)^{-1}L_k(A)\nabla u
 +\frac{2(n-k)}{n-2k}u^{-1}|\nabla u|^2
    \nabla u^TL_k(A)\nabla u
 \\
&\qquad
 +\frac{(n-k)^2}{(n-2k)^2}u^{-2}|\nabla u|^4
    \nabla u^TT_{k-1}(A)\nabla u
 \\
&\leq
 \frac{n-k}{n}|\nabla u|^2
    \tr\bigl(L_k(A)A\bigr)
 +\frac{2(n-k)}{n-2k}u^{-1}|\nabla u|^2
    \nabla u^TL_k(A)\nabla u
 \\
&\qquad
 +\frac{(n-k)^2}{(n-2k)^2}u^{-2}|\nabla u|^4
    \nabla u^TT_{k-1}(A)\nabla u
 \\
&=
\frac{n-k}{n}|\nabla u|^2
\Bigl[
    \tr\bigl(L_k(A)A\bigr)
    +\frac{2n}{n-2k}u^{-1}\nabla u^TL_k(A)\nabla u
    +\frac{n(n-k)}{(n-2k)^2}u^{-2}|\nabla u|^2
       \nabla u^TT_{k-1}(A)\nabla u
\Bigr].
\end{aligned}
\end{equation}
By using \eqref{eq:5.2},    multiplying \eqref{eq:5.6} by
\(u^{-4k/(n-2k)}\) and using \eqref{eq:5.3}, we get
\[
 \mathcal J^TT_{k-1}(A)^{-1}\mathcal J
 \leq
 \frac{n-k}{n}u^{-2k/(n-2k)}|\nabla u|^2
 \operatorname{div}\mathcal J.
\]
Together with the Cauchy--Schwarz inequality above, this proves
\eqref{eq:5.5}.
\end{proof}

\begin{proposition}
\label{Prop:5.2}
Let \(0\leq\eta\leq1\) be compactly supported and Lipschitz, and let
\(\ell>0\).  Then
\begin{equation}\label{eq:5.7}
    \int_{\R^n}\eta^{\ell+2}
       \operatorname{div}\mathcal J\dd x
    \leq
    C
    \int_{\R^n}
    \eta^\ell|\nabla\eta|^2
    \sigma_{k-1}(A)|\nabla u|^2u^{-2k/(n-2k)}\dd x.
\end{equation}
The constant depends only on \(n,k\), and \(\ell\).
\end{proposition}

\begin{proof}
We first assume that \(\eta\) is smooth.  Taking \(Y=\nabla\eta\) in
\eqref{eq:5.5}, we obtain
\[
 |\mathcal J\cdot\nabla\eta|^2
 \leq
 \frac{n-k}{n}u^{-2k/(n-2k)}|\nabla u|^2
 \operatorname{div}\mathcal J\,
 \nabla\eta^TT_{k-1}(A)\nabla\eta.
\]
Using integration by parts and the above inequality, we have
\begin{align*}
      &\int\eta^{\ell+2}\operatorname{div}\mathcal J\dd x
      =-(\ell+2)\int\eta^{\ell+1}\mathcal J\cdot\nabla\eta\dd x\\
      &\leq
    (\ell+2)\sqrt{\frac{n-k}{n}}
    \int
    \left(\eta^{\ell+2}\operatorname{div}\mathcal J\right)^{1/2}
    \left(\eta^\ell u^{-2k/(n-2k)}|\nabla u|^2
       \nabla\eta^TT_{k-1}(A)\nabla\eta\right)^{1/2}\dd x\\
      &\leq
    \frac12\int\eta^{\ell+2}\operatorname{div}\mathcal J\dd x
    +C\int\eta^\ell u^{-2k/(n-2k)}|\nabla u|^2
       \nabla\eta^TT_{k-1}(A)\nabla\eta\dd x.
\end{align*}
Thus, one gets
\begin{align*}
 \int\eta^{\ell+2}\operatorname{div}\mathcal J\dd x
 \leq C\int\eta^\ell u^{-2k/(n-2k)}|\nabla u|^2
       \nabla\eta^TT_{k-1}(A)\nabla\eta\dd x.
\end{align*}
Since \(T_{k-1}(A)>0\), we have
\[
    \nabla\eta^TT_{k-1}(A)\nabla\eta
    \leq
    \tr\bigl(T_{k-1}(A)\bigr)|\nabla\eta|^2
    =(n-k+1)\sigma_{k-1}(A)|\nabla\eta|^2.
\]
This proves \eqref{eq:5.7} for smooth \(\eta\).
The calculation is first justified for compactly supported smooth
functions.  The Lipschitz case follows by compactly supported smooth
approximation converging in \(W^{1,2}\) and uniformly.
\end{proof}

We aim to show that \(\operatorname{div}\mathcal J\equiv0\), which is the key property to show the desired classification result.  We divide the proof into two
cases.
\subsection{\texorpdfstring{The proof of \(\operatorname{div}\mathcal J\equiv0\) for case \(2k<n<4k\)}
{The case 2k<n<4k}}

We first consider the case \(2k<n<4k\), where   the exponent of the last term in \eqref{eq:5.7} satisfies
\begin{equation}\label{eq:5.8}
    2k/(n-2k)=1+\delta, \quad \delta=\frac{4k-n}{n-2k}>0,
\end{equation}
which enables us to use \eqref{eq:3.4}.  For every \(R>0\), we
take the function \(\chi\) in \eqref{eq:3.2}.  For every fixed
\(\ell>0\), Proposition~\ref{Prop:5.2} gives
\begin{equation}\label{eq:5.9}
    \int_{\R^n}\chi^{\ell+2}\operatorname{div}\mathcal J\dd x
    \leq
    CR^{-2}\int_{\R^n}
    \sigma_{k-1}(A)|\nabla u|^2
       u^{-1-\delta}\chi^\ell\dd x.
\end{equation}
Here, we have used \(|\nabla\chi|\leq C/R\).  Since
\(-\delta-1=-2k/(n-2k)\) by \eqref{eq:5.8}, the integral on the
right-hand side of \eqref{eq:5.9} is the
\(\Lambda_1\)-term in \eqref{eq:3.4} when \(\ell=4k\).  We
estimate this integral in the following proposition.
\begin{proposition}
\label{Prop:5.3}
For all sufficiently large \(R\),
\begin{equation}\label{eq:5.10}
\int_{\R^n}
    \sigma_{k-1}(A)|\nabla u|^2
    u^{-\delta-1}\chi^{4k}\dd x
    \leq
    \begin{cases}
        CR^{2-1/k},&n=2k+1,\\[1mm]
        CR^{2k/(k+1)},&2k+2\leq n<4k.
    \end{cases}
\end{equation}
Here \(C\) may depend on \(n,k,u\) and the fixed cutoff profile, but
is independent of \(R\).
\end{proposition}

\begin{proof}
We consider the two cases $n=2k+1$ and $n\geq 2k+2$ separately.

\textit{Case 1. \(n=2k+1\).}
In this case, \(\delta=2k-1\) and \(k-\delta=1-k\).  Using
\eqref{eq:3.47} with \(R_0=1\), we have
\(u(x)\geq C|x|^{-1/k}, \qquad |x|>1\).
Thus,
\(u(x)^{1-k}\leq C|x|^{(k-1)/k}\).
On the unit ball, \(u^{1-k}\) is bounded since \(u>0\) has a positive
minimum there.  Therefore, using \(\supp\chi\subset B_{2R}\), we obtain
\begin{align*}
    R^{-2k}\int u^{1-k}\chi^{2k}\dd x
    &\leq
    CR^{-2k}
    +CR^{-2k}\int_{1}^{2R}r^{n-1+(k-1)/k}\dd r
    \leq CR^{2-1/k},
\end{align*}
Note that all the terms on the left-hand side of
\eqref{eq:3.4} are nonnegative.  Thus the above estimate and
\eqref{eq:3.4} imply
\[
\int_{\R^n}
    \sigma_{k-1}(A)|\nabla u|^2
    u^{-\delta-1}\chi^{4k}\dd x
    =\Lambda_1
    \leq C R^{-2k}\int u^{1-k}\chi^{2k}\dd x
    \leq CR^{2-1/k}.
\]
This proves the first estimate in \eqref{eq:5.10}.

\bigskip
\textit{Case 2. \(2k+2\leq n<4k\).}
Using \eqref{eq:5.8}, we have
\(
    k-\delta
     =\frac{(k+1)n-2k(k+2)}{n-2k}\geq1\).

Young's inequality gives
\begin{align}
    R^{-2k}u^{k-\delta}\chi^{2k}
    \leq{}
    \varepsilon u^{p_*-\delta}\chi^{4k}
     +C_\varepsilon
    R^{-2k(p_*-\delta)/(p_*-k)}
    \chi^{[2k(p_*-\delta)-4k(k-\delta)]/(p_*-k)}.
    \label{eq:5.11}
\end{align}
The exponent of \(\chi\) in the second term is nonnegative.  Since
\(0\leq\chi\leq1\) and \(\supp\chi\subset B_{2R}\), integrating
\eqref{eq:5.11} gives
\begin{align}
    R^{-2k}\int u^{k-\delta}\chi^{2k}\dd x
    \leq{}&
    \varepsilon\int u^{p_*-\delta}\chi^{4k}\dd x
    +C_\varepsilon R^{2k/(k+1)}.
    \label{eq:5.12}
\end{align}
We choose \(\varepsilon\) sufficiently small and absorb
\(\varepsilon\int u^{p_*-\delta}\chi^{4k}\dd x\) in
\eqref{eq:5.12} into the first term on the left of
\eqref{eq:3.4}.  The nonnegativity of the remaining terms then
gives
\[
\int_{\R^n}
    \sigma_{k-1}(A)|\nabla u|^2
    u^{-\delta-1}\chi^{4k}\dd x
    {=\Lambda_1}
    \leq {C}R^{-2k}\int u^{k-\delta}\chi^{2k}\dd x
    \leq CR^{2k/(k+1)}.
\]
This proves the second estimate in
\eqref{eq:5.10}.
\end{proof}

With the help of the estimates in Proposition~\ref{Prop:5.3}, we can now
show that the vector field \(\mathcal J\) is divergence-free.
\begin{corollary}
\label{Cor:5.4}
Under the hypotheses of Theorem~\ref{Thm:1.6}, {assume in addition that \(u\not\equiv0\) and}
\(2k<n<4k\){. Then}
\begin{equation}\label{eq:5.13}
    \operatorname{div}\mathcal J\equiv0.
\end{equation}
\end{corollary}

\begin{proof}
Taking \(\ell=4k\) in \eqref{eq:5.9}, using \(\chi=1\) on
\(B_R\), and applying Proposition~\ref{Prop:5.3}, we obtain
\[
    \int_{B_R}\operatorname{div}\mathcal J\dd x
    \leq
    \begin{cases}
        CR^{-1/k},&n=2k+1,\\[1mm]
        CR^{-2/(k+1)},&2k+2\leq n<4k.
    \end{cases}.
\]
Fix arbitrarily \(r>0\).  For every \(R>r\), the nonnegativity in
\eqref{eq:5.4} gives
\(0\leq\int_{B_r}\operatorname{div}\mathcal J\dd x \leq\int_{B_R}\operatorname{div}\mathcal J\dd x\).
Letting \(R\to\infty\) in the preceding estimate, we obtain
\(\int_{B_r}\operatorname{div}\mathcal J\dd x=0\).  Since \(r>0\) is
arbitrary and \(\operatorname{div}\mathcal J\geq0\), we conclude that
\eqref{eq:5.13} holds.
\end{proof}
\subsection{\texorpdfstring{The proof of
\(\operatorname{div}\mathcal J\equiv0\) for \(n=4k\)}
{The proof of div J=0 for n=4k}}
\label{Subsec:5.3}

We next consider the case \(n=4k\), where
\(M=-1\) and \(\delta=0\). Then   the positive iteration used in
the preceding subsection degenerates.  We overcome this new difficulty by introducing a carefully chosen weighted function.  Note that \(p_*=2k+1\), and
\eqref{eq:5.1} gives \(\tau=3/2\).  Therefore, equations
\eqref{eq:5.2}, \eqref{eq:5.3}, and
\eqref{eq:5.4} become
\begin{equation}\label{eq:5.14}
\begin{aligned}
    \mathcal J={}&-u^{-1}L_k(A)\nabla u
    -\frac32u^{-2}|\nabla u|^2T_{k-1}(A)\nabla u,\\
    \operatorname{div}\mathcal J={}&u^{-1}\Bigl[
    \tr\bigl(L_k(A)A\bigr)
    +4u^{-1}\nabla u^TL_k(A)\nabla u
    +3u^{-2}|\nabla u|^2
       \nabla u^TT_{k-1}(A)\nabla u\Bigr]\geq0.
\end{aligned}
\end{equation}
Proposition~\ref{Prop:5.1} also yields
\begin{equation}\label{eq:5.15}
    |\mathcal J\cdot Y|^2
    \leq\frac34u^{-1}|\nabla u|^2
       \operatorname{div}\mathcal J\,
       Y^TT_{k-1}(A)Y.
\end{equation}

Since \(\delta=-M-1=0\), \eqref{eq:3.4} cannot be applied directly
to the following term on the right-hand side of \eqref{eq:5.7}:
\(\sigma_{k-1}(A)|\nabla u|^2u^{-1}\).
The missing estimate would be recovered by introducing the
stronger weight \(u^{-1-\widetilde\delta}\), with some
\(\widetilde\delta>0\), in place of \(u^{-1}\). Our key idea is to replace the vector $\mathcal J$ by $h^{-\varepsilon} \mathcal J$  with the weighted function $h$ properly defined. 
Now we find the weighted function in
\(h^{-\varepsilon}\mathcal J\) as follows.  The product rule gives
\begin{equation}\label{eq:5.16}
    \operatorname{div}(h^{-\varepsilon}\mathcal J)
    =h^{-\varepsilon}
    \bigl(\operatorname{div}\mathcal J
       -\varepsilon\mathcal J\cdot\nabla\log h\bigr).
\end{equation}
It is enough to find a positive function \(h\) such that
\begin{equation}\label{eq:5.17}
    h\geq c_k u^{1/k},
    \qquad
    |\mathcal J\cdot\nabla\log h|
    \leq C_k\operatorname{div}\mathcal J.
\end{equation}
We now define \(h\).  Set \(w=u^{-1}\).  For a standard bubble,
\(w=a_0+b_0|x-x_0|^2, \qquad D^2w=2b_0\Id\).
We want \(h=b_0\) for this function.  Since
\[
    \nabla(4wh-|\nabla w|^2)
    =4w\nabla h-2(D^2w-2h\Id)\nabla w,
\]
we take \(h^{k-1}(4wh-|\nabla w|^2)\) to be constant.  Using
\eqref{eq:1.9}, we define \(h>0\) by
\begin{equation}\label{eq:5.18}
    \binom{4k}{k}2^{k-2}h^{k-1}
    \left(4u^{-1}h-u^{-4}|\nabla u|^2\right)=1,
    \qquad
    4u^{-1}h>u^{-4}|\nabla u|^2.
\end{equation}
For fixed \(u>0\) and \(\nabla u\), the function
\(t\longmapsto t^{k-1}\left(4u^{-1}t-u^{-4}|\nabla u|^2\right)\)
is strictly increasing from \(0\) to \(+\infty\) on
\((|\nabla u|^2/(4u^3),+\infty)\).  Hence
\eqref{eq:5.18} has a unique solution in this interval.
Moreover, the implicit function theorem gives \(h\in C^1\).

Multiplying \eqref{eq:5.18} by \(u^{2k+1}\), we obtain
\begin{equation}\label{eq:5.19}
    \binom{4k}{k}2^{k-2}(u^2h)^{k-1}
    \left(4u^2h-\frac{|\nabla u|^2}{u}\right)
    =u^{2k+1},
    \qquad
    4u^2h>\frac{|\nabla u|^2}{u}.
\end{equation}
Since \(|\nabla u|^2/u\geq0\), we have
\[
    \binom{4k}{k}2^{k-2}(u^2h)^{k-1}
    \left(4u^2h-\frac{|\nabla u|^2}{u}\right)
    \leq\binom{4k}{k}2^k(u^2h)^k.
\]
Thus \eqref{eq:5.19} gives
\begin{equation}\label{eq:5.20}
    h\geq \left(2^k\binom{4k}{k}\right)^{-1/k}u^{1/k}.
\end{equation}

Differentiating \eqref{eq:5.18} and rearranging, we obtain
\[
    \bigl(4ku^{-1}h-(k-1)|\nabla(u^{-1})|^2\bigr)
    \nabla\log h
    =
    2\bigl(D^2(u^{-1})-2h\Id\bigr)\nabla(u^{-1}).
\]
The denominator below is positive, since \eqref{eq:5.19}
gives
\(4ku^2h-(k-1)u^{-1}|\nabla u|^2>4u^2h>0\).
Since
\[
    \nabla(u^{-1})=-u^{-2}\nabla u,
    \qquad
    D^2(u^{-1})=u^{-2}A+2u^{-3}\nabla u\otimes\nabla u,
\]
it follows that
\begin{equation}\label{eq:5.21}
    \sqrt u\,\nabla\log h
    =
    \frac{2\bigl(A+(2u^{-1}|\nabla u|^2-2u^2h)\Id\bigr)
    (-\nabla u/\sqrt u)}
    {4ku^2h-(k-1)u^{-1}|\nabla u|^2}.
\end{equation}
We now prove the second inequality in \eqref{eq:5.17}.
\begin{proposition}\label{Prop:5.5}
There is \(C_k<\infty\), depending only on \(k\), such that
\begin{equation}\label{eq:5.22}
    |\mathcal J\cdot\nabla\log h|\leq C_k\operatorname{div}\mathcal J
    \qquad\text{in }\R^n.
\end{equation}
\end{proposition}

{
\begin{proof}
Taking \(Y=\nabla\log h\) in \eqref{eq:5.15}, it is enough to prove
\begin{equation}
 |\nabla u|^2(\nabla\log h)^T T_{k-1}(A)\nabla\log h
 \leq C_k u\operatorname{div}\mathcal J.
\label{eq:5.23}
\end{equation}
Then \eqref{eq:5.15} and \eqref{eq:5.23} give
\(|\mathcal J\cdot\nabla\log h|^2 \leq C_k(\operatorname{div}\mathcal J)^2\).
Since \(\operatorname{div}\mathcal J\geq0\) by \eqref{eq:5.14},
this implies \eqref{eq:5.22}.

After substituting \eqref{eq:5.21} into \eqref{eq:5.23}, both
\eqref{eq:5.19} and \eqref{eq:5.23} depend on \(h\) only through
\(u^2h\).  Fix \(x\in\R^n\).  At this point, replace
\[
 \left(u,A,-\frac{\nabla u}{\sqrt u},u^2h\right)
 \quad\hbox{by}\quad
 \left((u^2h)^{-k/(2k+1)}u,(u^2h)^{-1}A,
 (u^2h)^{-1/2}\frac{-\nabla u}{\sqrt u},1\right).
\]
Then both sides of \eqref{eq:5.19} are multiplied by
\((u^2h)^{-k}\), while both sides of \eqref{eq:5.23} are multiplied by
\((u^2h)^{-k-1}\).  Moreover,
\[
 \sigma_k\bigl((u^2h)^{-1}A\bigr)
 =(u^2h)^{-k}\sigma_k(A)
 =\left((u^2h)^{-k/(2k+1)}u\right)^{2k+1},
\]
so \(\sigma_k(A)=u^{2k+1}\) still holds.  Therefore, we may assume
\(u^2h=1\).
This is only a pointwise algebraic normalization at \(x\), and the
factor \(u^2h\) is not differentiated.  We have used
\[
 L_k\bigl((u^2h)^{-1}A\bigr)
 =(u^2h)^{-k}L_k(A),
 \qquad
 T_{k-1}\bigl((u^2h)^{-1}A\bigr)
 =(u^2h)^{-(k-1)}T_{k-1}(A).
\]

Under this normalization, using
\(\binom{4k}{k}=4\binom{4k-1}{k-1}\),
\eqref{eq:5.19} becomes
\begin{equation}\label{eq:5.24}
 u^{2k+1}
 =2^k\binom{4k-1}{k-1}
 \left(4-\frac{|\nabla u|^2}{u}\right),
 \qquad 0\leq\frac{|\nabla u|^2}{u}<4.
\end{equation}

If \(\nabla u=0\), then \eqref{eq:5.23} follows from
\eqref{eq:5.14}.  Therefore, in what follows, we
assume that \(\nabla u\ne0\) and diagonalize \(A\).  Let
\(e_1,\ldots,e_{4k}\) be an orthonormal eigenbasis and write
\(-\frac{\nabla u}{|\nabla u|}=\sum_{i=1}^{4k}\omega_i e_i, \qquad \sum_{i=1}^{4k}\omega_i^2=1\).
By \eqref{eq:5.14} and \eqref{eq:5.21}, we have
\begin{align*}
&|\nabla u|^2(\nabla\log h)^TT_{k-1}(A)\nabla\log h\\
&\quad=
\frac{4|\nabla u|^4/u^2}
{\left(4k-(k-1)|\nabla u|^2/u\right)^2}
\sum_{i=1}^{4k}\omega_i^2
\bigl(T_{k-1}(A)\bigr)_{ii}
\left(A_{ii}+2\frac{|\nabla u|^2}{u}-2\right)^2,
\end{align*}
and
\begin{align*}
u\operatorname{div}\mathcal J
=\sum_{i=1}^{4k}\omega_i^2\left[
\tr\bigl(L_k(A)A\bigr)
+4\frac{|\nabla u|^2}{u}\bigl(L_k(A)\bigr)_{ii}
+3\frac{|\nabla u|^4}{u^2}
\bigl(T_{k-1}(A)\bigr)_{ii}\right].
\end{align*}
Since \(4k-(k-1)\frac{|\nabla u|^2}{u}
=4+(k-1)\left(4-\frac{|\nabla u|^2}{u}\right)\geq4\),
it is enough to prove, for every \(i\),
\begin{align}
&\frac{|\nabla u|^4}{u^2}
 \bigl(T_{k-1}(A)\bigr)_{ii}
 \left(A_{ii}+2\frac{|\nabla u|^2}{u}-2\right)^2
\notag\\
&\qquad\leq C_k\left[
 \tr\bigl(L_k(A)A\bigr)
 +4\frac{|\nabla u|^2}{u}\bigl(L_k(A)\bigr)_{ii}
 +3\frac{|\nabla u|^4}{u^2}
 \bigl(T_{k-1}(A)\bigr)_{ii}\right].
\label{eq:5.25}
\end{align}

Fix one index and relabel it as \(1\).  By
\eqref{eq:2.6} and \eqref{eq:2.9}, we find
\begin{equation}\label{eq:5.26}
 \bigl(T_{k-1}(A)\bigr)_{11}=\sigma_{k-1}(\lambda(A)|1)>0,
 \qquad
 \bigl(L_k(A)\bigr)_{11}=\frac34\sigma_k(A)-\sigma_k(\lambda(A)|1).
\end{equation}
The corresponding term on the right-hand side of
\eqref{eq:5.25} satisfies
\begin{align}
&\tr\bigl(L_k(A)A\bigr)
 +4\frac{|\nabla u|^2}{u}\bigl(L_k(A)\bigr)_{11}
 +3\frac{|\nabla u|^4}{u^2}
    \bigl(T_{k-1}(A)\bigr)_{11}
\notag\\
&=\frac{4}{3\bigl(T_{k-1}(A)\bigr)_{11}}
 \Biggl\{
 \left[
   \bigl(L_k(A)\bigr)_{11}
   +\frac32\frac{|\nabla u|^2}{u}
      \bigl(T_{k-1}(A)\bigr)_{11}
 \right]^2
\notag\\
&\hspace{35mm}
 +\left[
   \frac34\tr\bigl(L_k(A)A\bigr)
      \bigl(T_{k-1}(A)\bigr)_{11}
   -\bigl(L_k(A)\bigr)_{11}^2
 \right]
 \Biggr\}.
\label{eq:5.27}
\end{align}
By \eqref{eq:2.13} with \(n=4k\) and \(i=1\), the Newton deficit
in \eqref{eq:5.27} has the decomposition
\begin{align}
&\frac34\tr\bigl(L_k(A)A\bigr)
   \sigma_{k-1}(\lambda(A)|1)
 -\bigl(L_k(A)\bigr)_{11}^2
\notag\\
&=\frac{3\sigma_k(A)}{16}
 \Bigl[
   3\sigma_1(\lambda(A)|1)\sigma_{k-1}(\lambda(A)|1)
   -(4k-1)\sigma_k(\lambda(A)|1)
 \Bigr]
\notag\\
&\quad+\frac14
 \Bigl[
   (3k-1)\sigma_k(\lambda(A)|1)^2
   -3(k+1)\sigma_{k+1}(\lambda(A)|1)
      \sigma_{k-1}(\lambda(A)|1)
 \Bigr]
 \geq0.
\label{eq:5.28}
\end{align}
Thus the Newton deficit in \eqref{eq:5.27} is nonnegative.  This gives
one part of the required estimate.  We still
have to control
\(
A_{11}+2|\nabla u|^2/u-2
\)
on the left-hand side of \eqref{eq:5.25}.

Expanding \(\sigma_k(A)\) in the first direction and using
\eqref{eq:5.26}, we obtain
\begin{equation*}
 \bigl(T_{k-1}(A)\bigr)_{11}A_{11}
=\sigma_k(A)-\sigma_k(\lambda(A)|1)
=\bigl(L_k(A)\bigr)_{11}+\frac14\sigma_k(A).
\end{equation*}
Hence
\begin{equation*}
\bigl(T_{k-1}(A)\bigr)_{11}
 \left(A_{11}+2\frac{|\nabla u|^2}{u}-2\right)
=\bigl(L_k(A)\bigr)_{11}+\frac14\sigma_k(A)
 +\left(2\frac{|\nabla u|^2}{u}-2\right)
 \bigl(T_{k-1}(A)\bigr)_{11}.
\end{equation*}
Using \eqref{eq:5.24} and
\(2\frac{|\nabla u|^2}{u}-2
=\frac32\frac{|\nabla u|^2}{u}
-\frac12\left(4-\frac{|\nabla u|^2}{u}\right)\),
we get
\begin{align}
&\bigl(T_{k-1}(A)\bigr)_{11}
 \left(A_{11}+2\frac{|\nabla u|^2}{u}-2\right)
\notag\\
&\quad=\bigl(L_k(A)\bigr)_{11}
 +\frac32\frac{|\nabla u|^2}{u}
 \bigl(T_{k-1}(A)\bigr)_{11}
 +\frac12\left(4-\frac{|\nabla u|^2}{u}\right)
 \left[2^{k-1}\binom{4k-1}{k-1}
 -\bigl(T_{k-1}(A)\bigr)_{11}\right].
\label{eq:5.29}
\end{align}

We control the square of the last term,
\[
 \frac{|\nabla u|^4}{u^2}
 \left(4-\frac{|\nabla u|^2}{u}\right)^2
 \left[2^{k-1}\binom{4k-1}{k-1}
 -\bigl(T_{k-1}(A)\bigr)_{11}\right]^2
\]
by \eqref{eq:5.30} in Lemma~\ref{Lem:5.6} below, whose proof is
postponed because it is the most involved part and would interrupt the
proof of the proposition.
It follows from \eqref{eq:5.29}, \eqref{eq:5.30}, and
\((a+b)^2\leq2a^2+2b^2\) that
\begin{align*}
&\frac{|\nabla u|^4}{u^2}
 \bigl(T_{k-1}(A)\bigr)_{11}
 \left(A_{11}+2\frac{|\nabla u|^2}{u}-2\right)^2\\
&\quad\leq\frac{C_k}{\bigl(T_{k-1}(A)\bigr)_{11}}
\Biggl\{
 \left[\bigl(L_k(A)\bigr)_{11}
 +\frac32\frac{|\nabla u|^2}{u}
 \bigl(T_{k-1}(A)\bigr)_{11}\right]^2\\
&\hspace{25mm}
 +\frac{|\nabla u|^4}{u^2}
 \left(4-\frac{|\nabla u|^2}{u}\right)^2
 \left[2^{k-1}\binom{4k-1}{k-1}
 -\bigl(T_{k-1}(A)\bigr)_{11}\right]^2\Biggr\}\\
&\quad\leq\frac{C_k}{\bigl(T_{k-1}(A)\bigr)_{11}}
\Biggl\{
 \left[\bigl(L_k(A)\bigr)_{11}
 +\frac32\frac{|\nabla u|^2}{u}
 \bigl(T_{k-1}(A)\bigr)_{11}\right]^2\\
&\hspace{42mm}
 +\frac34\tr\bigl(L_k(A)A\bigr)
 \bigl(T_{k-1}(A)\bigr)_{11}
 -\bigl(L_k(A)\bigr)_{11}^2\Biggr\}.
\end{align*}
By \eqref{eq:5.27}, this is \eqref{eq:5.25} for \(i=1\).
The same argument applies to every \(i\).  Multiplying by
\(\omega_i^2\), summing over \(i\), and using
\(\sum_i\omega_i^2=1\), we obtain \eqref{eq:5.23}.  This proves the
proposition.
\end{proof}

\begin{lemma}\label{Lem:5.6}
Under the pointwise normalization \(u^2h=1\), for every
\(1\leq i\leq4k\),
\begin{align}
&\frac{|\nabla u|^4}{u^2}
 \left(4-\frac{|\nabla u|^2}{u}\right)^2
 \left[\bigl(T_{k-1}(A)\bigr)_{ii}
 -2^{k-1}\binom{4k-1}{k-1}\right]^2
\notag\\
&\quad\leq C_k\Biggl\{
 \left[\bigl(L_k(A)\bigr)_{ii}
 +\frac32\frac{|\nabla u|^2}{u}
 \bigl(T_{k-1}(A)\bigr)_{ii}\right]^2
 +\frac34\tr\bigl(L_k(A)A\bigr)
 \bigl(T_{k-1}(A)\bigr)_{ii}
 -\bigl(L_k(A)\bigr)_{ii}^2\Biggr\}.
\label{eq:5.30}
\end{align}
\end{lemma}

\begin{proof}
If \(\nabla u=0\), then \eqref{eq:5.30} follows directly from
\eqref{eq:5.28}. We assume that \(\nabla u\ne0\). It is enough to
consider \(i=1\). From
\eqref{eq:5.24}, we find
\begin{align}
&\bigl(L_k(A)\bigr)_{11}
 +\frac32\frac{|\nabla u|^2}{u}\bigl(T_{k-1}(A)\bigr)_{11}
\notag\\
&\qquad=\frac32\left[
 2^{k-1}\binom{4k-1}{k-1}
 \left(4-\frac{|\nabla u|^2}{u}\right)
 +\frac{|\nabla u|^2}{u}\bigl(T_{k-1}(A)\bigr)_{11}\right]
 -\sigma_k(\lambda(A)|1).
\label{eq:5.31}
\end{align}
We also have
\begin{align}
&\frac{|\nabla u|^2}{u}
 \left(4-\frac{|\nabla u|^2}{u}\right)
 \left|\bigl(T_{k-1}(A)\bigr)_{11}
 -2^{k-1}\binom{4k-1}{k-1}\right|
\notag\\
&\qquad\leq4\left[
 2^{k-1}\binom{4k-1}{k-1}
 \left(4-\frac{|\nabla u|^2}{u}\right)
 +\frac{|\nabla u|^2}{u}\bigl(T_{k-1}(A)\bigr)_{11}\right].
\label{eq:5.32}
\end{align}
This follows from \(0\leq|\nabla u|^2/u<4\),
\(\bigl(T_{k-1}(A)\bigr)_{11}>0\), and
\begin{align*}
&\frac{|\nabla u|^2}{u}
 \left(4-\frac{|\nabla u|^2}{u}\right)
 \left|\bigl(T_{k-1}(A)\bigr)_{11}
 -2^{k-1}\binom{4k-1}{k-1}\right|\\
&\quad\leq
\frac{|\nabla u|^2}{u}
 \left(4-\frac{|\nabla u|^2}{u}\right)
 \left[\bigl(T_{k-1}(A)\bigr)_{11}
 +2^{k-1}\binom{4k-1}{k-1}\right]\\
&\quad\leq4\left[
 \frac{|\nabla u|^2}{u}\bigl(T_{k-1}(A)\bigr)_{11}
 +2^{k-1}\binom{4k-1}{k-1}
 \left(4-\frac{|\nabla u|^2}{u}\right)\right].
\end{align*}

We divide the proof of \eqref{eq:5.30} into three cases according to
the signs of \(\sigma_k(\lambda(A)|1)\) and
\(\sigma_{k+1}(\lambda(A)|1)\).

\medskip
\noindent\emph{Case 1:
\(\sigma_k(\lambda(A)|1)\leq0\).}
The right-hand side of
\eqref{eq:5.31} is at least
\[
 \frac32\left[
 2^{k-1}\binom{4k-1}{k-1}
 \left(4-\frac{|\nabla u|^2}{u}\right)
 +\frac{|\nabla u|^2}{u}\bigl(T_{k-1}(A)\bigr)_{11}\right].
\]
Thus \eqref{eq:5.32} proves
\eqref{eq:5.30} in this case.

\medskip
\noindent\emph{Case 2:
\(\sigma_k(\lambda(A)|1)>0\) and
\(\sigma_{k+1}(\lambda(A)|1)\leq0\).}
Since \(\sigma_{k+1}(\lambda(A)|1)\leq0\), identity
\eqref{eq:5.28} gives
\[
 \frac34\tr\bigl(L_k(A)A\bigr)\bigl(T_{k-1}(A)\bigr)_{11}
 -\bigl(L_k(A)\bigr)_{11}^2
 \geq\frac{3k-1}{4}\sigma_k(\lambda(A)|1)^2.
\]
Equations \eqref{eq:5.31} and
\eqref{eq:5.32} {yield}
\begin{align*}
&\frac{|\nabla u|^2}{u}
 \left(4-\frac{|\nabla u|^2}{u}\right)
 \left|\bigl(T_{k-1}(A)\bigr)_{11}
 -2^{k-1}\binom{4k-1}{k-1}\right|
\\
&\qquad\leq C_k\left(
 \left|\bigl(L_k(A)\bigr)_{11}
 +\frac32\frac{|\nabla u|^2}{u}
   \bigl(T_{k-1}(A)\bigr)_{11}\right|
 +\sigma_k(\lambda(A)|1)\right).
\end{align*}
Squaring this inequality and using
\[
 \frac34\tr\bigl(L_k(A)A\bigr)\bigl(T_{k-1}(A)\bigr)_{11}
 -\bigl(L_k(A)\bigr)_{11}^2
 \geq\frac{3k-1}{4}\sigma_k(\lambda(A)|1)^2
\]
proves \eqref{eq:5.30} in this case.

\medskip
\noindent\emph{Case 3:
\(\sigma_k(\lambda(A)|1)>0\) and
\(\sigma_{k+1}(\lambda(A)|1)>0\).}
Corollary~\ref{Cor:2.3} and the assumptions of Case 3 give
\(\lambda(A)|1\in\Gamma_{k+1}\).  For
\(1\leq j\leq k+1\), set
\begin{equation}\label{eq:5.33}
 \rho_j
 =\frac{\sigma_j(\lambda(A)|1)/\binom{4k-1}{j}}
        {\sigma_{j-1}(\lambda(A)|1)/\binom{4k-1}{j-1}}
 =\frac{j\,\sigma_j(\lambda(A)|1)}
        {(4k-j)\sigma_{j-1}(\lambda(A)|1)}.
\end{equation}
Equation \eqref{eq:2.7} gives
\begin{equation}\label{eq:5.34}
 \rho_1\geq\rho_2\geq\cdots\geq\rho_{k+1}>0.
\end{equation}
By \eqref{eq:5.33}, we get
\begin{equation}\label{eq:5.35}
 \frac{\bigl(T_{k-1}(A)\bigr)_{11}}
 {2^{k-1}\binom{4k-1}{k-1}}
 =\prod_{j=1}^{k-1}\frac{\rho_j}{2},
 \qquad
 \sigma_k(\lambda(A)|1)
 =3\bigl(T_{k-1}(A)\bigr)_{11}\rho_k.
\end{equation}
Substituting \eqref{eq:5.35} into \eqref{eq:5.28}, we obtain
\begin{align}
&\frac34\tr\bigl(L_k(A)A\bigr)\bigl(T_{k-1}(A)\bigr)_{11}
 -\bigl(L_k(A)\bigr)_{11}^2
\notag\\
&=\frac{9(4k-1)}{16}\sigma_k(A)
 \bigl(T_{k-1}(A)\bigr)_{11}(\rho_1-\rho_k)
 +\frac94(3k-1)\bigl(T_{k-1}(A)\bigr)_{11}^2
 \rho_k(\rho_k-\rho_{k+1}).
\label{eq:5.36}
\end{align}

The two nonnegative quantities
\[
 \sigma_k(A)\bigl(T_{k-1}(A)\bigr)_{11}(\rho_1-\rho_k)
 \quad\text{and}\quad
 \bigl(T_{k-1}(A)\bigr)_{11}^2
 \rho_k(\rho_k-\rho_{k+1})
\]
may both be small.  We use the normalized deficit
\((\rho_k-\rho_{k+1})/\rho_k\) to control all the quotients
\(\rho_j/\rho_k\).

We first claim that
\begin{equation}\label{eq:5.37}
 \max_{1\leq j\leq k+1}
 \left|\frac{\rho_j}{\rho_k}-1\right|
 \leq C_k
 \left(\frac{\rho_k-\rho_{k+1}}{\rho_k}\right)^{1/2}
\end{equation}
whenever
\(\frac{\rho_k-\rho_{k+1}}{\rho_k} \leq\frac{1}{16k(k+1)}\).
To prove \eqref{eq:5.37}, start from the real-rooted polynomial
\(\prod_{i=2}^{4k}(\zeta+\lambda_i)\).  Direct differentiation gives
\[
 \frac{\dd^{\,3k-2}}{\dd\zeta^{\,3k-2}}
 \prod_{i=2}^{4k}(\zeta+\lambda_i)
 =
 \sum_{j=0}^{k+1}
 \frac{(4k-1-j)!}{(k+1-j)!}
 \sigma_j(\lambda(A)|1)\zeta^{k+1-j}.
\]
By Rolle's theorem, this polynomial of degree \(k+1\) has only real
roots.  Its constant term is
\((3k-2)!\sigma_{k+1}(\lambda(A)|1)\ne0\).
Thus none of its roots is zero.  Reversing the coefficients and
normalizing the constant term, we obtain
\begin{align*}
&\frac{(k+1)!}{(4k-1)!}\zeta^{k+1}
 \left.
 \frac{\dd^{\,3k-2}}{\dd t^{\,3k-2}}
 \prod_{i=2}^{4k}(t+\lambda_i)
 \right|_{t=\zeta^{-1}}\\
&\qquad=
 \sum_{j=0}^{k+1}
 \frac{(k+1)!}{(4k-1)!}
 \frac{(4k-1-j)!}{(k+1-j)!}
 \sigma_j(\lambda(A)|1)\zeta^j\\
&\qquad=
 \sum_{j=0}^{k+1}\binom{k+1}{j}
 \frac{\sigma_j(\lambda(A)|1)}
      {\binom{4k-1}{j}}\zeta^j.
\end{align*}
Its roots are the reciprocals of the nonzero roots above and hence
are real.  All its coefficients are positive in Case 3, so it is
positive for \(\zeta\geq0\).  Therefore, its roots can be written as
\(-\xi_1^{-1},\ldots,-\xi_{k+1}^{-1}\), where every \(\xi_i>0\).
Its constant term is one, and consequently
\[
 \sum_{j=0}^{k+1}\binom{k+1}{j}
 \frac{\sigma_j(\lambda(A)|1)}
      {\binom{4k-1}{j}}\zeta^j
 =\prod_{i=1}^{k+1}(1+\xi_i\zeta).
\]
Comparing coefficients gives
\begin{equation}\label{eq:5.38}
 \frac{\sigma_j(\lambda(A)|1)}{\binom{4k-1}{j}}
 =\frac{\sigma_j(\xi_1,\ldots,\xi_{k+1})}{\binom{k+1}{j}},
 \qquad 0\leq j\leq k+1.
\end{equation}
The complementary symmetric-function identities, together with
\eqref{eq:5.33} and
\eqref{eq:5.38}, {yield}
\[
 \rho_k
 =\frac{k}{2}
 \frac{\sum_{i=1}^{k+1}\xi_i^{-1}}
      {\sum_{1\leq i<j\leq k+1}(\xi_i\xi_j)^{-1}},
 \qquad
 \rho_{k+1}
 =\frac{k+1}{\sum_{i=1}^{k+1}\xi_i^{-1}}.
\]
Here we used the complementary identities
\[
\begin{aligned}
 \sigma_k(\xi_1,\ldots,\xi_{k+1})
 &=\sigma_{k+1}(\xi_1,\ldots,\xi_{k+1})
 \sum_{i=1}^{k+1}\xi_i^{-1},\\
 \sigma_{k-1}(\xi_1,\ldots,\xi_{k+1})
 &=\sigma_{k+1}(\xi_1,\ldots,\xi_{k+1})
 \sum_{1\leq i<j\leq k+1}(\xi_i\xi_j)^{-1}.
\end{aligned}
\]
The identity \(2\sum_{1\leq i<j\leq k+1}(\xi_i\xi_j)^{-1}
=\left(\sum_{i=1}^{k+1}\xi_i^{-1}\right)^2
-\sum_{i=1}^{k+1}\xi_i^{-2}\),
now gives
\[
 1-\frac{\rho_{k+1}}{\rho_k}
 =\frac{(k+1)\displaystyle\sum_{i=1}^{k+1}\xi_i^{-2}
 -\left(\displaystyle\sum_{i=1}^{k+1}\xi_i^{-1}\right)^2}
 {k\left(\displaystyle\sum_{i=1}^{k+1}\xi_i^{-1}\right)^2},
\]
and
\[
 (k+1)\sum_{i=1}^{k+1}\xi_i^{-2}
 -\left(\sum_{i=1}^{k+1}\xi_i^{-1}\right)^2
 =(k+1)\sum_{i=1}^{k+1}
 \left(\xi_i^{-1}
 -\frac1{k+1}\sum_{\ell=1}^{k+1}\xi_\ell^{-1}\right)^2.
\]
Therefore, we get
\begin{align}
 \frac{\rho_k-\rho_{k+1}}{\rho_k}
 =\frac{\displaystyle\sum_{i=1}^{k+1}
 \left(\xi_i^{-1}
 -\frac1{k+1}\sum_{\ell=1}^{k+1}\xi_\ell^{-1}\right)^2}
 {\displaystyle k(k+1)
 \left(\frac1{k+1}\sum_{\ell=1}^{k+1}\xi_\ell^{-1}\right)^2}.
\label{eq:5.39}
\end{align}
Multiplying all the \(\xi_i\) by the same positive number multiplies
each quotient of two consecutive terms on the right of
\eqref{eq:5.38} by that number.  Hence it does not change
\(\rho_j/\rho_k\) or
\((\rho_k-\rho_{k+1})/\rho_k\).  We may assume that the mean of the
\(\xi_i^{-1}\) is one.  Then
\eqref{eq:5.39} gives
\[
 \sum_{i=1}^{k+1}(\xi_i^{-1}-1)^2
 =k(k+1)\frac{\rho_k-\rho_{k+1}}{\rho_k}\leq\frac1{16}.
\]
In particular, \(4/5\leq\xi_i\leq4/3\).  The denominators in the
quotients of consecutive elementary symmetric functions are bounded
away from zero on this interval.  By the mean value theorem and
\eqref{eq:5.38},
\[
 \left|\frac{\rho_j}{\rho_k}-1\right|
 \leq C_k\max_{1\leq i\leq k+1}|\xi_i^{-1}-1|
 \leq C_k
 \left(\frac{\rho_k-\rho_{k+1}}{\rho_k}\right)^{1/2},
\]
which proves \eqref{eq:5.37}.

With  \eqref{eq:5.37} in hand, we are able to continue our proof. Set
\begin{equation}\label{eq:5.40}
 Z=\frac{\bigl(T_{k-1}(A)\bigr)_{11}}
 {2^{k-1}\binom{4k-1}{k-1}}.
\end{equation}
By \eqref{eq:5.35}, we have
\(Z=\left(\frac{\rho_k}{2}\right)^{k-1} \prod_{j=1}^{k-1}\frac{\rho_j}{\rho_k}\).
Taking the positive \((k-1)\)-st root gives
\(\frac{2Z^{1/(k-1)}}{\rho_k}
=\allowbreak\left(\prod_{j=1}^{k-1}\frac{\rho_j}{\rho_k} \right)^{1/(k-1)}\).
By \eqref{eq:5.34},
\(1\leq \left(\prod_{j=1}^{k-1}\frac{\rho_j}{\rho_k}\right)^{1/(k-1)}
\allowbreak\leq\max_{1\leq j\leq k-1}\frac{\rho_j}{\rho_k}\).
Therefore, \eqref{eq:5.37} directly gives
\begin{equation}\label{eq:5.41}
 \left|\rho_k-2Z^{1/(k-1)}\right|
 \leq C_k\rho_k
 \left(\frac{\rho_k-\rho_{k+1}}{\rho_k}\right)^{1/2}
\end{equation}
when
\((\rho_k-\rho_{k+1})/\rho_k\leq1/[16k(k+1)]\).

Equations \eqref{eq:5.24} and
\eqref{eq:5.35} also {yield}
\begin{equation}\label{eq:5.42}
\frac{2}
 {3\cdot2^{k-1}\binom{4k-1}{k-1}}
 \left[
 \bigl(L_k(A)\bigr)_{11}
 +\frac32\frac{|\nabla u|^2}{u}
   \bigl(T_{k-1}(A)\bigr)_{11}\right]
=
 \left(4-\frac{|\nabla u|^2}{u}\right)
 +\frac{|\nabla u|^2}{u}Z-2Z\rho_k.
\end{equation}
By \eqref{eq:5.34},
\[
 \frac{9(4k-1)}{16}\sigma_k(A)
 \bigl(T_{k-1}(A)\bigr)_{11}(\rho_1-\rho_k)\geq0.
\]
Hence, by
\eqref{eq:5.36},
\eqref{eq:5.40}, and
\eqref{eq:5.42}, it is enough to prove
\begin{equation}\label{eq:5.43}
\frac{|\nabla u|^4}{u^2}
 \left(4-\frac{|\nabla u|^2}{u}\right)^2(Z-1)^2
 \leq C_k\Biggl\{
 \left[
 \left(4-\frac{|\nabla u|^2}{u}\right)
 +\frac{|\nabla u|^2}{u}Z-2Z\rho_k
 \right]^2
 +Z^2\rho_k(\rho_k-\rho_{k+1})\Biggr\}.
\end{equation}

We divide the proof of \eqref{eq:5.43} into two cases.

\medskip
\noindent\emph{Case 3(a):
\(\displaystyle
 \frac{\rho_k-\rho_{k+1}}{\rho_k}
 \leq\frac{1}{16k(k+1)}\).}
We first write
\begin{align*}
4Z^{k/(k-1)}
 -\frac{|\nabla u|^2}{u}Z
 -\left(4-\frac{|\nabla u|^2}{u}\right)
&=4\left(Z^{k/(k-1)}-1\right)
 -\frac{|\nabla u|^2}{u}(Z-1)\\
&\quad=\left(Z^{1/(k-1)}-1\right)
 \left[
 \left(4-\frac{|\nabla u|^2}{u}\right)
 \sum_{j=0}^{k-2}Z^{j/(k-1)}+4Z\right],
\end{align*}
while
\(Z-1=\left(Z^{1/(k-1)}-1\right) \sum_{j=0}^{k-2}Z^{j/(k-1)}\).
Since \(0<|\nabla u|^2/u<4\), these two identities give
\begin{align}
&\frac{|\nabla u|^2}{u}
 \left(4-\frac{|\nabla u|^2}{u}\right)|Z-1|\leq4\left|
 4Z^{k/(k-1)}
 -\frac{|\nabla u|^2}{u}Z
 -\left(4-\frac{|\nabla u|^2}{u}\right)\right|.
\label{eq:5.44}
\end{align}
Adding and subtracting \(2Z\rho_k\) {yields}
\begin{equation*}
4Z^{k/(k-1)}
 -\frac{|\nabla u|^2}{u}Z
 -\left(4-\frac{|\nabla u|^2}{u}\right)
=-\left[
 \left(4-\frac{|\nabla u|^2}{u}\right)
 +\frac{|\nabla u|^2}{u}Z-2Z\rho_k
 \right]
 +2Z\left(2Z^{1/(k-1)}-\rho_k\right).
\end{equation*}
Applying \eqref{eq:5.41} to
\(2Z\bigl(2Z^{1/(k-1)}-\rho_k\bigr)\), we obtain
\begin{align*}
&\left|
 4Z^{k/(k-1)}
 -\frac{|\nabla u|^2}{u}Z
 -\left(4-\frac{|\nabla u|^2}{u}\right)\right|
\\
&\quad\leq
 \left|
 \left(4-\frac{|\nabla u|^2}{u}\right)
 +\frac{|\nabla u|^2}{u}Z-2Z\rho_k
 \right|
 +C_kZ\rho_k
 \left(\frac{\rho_k-\rho_{k+1}}{\rho_k}\right)^{1/2}.
\end{align*}
Combining this estimate with \eqref{eq:5.44} and Cauchy's inequality,
we obtain
\begin{align*}
&\frac{|\nabla u|^4}{u^2}
 \left(4-\frac{|\nabla u|^2}{u}\right)^2(Z-1)^2
\\
&\quad\leq C_k\Biggl\{
 \left[
 \left(4-\frac{|\nabla u|^2}{u}\right)
 +\frac{|\nabla u|^2}{u}Z-2Z\rho_k
 \right]^2
 +Z^2\rho_k^2
 \frac{\rho_k-\rho_{k+1}}{\rho_k}\Biggr\}
\\
&\quad=C_k\Biggl\{
 \left[
 \left(4-\frac{|\nabla u|^2}{u}\right)
 +\frac{|\nabla u|^2}{u}Z-2Z\rho_k
 \right]^2
 +Z^2\rho_k(\rho_k-\rho_{k+1})\Biggr\}.
\end{align*}
This proves \eqref{eq:5.43} in Case 3(a).

\medskip
\noindent\emph{Case 3(b):
\(\displaystyle
 \frac{\rho_k-\rho_{k+1}}{\rho_k}
 >\frac{1}{16k(k+1)}\).}
As in \eqref{eq:5.32}, one has
\begin{align}
&\frac{|\nabla u|^2}{u}
 \left(4-\frac{|\nabla u|^2}{u}\right)|Z-1|
\leq4\left[
 \left(4-\frac{|\nabla u|^2}{u}\right)
 +\frac{|\nabla u|^2}{u}Z\right].
\label{eq:5.45}
\end{align}
If \(\left|\left(4-\frac{|\nabla u|^2}{u}\right)
+\frac{|\nabla u|^2}{u}Z-2Z\rho_k\right|
\geq\frac12\left[\left(4-\frac{|\nabla u|^2}{u}\right)
+\frac{|\nabla u|^2}{u}Z\right]\), then \eqref{eq:5.45} proves
\eqref{eq:5.43}.  Otherwise,
\(\left(4-\frac{|\nabla u|^2}{u}\right) +\frac{|\nabla u|^2}{u}Z \leq4Z\rho_k\).
It follows from \eqref{eq:5.45} that
\begin{align*}
&\frac{|\nabla u|^4}{u^2}
 \left(4-\frac{|\nabla u|^2}{u}\right)^2(Z-1)^2
 \leq C_k Z^2\rho_k^2
 \leq C_k Z^2\rho_k(\rho_k-\rho_{k+1}).
\end{align*}
The last inequality follows from the assumption of Case 3(b), which
gives
\(\rho_k^2 <16k(k+1)\rho_k(\rho_k-\rho_{k+1})\),
which is the last inequality above.
This proves \eqref{eq:5.43}, and hence
\eqref{eq:5.30} in Case 3(b).

The three cases prove \eqref{eq:5.30} and complete the proof.
\end{proof}
}

We now carry out the weighted integration by parts.  Choose
\(\varepsilon\) sufficiently small such that
\(0<\varepsilon<\min\left\{1,\frac1{2C_k}\right\}\).
It follows from \eqref{eq:5.16} and
Proposition~\ref{Prop:5.5} that
\begin{equation}\label{eq:5.46}
    \operatorname{div}(h^{-\varepsilon}\mathcal J)
    \geq\frac12h^{-\varepsilon}\operatorname{div}\mathcal J.
\end{equation}
Multiplying \eqref{eq:5.46} by
\(\chi^{8k+2}\) and integrating by parts, we obtain
\begin{align*}
    \frac12\int\chi^{8k+2}h^{-\varepsilon}
       \operatorname{div}\mathcal J\dd x
    \leq\int\chi^{8k+2}
       \operatorname{div}(h^{-\varepsilon}\mathcal J)\dd x
    \leq(8k+2)\int\chi^{8k+1}h^{-\varepsilon}
       |\mathcal J\cdot\nabla\chi|\dd x.
\end{align*}
Using \eqref{eq:5.15}, Young's inequality,
\(\tr(T_{k-1}(A))=(3k+1)\sigma_{k-1}(A)\), and
\(|\nabla\chi|\leq C/R\), we obtain
\begin{equation}\label{eq:5.47}\begin{split}
    \int\chi^{8k+2}h^{-\varepsilon}
       \operatorname{div}\mathcal J\dd x
    &\leq CR^{-2}\int\chi^{8k}h^{-\varepsilon}
       \sigma_{k-1}(A)|\nabla u|^2u^{-1}\dd x\\
       &\leq CR^{-2}\int\chi^{8k}\sigma_{k-1}(A)|\nabla u|^2
       u^{-1-\varepsilon/k}\dd x.
\end{split}
\end{equation}
Here the second inequality follows from \eqref{eq:5.20}.
It remains to estimate the right-hand side of
\eqref{eq:5.47}.  We give the estimate for a general
parameter \(\kappa\), and will later take \(\kappa=\varepsilon/k\).
\begin{lemma}\label{Lem:5.7}
Let \(0<\kappa<k\), and let \(\chi\) be as in
\eqref{eq:3.2}.  Then
\begin{equation}\label{eq:5.48}
    \int
    \sigma_{k-1}(A)|\nabla u|^2u^{-1-\kappa}
    \chi^{8k}\dd x
    \leq
    C_\kappa R^{\frac{2k}{k+1}(1+\kappa)}.
\end{equation}
The constant \(C_\kappa\) depends only on \(k,\kappa\) and the fixed
cutoff profile, and is independent of \(R\).
\end{lemma}

\begin{proof}
We apply \eqref{eq:3.4} with \(p=2k+1\), \(\delta=\kappa\),
and \(\theta=8k\).  Then the left-hand side of
\eqref{eq:5.48} is \(\Lambda_1\), while the terminal term on the
right-hand side of \eqref{eq:3.4} is
\(CR^{-2k}\int u^{k-\kappa}\chi^{6k}\dd x\).
Using Young's inequality, for every \(\eta>0\), we obtain
\[
\begin{aligned}
    CR^{-2k}u^{k-\kappa}\chi^{6k}
    &\leq
    \eta u^{2k+1-\kappa}\chi^{8k}
    +C_{\kappa,\eta}
    R^{-\frac{2k(2k+1-\kappa)}{k+1}}
    \chi^{\frac{2k(2k+3+\kappa)}{k+1}}.
\end{aligned}
\]
Since \(0\leq\chi\leq1\) and
\(\lvert\operatorname{supp}\chi\rvert\leq CR^{4k}\), integrating
this inequality gives
\[
\begin{aligned}
    CR^{-2k}\int u^{k-\kappa}\chi^{6k}\dd x
    \leq
    \eta\int u^{2k+1-\kappa}\chi^{8k}\dd x
    +C_{\kappa,\eta}
    R^{\frac{2k}{k+1}(1+\kappa)}.
\end{aligned}
\]
Choosing \(\eta\) sufficiently small, we absorb the first term into
the left-hand side of \eqref{eq:3.4}.  This proves
\eqref{eq:5.48}.
\end{proof}

We can now complete the weighted integration by parts.
\begin{corollary}\label{Cor:5.8}
Every positive \(C^2\) solution of \eqref{eq:1.8} in
dimension \(n=4k\)
satisfies \(\operatorname{div}\mathcal J\equiv0\).
\end{corollary}

\begin{proof}
Applying \eqref{eq:5.47} and
Lemma~\ref{Lem:5.7} with
\(\kappa=\varepsilon/k\), we obtain
\begin{equation}\label{eq:5.49}
    \int_{B_R}h^{-\varepsilon}\operatorname{div}\mathcal J\dd x
    \leq
    C R^{-2+\frac{2k}{k+1}(1+\varepsilon/k)}
    =
    C R^{-\frac{2(1-\varepsilon)}{k+1}}
    \longrightarrow0.
\end{equation}
Fix \(r>0\).  Since \(h>0\) and
\(\operatorname{div}\mathcal J\geq0\), for every \(R>r\),
\[
 0\leq\int_{B_r}h^{-\varepsilon}\operatorname{div}\mathcal J\dd x
 \leq\int_{B_R}h^{-\varepsilon}\operatorname{div}\mathcal J\dd x.
\]
Letting \(R\to\infty\) in \eqref{eq:5.49} shows that the
first integral is zero.  Since \(r\) is arbitrary,
\(\operatorname{div}\mathcal J\equiv0\).
\end{proof}
\subsection{The proof of Theorem~\ref{Thm:1.6}}
\label{Subsec:5.4}
The preceding two subsections give
\(\operatorname{div}\mathcal J\equiv0\).  We now apply this key information to prove that
\(D^2(u^{-2k/(n-2k)})\) is a positive scalar matrix, and hence establish the classification results in Theorem~\ref{Thm:1.6}.

\begin{lemma}\label{Lem:5.9}
Suppose that, at some point,
\begin{align}
0={}&
    \tr\bigl(L_k(A)A\bigr)
    +\frac{2n}{n-2k}u^{-1}\nabla u^TL_k(A)\nabla u
    +\frac{n(n-k)}{(n-2k)^2}u^{-2}|\nabla u|^2
       \nabla u^TT_{k-1}(A)\nabla u.
\label{eq:5.50}
\end{align}
Then:

\begin{enumerate}[label=\textup{(\roman*)}]
\item If \(\nabla u=0\), then
\(A=\left(\frac{\sigma_k(A)}{\binom nk}\right)^{1/k}\Id\).

\item If \(\nabla u\neq0\), then \(\nabla u\) is an eigenvector of
\(A\).  If \(a\) is the eigenvalue in the \(\nabla u\)-direction and
\(b\) is the common eigenvalue on \((\nabla u)^\perp\), then
\begin{equation}\label{eq:5.51}
    b>0,
    \qquad
    \frac{|\nabla u|^2}{u}=\frac{n-2k}{n}(b-a).
\end{equation}
\end{enumerate}
\end{lemma}

\begin{proof}
Fix a point \(x\in\R^n\) arbitrarily.  If \(\nabla u(x)=0\), then
\eqref{eq:5.50} gives
\(\tr(L_k(A)A)=0\).  By \eqref{eq:2.18}, we obtain
part~\textup{(i)}.

Assume now that \(\nabla u(x)\neq0\).  By
\eqref{eq:5.50}, the right-hand side of
\eqref{eq:5.6} vanishes.  Since \(T_{k-1}(A)>0\), the
left-hand side vanishes as well, and hence
\begin{equation}\label{eq:5.52}
  L_k(A)\nabla u
  +\frac{n-k}{n-2k}\frac{|\nabla u|^2}{u}
  T_{k-1}(A)\nabla u=0.
\end{equation}
Taking the scalar product with \(\nabla u\) and substituting the result
into \eqref{eq:5.50}, we obtain
\begin{equation}\label{eq:5.53}
    \tr\bigl(L_k(A)A\bigr)
    =
    \frac{n(n-k)}{(n-2k)^2}u^{-2}|\nabla u|^2
       \nabla u^TT_{k-1}(A)\nabla u.
\end{equation}

{At the point under consideration, choose orthonormal
coordinates in which \(A\) is diagonal,
and denote the corresponding eigenvalues by
\(\lambda_1,\ldots,\lambda_n\).}  The matrices \(L_k(A)\) and
\(T_{k-1}(A)\) are diagonal in these coordinates.  Thus, whenever
\({\partial_i u}\neq0\), the \(i\)-th
component of \eqref{eq:5.52} gives
\begin{equation}\label{eq:5.54}
      \bigl(L_k(A)\bigr)_{ii}
      =-\frac{n-k}{n-2k}\frac{|\nabla u|^2}{u}
       \bigl(T_{k-1}(A)\bigr)_{ii}.
\end{equation}
The \(i\)-th diagonal component of \eqref{eq:2.12} is
\[
 \bigl(L_k(A)\bigr)_{ii}^2
 \leq\frac{n-k}{n}\tr\bigl(L_k(A)A\bigr)
       \bigl(T_{k-1}(A)\bigr)_{ii}.
\]
Combining this inequality with \eqref{eq:5.54} and
\eqref{eq:5.53}, we obtain, for
\({\partial_i u}\neq0\),
\begin{equation*}
 \frac{(n-k)^2}{(n-2k)^2}
   \frac{|\nabla u|^4}{u^2}
   \bigl(T_{k-1}(A)\bigr)_{ii}^2
 \leq
 \frac{n-k}{n}
 \frac{n(n-k)}{(n-2k)^2}u^{-2}|\nabla u|^2
 \nabla u^TT_{k-1}(A)\nabla u\,
 \bigl(T_{k-1}(A)\bigr)_{ii}.
\end{equation*}
Since \(\nabla u\neq0\) and \(T_{k-1}(A)>0\), we can divide by the
positive common factors.  Thus
\begin{equation}\label{eq:5.55}
 \bigl(T_{k-1}(A)\bigr)_{ii}
 \leq
 \frac{\nabla u^TT_{k-1}(A)\nabla u}{|\nabla u|^2}
 \qquad\text{whenever }{\partial_i u}\neq0.
\end{equation}
Multiplying \eqref{eq:5.55} by
\({(\partial_i u)^2}\) and summing over all
indices for which \({\partial_i u}\neq0\), we get
\begin{align*}
 \sum_{\{i:\,{\partial_i u}\neq0\}}
 \bigl(T_{k-1}(A)\bigr)_{ii}{(\partial_i u)^2}
 &\leq
 \frac{\nabla u^TT_{k-1}(A)\nabla u}{|\nabla u|^2}
 \sum_{\{i:\,{\partial_i u}\neq0\}}
 {(\partial_i u)^2}
 =\nabla u^TT_{k-1}(A)\nabla u.
\end{align*}
On the other hand, because \(T_{k-1}(A)\) is diagonal in the chosen
basis, one has
\[
 \sum_{\{i:\,{\partial_i u}\neq0\}}
 \bigl(T_{k-1}(A)\bigr)_{ii}{(\partial_i u)^2}
 =\nabla u^TT_{k-1}(A)\nabla u.
\]
Therefore, equality holds in the preceding sum.  Since every
individual inequality is multiplied by the strictly positive number
\({(\partial_i u)^2}\), equality must hold for every
index with \({\partial_i u}\neq0\).  Hence
\begin{equation}\label{eq:5.56}
 \bigl(T_{k-1}(A)\bigr)_{ii}
 =\frac{\nabla u^TT_{k-1}(A)\nabla u}{|\nabla u|^2}.
\end{equation}
Equations \eqref{eq:5.54}, \eqref{eq:5.53}, and
\eqref{eq:5.56} now give
\begin{align*}
 \bigl(L_k(A)\bigr)_{ii}^2
 =\frac{(n-k)^2}{(n-2k)^2}
   \frac{|\nabla u|^4}{u^2}
   \bigl(T_{k-1}(A)\bigr)_{ii}^2 =\frac{n-k}{n}\tr\bigl(L_k(A)A\bigr)
   \bigl(T_{k-1}(A)\bigr)_{ii}.
\end{align*}
Thus equality holds in \eqref{eq:2.12} in every eigendirection for
which \({\partial_i u}\neq0\).

Fix such an index \(i\).  Lemma~\ref{Lem:2.5} shows that
all \(\lambda_j\), \(j\neq i\), are equal.  Moreover,
Corollary~\ref{Cor:2.3} gives \(\lambda(A)|i\in\Gamma_{k-1}\).
Thus, for every \(j\neq i\),
\(0<\sigma_1(\lambda(A)|i) =\sum_{\ell\neq i}\lambda_\ell =(n-1)\lambda_j\),
and hence \(\lambda_j>0\).

There cannot be two nonzero components
\({\partial_i u},{\partial_j u}\).
Indeed, suppose that \({\partial_i u}\neq0\) and
\({\partial_j u}\neq0\) for some \(i\neq j\).  Choose an
index \(\ell\neq i,j\), which is possible because \(n>2k\geq4\).
Applying the preceding conclusion to the \(i\)-th direction gives
\(\lambda_j=\lambda_\ell\), while applying it to the \(j\)-th
direction {yields} \(\lambda_i=\lambda_\ell\).  Hence all eigenvalues of
\(A\) are equal.  Then \(L_k(A)=0\), and
\eqref{eq:5.52} contradicts
\(T_{k-1}(A)>0\) and \(\nabla u\neq0\).  Hence \(\nabla u\) has exactly
one nonzero component and is an eigenvector of \(A\).

Fix \(j\neq i\).  The preceding positivity shows that
\(\lambda_j>0\).  By \eqref{eq:2.6},
\begin{align*}
    \bigl(T_{k-1}(A)\bigr)_{ii}
    &=\binom{n-1}{k-1}\lambda_j^{k-1},\\
    \sigma_k(A)
    &=\lambda_i\binom{n-1}{k-1}\lambda_j^{k-1}
      +\binom{n-1}{k}\lambda_j^k =\bigl(T_{k-1}(A)\bigr)_{ii}
      \left(\lambda_i+\frac{n-k}{k}\lambda_j\right),\\
    \bigl(T_k(A)\bigr)_{ii}
    &=\binom{n-1}{k}\lambda_j^k
     =\bigl(T_{k-1}(A)\bigr)_{ii}
      \frac{n-k}{k}\lambda_j.
\end{align*}
Therefore, by \eqref{eq:2.9},
\begin{align*}
    \frac{\bigl(L_k(A)\bigr)_{ii}}
         {\bigl(T_{k-1}(A)\bigr)_{ii}}
     =\frac{n-k}{n}
      \left(\lambda_i+\frac{n-k}{k}\lambda_j\right)
      -\frac{n-k}{k}\lambda_j =\frac{n-k}{n}(\lambda_i-\lambda_j).
\end{align*}
Substituting this equality into \eqref{eq:5.54}, we obtain
\(\frac{|\nabla u|^2}{u} =\frac{n-2k}{n}(\lambda_j-\lambda_i)\),
which is \eqref{eq:5.51}.
\end{proof}

Lemma~\ref{Lem:5.9} gives the following consequence for the
Hessian of \(u^{-2k/(n-2k)}\).

\begin{proposition}\label{Prop:5.10}
Assume that   \( \operatorname{div}\mathcal J\equiv 0\), and let
\begin{equation}\label{eq:5.57}
    w=u^{-2k/(n-2k)}.
\end{equation}
Then
\begin{equation}\label{eq:5.58}
    D^2w=\lambda(x)\Id
\end{equation}
for some continuous positive function \(\lambda\).
\end{proposition}

\begin{proof}
Differentiating \eqref{eq:5.57} twice gives
\begin{equation}\label{eq:5.59}
    D^2w
    =
    \frac{2k}{n-2k}
    u^{-n/(n-2k)}
    \left(
        A+\frac{n}{n-2k}
        \frac{\nabla u\otimes\nabla u}{u}
    \right).
\end{equation}
If \(\nabla u=0\), Lemma~\ref{Lem:5.9} shows that \(A\) is
scalar, and \eqref{eq:5.59} proves the assertion.  If
\(\nabla u\neq0\), choose \(i\) so that \(\nabla u\) is the
\(\lambda_i\)-eigendirection.  For
any \(j\neq i\), Lemma~\ref{Lem:5.9} and
\eqref{eq:5.51} give
\[
    A\xi=\lambda_j\xi
    \quad\text{for }\xi\perp\nabla u,
    \qquad
    \lambda_i+\frac{n}{n-2k}\frac{|\nabla u|^2}{u}
    =\lambda_j.
\]
Hence
\(A+\frac{n}{n-2k} \frac{\nabla u\otimes\nabla u}{u} =\lambda_j\Id\).
Therefore, \eqref{eq:5.59} gives \eqref{eq:5.58}.  Its scalar
coefficient is positive because
\(\lambda_j>0\) when \(\nabla u\neq0\), while \(A\in\Gamma_k\) when
\(\nabla u=0\).  Its continuity follows from that of \(D^2w\).
\end{proof}

\begin{proof}[Proof of Theorem~\ref{Thm:1.6}]
By Corollaries~\ref{Cor:5.4} and
\ref{Cor:5.8}, we conclude that
\(\operatorname{div}\mathcal J\equiv0 \qquad\text{for }2k<n\leq4k\).
By \eqref{eq:5.3}, \eqref{eq:5.50} holds at
every point.  Proposition~\ref{Prop:5.10} gives
\(D^2w=\lambda(x)\Id\).

Fix \(j\in\{1,\ldots,n\}\) and choose \(i\neq j\).  In the sense of
distributions,
\(\partial_j\lambda ={\partial_j\partial_i^2w} ={\partial_i^2\partial_jw} =0\).
Thus \(\lambda\) is constant.  Writing \(\lambda\equiv2b_0\), where
\(b_0>0\), and integrating \eqref{eq:5.58}, we obtain
\(w(x)=a_0+b_0|x-x_0|^2\).
Here we completed the square.  Since \(w>0\) in \(\R^n\), \(a_0>0\).
The relation \eqref{eq:5.57} now gives \eqref{eq:1.7}.

At \(x=x_0\), equation \eqref{eq:5.59} {yields}
\begin{equation}
    -D^2u(x_0)
    =
    \frac{n-2k}{k}b_0a_0^{-n/(2k)}\Id.
\end{equation}
Then we deduce from the equation  \(\sigma_k(-D^2u)=u^{p_*}\) that
\(
\begin{aligned}
    a_0^{-(n+2)/2}
    &=
    \binom nk
    \left(\frac{n-2k}{k}b_0\right)^k
    a_0^{-n/2},
\end{aligned}\)
 and hence
\(\binom nk a_0 \left(\frac{n-2k}{k}b_0\right)^k=1\).
This is \eqref{eq:1.9}.

Next, we show the bubble is a solution.
A direct calculation gives \(\sigma_j(-D^2u)>0\) for \(1\leq j\leq k\). Moreover,
\eqref{eq:1.9} {yields}
\begin{equation}
    \sigma_k(-D^2u)
    =
    \bigl(a_0+b_0|x-x_0|^2\bigr)^{-(n+2)/2}
    =
    u^{p_*}.
\end{equation}
This completes our proof of the unconditional classification results in Theorem \ref{Thm:1.6} for $2k<n\leq 4k$.
\end{proof}


\section{\texorpdfstring{The unconditional classification results for \(n>4k\)}
{The classification result for n>4k}}
\label{Sec:6}

In this Section, we will prove the unconditional classification results in Theorem~\ref{Thm:1.6} for
\(n>4k\). In Subsection~\ref{Subsec:6.1}, we prove the classification results under
\eqref{eq:6.3}.  In
Subsection~\ref{Subsec:6.2}, we prove that every solution
of \eqref{eq:1.8} satisfies \eqref{eq:6.3} and hence conclude our proof of Theorem~\ref{Thm:1.6}.

\subsection{Classification under an integral growth condition}
\label{Subsec:6.1}

Throughout this subsection, let \(k\geq2\), \(n>4k\), and let \(u\) be a
nonzero solution of \eqref{eq:1.8}.  Lemma~\ref{Lem:2.2} gives
\(u>0\) and \(-D^2u\in\Gamma_k\).

The right-hand side of  \eqref{eq:5.7} can be estimated
by \eqref{eq:3.5} with
\(\delta=-\frac{n-4k}{n-2k}\in (-1,0)\) as follows.   

\begin{proposition}\label{Prop:6.1}
For every \(R\geq1\), there holds
\begin{equation}\label{eq:6.1}
    \int_{B_R}\sigma_{k-1}(-D^2u)|\nabla u|^2
       u^{-\frac{2k}{n-2k}}\dd x
    \leq C\int_{B_{2R}}
       u^{\frac{n(k+1)-2k}{n-2k}}\dd x
       +CR^{2k/(k+1)},
\end{equation}
where \(C\) depends only on \(n\) and \(k\).
\end{proposition}

\begin{proof}
Let \(\chi=\chi_R\) be the cutoff function in
\eqref{eq:3.2}.  Applying
\eqref{eq:3.5} with
\(\delta=-\frac{n-4k}{n-2k}\) and
\(\theta=2n\), we obtain
\begin{align}\label{eq:6.2}
  \int\chi^{2n}\sigma_{k-1}(-D^2u)|\nabla u|^2
       u^{-\frac{2k}{n-2k}}\dd x
   \leq C\int_{B_{2R}}u^{\frac{n(k+1)-2k}{n-2k}}\dd x
       +CR^{-2k}\int_{B_{2R}}u^{k+\frac{n-4k}{n-2k}}\dd x.
\end{align}
Young's inequality 
gives
\(R^{-2k}\int_{B_{2R}}u^{k+\frac{n-4k}{n-2k}}\dd x
\leq C\int_{B_{2R}}u^{\frac{n(k+1)-2k}{n-2k}}\dd x
+CR^{2k/(k+1)}\).  Since \(\chi=1\) on \(B_R\), substituting this
in \eqref{eq:6.2} proves \eqref{eq:6.1}.
\end{proof}


\begin{theorem}\label{Thm:6.2}
Let \(k\geq2\), \(n>4k\), and let \(u\in C^2(\R^n)\) solve
\eqref{eq:1.8}.  Suppose that, for some \(C>0\) and all sufficiently
large \(R\),
\begin{equation}\label{eq:6.3}
 \int_{B_R}u^{\frac{n(k+1)-2k}{n-2k}}\dd x
 \leq CR^2\log R.
\end{equation}
Then either \(u\equiv0\), or \(u\) is the Hessian--Sobolev bubble
\eqref{eq:1.7} for some \(x_0\in\R^n\) and \(a_0,b_0>0\) satisfying
\eqref{eq:1.9}.
\end{theorem}

\begin{proof}[Proof of Theorem~\ref{Thm:6.2}]
If \(u\equiv0\), the conclusion follows.  Suppose that \(u\not\equiv0\).
By \eqref{eq:6.3} and \eqref{eq:6.1}, for \(R\geq1\),
\begin{equation}\label{eq:6.4}
 \int_{B_R}\sigma_{k-1}(-D^2u)|\nabla u|^2
       u^{-\frac{2k}{n-2k}}\dd x \leq CR^2\log(e+R)+CR^{2k/(k+1)}
 \leq CR^2\log(e+R).
\end{equation}


For two integers \(N\geq J\geq 1\), set
\(H_{J,N}=\sum_{j=J}^N(1+j)^{-1}\).
Define the radial Lipschitz function \(\eta_{J,N}\) by
\[
\eta_{J,N}(x)=
\begin{cases}
1,
    & |x|\leq 2^J,\\[2mm]
\displaystyle
1-\frac1{H_{J,N}}
\left(
\sum_{\ell=J}^{j-1}\frac1{1+\ell}
+\frac{|x|-2^j}{2^j(1+j)}
\right),
    & 2^j<|x|\leq 2^{j+1},\quad J\leq j\leq N,\\[3mm]
0,
    & |x|>2^{N+1},
\end{cases}
\]
where the sum is zero when \(j=J\).  Then \(\eta_{J,N}\) is
Lipschitz, piecewise linear in \(|x|\), \(0\leq\eta_{J,N}\leq1\), and
\(|\nabla\eta_{J,N}|=[2^j(1+j)H_{J,N}]^{-1}\) a.e. in
\(B_{2^{j+1}}\setminus B_{2^j}\).  By \eqref{eq:6.4},
\(\int_{B_{2^{j+1}}}\sigma_{k-1}(-D^2u)|\nabla u|^2
u^{-\frac{2k}{n-2k}}\dd x\leq C4^j(1+j)\).  Hence
\begin{align*}
 \int_{\R^n}|\nabla\eta_{J,N}|^2
  \sigma_{k-1}(-D^2u)|\nabla u|^2
     u^{-\frac{2k}{n-2k}}\dd x
& \leq
\sum_{j=J}^N
\frac1{4^j(1+j)^2H_{J,N}^2}
\int_{B_{2^{j+1}}\setminus B_{2^j}}
\sigma_{k-1}(-D^2u)|\nabla u|^2
   u^{-\frac{2k}{n-2k}}\dd x\\
&\quad\leq
\frac{C}{H_{J,N}^2}
\sum_{j=J}^N\frac1{1+j}
=\frac{C}{H_{J,N}}.
\end{align*}

Fix \(r>0\) and choose \(J\) so that \(B_r\subset B_{2^J}\).
Equations \eqref{eq:5.4} and \eqref{eq:5.7} give
\begin{align*}
 0 {\leq}
   \int_{B_r}\operatorname{div}\mathcal J\dd x
   \leq\int_{\R^n}\eta_{J,N}^3
      \operatorname{div}\mathcal J\dd x {\leq}
 C\int_{\R^n}\eta_{J,N}|\nabla\eta_{J,N}|^2
   \sigma_{k-1}(-D^2u)|\nabla u|^2
      u^{-\frac{2k}{n-2k}}\dd x
 \leq\frac{C}{H_{J,N}}.
\end{align*}
Since \(H_{J,N}\to\infty\), letting \(N\to\infty\) yields
\(\int_{B_r}\operatorname{div}\mathcal J\dd x=0\).
Since \(r>0\) is arbitrary and
\(\operatorname{div}\mathcal J\geq0\) by \eqref{eq:5.4}, we
conclude that \(\operatorname{div}\mathcal J\equiv0\).  Equation
\eqref{eq:5.3} then gives \eqref{eq:5.50}.
Applying Proposition~\ref{Prop:5.10} to
\(w=u^{-2k/(n-2k)}\), we obtain \(D^2w=\lambda\Id\). Since, for
\(i\neq j\),
\(\partial_j\lambda=\partial_j\partial_i^2w
=\partial_i^2\partial_jw=0\), \(\lambda\) is constant. Hence
\(w(x)=a_0+b_0|x-x_0|^2\), where \(a_0,b_0>0\), which gives
\eqref{eq:1.7}. Substituting  \eqref{eq:1.7} into
\eqref{eq:1.8} {yields}
\eqref{eq:1.9}.
\end{proof}

\subsection{Proof of the integral growth condition}
\label{Subsec:6.2}

In this subsection, we show the integral growth condition \eqref{eq:6.3}. We use only
\(b=(n-2k)/k\) and \(\gamma=2k/(n-2k)\).
Thus \(bk=n-2k\), \(bp_*=n+2\), and
\(b+1=(n-k)/k\).

We first prove
\(u\in L^\infty(\R^n)\) and
\(\int_{\R^n}u^{p_*}\dd x<\infty\).  Since
\(\frac{n(k+1)-2k}{n-2k}-p_*=(n-4k)/(n-2k)>0\),
these two conclusions imply \eqref{eq:6.3}. Lemmas~\ref{Lem:6.3}, \ref{Lem:6.4}, and
\ref{Lem:6.5}
provide the new \(k\)-Hessian compactness and Newton tensor argument. The construction of the $P$-function is crucial, which is inspired by \eqref{eq:5.18}. We also developed a B\^ocher expansion for solutions with single atomic Hessian measure in
Theorem~\ref{Thm:6.7}, and the stability estimate for the Pohozaev
integral in Lemma~\ref{Lem:6.9}.  Theorem~\ref{Thm:6.7},
Lemma~\ref{Lem:6.9}, and
Proposition~\ref{Prop:6.10} provide the new isolated-pole
argument used below. The \(k\)-Hessian compactness, Newton tensor argument, B\^ocher expansion, stability estimate for the Pohozaev
integral and isolated-pole
argument are new tools introduced by us. We also use the ideas of point selection and
the choice of first-contact radius from \cite{Zhang} in Propositions \ref{Prop:6.6}, \ref{Prop:6.13} and \ref{Prop:6.14}.

Let \(Z\geq0\) satisfy \(-Z\in\Phi^k(B_{3R}(x_0))\) and
\(\mu_k[-Z]=\mu\). The lower estimate in
\cite[Theorem~7.2]{PV2008} gives
\(Z(x_0)\geq c\int_0^{R/8}
\left(\frac{\mu(B_t(x_0))}{t^{n-2k}}\right)^{1/k}\frac{\dd t}{t}\).
Restricting this integral to \(r<t<2r\), where \(0<r<R/16\), yields
\begin{equation}\label{eq:6.5}
\mu(B_r(x_0))\leq CZ(x_0)^kr^{n-2k},\qquad 0<r<R/16.
\end{equation}
If \(Z>0\) is \(C^3\) and
\(\sigma_k(-D^2Z)=cZ^p\) in \(B_{2r}\), \(p>k\),
\(0\leq r^{2k}cZ^{p-k}\leq N\), and \(\sup\limits_{B_{2r}}Z\leq K\), then
\begin{equation}\label{eq:6.6}
\sup_{B_r}Z\leq C\inf_{B_r}Z,\qquad
|\nabla\log Z|\leq C/r\quad\hbox{in }B_r
\end{equation}
by \cite[Theorems~3.1 and~3.2]{MohammedPorru2026}.
{Indeed, with
\(S=\sup_{B_{2r}(x_0)}Z\) and \(V(y)=-S^{-1}Z(x_0+ry)\), one has
\(V<0\), \(D^2V\in\Gamma_k\), \(|V|\leq1\), and
\(\sigma_k(D^2V)=\lambda|V|^p\) on \(B_2\), where
\(\lambda=r^{2k}cS^{p-k}\leq N\). Thus \(g(t)=\lambda t^p\) and
\(h=0\) satisfy \textup{(p-g)}--\textup{(p-h)} with \(C_0=pN\).
The cited local estimates on \(B_1(y)\), followed by a finite Harnack
chain, prove \eqref{eq:6.6}.} We use the following consequence of the calculation in
\cite[Theorem~4.1, equations~(4.4)--(4.14)]{ChouWang2001}.
Let \(D\) be bounded and smooth, let \(I\Subset(-\infty,0)\), and let
\(q,\widehat q\in C^\infty(D)\cap C(\overline D)\) satisfy
\(D^2q,D^2\widehat q\in\Gamma_k\). Suppose that
\begin{equation}\label{eq:6.7}
 \sigma_k(D^2q)=\Psi(q),\qquad \widehat q>q\quad\hbox{in }D,
 \qquad \widehat q=q\quad\hbox{on }\partial D,
\end{equation}
where \(q(D)\subset I\), \(\Psi\in C^2(I)\), and
\(\Psi\geq\psi_0>0\). If
\(\|\nabla q\|_{L^\infty(D)}+\|\nabla\widehat q\|_{L^\infty(D)}\leq M\), then
\begin{equation}\label{eq:6.8}
 \sup_D(\widehat q-q)^4|D^2q|
 \leq C(n,k,M,\psi_0,\|\Psi\|_{C^2(I)},I).
\end{equation}
In order to prove \eqref{eq:6.8}, put
\(F=\sigma_k^{1/k}\), \(f=\Psi^{1/k}\), and
\(\rho=\widehat q-q\). Concavity and homogeneity of \(F\), followed by
the first and second derivatives of \eqref{eq:6.7}, give
\begin{align}
 F^{ij}\rho_{ij}&\geq-f(q), \qquad
 F^{ij}q_{ij\ell}=f'(q)q_\ell,\label{eq:6.9}\\
 F^{ij}q_{ij\ell\ell}
 &=f'(q)q_{\ell\ell}+f''(q)q_\ell^2
 -F^{ij,rs}q_{ij\ell}q_{rs\ell}\geq f'(q)q_{\ell\ell}+f''(q)q_\ell^2.
 \label{eq:6.10}
\end{align}
Here \(F\) and its derivatives are evaluated at \(D^2q\).
For a positive regular value \(\varepsilon\) of \(\rho\), let \(G\)
be a connected component of \(\{\rho>\varepsilon\}\) and put
\(\rho_\varepsilon=\rho-\varepsilon\). The function
\(\widehat q-\varepsilon\) is admissible, is larger than \(q\) in
\(G\), and equals \(q\) on \(\partial G\). The quantity used in
\cite[equation~(4.4)]{ChouWang2001} is
\begin{equation}\label{eq:6.11}
 \rho_\varepsilon^4
 \phi\left(\frac{|\nabla q|^2}{2}\right)q_{\xi\xi},
\end{equation}
where \(\phi\) is chosen in \cite{ChouWang2001} and \(\xi\) is a unit vector. At a
positive maximum of \eqref{eq:6.11}, choose coordinates for which
\(D^2q\) is diagonal and \(q_{11}\) is its largest eigenvalue. If
\(q_{11}\leq1\), the cone condition gives \(|D^2q|\leq C(n,k)\).
If \(q_{11}>1\), equations \eqref{eq:6.9} and
\eqref{eq:6.10} show that the only terms in
\cite[equations~(4.5)--(4.10)]{ChouWang2001} changed by the
nonconstant right-hand side are
\begin{equation}\label{eq:6.12}
 \frac{\phi'}{\phi}q_\ell F^{ij}q_{ij\ell}
 =\frac{\phi'}{\phi}f'(q)|\nabla q|^2,\qquad
 \frac{F^{ij}q_{ij11}}{q_{11}}
 \geq f'(q)+f''(q)\frac{q_1^2}{q_{11}}.
\end{equation}
Moreover,
\begin{equation}\label{eq:6.13}
 \|f\|_{C^2(I)}
 \leq C(\psi_0,\|\Psi\|_{C^2(I)},I).
\end{equation}
Since \(q_{11}>1\) and \(|\nabla q|\leq M\), equations
\eqref{eq:6.12}--\eqref{eq:6.13} give
\begin{equation}\label{eq:6.14}
 |f(q)|
 +\left|\frac{\phi'}\phi f'(q)|\nabla q|^2\right|
 +|f'(q)|
 +\left|f''(q)\frac{q_1^2}{q_{11}}\right|\leq C.
\end{equation}
All other terms are the same as
\cite[equations~(4.11)--(4.14)]{ChouWang2001}.  Equations \eqref{eq:6.12} and
\eqref{eq:6.14}, substituted in
\cite[equation~(4.14)]{ChouWang2001}, imply
\(\sup_G\rho_\varepsilon^4|D^2q|
\leq C(n,k,M,\psi_0,\|\Psi\|_{C^2(I)},I)\).
The constant is independent of \(G\) and \(\varepsilon\). Letting \(\varepsilon\rightarrow0\) proves
\eqref{eq:6.8}.

\begin{lemma}\label{Lem:6.3}
Let \(p>k\), let \(R_j\to\infty\), and let \(Z_j\) be smooth
solutions in \(B_{R_j}\) such that
\begin{equation}\label{eq:6.15}
 0<Z_j\leq C_0,\quad |\nabla Z_j|\leq C_1,\quad Z_j(0)\geq c_0>0,
 \qquad -D^2Z_j\in\Gamma_k,\quad
 \sigma_k(-D^2Z_j)=Z_j^p.
\end{equation}
Then
\begin{equation}\label{eq:6.16}
 |D^2Z_j(0)|\leq C(n,k,p,C_0,C_1,c_0)
\end{equation}
for all sufficiently large \(j\).
\end{lemma}

\begin{proof}
Put \(q_j=-Z_j\) and \(r_j=R_j/2\).  Choose
\(0<\varepsilon_j<\min\{r_j^{-4k},
\frac12\inf_{B_{r_j}}Z_j^p\}\).

Since \(\sigma_k(D^2q_j)=Z_j^p>2\varepsilon_j\), the function
\(q_j\) is a strict admissible subsolution on \(B_{R_j}\).  The smooth
Dirichlet theorem \cite{CNS1985,TW1999} gives a smooth admissible
solution \(H_j\) satisfying
\(\sigma_k(D^2H_j)=\varepsilon_j\) in \(B_{r_j}\) and
\(H_j=q_j\) on \(\partial B_{r_j}\).

Lemma~\ref{Lem:4.1} and subharmonicity {yield}
\begin{equation}\label{eq:6.17}
 -C_0\leq q_j<H_j\leq0\quad\hbox{in }B_{r_j}.
\end{equation}
We will show that
\begin{equation}\label{eq:6.18}
 H_j(0)\longrightarrow0.
\end{equation}
The normalization in \cite[Theorem~2.4]{Hu2026} satisfies
\begin{equation}
 \sigma_k\!\left(D^2(\varepsilon_j^{-1/k}H_j)\right)=1,\qquad
 \operatorname{osc}_{B_{r_j}}(\varepsilon_j^{-1/k}H_j)
 =\varepsilon_j^{-1/k}\operatorname{osc}_{B_{r_j}}H_j .
\end{equation}
{Apply
\cite[Theorem~2.4]{Hu2026} to
\(\varepsilon_j^{-1/k}(H_j-\sup_{B_{r_j}}H_j)\), whose Hessian is
unchanged and whose \(L^\infty\)-norm is
\(\varepsilon_j^{-1/k}\operatorname{osc}_{B_{r_j}}H_j\). With
\eqref{eq:6.17}, this yields}
\begin{equation}\label{eq:6.19}
 \varepsilon_j^{-1/k}\|\nabla H_j\|_{L^\infty(B_{r_j/2})}
 \leq \frac{C(n,k)}{r_j}
 \operatorname{osc}_{B_{r_j}}(\varepsilon_j^{-1/k}H_j)
 \leq\frac{C(n,k)C_0\varepsilon_j^{-1/k}}{r_j}.
\end{equation}
If \eqref{eq:6.18} failed, then \(H_j(0)\leq-c_2<0\) along a subsequence.
By \eqref{eq:6.17}--\eqref{eq:6.19}, \(Z_j\geq c_2/2\) on every fixed ball.  By \eqref{eq:6.15}, after selecting a diagonal subsequence, \(Z_j\to Z\) locally uniformly in
\(\R^n\), with \(c_2/2\leq Z\leq C_0\).

Maclaurin's inequality applied to \(-D^2Z_j\) gives
\begin{equation}\label{eq:6.20}
 -\Delta Z_j\geq
 n\binom nk^{-1/k}Z_j^{p/k}.
\end{equation}
Passing to the limit in \eqref{eq:6.20}, the spherical mean of \(Z\) satisfies
\((r^{n-1}\overline Z'(r))'\leq-cr^{n-1}\).  Hence
\(\overline Z(r)\leq\overline Z(0)-cr^2/(2n)<0\) for large \(r\),
contrary to \(Z\geq c_2/2\).  Thus \eqref{eq:6.18} holds.

By \eqref{eq:6.15} and \eqref{eq:6.18}, we have
\begin{equation}\label{eq:6.21}
 H_j(0)-q_j(0)\geq c_0/2
\end{equation}
for all large \(j\).  Let \(P_j\) be the harmonic extension in $B_{r_j}$
of \(q_j|_{\partial B_{r_j}}\).  Since \(H_j\) is subharmonic,
\(H_j\leq P_j\).  Fix \(0<\alpha<1\) and write
\(d_j(x)=\operatorname{dist}(x,\partial B_{r_j})\).  Since
\(q_j=-Z_j\), equation~\eqref{eq:6.15} gives
\(|q_j|\leq C_0\) and \(|\nabla q_j|\leq C_1\) on
\(\partial B_{r_j}\).  For \(0<d_j(x)\leq1\), choose
\(\xi_x\in\partial B_{r_j}\) with \(|x-\xi_x|=d_j(x)\).  The Poisson
formula and
\(|q_j(\xi)-q_j(\xi_x)|\leq
\min\{2C_0,C_1|\xi-\xi_x|\}\) yield
\begin{equation}\label{eq:6.22}
 |P_j(x)-q_j(\xi_x)|
 \leq\int_{\partial B_{r_j}}
 \frac{r_j^2-|x|^2}{r_j|\mathbb S^{n-1}||x-\xi|^n}
 \min\{2C_0,C_1|\xi-\xi_x|\}\dd S_\xi
 \leq C(C_0,C_1,\alpha)d_j(x)^\alpha .
\end{equation}
Since \(|q_j(x)-q_j(\xi_x)|\leq C_1d_j(x)\),
\eqref{eq:6.22} implies
\begin{equation}\label{eq:6.23}
 0<H_j(x)-q_j(x)\leq P_j(x)-q_j(x)
 \leq C(C_0,C_1,\alpha)d_j(x)^\alpha,
 \qquad\hbox{if }d_j(x)\leq1.
\end{equation}
For every \(\eta\in[c_0/8,c_0/4]\),
\eqref{eq:6.23} gives
\[
 H_j(x)-q_j(x)>\eta\quad\Longrightarrow\quad
 d_j(x)\geq d_*:=\min\left\{1,
 \left(\frac{c_0}{8C(C_0,C_1,\alpha)}\right)^{1/\alpha}\right\}.
\]
By applying the local gradient estimate \cite[Theorem~3.2]{ChouWang2001} on \(B_{d_*/2}(x)\), we have
\begin{equation}\label{eq:6.24}
 |\nabla H_j|\leq C(n,k,C_0,d_*)\quad
 \hbox{on }\{H_j-q_j>c_0/8\}.
\end{equation}
Choose a regular value
\(\eta_j\in(c_0/8,c_0/4)\), and let \(G_j\) be the connected
component of \(\{H_j-q_j>\eta_j\}\) containing the origin.  It is a
smooth domain compactly contained in \(B_{r_j}\).  By \eqref{eq:6.17},
\(Z_j\geq H_j-q_j>\eta_j\geq c_0/8\) in \(G_j\).

The function \(H_j-\eta_j\) is admissible and satisfies
\begin{equation}\label{eq:6.25}
 H_j-\eta_j>q_j\quad\hbox{in }G_j,\qquad
 H_j-\eta_j=q_j\quad\hbox{on }\partial G_j.
\end{equation}
Choose a fixed open interval \(I\Subset(-\infty,0)\) containing
\([-C_0,-c_0/8]\).  On \(I\), the function
\(\Psi(z)=(-z)^p\) has a positive lower bound and a uniform
\(C^2\)-norm. On \(G_j\), equations \eqref{eq:6.15} and
\eqref{eq:6.24} give
\(\|\nabla q_j\|_\infty+\|\nabla(H_j-\eta_j)\|_\infty\leq C_1+C\). By \eqref{eq:6.25}, $H_j-\eta_j$ is the strict admissible upper barrier $\widehat{q}$
required in \eqref{eq:6.8}.
Thus \eqref{eq:6.8} implies
\begin{equation}\label{eq:6.26}
 (H_j-q_j-\eta_j)^4|D^2q_j|\leq C
 \quad\hbox{on }G_j.
\end{equation}
It follows from \eqref{eq:6.21} and \(\eta_j<c_0/4\) that
\(H_j(0)-q_j(0)-\eta_j\geq c_0/4\).  Since
\(|D^2q_j|=|D^2Z_j|\), \eqref{eq:6.26} proves \eqref{eq:6.16}.
\end{proof}

Set
\begin{equation}
 h_0:=\frac{\gamma}{2\binom nk^{1/k}},
 \qquad U(y):=(1+h_0|y|^2)^{-1/\gamma}.
 \label{eq:6.27}
\end{equation}
Substitution of \eqref{eq:6.27} into \eqref{eq:1.8} gives
\(-D^2U\in\Gamma_k\),
\(\sigma_k(-D^2U)=U^{p_*}\), and \(U(0)=1\).

We use a \(P\)-function argument together with the identities and the
quantitative Newton--Maclaurin inequality for matrix in Section~\ref{Sec:2}.

Let \(v>0\) be a smooth solution of \eqref{eq:1.8}, and put
\(w=v^{-\gamma}\) and
\(B=D^2w-\frac{n}{2kw}\nabla w\otimes \nabla w\).

Differentiating \(w=v^{-\gamma}\) and using \eqref{eq:1.8} gives
\begin{equation}\label{eq:6.28}
 B\in\Gamma_k,
 \qquad
 \sigma_k(B)=\frac{\gamma^k}{w}.
\end{equation}
Inspired by \eqref{eq:5.18}, we define \(h>0\) by
\begin{equation}\label{eq:6.29}
 \binom nk2^{k-2}h^{k-1}(4wh-|\nabla w|^2)=\gamma^k.
\end{equation}
For \(h>|\nabla w|^2/(4w)\), its \(h\)-derivative is
\(\binom nk2^{k-2}h^{k-2}(4kwh-(k-1)|\nabla w|^2)>0\).
Thus \eqref{eq:6.29} has one positive root, which is smooth by the
implicit function theorem.
Equivalently, for
\(K=k-(k-1)|\nabla w|^2/(4wh)\), one has
\begin{equation}\label{eq:6.30}
 0\leq\frac{|\nabla w|^2}{4wh}<1,\qquad
 h^kw\left(1-\frac{|\nabla w|^2}{4wh}\right)=h_0^k.
\end{equation}
Differentiation of \eqref{eq:6.29} yields
\begin{equation}\label{eq:6.31}
 4whK\,\nabla\log h=2(D^2w-2h\Id)\nabla w.
\end{equation}
The linearized matrix of \eqref{eq:6.28} is
\(\mathscr A=wT_{k-1}(B)>0\).

Let \(\phi>0\), put
\begin{equation}
  P=\log(h\phi(w)),
\end{equation}
and write
\(\eta=-d\log\phi/d\log w\) and
\(\dot\eta=d\eta/d\log w\).  At a point where \(\nabla w\ne0\), choose
\(e_1=\nabla w/|\nabla w|\), write
\(M=B/h=\begin{pmatrix}m_{11}&z^T\\ z&C\end{pmatrix}\),
and put \(N=T_{k-1}(M)\), \(\Sigma=\sigma_k(M)\),
\(T=\sigma_{k-1}(C)\), \(Q=\tr(N M^2)\), and
\(J=k-(k-1)K\); for vectors \(X,Y\),
{write \(N[X,Y]=X^TNY\)}.  Then
\begin{equation}\label{eq:6.32}
 P_i=(\log h)_i-\eta\frac{w_i}{w},\qquad
 P_{ij}=(\log h)_{ij}-\eta\frac{w_{ij}}w
 +(\eta-\dot\eta)\frac{w_iw_j}{w^2}.
\end{equation}
Differentiation of \eqref{eq:6.31} yields
\begin{equation}\label{eq:6.33}
 4whK(\log h)_{ij}
 =2w_{i\ell j}w_\ell
 +2(w_{i\ell}-2h\delta_{i\ell})w_{\ell j}
 -4h_jw_i-4(whK)_j(\log h)_i.
\end{equation}
Differentiation of \eqref{eq:6.28}, followed by contraction with
\(\nabla w\), implies
\begin{equation}\label{eq:6.34}
 N^{ij}w_{ij\ell}w_\ell
 =-\frac{h\Sigma}{w}|\nabla w|^2
 +\frac n{kw}N[D^2w\nabla w,\nabla w]
 -\frac n{2kw^2}|\nabla w|^2N[\nabla w,\nabla w].
\end{equation}
Equation \eqref{eq:6.31} also gives
\begin{equation}
 \nabla\left(\frac{|\nabla w|^2}{4wh}\right)
 =\left(1-\frac{|\nabla w|^2}{4wh}\right)
 \left(\frac{\nabla w}{w}+k\nabla\log h\right),\qquad
 \nabla\log(whK)=\frac{\nabla w}{wK}+\frac JK\nabla\log h.
\end{equation}
{Indeed,
\(D^2w/h=M+2(n/k)|\nabla w|^2e_1\otimes e_1/(4wh)\),
\(N^{11}=T\), \(NM=MN\), and \(\tr(NM)=k\Sigma\), which yield}
\begin{equation}\label{eq:6.35}
 \frac1h \tr(N D^2w)=k\Sigma
 +2\frac nk\frac{|\nabla w|^2}{4wh}T,\qquad
 \frac{N[\nabla w,\nabla w]}{wh}=4\frac{|\nabla w|^2}{4wh}T,
\end{equation}
and
\begin{equation}\label{eq:6.36}
\frac{N[D^2w\nabla w,\nabla w]}{wh^2}
 =4\frac{|\nabla w|^2}{4wh}\,e_1^TNMe_1
 +8\frac nk\left(\frac{|\nabla w|^2}{4wh}\right)^2T.
\end{equation}
Using \eqref{eq:6.32}--\eqref{eq:6.36}, we obtain
\begin{equation}\label{eq:6.37}
\begin{aligned}
 &h^{-k}\left[\mathscr A^{ij}P_{ij}+\mathscr A\left(
 \frac{2+\eta J}{wK}\nabla w+\frac JK\nabla P
 \right)\cdot \nabla P\right]\\
 &=\frac1{2K}Q
 -\frac{2k+4|\nabla w|^2/(4wh)}{2K}\Sigma
 +\frac4K\frac nk\frac{|\nabla w|^2}{4wh}\,e_1^TNMe_1\\
 &\quad+\frac1{2K}\left[
 12\frac{n^2}{k^2}\left(\frac{|\nabla w|^2}{4wh}\right)^2
 -4\frac nk\frac{|\nabla w|^2}{4wh}
 -8\frac nk\left(\frac{|\nabla w|^2}{4wh}\right)^2\right]T\\
 &\quad-\eta\left(k\Sigma+2\frac nk\frac{|\nabla w|^2}{4wh}T\right)
 +4\frac{|\nabla w|^2}{4wh}T\left[
 \eta-\dot\eta-\frac{\eta(2+\eta J)}K
 \right]
 -\frac{\eta J|\nabla w|}{hK}N[\nabla P,e_1].
\end{aligned}
\end{equation}
Choose
\begin{equation}\label{eq:6.38}
 1<s_\infty<\frac nk-1,
 \qquad
 0<\tau<\frac{\frac{n}{k}-2}{2}
 \left(1-\frac{s_\infty}{\frac{n}{k}-1}\right),
\end{equation}
and define
\begin{equation}\label{eq:6.39}
 \phi_m(w)=
 \left(\frac{1+m}{m+w^\tau}\right)^{s_\infty/\tau}.
\end{equation}
\begin{lemma}\label{Lem:6.4}
Let \(k\geq2\) and \(w\geq1\).  There is a number
\(\overline\ell>h_0\) such that the functions in
\eqref{eq:6.39} satisfy
\begin{equation}\label{eq:6.40}
  \mathscr A^{ij}(P_m)_{ij}>0
\end{equation}
at every critical point of \(\{h\phi_m(w)>\overline\ell\}\), where \(P_m=\log(h\phi_m(w))\).  Moreover, for every
\(\ell>h_0\), there are smooth functions
\(0<\phi_{m\ell}\leq1\) such that
\begin{equation}\label{eq:6.41}
 \phi_{m\ell}(t)\longrightarrow1
 \quad\hbox{locally uniformly for }t\geq1,
 \qquad
 \phi_{m\ell}(t)=O_{m\ell}(t^{-s_\infty}),\quad s_\infty>1,
\end{equation}
and every critical point of \(P_{m\ell}=\log(h\phi_{m\ell}(w))\) in
\(\{h\phi_{m\ell}(w)>\ell\}\) has the strict sign
\begin{equation}\label{eq:6.42}
 \mathscr A^{ij}(P_{m\ell})_{ij}>0.
\end{equation}
\end{lemma}

\begin{proof}
Let
\(\eta_m=-d\log\phi_m/d\log w\) and
\(\dot\eta_m=d\eta_m/d\log w\). In the direction
\(e_1=\nabla w/|\nabla w|\), the eigenvalue of \(B/h\) is
\(b_1=2-2\frac nk\frac{|\nabla w|^2}{4wh}+2K\eta_m\).

When \(|\nabla w|^2/(4wh)\) is sufficiently close to one, \(b_1<0\).  Writing the
remaining eigenvalues as \(y_2,\ldots,y_n\), put
\(\Sigma=\binom nk2^k(1-|\nabla w|^2/(4wh))\),
\(T=\sigma_{k-1}(y_2,\ldots,y_n)\), and
\(Q=\tr(T_{k-1}(B/h)(B/h)^2)\). Corollary~\ref{Cor:2.3} gives
\(T>0\) and
\(\tr T_{k-1}(B/h)\leq(n-k+1)T\).
We have
\(\sigma_k(y')=(-b_1+\Sigma/T)\sigma_{k-1}(y')\), where
\(y'=(y_2,\ldots,y_n)\).

Since \(M\in\Gamma_k\) and \(b_1<0\),
\((y_2,\ldots,y_n)\in\Gamma_k\). From Lemma~\ref{Lem:2.1}, in
dimension \(n-1\), we deduce that
\(\sigma_1(y')\geq k(n-1)(-b_1+\Sigma/T)/(n-k)\) and
\((k+1)\sigma_{k+1}(y')/\sigma_{k-1}(y')
\leq k(n-k-1)(-b_1+\Sigma/T)^2/(n-k)\).

Since
\(Q/T=(\Sigma/T)(b_1+\sigma_1(y'))
+(k+1)[(-b_1)(-b_1+\Sigma/T)
-\sigma_{k+1}(y')/\sigma_{k-1}(y')]\), we obtain
\begin{equation}\label{eq:6.43}
 \frac QT\geq
 \frac{nb_1^2-k(k+1)b_1\Sigma/T+k^2(\Sigma/T)^2}{n-k}.
\end{equation}
By Lemma~\ref{Lem:2.1}, we also have
\begin{equation}\label{eq:6.44}
 q\binom{n-1}{k-1}
 \left[\frac{k(-b_1+q)}{n-k}\right]^{k-1}\leq\Sigma
\end{equation}
with $q=\frac{\Sigma}{T}$. Here \eqref{eq:6.44} is equivalent to
\(T\geq\binom{n-1}{k-1}
[k(-b_1+\Sigma/T)/(n-k)]^{k-1}\).

Let \(q_*(s,\theta)\), for scalar variables \(s\) near one and
\(0\leq\theta\leq s_\infty\), be the unique nonnegative solution
obtained by replacing the inequality in \eqref{eq:6.44} by equality,
with \(|\nabla w|^2/(4wh)\) and \(\eta_m\) replaced by \(s\) and \(\theta\),
respectively.  Then
\(0\leq\Sigma/T\leq q_*(|\nabla w|^2/(4wh),\eta_m)\), and equality is attained by
\[
 y_2=\cdots=y_n=\frac{k}{n-k}\left[
 -2+2\frac nk\frac{|\nabla w|^2}{4wh}-2K\eta_m
 +q_*\left(\frac{|\nabla w|^2}{4wh},\eta_m\right)\right].
\]

At \(\nabla P_m=0\), equations \eqref{eq:6.43} and
\eqref{eq:6.37} yield
\begin{equation}\label{eq:6.45}
\begin{aligned}
 \frac{h^{-k}}T\mathscr A^{ij}(P_m)_{ij}
 \geq{}&
 \frac{nb_1^2-k(k+1)b_1\Sigma/T+k^2(\Sigma/T)^2}{2K(n-k)}\\
 &-\left(\frac{k+2|\nabla w|^2/(4wh)}{K}+k\eta_m\right)\frac\Sigma T
 +\frac{4(n/k)(|\nabla w|^2/(4wh))b_1}{K}\\
 &+\frac{6(n/k)^2(|\nabla w|^2/(4wh))^2
 -2(n/k)|\nabla w|^2/(4wh)-4(n/k)(|\nabla w|^2/(4wh))^2}{K}\\
 &-2\frac nk\frac{|\nabla w|^2}{4wh}\eta_m
 +4\frac{|\nabla w|^2}{4wh}\eta_m
 -\frac{4(|\nabla w|^2/(4wh))\eta_m(2+\eta_m J)}K
 -4\frac{|\nabla w|^2}{4wh}\dot\eta_m .
\end{aligned}
\end{equation}
The number \(q_*=q_*(s,\theta)\) is determined by
\begin{equation}\label{eq:6.46}
 q_*\binom{n-1}{k-1}
 \left\{\frac{k}{n-k}\left[
 -2+2\frac nk s-2\bigl(k-(k-1)s\bigr)\theta+q_*\right]\right\}^{k-1}
 =\binom nk2^k(1-s).
\end{equation}
For \(|\nabla w|^2/(4wh)\) in a fixed neighborhood of $1$,
\(-b_1\) is bounded below by a positive constant depending only on
\(n,k,s_\infty\).  Differentiating the left-hand side of
\eqref{eq:6.46} with respect to \(q_*\) and applying the
implicit function theorem gives
\begin{equation}\label{eq:6.47}
 0\leq q_*\leq C\left(1-\frac{|\nabla w|^2}{4wh}\right),\qquad
 |\partial_1q_*|+|\partial_2q_*|\leq C
\end{equation}
for \(0\leq\eta_m\leq s_\infty\). Let
\(\mathcal E(|\nabla w|^2/(4wh),\eta_m,q)\) denote the right-hand side of
\eqref{eq:6.45} after
\(\dot\eta_m=\tau\eta_m\left(1-\frac{\eta_m}{s_\infty}\right)\) is substituted and
\(\Sigma/T\) is replaced by the scalar variable \(q\). Its
\(q\)-derivative satisfies
\begin{equation}\label{eq:6.48}
\begin{aligned}
 &\left[
 \frac{-k(k+1)b_1+2k^2q}{2K(n-k)}
 -\left(\frac{k+2|\nabla w|^2/(4wh)}{K}+k\eta_m\right)
 \right]_{|\nabla w|^2/(4wh)=1,q=0}\\
 &\qquad=-1-\eta_m\left(k+\frac{k+1}{\frac{n}{k}-1}\right)\leq-1.
\end{aligned}
\end{equation}
Equations \eqref{eq:6.47} and
\eqref{eq:6.48} show that the right-hand side of
\eqref{eq:6.45} decreases on \(0\leq q\leq q_*\) for \(|\nabla w|^2/(4wh)\) in a fixed neighborhood of $1$.

When \(\eta_m=0\), equation \eqref{eq:6.46} gives
\(q_*(|\nabla w|^2/(4wh),0)=2(n/k)(1-|\nabla w|^2/(4wh))\), and equality in \eqref{eq:6.45} is attained by \(y_2=\cdots=y_n=2\).  Thus \(\mathcal E(|\nabla w|^2/(4wh),\eta_m,q_*(|\nabla w|^2/(4wh),0))=0\). Equation
\eqref{eq:6.47} and
\(\dot\eta_m=\tau\eta_m(1-\eta_m/s_\infty)\) show that \(\mathcal E(|\nabla w|^2/(4wh),\eta_m,q)/\eta_m\), evaluated at \(q=q_*(|\nabla w|^2/(4wh),\eta_m)\), extends continuously to \(\eta_m=0\), uniformly
for \(|\nabla w|^2/(4wh)\) in a fixed neighborhood of $1$.  Before division by \(\eta_m\), we have
\begin{equation}\label{eq:6.49}
 \mathcal E(|\nabla w|^2/(4wh),\eta_m,q)\mid_{|\nabla w|^2/(4wh)=1,q=q_*}=2\eta_m\left(\frac nk-2\right)
 \left(1-\frac{\eta_m}{n/k-1}\right)-4\dot\eta_m.
\end{equation}
Equations \eqref{eq:6.38} and \eqref{eq:6.49} imply that
\[
 \frac{2\eta_m(n/k-2)(1-\eta_m/(n/k-1))-4\dot\eta_m}{\eta_m}
 \geq2\left(\frac nk-2\right)
 \left(1-\frac{s_\infty}{n/k-1}\right)-4\tau>0.
\]
Equations \eqref{eq:6.47},
\eqref{eq:6.48}, and \eqref{eq:6.49} give
\(\delta,c_1>0\), depending only on
\(n,k,s_\infty,\tau\), such that
\begin{equation}\label{eq:6.50}
 \partial_q\mathcal E\left(\frac{|\nabla w|^2}{4wh},\eta_m,q\right)\leq-\frac12,
 \qquad
 \mathcal E\left(\frac{|\nabla w|^2}{4wh},\eta_m,
 q_*\left(\frac{|\nabla w|^2}{4wh},\eta_m\right)\right)\geq c_1\eta_m
\end{equation}
when \(1-\delta<|\nabla w|^2/(4wh)<1\), \(0\leq\eta_m\leq s_\infty\), and
\(0\leq q\leq q_*(|\nabla w|^2/(4wh),\eta_m)\). Since
\(\Sigma/T\leq q_*(|\nabla w|^2/(4wh),\eta_m)\), equations
\eqref{eq:6.45} and \eqref{eq:6.50} {yield}
\begin{equation}
 \frac{h^{-k}}T\mathscr A^{ij}(P_m)_{ij}\geq c_1\eta_m
 \qquad\hbox{if }\frac{|\nabla w|^2}{4wh}>1-\delta.
\end{equation}

Choose \(\overline\ell>h_0\delta^{-1/k}\). On \(\{h\phi_m(w)>\overline\ell\}\), \eqref{eq:6.30} gives
\(|\nabla w|^2/(4wh)>1-\delta\), so \eqref{eq:6.40} holds at every critical point.

\medskip

Next, fix arbitrarily \(\ell>h_0\), and put
\(\delta_\ell=1-(h_0/\ell)^k>0\).

{Decrease the \(\delta\) in
\eqref{eq:6.50} so that the extension of \(\mathcal E/\eta\) at
\(\eta=0\) used in \eqref{eq:6.49} ensures, whenever
\(s>1-\delta\), \(0\leq\eta\leq s_\infty\),
\(0\leq q\leq q_*(s,\eta)\), and
\(\dot\eta\leq\tau\eta(1-\eta/s_\infty)\), the right-hand side of
\eqref{eq:6.45} is at least
\(\eta[(n/k-2)(1-s_\infty/(n/k-1))-2\tau]\).} Then choose \(T_\ell>0\) so that
\((h_0/\ell)^ke^{-T_\ell}<\delta\). In \(\{h\phi_{m\ell}(w)>\ell\}\), \eqref{eq:6.30} gives
\(1-|\nabla w|^2/(4wh)<(h_0/\ell)^k\phi_{m\ell}(w)^k/w\). Thus \(\log w\geq T_\ell\) and
\(h\phi_{m\ell}(w)>\ell\) imply \(|\nabla w|^2/(4wh)>1-\delta\).

For \(0\leq t\leq T_\ell\), choose
\begin{equation}\label{eq:6.51}
 \eta_{m\ell}(t)=\varepsilon_{m\ell}e^{-At},
 \qquad
 A>\frac{C_{\delta_\ell}}{4\delta_\ell},
\end{equation}
where \(\varepsilon_{m\ell}\downarrow0\) as $m\rightarrow+\infty$ and
\(0<\varepsilon_{m\ell}<\min\{s_\infty/2,\varepsilon_0\}\).
Here \(\varepsilon_0=\varepsilon_0(n,k,\delta_\ell)>0\) is chosen so
that
\(k\varepsilon_0/(n/k-1)\leq1-2^{-1/(k-1)}\), and so that the
implicit-function construction in \eqref{eq:6.57}
is uniform for \(|\nabla w|^2/(4wh)\in[\delta_\ell,1]\) and
\(0\leq\eta_{m\ell}\leq\varepsilon_0\).  Fix \(d_0>0\), and choose
\(\chi\in C^\infty([0,\infty))\) with \(0\leq\chi\leq1\),
\(\chi=1\) on \([0,T_\ell+d_0]\), and \(\chi=0\) on
\([T_\ell+2d_0,\infty)\).  Put
\begin{equation}\label{eq:6.52}
 r_{m\ell}(t)=-\frac{A\chi(t)}{1-\varepsilon_{m\ell}e^{-At}/s_\infty}
 +\tau(1-\chi(t))\leq\tau,
\end{equation}
and let \(\eta_{m\ell}\) solve
\begin{equation}\label{eq:6.53}
 \dot\eta_{m\ell}=r_{m\ell}(t)\eta_{m\ell}\left(1-\frac{\eta_{m\ell}}{s_\infty}\right),
 \qquad \eta_{m\ell}(0)=\varepsilon_{m\ell}.
\end{equation}
Uniqueness in \eqref{eq:6.53} and
\eqref{eq:6.52} gives
\begin{equation}\label{eq:6.54}
\begin{aligned}
 &\eta_{m\ell}(t)=\varepsilon_{m\ell}e^{-At}
 \quad(0\leq t\leq T_\ell+d_0),
 \qquad 0<\eta_{m\ell}<s_\infty,\\
 & \dot\eta_{m\ell}\leq\tau\eta_{m\ell}
 (1-\eta_{m\ell}/s_\infty),
 \qquad \eta_{m\ell}(t)\longrightarrow s_\infty.
\end{aligned}
\end{equation}
Define \(\phi_{m\ell}(e^t)=\exp(-\int_0^t\eta_{m\ell}(s)\,ds)\).
Equation \eqref{eq:6.54} implies \(0<\phi_{m\ell}\leq1\),
local uniform convergence \(\phi_{m\ell}\to1\), and
\(\phi_{m\ell}(w)=O_{m\ell}(w^{-s_\infty})\).

Assume
\(0\leq t\leq T_\ell+d_0\).  Then, \eqref{eq:6.54} yields
\(0<\eta_{m\ell}(t)\leq\varepsilon_0\).
Set \(T_0=\binom{n-1}{k-1}2^{k-1}\) and
\(\rho=(T/T_0)^{1/(k-1)}\).

Since \(C\in\Gamma_{k-1}\), Lemma~\ref{Lem:2.1} gives
\begin{equation}\label{eq:6.55}
 \Sigma\leq(2-2(n/k)|\nabla w|^2/(4wh)+2K\eta_{m\ell})T_0\rho^{k-1}
 +\binom{n-1}{k}2^k\rho^k.
\end{equation}

For \(0<\rho<1\), one has
\begin{equation}\label{eq:6.56}
 \Sigma-\left[(2-2(n/k)|\nabla w|^2/(4wh))T_0\rho^{k-1}
 +\binom{n-1}{k}2^k\rho^k\right]
 \geq2T_0\left(\frac nk-1\right)\rho^{k-1}(1-\rho).
\end{equation}
It follows from \eqref{eq:6.55} and
\eqref{eq:6.56} that
\((n/k-1)(1-\rho)\leq K\eta_{m\ell}\).  If \(\rho\geq1\), then
\(T=T_0\rho^{k-1}\geq T_0\).  Hence
the choice of \(\varepsilon_0\) in \eqref{eq:6.51} gives
\begin{equation}
 \rho\geq1-\frac{k\eta_{m\ell}}{n/k-1}\geq2^{-1/(k-1)},\qquad
 T=T_0\rho^{k-1}\geq\frac{T_0}{2}.
\end{equation}

For \(|\nabla w|^2/(4wh)\geq\delta_\ell\), choose \(a>0\) and \(x_0\) by
\begin{equation}\label{eq:6.57}
 a\left(2-2\frac nk\frac{|\nabla w|^2}{4wh}+2K\eta_{m\ell}\right)-2+2\frac nkx_0=0,
 \qquad a^k(1-|\nabla w|^2/(4wh))=1-x_0.
\end{equation}
After \(x_0\) is eliminated, the derivative of the left-hand side of the first equation in
\eqref{eq:6.57} with respect to \(a\), at
\((a,\eta_{m\ell})=(1,0)\), equals
\(-2((n/k)K-1)\neq0\). Hence
\begin{equation}\label{eq:6.58}
 |a-1|+|x_0-|\nabla w|^2/(4wh)|\leq C_{\delta_\ell}\eta_{m\ell}.
\end{equation}
After decreasing \(\varepsilon_0\), equation
\eqref{eq:6.58} implies $\frac12\leq a\leq2$ and $\frac{\delta_\ell}{2}\leq x_0\leq1$.

Put \(M=B/h=\operatorname{diag}(m_{11},C)\),
\(\zeta=\nabla w/\sqrt{wh}=2\sqrt{|\nabla w|^2/(4wh)}\,e_1\), and \(N=T_{k-1}(M)\).
Retain \(\Sigma=\sigma_k(M)\) and \(T=\sigma_{k-1}(C)\). For $\eta_{m\ell}=0$, we have \(D^2w\,\nabla w=2h\nabla w\), and therefore,
\begin{equation}\label{eq:6.59}
 m_{11}=2-2\frac nk\frac{|\nabla w|^2}{4wh},\qquad
 N[e_1,e_1]=T,\qquad
 T_k(M)[e_1,e_1]=\Sigma-m_{11}T,\qquad
 \sigma_1(M)=m_{11}+\tr C.
\end{equation}

By the Newton--Maclaurin inequality in Lemma \ref{Lem:2.1}, we can derive that
\begin{equation}\label{eq:6.60}
 T\geq T_0,\qquad \tr C\geq2(n-1).
\end{equation}
Indeed, if \(\rho=(T/T_0)^{1/(k-1)}<1\), the Newton--Maclaurin inequality in Lemma \ref{Lem:2.1} yields
\[\Sigma\leq(2-2(n/k)|\nabla w|^2/(4wh))T_0\rho^{k-1}
+\binom{n-1}{k}2^k\rho^k.\]  The left-hand side minus the right-hand side in the above inequality is
at least \(2T_0(\frac nk-1)\rho^{k-1}(1-\rho)>0\), a contradiction.
The trace inequality in \eqref{eq:6.60} then follows from the Newton--Maclaurin inequality at orders \(1\) and \(k-1\).

Set \(L_k(M)=\frac{n-k}{n}\Sigma\Id-T_k(M)\) and define
\begin{equation}
\begin{aligned}
 \mathcal R={}&\tr(L_k(M)M)
 +\frac nk\zeta^TL_k(M)\zeta
 +\frac{n(n-k)}{4k^2}|\zeta|^2\zeta^TN\zeta,\\
 \mathcal S={}&\tr C-2(n-1)
 +\frac{n|\zeta|^2}{2kT_0}(T-T_0).
\end{aligned}
\end{equation}
From \eqref{eq:2.12}, we get
\begin{equation}\label{eq:6.61}
 \mathcal R\geq
 \left[
 \sqrt{\tr(L_k(M)M)}
 -\frac{\sqrt{n(n-k)}}{2k}
 \sqrt{|\zeta|^2\zeta^TN\zeta}
 \right]^2\geq0.
\end{equation}
Equation \eqref{eq:6.60} implies \(\mathcal S\geq0\).
For $\eta_{m\ell}=0$, equations \eqref{eq:6.37},
\eqref{eq:6.59}, and
\(\tr(T_{k-1}(M)M^2)
=\tr(L_k(M)M)+(k/n)\sigma_1(M)\Sigma\) imply
\begin{equation}\label{eq:6.62}
\begin{aligned}
 2K h^{-k}\mathscr A^{ij}(\log h)_{ij}
 &=\mathcal R+\frac kn\Sigma\mathcal S \\
 &=Q-\left(2k+4\frac{|\nabla w|^2}{4wh}\right)\Sigma
 +2\frac nk\frac{|\nabla w|^2}{4wh}
 \left(6-2\frac nk\frac{|\nabla w|^2}{4wh}
 -4\frac{|\nabla w|^2}{4wh}\right)T.
 \end{aligned}
\end{equation}
Equations \eqref{eq:6.61}, \eqref{eq:6.60}, and
\eqref{eq:6.62} prove that
\begin{equation}
 2K h^{-k}\mathscr A^{ij}(\log h)_{ij}
 =\mathcal R+\frac{k}{n}\Sigma\mathcal S
 \geq0.
\end{equation}
At a critical point \(\nabla P_{m\ell}=0\), define
\[
 \mathcal H\left(\frac{|\nabla w|^2}{4wh},\eta_{m\ell},M\right)
 :=
 h^{-k}\mathscr A^{ij}
 \bigl[(\log h)_{ij}-\eta_{m\ell}(\log w)_{ij}\bigr].
\]
Equations \eqref{eq:6.37} and
\(T\geq T_0/2\) give
\begin{equation}\label{eq:6.63}
 \mathcal H(|\nabla w|^2/(4wh),\eta_{m\ell},M)\geq\frac{Q}{2K}
 -C_{\delta_\ell}(T+\Sigma),\qquad
 \frac{\Sigma}{T}\leq C_{\delta_\ell}.
\end{equation}
Thus \(\mathcal H(|\nabla w|^2/(4wh),\eta_{m\ell},M)\geq0\) if
\(Q\geq4kC_{\delta_\ell}(T+\Sigma)\). If
\(Q<4kC_{\delta_\ell}(T+\Sigma)\), equations \eqref{eq:6.57} and
\eqref{eq:6.58} give
\begin{equation}\label{eq:6.64}
 |x_0-|\nabla w|^2/(4wh)|+|a-1|+|a^{-1}-1|+|a^{-2}-1|
 \leq C_{\delta_\ell}\eta_{m\ell}.
\end{equation}
The three terms $Q$, $\Sigma$ and $T$ in \eqref{eq:6.37} satisfy
\begin{equation}\label{eq:6.65}
 Q(aM)=a^{k+1}Q,\qquad
 \sigma_k(aM)=a^k\Sigma,\qquad
 \sigma_{k-1}(aC)=a^{k-1}T.
\end{equation}
At a critical point, equation \eqref{eq:6.31} implies
\(e_1^TNMe_1=\left(2-2\frac nk\frac{|\nabla w|^2}{4wh}+2K\eta_{m\ell}\right)T\).
The coefficients of \(Q,\Sigma,T\) in \eqref{eq:6.37}, with
\(\nabla P_{m\ell}=0\) and \(\dot\eta_{m\ell}=0\), are smooth for
\(\delta_\ell\leq |\nabla w|^2/(4wh)\leq1\) and
\(0\leq\eta_{m\ell}\leq\varepsilon_0\). Equations
\eqref{eq:6.64} and
\eqref{eq:6.65} yield
\begin{equation}\label{eq:6.66}
 \left|\mathcal H(|\nabla w|^2/(4wh),\eta_{m\ell},M)
 -a^{-(k+1)}\mathcal H(x_0,0,aM)\right|
 \leq C_{\delta_\ell}\eta_{m\ell}(Q+T+\Sigma)
 \leq C_{\delta_\ell}\eta_{m\ell} T.
\end{equation}
Equation \eqref{eq:6.62} implies
\(\mathcal H(x_0,0,aM)\geq0\).  Hence it follows from
\eqref{eq:6.63} and
\eqref{eq:6.66} that
\begin{equation}\label{eq:6.67}
 h^{-k}\mathscr A^{ij}\bigl[(\log h)_{ij}-\eta_{m\ell}(\log w)_{ij}\bigr]
 \geq-C_{\delta_\ell}\eta_{m\ell} T.
\end{equation}

For \(0\leq t\leq T_\ell+d_0\), equations
\eqref{eq:6.32} and \eqref{eq:6.35} give
\[
h^{-k}\mathscr A^{ij}(P_{m\ell})_{ij}
=h^{-k}\mathscr A^{ij}\bigl[(\log h)_{ij}
-\eta_{m\ell}(\log w)_{ij}\bigr]
-4\frac{|\nabla w|^2}{4wh}T\dot\eta_{m\ell}.
\]
Since \(|\nabla w|^2/(4wh)\geq\delta_\ell\), equations
\eqref{eq:6.54} and \eqref{eq:6.67}
therefore imply
\begin{equation}\label{eq:6.68}
 h^{-k}\mathscr A^{ij}(P_{m\ell})_{ij}
 \geq(4A\delta_\ell-C_{\delta_\ell})\eta_{m\ell}T>0.
\end{equation}
For \(t\geq T_\ell\), the choice of \(T_\ell\) gives
\(|\nabla w|^2/(4wh)>1-\delta\). Equations \eqref{eq:6.45},
\eqref{eq:6.50}, and
\eqref{eq:6.54} imply
\begin{equation}\label{eq:6.69}
 h^{-k}\mathscr A^{ij}(P_{m\ell})_{ij}
 \geq\eta_{m\ell}T\left[
 \left(\frac nk-2\right)
 \left(1-\frac{s_\infty}{n/k-1}\right)-2\tau\right]>0.
\end{equation}
Equation \eqref{eq:6.68} holds for
\(0\leq t\leq T_\ell+d_0\), while
\eqref{eq:6.69} holds for \(t\geq T_\ell\). Hence \eqref{eq:6.42} holds for
every \(t\geq0\), which proves
\eqref{eq:6.41}--\eqref{eq:6.42}.
\end{proof}

\begin{lemma}\label{Lem:6.5}
{
Let \(v_j>0\) be admissible solutions of \eqref{eq:1.8} and be smooth near
points \(y_j\), set \(w_j=v_j^{-\gamma}\), and define \(h_j\) by
\eqref{eq:6.29}. Fix \(\phi=\phi_m\) in \eqref{eq:6.39} and
\(\ell_0=\overline\ell\), or fix \(\phi=\phi_{m\ell}\) from
Lemma~\ref{Lem:6.4} and \(\ell_0=\ell>h_0\). In the first case, assume
\(w_j(y_j)\geq1\). In the second case, take \(A\) and
\(\varepsilon_{m\ell}=\eta_{m\ell}(0)\) from
\eqref{eq:6.51}--\eqref{eq:6.54}, extend to \(t<0\) by
\(\eta_{m\ell}(t)=\varepsilon_{m\ell}e^{-At}\) and
\(\phi_{m\ell}(e^t)=\exp(-\int_0^t\eta_{m\ell}(s)\,ds)\), and assume
\(\log w_j(y_j)=o(1)\) whenever \(w_j(y_j)<1\).
Set \(P_j=\log(h_j\phi(w_j))\). Suppose that, for some
\(\varepsilon_*>0\),
\begin{equation}\label{eq:6.70}
 |\nabla P_j(y_j)|=o(1),\qquad D^2P_j(y_j)\preceq o(1)\Id,
\end{equation}
and
\begin{equation}\label{eq:6.71}
 h_j(y_j)\phi(w_j(y_j))\geq\ell_0+\varepsilon_*.
\end{equation}
Assume also that there exists \(C_0\geq1\), independent of \(j\), such that
\begin{equation}\label{eq:6.72}
 C_0^{-1}\leq w_j(y_j)\leq C_0,\qquad
 C_0^{-1}\leq h_j(y_j)\leq C_0,\qquad
 |D^2v_j(y_j)|\leq C_0.
\end{equation}
Then no such sequence $\{y_j\}$ exists.
}
\end{lemma}

\begin{proof}
{
All quantities below are evaluated at \(y_j\). Set
\(x_j=|\nabla w_j|^2/(4w_jh_j)\). We infer from
\eqref{eq:6.30} and \eqref{eq:6.71} that
\(x_j\geq c>0\) and \(|\nabla w_j|^2=4w_jh_jx_j\geq c>0\), uniformly in \(j\). Since
\(w_j=v_j^{-\gamma}\), \eqref{eq:6.72} and
\(D^2w_j=-\gamma v_j^{-\gamma-1}D^2v_j
+\gamma(\gamma+1)v_j^{-\gamma-2}\nabla v_j\otimes \nabla v_j\)
show that \(D^2w_j\) is bounded. Hence
\begin{equation}\label{eq:6.73}
 M_j:=\frac1{h_j}\left(D^2w_j-
 \frac{n}{2kw_j}\nabla w_j\otimes \nabla w_j\right)
\end{equation}
is bounded.  Moreover, one has
\begin{equation}\label{eq:6.74}
 0<\frac{1}{C_1}\leq \sigma_k(M_j)=\frac{\gamma^k}{w_jh_j^k}\leq C_1.
\end{equation}
By
\eqref{eq:6.73}--\eqref{eq:6.74} and the
Newton--Maclaurin inequalities in Lemma \ref{Lem:2.1}, we have
\begin{equation}\label{eq:6.75}
 \frac{\sigma_s(M_j)}{\binom ns}
 \geq\left(\frac{\sigma_k(M_j)}{\binom nk}\right)^{s/k}
 \geq c>0\qquad(1\leq s<k)
\end{equation}
Thus \(M_j\) stays in a fixed compact subset of \(\Gamma_k\), and
\(T_{k-1}(M_j)\) and its trace are bounded.

Choose \(e_{1,j}=\nabla w_j/|\nabla w_j|\) and write
\(M_j=\left(\begin{smallmatrix}m_{11,j}&z_j^T\\z_j&C_j\end{smallmatrix}\right)\).
Put \(K_j=k-(k-1)x_j\), \(J_j=k-(k-1)K_j\), and
\(\eta_j=-d\log\phi/d\log w\vert_{w=w_j}\).

The first derivative identity \eqref{eq:6.31} gives
\begin{equation}\label{eq:6.76}
 m_{11,j}=2-2\frac nkx_j+2K_j\eta_j
 +\frac{2w_jK_j}{|\nabla w_j|}(P_j)_1,\qquad
 (z_j)_i=\frac{2w_jK_j}{|\nabla w_j|}(P_j)_i
 \quad(2\leq i\leq n).
\end{equation}
By passing to a subsequence, we have
\((M_j,e_{1,j},x_j,K_j,\eta_j,C_j,w_j,h_j)\to
(M,e_1,x,K,\eta,C,w,h)\). Equations \eqref{eq:6.70} and
\eqref{eq:6.76} yield \(z_j\to0\) and
\(m_{11,j}-(2-2\frac nkx_j+2K_j\eta_j)\to0\). Hence
\(M=\operatorname{diag}(2-2\frac nkx+2K\eta,C)\in\Gamma_k\) by
\eqref{eq:6.75}, so \(C\in\Gamma_{k-1}\),
\(\sigma_{k-1}(C)>0\), and
\(h\phi(w)\geq\ell_0+\varepsilon_*\). By selecting a further subsequence, one has, either \(w_j\geq1\), where \eqref{eq:6.40} or \eqref{eq:6.42}
applies, or \(w_j<1\). In the latter case,
\(\log w_j=o(1)<0\), \(\eta_j=\varepsilon_{m\ell}(1+o(1))\),
\(\dot\eta_{m\ell}(\log w_j)=-A\eta_j\), and \eqref{eq:6.30}--
\eqref{eq:6.71} yield \(x_j\geq\delta_\ell+o(1)\). Since \(t_j=\log w_j\to0^{-}\), \(\eta_j\to\varepsilon_{m\ell}\),
\(x\geq\delta_\ell\), and \(\sigma_{k-1}(C)>0\), the right-hand side
of \eqref{eq:6.68} is positive at the limiting point \(t=0\). By continuity
of the left-hand side of \eqref{eq:6.37}, with
\(\mathscr A_j=w_jh_j^{k-1}T_{k-1}(M_j)\), for all large \(j\),
\begin{equation}\label{eq:6.77}
 h_j^{-k}\left[\mathscr A_j^{ab}(P_j)_{ab}+\mathscr A_j\left(
 \frac{2+\eta_jJ_j}{w_jK_j}\nabla w_j+\frac {J_j}{K_j}\nabla P_j
 \right)\cdot \nabla P_j\right]
 \geq c>0.
\end{equation}

On the other hand, \(\mathscr A_j>0\) has bounded trace, and
\eqref{eq:6.30}, \eqref{eq:6.70}, and \eqref{eq:6.72} imply that
\(\mathscr A_j^{ab}(P_j)_{ab}\leq o(1)\tr\mathscr A_j=o(1)\) and
\(\mathscr A_j((2+\eta_jJ_j)\nabla w_j/(w_jK_j)
+J_j\nabla P_j/K_j)\cdot\nabla P_j=o(1)\). Since \(h_j^{-1}\) is
bounded, this contradicts \eqref{eq:6.77}.
}
\end{proof}

\begin{proposition}\label{Prop:6.6}
Let \(\Gamma_j\to\infty\), and let \(v_j\) be smooth positive admissible solutions of
\begin{equation}\label{eq:6.78}
 \sigma_k(-D^2v_j)=v_j^{p_*}\quad\hbox{in }B_{4\Gamma_j}.
\end{equation}
Assume
\begin{equation}\label{eq:6.79}
 v_j(0)=1,\qquad 0<v_j\leq C_0,
\end{equation}
and, whenever \(R_j=o(\Gamma_j)\),
\begin{equation}\label{eq:6.80}
 \sup_{B_{R_j}}v_j\leq1+o(1).
\end{equation}
Then every locally uniform subsequential limit of $v_{j}$ equals
\begin{equation}\label{eq:6.81}
 U(x)=(1+h_0|x|^2)^{-1/\gamma}.
\end{equation}
\end{proposition}

\begin{proof}
Put \(w_j=v_j^{-\gamma}\) and define \(h_j\) by \eqref{eq:6.29}. Choose
\(R_j\to\infty\), \(R_j=o(\Gamma_j)\).  Applying \eqref{eq:6.6} on \(B_1(x)\subset B_{3R_j}\) and \eqref{eq:6.79}, we have
\(|\nabla\log v_j|\leq C\) in \(B_{2R_j}\), and hence
\begin{equation}\label{eq:6.82}
 |\nabla w_j|\leq Cw_j,\qquad h_j\leq Cw_j.
\end{equation}
Moreover, \eqref{eq:6.80} and \eqref{eq:6.30} give
\(w_j\geq1-o(1)\) and \(h_j(0)\geq h_0\).

Fix \(L<\infty\), \(\ell>h_0\), and one function
\(\phi_{m\ell}\). Equation \eqref{eq:6.80} gives
\(w_j\geq1-o(1)\), although \(w_j\) may be slightly smaller than
one. For the corresponding values
\(\log(1-o(1))\leq t=\log w_j<0\), extend \(\phi_{m\ell}\) by setting
\[
 \eta_{m\ell}(t)=\varepsilon_{m\ell}e^{-At},
 \qquad
 \phi_{m\ell}(e^t)=\exp\left(-\int_0^t\eta_{m\ell}(s)\,ds\right).
\]
Then \(\dot\eta_{m\ell}=-A\eta_{m\ell}\), and
\(\eta_{m\ell}=\varepsilon_{m\ell}(1+o(1))\leq\varepsilon_0\) for all large
\(j\).   Let \(\psi_j\) satisfy
\begin{equation}\label{eq:6.83}
 \begin{gathered}
  0\leq\psi_j\leq1,\qquad
  \psi_j=0\text{ on }B_{R_j/2},\qquad
  \psi_j=1\text{ on }B_{R_j}\setminus B_{3R_j/4},\\
 |\nabla\psi_j|\leq CR_j^{-1},\qquad
 |D^2\psi_j|\leq CR_j^{-2}.
 \end{gathered}
\end{equation}
The decay exponent in \eqref{eq:6.41} and \eqref{eq:6.82} implies
\begin{equation}\label{eq:6.84}
 h_j\phi_{m\ell}(w_j)\leq {C}\quad\hbox{in }B_{2R_j}.
\end{equation}
Suppose instead that, for some
\(\varepsilon>0\) and a subsequence,
\begin{equation}\label{eq:6.85}
\sup_{B_L}h_j\phi_{m\ell}(w_j)>\ell+\varepsilon.
\end{equation}
Put {\(P_j=\log(h_j\phi_{m\ell}(w_j))\)}. Choose
{\(A_m>0\)} so
large, using \eqref{eq:6.83} and \eqref{eq:6.84},
such that the maximum of
{\(P_j-A_m\psi_j\)} over \(\overline B_{R_j}\) cannot be attained in
\(B_{R_j}\setminus B_{3R_j/4}\).  Denote a maximum point by
\(y_j\in B_{3R_j/4}\).  At $y_j$, one has \(h_j\phi_{m\ell}(w_j)>\ell+\varepsilon\), and
\begin{equation}\label{eq:6.86}
 |\nabla {P_j}|={O}(R_j^{-1}),\qquad
 D^2{P_j}{\preceq O}(R_j^{-2})\Id.
\end{equation}
If \(w_j\geq1\), then \(\phi_{m\ell}\leq1\); if
\(1-o(1)\leq w_j<1\), the chosen extension
\(\phi_{m\ell}=1+o(1)\).  Hence \(w_j/\phi_{m\ell}(w_j)^k\geq1-o(1)\).

Equations \eqref{eq:6.30} and \eqref{eq:6.85} therefore imply, for all large \(j\),
\begin{equation}\label{eq:6.87}
 h_j^kw_j>\ell^k,\qquad
 {\frac{|\nabla w_j|^2}{4w_jh_j}}>\delta_\ell>0.
\end{equation}
By \eqref{eq:6.41}, \eqref{eq:6.82}, and \eqref{eq:6.85},
\(\frac{1}{C_{0}}\leq w_j,h_j,h_j^{-1}\leq C_0\).
Equation \eqref{eq:6.6} gives \(|\nabla v_j|\leq C\) on
\(B_{2R_j}\), and the upper bound for \(w_j(y_j)\) implies
\(v_j(y_j)\geq {c}>0\).  Since
\(B_{R_j/4}(y_j)\subset B_{R_j}\), by applying Lemma~\ref{Lem:6.3} to
\(v_j(y_j+\cdot)\), we have
\begin{equation}\label{eq:6.88}
 |D^2v_j(y_j)|\leq {C}.
\end{equation}
{Equations \eqref{eq:6.85}--
\eqref{eq:6.86} give \eqref{eq:6.70}--\eqref{eq:6.71}; equations
\eqref{eq:6.41}, \eqref{eq:6.82}, \eqref{eq:6.85}, \eqref{eq:6.87},
and \eqref{eq:6.88} yield \eqref{eq:6.72}; and \eqref{eq:6.80} ensures
\(\log w_j(y_j)=o(1)\) if \(w_j(y_j)<1\). Lemma~\ref{Lem:6.5}, with
the resulting \(P_j\), contradicts \eqref{eq:6.86}.} Thus
\eqref{eq:6.85} is impossible. Consequently,
\(\limsup\limits_j\sup\limits_{B_L}h_j\phi_{m\ell}(w_j)\leq\ell\).

Equation \eqref{eq:6.6} yields \(w_j\leq C_L\) on \(B_L\).  Hence \eqref{eq:6.41},
followed by \(m\to\infty\) and \(\ell\downarrow h_0\), implies
\begin{equation}\label{eq:6.89}
 \limsup_{j\to\infty}\sup_{B_L}h_j\leq h_0.
\end{equation}
By \eqref{eq:6.30},
\(|\nabla w_j|^2=4h_j[w_j-(h_0/h_j)^k]\).  By \eqref{eq:6.80} and \eqref{eq:6.89}, we may
choose numbers \(\varepsilon_j=\varepsilon_j(L)\downarrow0\) such that
\(|\nabla w_j|^2\leq4(h_0+\varepsilon_j)(w_j-1+\varepsilon_j)\) on \(B_L\).
Since \(w_j(0)=1\), integration of
\(\left|\nabla\sqrt{w_j-1+\varepsilon_j}\right|
\leq\sqrt{h_0+\varepsilon_j}\)
along the line segment from the origin to \(x\) gives
\(\sqrt{w_j(x)-1+\varepsilon_j}
\leq\sqrt{\varepsilon_j}+\sqrt{h_0+\varepsilon_j}|x|\).
Letting \(j\to\infty\), we have
\begin{equation}\label{eq:6.90}
 \limsup_{j\to\infty}w_j(x)\leq1+h_0|x|^2.
\end{equation}
If \(v_j\to V\) locally uniformly, Hessian-measure stability (\cite[Lemma~2.2 and Theorem~1.1]{TW1999}) and
\eqref{eq:6.90} imply
\(\sigma_k(-D^2V)=V^{p_*}\), \(V\geq U\), and \(V(0)=U(0)\).
Set \(Z_1=-V\), \(Z_0=-U\), and \(q=Z_0-Z_1\geq0\).  Concavity of
\(\sigma_k^{1/k}\) at \(D^2Z_0\) gives
\begin{equation}\label{eq:6.91}
 a^{ij}q_{ij}+c(x)q\leq0
\end{equation}
in the viscosity sense, where
\(a^{ij}=F^{ij}(D^2Z_0)\), \(F=\sigma_k^{1/k}\), and
\(c=(V^{p_*/k}-U^{p_*/k})/(V-U)>0\), with its continuous extension
where \(V=U\). Corollary~\ref{Cor:2.3} makes \(\{a^{ij}\}\) locally
uniformly elliptic. Since \(q\geq0\), \eqref{eq:6.91} gives
\(a^{ij}q_{ij}\leq0\); the strong minimum principle and \(q(0)=0\)
give \(q\equiv0\).  This proves \eqref{eq:6.81}.
\end{proof}

For \(\kappa>0\), put $\Phi_\kappa(x)=\kappa|x|^{-b}$ and $\psi_\kappa(x)=-\Phi_\kappa(x)$.
The eigenvalues of \(D^2\psi_\kappa\) are
\begin{equation}\label{eq:6.92}
 -(b+1)\rho,\ \rho,\ldots,\rho
 \qquad \text{with}\,\, \rho=\kappa b|x|^{-n/k}.
\end{equation}
Consequently, we have $\mu_k[\psi_\kappa]=M\delta_0$ and $M=\frac{|\mathbb S^{n-1}|}{k}
 \binom{n-1}{k-1}(\kappa b)^k$.

\begin{theorem}\label{Thm:6.7}
Suppose that
\[
 q\in C(B_{r_0}\setminus\{0\})\cap
 \Phi^k(B_{r_0}\setminus\{0\}),
 \qquad
 \mu_k[q]=0\quad\hbox{in }B_{r_0}\setminus\{0\}.
\]
If
\begin{equation}\label{eq:6.93}
 |q-\psi_\kappa|\leq C
 \qquad\hbox{in }B_{r_0}\setminus\{0\},
\end{equation}
then there is a constant \(c\in\R\) such that
\begin{equation}\label{eq:6.94}
 q(r\,\cdot)-\psi_\kappa(r\,\cdot)
 \longrightarrow c
 \quad\hbox{in }C^2(A)
\end{equation}
as \(r\downarrow0\) for every compact annulus
\(A\Subset\R^n\setminus\{0\}\).  In particular,
\begin{equation}\label{eq:6.95}
 q=\psi_\kappa+c+o(1),
 \qquad
 |x||\nabla(q-\psi_\kappa)|
 +|x|^2|D^2(q-\psi_\kappa)|=o(1).
\end{equation}
Moreover, \(q\) is smooth in a sufficiently small punctured ball.
\end{theorem}

\begin{proof}
For a symmetric matrix \(A\), define its least principal G{\aa}rding
root by
\begin{equation}
 \mathscr F(A)=\sup\{t\in\R\mid
 A-t{\Id}\in\overline{\Gamma_k}\}.
\end{equation}
G{\aa}rding theory implies
\begin{equation}\label{eq:6.96}
 \begin{gathered}
 \mathscr F(A)\geq0
 \quad\Longleftrightarrow\quad
 A\in\overline{\Gamma_k},
 \qquad
 \mathscr F(A+s{\Id})=\mathscr F(A)+s,\\
 \mathscr F(\lambda A)=\lambda\mathscr F(A)
 \quad(\lambda>0),
 \qquad
 N\geq0\quad\Longrightarrow\quad
 \mathscr F(A+N)\geq\mathscr F(A).
 \end{gathered}
\end{equation}
The function \(\mathscr F\) is concave and
\begin{equation}\label{eq:6.97}
 |\mathscr F(A)-\mathscr F(B)|
 \leq\|A-B\|.
\end{equation}
Equations \eqref{eq:6.96} and
\eqref{eq:6.97} follow from
\cite[Corollary~2.13 and Theorem~2.15]{HarveyLawson}.

For the increasing operator \(\mathscr F\), we say $q$ is the viscosity solution to $\mathscr F(D^2q)=0$ in the sense that
\begin{equation}
 \begin{aligned}
 q-\varphi\text{ has a local maximum (minimum) at }x_0
 &\quad\Longrightarrow\quad \mathscr F(D^2\varphi(x_0))\geq(\leq)0.
 \end{aligned}
\end{equation}
We can show the following equivalence
\begin{equation}\label{eq:6.98}
 q\in C(\Omega)\cap\Phi^k(\Omega),\quad
 \mu_k[q]=0
 \quad\Longleftrightarrow\quad
 \mathscr F(D^2q)=0.
\end{equation}

Indeed, a test function touching $q$
from above has Hessian in \(\overline{\Gamma_k}\) by the cone
characterization of \(k\) convexity
\cite[pp.~582 and 583]{TW1999}.  Hence the upper test inequality for
\(\mathscr F=0\) holds. Let a \(C^2\) function \(\varphi\) touch \(q\) from below at
\(x_0\).  If
\(\mathscr F(D^2\varphi(x_0))>0\), then
\(D^2\varphi(x_0)\in\Gamma_k\).  By continuity there is a ball
\(B_\rho(x_0)\Subset\Omega\) on which
\(D^2\varphi\in\Gamma_k\).  Restrict both \(q\) and the equation
\(\mu_k[q]=0\) to this ball.  The function \(\varphi\) is now an
admissible lower test in the Trudinger--Wang viscosity definition.
The Hessian measure and viscosity equivalence
\cite[Remark after Theorem~2.7, p.~586]{TW1999} gives
\(\sigma_k(D^2\varphi(x_0))\leq0\).  This contradicts
\(D^2\varphi(x_0)\in\Gamma_k\).  Therefore, the lower test
inequality \(\mathscr F(D^2\varphi(x_0))\leq0\) holds.

Conversely, suppose that \(\mathscr F(D^2q)=0\) in the standard
viscosity sense.  Every upper test has Hessian in
\(\overline{\Gamma_k}\), so the cone characterization
\cite[pp.~582 and 583]{TW1999} gives \(q\in\Phi^k(\Omega)\).
If an admissible \(C^2\) function \(\varphi\)
touches from below at \(x_0\), then
\(D^2\varphi(x_0)\in\overline{\Gamma_k}\), so
\(\mathscr F(D^2\varphi(x_0))\geq0\).  The viscosity supersolution
inequality gives the reverse inequality.  Hence
\(\mathscr F(D^2\varphi(x_0))=0\), and therefore
\(\sigma_k(D^2\varphi(x_0))=0\).  After restricting to a ball
compactly contained in \(\Omega\), the Hessian-measure and viscosity
equivalence in \cite[Remark after Theorem~2.7, p.~586]{TW1999} gives
\(\mu_k[q]=0\).  This
proves \eqref{eq:6.98}.

For \(1\leq j<k\), equation \eqref{eq:6.92} gives
\begin{equation}
 \sigma_j(D^2\psi_\kappa)
 =\binom{n-1}{j}\rho^j
 \frac{n(k-j)}{k(n-j)}>0,
 \qquad
 \sigma_k(D^2\psi_\kappa)=0.
\end{equation}
The radial eigenvalue of
\(T_{k-1}(D^2\psi_\kappa)\) is
\(\binom{n-1}{k-1}\rho^{k-1}\).  The ratio of each tangential
eigenvalue to the radial eigenvalue is
\begin{equation}\label{eq:6.99}
 \eta_0=\frac{n-k}{k(n-1)}>0.
\end{equation}
Thus the zero root of
\(t\mapsto\sigma_k(D^2\psi_\kappa-t{\Id})\) is simple.  The implicit
equation \(\sigma_k(A-\mathscr F(A)\Id)=0\)
shows that \(\mathscr F\) is smooth near $D^2\psi_\kappa$ and
\begin{equation}\label{eq:6.100}
\tr(D\mathscr F(A)H)
 =\frac{\tr(T_{k-1}(A-\mathscr F(A)\Id)H)}
 {\tr T_{k-1}(A-\mathscr F(A)\Id)}.
\end{equation}
By \eqref{eq:6.99} and
\eqref{eq:6.100}, there exist a fixed neighborhood
\(\mathcal N\) of $\{D^2\psi_\kappa(y)\mid 1/4\leq|y|\leq4\}$ and constants \(0<\lambda\leq\Lambda\) such that
\begin{equation}\label{eq:6.101}
 \lambda \Id\leq D\mathscr F(A)\leq\Lambda \Id,
 \qquad |D^2\mathscr F(A)|\leq\Lambda
 \quad\hbox{for every }A\in\mathcal N.
\end{equation}
Put \(e=q-\psi_\kappa\) and
\(\widetilde e_r(y)=r^be(ry)\).
Homogeneity and \eqref{eq:6.98} give
\begin{equation}\label{eq:6.102}
 \mathscr F(D^2\psi_\kappa+D^2\widetilde e_r)=0,
 \qquad
 \|\widetilde e_r\|_{L^\infty(\mathcal{A})}\leq C_\mathcal{A} r^b,
\end{equation}
where $\mathcal{A}\subset\R^{n}$ is a compact annulus. Choose finitely many balls \(B_\varrho(y_\ell)\) covering
\(\{1/2\leq|y|\leq2\}\), with
\(B_{2\varrho}(y_\ell)\Subset\{1/4<|y|<4\}\).
The radius may be chosen independent of \(\ell\).  On \(B_1\), put $\xi=\frac{y-y_\ell}{2\varrho}$ and $u_{r,\ell}(\xi)=\frac{\widetilde e_r
 (y_\ell+2\varrho\xi)}{(2\varrho)^2}$. Then $\mathcal F_\ell(D^2u_{r,\ell},\nabla u_{r,\ell},
 u_{r,\ell},\xi)=0$, where
\begin{equation}
 \mathcal F_\ell(N,p,s,\xi)
 =\mathscr F\left(D^2\psi_\kappa
 (y_\ell+2\varrho\xi)+N\right).
\end{equation}

We verify the hypotheses of \cite[Theorem~1.3]{Savin}.  Its
hypothesis H1 is the global degenerate ellipticity
\begin{equation}\label{eq:6.103}
 P\geq0
 \quad\Longrightarrow\quad
 \mathcal F_\ell(N+P,p,s,\xi)
 \geq\mathcal F_\ell(N,p,s,\xi).
\end{equation}
Equation~\eqref{eq:6.103} follows from \eqref{eq:6.96}.
{Since
\(\{D^2\psi_\kappa(y_\ell+2\varrho\xi):\ell\geq1,\
\xi\in\overline B_1\}\Subset\mathcal N\),
\(D^2\psi_\kappa(y_\ell+2\varrho\xi)+N+\theta P\in\mathcal N\)
for sufficiently small \(|N|+|P|\), uniformly in \(\ell\), when
\(P\succeq0\) and \(0\leq\theta\leq1\). Integrating \eqref{eq:6.101}
in \(\theta\) yields
\(\lambda\tr P\leq\mathcal F_\ell(N+P,p,s,\xi)
-\mathcal F_\ell(N,p,s,\xi)\leq\Lambda\tr P\), i.e., H2 uniformly in
\(\ell\).}
Hypothesis H3 is
\(\mathcal F_\ell(0,0,0,\xi)
=\mathscr F(D^2\psi_\kappa(y_\ell+2\varrho\xi))=0\).
The operator is independent of \(p\) and \(s\).
Formula \eqref{eq:6.101} gives uniform bounds for \(D^2\mathscr F\) on
\(\mathcal N\).  The derivatives up to order two of the map
\(\xi\mapsto D^2\psi_\kappa(y_\ell+2\varrho\xi)\) are uniformly
bounded on any fixed compact annulus.  The chain rule
therefore provides the joint \(C^2\) bound in H4.

Finally \eqref{eq:6.102} implies
\begin{equation}
 \|u_{r,\ell}\|_{L^\infty(B_1)}
 \leq C(2\varrho)^{-2}r^b<\varepsilon_{\mathrm S}
\end{equation}
for all sufficiently small \(r<r_{1}\), where \(\varepsilon_{\mathrm S}\)
is the uniform smallness constant in \cite[Theorem~1.3]{Savin}.  Savin's theorem yields some
\(\alpha\in(0,1)\) such that
\begin{equation}
 \sup_{0<r<r_1}
 \|\widetilde e_r\|_{C^{2,\alpha}
 (\{1/2<|y|<2\})}\leq C.
\end{equation}
Together with the \(C^0\) convergence in
\eqref{eq:6.102}, compactness and a finite covering of
each fixed compact annulus give
\begin{equation}\label{eq:6.104}
 \widetilde e_r\longrightarrow0
 \quad\hbox{in }C^{2,\beta}_{\mathrm{loc}}
 (\R^n\setminus\{0\})
 \qquad(0<\beta<\alpha).
\end{equation}
The equation $\mathscr F(D^2q)=0$ is consequently a classical uniformly elliptic equation
on every sufficiently small annulus $\mathcal{A}=\{0<s<|x|<t\}$.  Standard Schauder estimates yield smoothness of \(q\) in $\mathcal{A}$.

Since both \(\sigma_k(D^2 q)=0\) and \(\sigma_k(D^2\psi_\kappa)=0\), we have
\begin{equation}\label{eq:6.105}
 B^{ij}e_{ij}=0,
 \qquad
 B=\int_0^1T_{k-1}
 (D^2\psi_\kappa+tD^2e)\dd t.
\end{equation}
Let
\begin{equation}\label{eq:6.106}
 s_0=\frac{n(k-1)}k,
 \qquad
 A_r(y)=r^{s_0}B(ry),
 \qquad
 E_r(y)=e(ry).
\end{equation}
Equations \eqref{eq:6.105} and
\eqref{eq:6.106} give
\begin{equation}
 A_r^{ij}(E_r)_{ij}=0,
 \qquad
 A_r(y)=\int_0^1T_{k-1}
 \left(D^2\psi_\kappa(y)+tD^2\widetilde e_r(y)\right)\dd t
 \longrightarrow T_{k-1}(D^2\psi_\kappa(y))
\end{equation}
in \(C^\beta\) on compact annuli.  The matrices \(A_r\) are
uniformly elliptic.  Indeed, for every compact annulus
\(K\Subset\R^n\setminus\{0\}\), equations
\eqref{eq:6.92} and
\eqref{eq:6.99} give constants
\(0<\lambda_K\leq\Lambda_K\) such that
\(\lambda_K\Id\leq T_{k-1}(D^2\psi_\kappa)\leq\Lambda_K\Id\) on \(K\).
Equation \eqref{eq:6.104} and continuity of
\(T_{k-1}\) imply, for all sufficiently small \(r\),
\begin{equation}
 \frac{\lambda_K}{2}\Id\leq
 T_{k-1}\left(D^2\psi_\kappa(y)+tD^2\widetilde e_r(y)\right)
 \leq2\Lambda_K\Id,
 \qquad
 \frac{\lambda_K}{2}\Id\leq A_r(y)\leq2\Lambda_K\Id
\end{equation}
for \(y\in K\) and \(0\leq t\leq1\).  The constants do not depend on
\(r\) or \(t\).  The linear Schauder estimate and
\eqref{eq:6.93} make \(E_r\) precompact in \(C^2\)
on every compact annulus.

Consequently, for every sequence \(r_j\downarrow0\), there exist a
subsequence, still denoted by \(r_j\), and a bounded function
\(E\) on \(\R^n\setminus\{0\}\) such that $E_{r_j}\longrightarrow E
 \quad\text{in }C^2_{\mathrm{loc}}
 (\R^n\setminus\{0\})$.
The limit \(E\) satisfies
\begin{equation}
 E_{rr}+\frac{b+1}{r}E_r
 +\frac{\eta_0}{r^2}\Delta_{\mathbb S^{n-1}}E=0.
\end{equation}
For a spherical harmonic of degree \(\ell\), the two radial powers
are the roots of
\begin{equation}\label{eq:6.107}
 \mu(\mu+b)-\eta_0\ell(\ell+n-2)=0.
\end{equation}
Indeed, if \(Y_{\ell,m}\) is a spherical harmonic and
\(c_{\ell,m}(r)=\int_{\mathbb S^{n-1}}E(r,\theta)
Y_{\ell,m}(\theta)\dd\theta\), then
\begin{equation}\label{eq:6.108}
 c_{\ell,m}''+\frac{b+1}{r}c_{\ell,m}'
 -\frac{\eta_0\ell(\ell+n-2)}{r^2}c_{\ell,m}=0.
\end{equation}
For \(\ell=0\), equations \eqref{eq:6.107} and
\eqref{eq:6.108} give $c_{0,m}(r)=A_{0,m}+B_{0,m}r^{-b}$.
For \(\ell\geq1\), one has $c_{\ell,m}(r)=A_{\ell,m}r^{\mu_\ell^+}
 +B_{\ell,m}r^{\mu_\ell^-}\quad(\ell\geq1)$, where $\mu_\ell^+$ and $\mu_\ell^-$ are roots for \eqref{eq:6.107}.  Since \(E\) is bounded both at
zero and at infinity, we have $E\equiv C$.

It remains to show that all subsequences give the same constant.
Concavity of \(\mathscr F\) and differentiability on $D^2\psi_\kappa$ imply
\begin{equation}\label{eq:6.109}
 \tr(T_{k-1}(D^2\psi_\kappa)D^2e)\geq0.
\end{equation}
Indeed,
\[
 0=\mathscr F(D^2q)
 \leq\mathscr F(D^2\psi_\kappa)
 +\tr(D\mathscr F(D^2\psi_\kappa)D^2e),
\]
and \eqref{eq:6.100} turns this inequality into
\eqref{eq:6.109}.
Let \(\overline e(r)\) be the spherical average of \(e\), put
\(t=-\log r\), and set \(m(t)=\overline e(e^{-t})\).  Averaging
\eqref{eq:6.109} gives
\begin{equation}
 m''-bm'\geq0.
\end{equation}
Thus \(e^{-bt}m'\) is nondecreasing.  It cannot be positive at any
time because otherwise \(m'(t)\geq c_1e^{bt}\) thereafter, contradicting the boundedness of \(m\).  Hence \(m'\leq0\), and the bounded
function \(m\) has a finite limit \(c\).  If
\(E_{r_j}\to C\) on compact annuli, then
\[
 C=\frac1{|\mathbb S^{n-1}|}
 \int_{\mathbb S^{n-1}}E_{r_j}(1,\theta)\dd\theta
 =\overline e(r_j)\longrightarrow c.
\]
Thus \(E_{r}=e(r\cdot)\rightarrow c\) in $C^2(\mathcal{A})$ for every compact annulus $\mathcal{A}\subset\R^{n}$. This proves
\eqref{eq:6.94} and \eqref{eq:6.95}.
\end{proof}

We approximate the pole by smooth radial solutions.

For \(0<\tau<1\), put
\begin{equation}
 f_\tau(x)=
 \frac{M e^{-|x|^2/\tau^2}}
 {\displaystyle\int_{B_8}e^{-|y|^2/\tau^2}\dd y},
 \qquad
 m_\tau(r)=\int_{B_r}f_\tau\dd x.
\end{equation}
Thus \(f_\tau\dd x\rightharpoonup M\delta_0\).  Define
\begin{equation}\label{eq:6.110}
 q_\tau'(r)=
 \left(\frac{k m_\tau(r)}
 {|\mathbb S^{n-1}|\binom{n-1}{k-1}}\right)^{1/k}
 r^{-(n-k)/k},
 \qquad q_\tau(r)=-\int_r^8q_\tau'(t)\dd t.
\end{equation}
Since \(m_\tau(r)=r^nh_\tau(r^2)\) for a positive smooth function
\(h_\tau\), the function \(q_\tau\) is smooth at zero.  The radial
Hessian formula implies
\begin{equation}\label{eq:6.111}
 \sigma_k(D^2q_\tau)=f_\tau,
 \qquad
 q_\tau''+\frac{n-k}{k}\frac{q_\tau'}r
 =\frac{q_\tau'}k\frac{m_\tau'}{m_\tau}>0.
\end{equation}
For \(1\leq j<k\), we have
\begin{equation}\label{eq:6.112}
 \sigma_j(D^2q_\tau)=\binom{n-1}{j-1}
 \left(\frac{q_\tau'}r\right)^{j-1}
 \left(q_\tau''+\frac{n-j}{j}\frac{q_\tau'}r\right)>0.
\end{equation}
Equations \eqref{eq:6.111} and
\eqref{eq:6.112} yield
\(D^2q_\tau\in\Gamma_k\) in \(B_8\).

For fixed \(R>0\) and \(d\in\R\), let
\begin{equation}\label{eq:6.113}
 V_{R,\tau}(r)=\kappa8^{-b}+R^bd-q_\tau(r).
\end{equation}
For all sufficiently small \(R\), the function \(V_{R,\tau}\) is
positive and \(-D^2V_{R,\tau}\in\Gamma_k\).  Moreover,
\begin{equation}\label{eq:6.114}
 V_{R,\tau}\longrightarrow\Phi_\kappa+R^bd
 \quad\hbox{in }C^2_{\mathrm{loc}}(B_8\setminus\{0\}).
\end{equation}
Indeed, \(m_\tau\to M\) and \(m_\tau'\to0\) locally uniformly in $B_8\setminus\{0\}$.  Formula \eqref{eq:6.110} and its derivative then
give \eqref{eq:6.114}.

\begin{lemma}\label{Lem:6.8}
For \(0<r<8\), let
\begin{equation}\label{eq:6.115}
 \mathscr P_{V_{R,\tau}}(r)
 =\frac{k}{k+1}m_\tau(r)
 \left[bV_{R,\tau}(r)-rq_\tau'(r)\right].
\end{equation}
Consequently, uniformly for \(1\leq r\leq2\),
\begin{equation}\label{eq:6.116}
 \mathscr P_{V_{R,\tau}}(r)\longrightarrow \frac{n-2k}{k+1}MR^bd
 \qquad\hbox{as }\tau\downarrow0.
\end{equation}
\end{lemma}

\begin{proof}
For a radial Hessian \(D^2q_\tau\),
\((T_k)_{rr}=\binom{n-1}{k}(q_\tau'/r)^k\) and
\((T_{k-1})_{rr}=\binom{n-1}{k-1}(q_\tau'/r)^{k-1}\).
Substituting the formulae for \((T_k)_{rr}\),
\((T_{k-1})_{rr}\), and \((V_{R,\tau})_r=-q_\tau'\) in
\eqref{eq:2.34}, and using
\(m_\tau(r)=\frac{|\mathbb S^{n-1}|}{k}
\binom{n-1}{k-1}r^{n-k}(q_\tau')^k\),
we get \eqref{eq:6.115}.  Equations
\eqref{eq:6.114} and
\eqref{eq:6.110} imply
\(bV_{R,\tau}(r)-rq_\tau'(r)\to bR^bd\) uniformly on \([1,2]\).
The identity
\(kb/(k+1)=(n-2k)/(k+1)\) proves \eqref{eq:6.116}.
\end{proof}

\begin{lemma}\label{Lem:6.9}
Let \(U_0,U_1>0\) be smooth in \(B_8\), with
\(-D^2U_i\in\Gamma_k\). Set
\(f_i=\sigma_k(-D^2U_i)\), \(H=U_1-U_0\), and
\(U_t=U_0+tH\). {Also set
\(f_t=\sigma_k(-D^2U_t)\) for \(0\leq t\leq1\)}.
Let
\begin{equation}
 \{1\leq|x|\leq2\}\Subset A_0=\{3/4<|x|<5/2\}
 \Subset A_1=\{1/2<|x|<3\}
 \Subset A_2=\{1/4<|x|<4\}.
\end{equation}
Assume
\(\sup\limits_{A_2}(|U_0|+|U_1|+|\nabla U_0|+|\nabla U_1|)\leq L\) and
\begin{equation}
 \int_0^1\int_{A_2}
 \left[\tr T_{k-1}(-D^2U_t)
 +\sigma_k(-D^2U_t)\right]\dd x\dd t\leq K_0.
\end{equation}
Set \(\delta=\|H\|_{L^\infty(A_2)}\) and
\(S=\int_{A_2}(f_0+f_1)\dd x\).
If \(\zeta\in C_c^\infty((1,2))\) and \(\int_1^2\zeta=1\), then
\begin{equation}\label{eq:6.117}
 \left|\int_1^2\zeta(r)
 \bigl(\mathscr P_{U_1}(r)-\mathscr P_{U_0}(r)\bigr)\dd r\right|
 \leq C\left(\delta+\sqrt{\delta S}\right),
\end{equation}
where \(C\) depends only on \(L,K_0,n,k\), and \(\zeta\).
\end{lemma}

\begin{proof}
Write \(T_t=T_{k-1}(-D^2U_t)\) and
\(Z_t=x\cdot \nabla U_t+\frac{n-2k}{k+1}U_t\).
The divergence-free identity for Newton tensors implies
\begin{equation}\label{eq:6.118}
 \operatorname{div}\left(\int_0^1T_t\dd t\,\nabla H\right)=f_0-f_1.
\end{equation}
Test \eqref{eq:6.118} by \(\chi^2H\), where
\(\chi\in C_c^\infty(A_2)\) equals one on \(A_1\).  The local
Hessian mass estimate \cite[Theorem~3.1]{TW1999} at orders \(k\)
and \(k-1\) gives
\begin{equation}\label{eq:6.119}
\begin{aligned}
 E:=\int\chi^2\int_0^1T_t[\nabla H,\nabla H]\dd t\dd x
 ={}&-2\int\chi H\int_0^1T_t[\nabla H,\nabla\chi]\dd t\dd x
 -\int\chi^2H(f_0-f_1)\dd x\\
 &\leq\frac12E+CK_0\delta^2+\delta S.
\end{aligned}
\end{equation}
Hence \(E\leq C(K_0\delta^2+\delta S)\).
Define
\begin{equation}\label{eq:6.120}
 \mathcal C=\int_0^1\mathcal C_t\dd t=\int_0^1
 \left[HT_t\nabla Z_t-Z_tT_t\nabla H+xH\sigma_k(-D^2U_t)\right]\dd t.
\end{equation}
The identities
\begin{equation}
 \partial_tf_t=-\operatorname{div}(T_t\nabla H),
 \quad
 \partial_tZ_t=x\cdot \nabla H+\frac{n-2k}{k+1}H,
 \quad
 \operatorname{div}(T_t\nabla Z_t)=-x\cdot \nabla f_t-\frac{k(n+2)}{k+1}f_t
\end{equation}
imply, for the integrand \(\mathcal C_t\) in
\eqref{eq:6.120},
\begin{equation}\label{eq:6.121}
 \operatorname{div}\mathcal C_t
 =f_t\left(x\cdot \nabla H+\frac{n-2k}{k+1}H\right)+Z_t\partial_tf_t
 =\partial_t(f_tZ_t).
\end{equation}
Integration of \eqref{eq:6.121} in \(t\)
yields $\operatorname{div}\mathcal C=f_1Z_1-f_0Z_0$.
By integration over \(B_r\)
and Lemma~\ref{Lem:2.9}, we obtain
\begin{equation}\label{eq:6.122}
 \int_{\partial B_r}\mathcal C\cdot\nu\dd S
 =\mathscr P_{U_1}(r)-\mathscr P_{U_0}(r).
\end{equation}

For a compactly supported \(C^1\) vector field \(X\) on \(A_1\),
integration by parts in the first term of
\eqref{eq:6.120} implies
\begin{equation}\label{eq:6.123}
 \int X\cdot\mathcal C\dd x
 =\int_0^1\!\int\left[-2Z_tT_t[\nabla H,X]-\tr(HZ_tT_t\nabla X)
 +(x\cdot X)H\sigma_k(-D^2U_t)\right]\dd x\dd t.
\end{equation}
Applying the Cauchy--Schwarz inequality to the positive-definite quadratic form induced by \(T_t\) and
\eqref{eq:6.119}, we derive
\[
 \left|\int_0^1\int Z_tT_t[\nabla H,X]\dd x\dd t\right|
 \leq C E^{1/2}
 \left(\int_0^1\int_{A_1}\tr T_t\dd x\dd t\right)^{1/2}
 \leq C\sqrt E.
\]
The second and third terms in \eqref{eq:6.123} are
bounded by \(CK_0\delta\).  Therefore, we obtain
\begin{equation}\label{eq:6.124}
 \left|\int X\cdot\mathcal C\dd x\right|
 \leq C\left(\delta+\sqrt{\delta S}\right).
\end{equation}
Take \(X(x)=\zeta(|x|)x/|x|\).  Equation
\eqref{eq:6.122} implies
\[
 \int X\cdot\mathcal C\,\dd x
 =\int_1^2\zeta(r)
 \bigl(\mathscr P_{U_1}(r)-\mathscr P_{U_0}(r)\bigr)\,\dd r.
\]
Equation \eqref{eq:6.124} now proves
\eqref{eq:6.117}.
\end{proof}

\begin{proposition}\label{Prop:6.10}
Suppose that \(W_j\) are positive and smooth in a fixed
ball \(B_{r_0}\), that \(-D^2W_j\in\Gamma_k\), and that
\begin{equation}
 \sigma_k(-D^2W_j)=\varepsilon_jW_j^{p_*},
 \qquad \varepsilon_j\longrightarrow0.
\end{equation}
Suppose also that \(W_j\to W\) locally uniformly on every compact subset of
\(B_{r_0}\setminus\{0\}\), and that
\begin{equation}
 W=\Phi_\kappa+d+o(1)\qquad(x\to0).
\end{equation}
Then \(d=0\).
\end{proposition}

\begin{proof}
Fix \(0<R<r_0/8\) and let $U_{j,R}(x)=R^bW_j(Rx)$.
Since \(bp_*=n+2\), we have
\begin{equation}\label{eq:6.125}
 \sigma_k(-D^2U_{j,R})
 =\varepsilon_jR^{-2}U_{j,R}^{p_*}.
\end{equation}
Let $\omega(\rho)=\sup\limits_{0<|x|\leq\rho}|W(x)-\Phi_\kappa(x)-d|.$
Then \(\omega(\rho)\to0\), and on fixed annuli in \(B_8\),
\begin{equation}\label{eq:6.126}
 R^bW(Rx)=\Phi_\kappa(x)+R^bd
 +O(R^b\omega(8R)).
\end{equation}
Compare \(U_{j,R}\) with \(V_{R,\tau}\) given by
\eqref{eq:6.113}.  For fixed \(R\), first let
\(j\to\infty\) and \(\tau\downarrow0\).  Equations
\eqref{eq:6.114} and
\eqref{eq:6.126} yield
\begin{equation}\label{eq:6.127}
 \limsup_{\tau\downarrow0}\limsup_{j\to\infty}
 \|U_{j,R}-V_{R,\tau}\|_{L^\infty(A_2)}
 \leq CR^b\omega(8R).
\end{equation}

Equations \eqref{eq:6.126} and
\eqref{eq:6.114} give
\begin{equation}\label{eq:6.128}
 C^{-1}\leq U_{j,R},V_{R,\tau}\leq C
 \quad\hbox{on }\{1/8<|x|<6\}
\end{equation}
for all sufficiently small \(R\), all sufficiently large \(j\)
depending on \(R\), and all sufficiently small \(\tau\).  The
constant is independent of \(R,j,\tau\).
For \(0<R<r_0/8\), let \(j\) large enough such that
\(\varepsilon_jR^{-2}\leq1\).  The local gradient estimate
\eqref{eq:6.6}, applied on a fixed finite cover of \(A_2\) by
balls compactly contained in \(\{1/8<|x|<6\}\), yields the bound for
\(U_{j,R}\). Equation \eqref{eq:6.110} gives the bound for
\(V_{R,\tau}\). Hence
\begin{equation}\label{eq:6.129}
 \|\nabla U_{j,R}\|_{L^\infty(A_2)}
 +\|\nabla V_{R,\tau}\|_{L^\infty(A_2)}\leq C.
\end{equation}

Set
\(U_t=(1-t)U_{j,R}+tV_{R,\tau}\). Fix
\(\zeta\in C_c^\infty((1,2))\) with \(\int_1^2\zeta=1\). For \(1\leq\ell\leq k\), positivity of the mixed Hessians on
\(\Gamma_k\) implies that
\begin{equation}\label{eq:6.130}
\begin{aligned}
 \int_0^1\sigma_\ell(-D^2U_t)\dd t
 &=\frac1{\ell+1}\sum_{m=0}^{\ell}
 \sigma_\ell((-D^2U_{j,R})[\ell-m],
 (-D^2V_{R,\tau})[m])\\
 &\leq\frac1{\ell+1}
 \sigma_\ell(-D^2U_{j,R}-D^2V_{R,\tau}).
\end{aligned}
\end{equation}
Here
\(\sigma_\ell((-D^2U_{j,R})[\ell-m],
(-D^2V_{R,\tau})[m])\) denotes the polarized \(\ell\)-Hessian
with \(\ell-m\) entries \(-D^2U_{j,R}\) and \(m\) entries
\(-D^2V_{R,\tau}\).
The local Hessian mass estimate on
\(\{1/8<|x|<6\}\), applied to
\(-U_{j,R}-V_{R,\tau}\), together with
\eqref{eq:6.128}, bounds the integral of the right-hand side of
\eqref{eq:6.130} over \(A_2\), uniformly for all
sufficiently small \(R\), large \(j\) depending on \(R\), and small
\(\tau\).  Since
\(\tr T_{k-1}=(n-k+1)\sigma_{k-1}\), equation
\eqref{eq:6.130} with
\(\ell=k-1,k\) implies
\begin{equation}\label{eq:6.131}
 \int_0^1\int_{A_2}
 \left[\tr T_{k-1}(-D^2U_t)
 +\sigma_k(-D^2U_t)\right]\dd x\dd t\leq C.
\end{equation}
The integral $S=\int_{A_2}(f_0+f_1)\dd x$ in
Lemma~\ref{Lem:6.9} satisfies
\begin{equation}\label{eq:6.132}
 S=\int_{A_2}(f_0+f_1)\dd x\leq C\varepsilon_jR^{-2}+\int_{A_2}f_\tau\dd x
 \quad\Longrightarrow\quad
 \lim_{\tau\downarrow0}\limsup_{j\to\infty}S=0.
\end{equation}
Equations \eqref{eq:6.128},
\eqref{eq:6.129},
\eqref{eq:6.131}, and
\eqref{eq:6.132} imply
\begin{align}\label{eq:6.133}
 &\quad \left|\int_1^2\zeta(r)
 \bigl(\mathscr P_{U_{j,R}}(r)-
 \mathscr P_{V_{R,\tau}}(r)\bigr)\dd r\right|\notag\\
 &\quad\leq C\Bigg\{
 \|U_{j,R}-V_{R,\tau}\|_{L^\infty(A_2)}
 +\left[\|U_{j,R}-V_{R,\tau}\|_{L^\infty(A_2)}
 \left(C\varepsilon_jR^{-2}+\int_{A_2}f_\tau\dd x\right)
 \right]^{1/2}\Bigg\}.
\end{align}
Formula \eqref{eq:2.36}, the identity
\(\frac{n-2k}{k+1}(p_*+1)=n\), and
\eqref{eq:6.125} imply
\begin{equation}\label{eq:6.134}
 \sup_{1\leq r\leq2}|\mathscr P_{U_{j,R}}(r)|
 =\sup_{1\leq r\leq2}
 \frac{\varepsilon_jR^{-2}r}{p_*+1}
 \int_{\partial B_r}U_{j,R}^{p_*+1}\dd S
 \leq C\varepsilon_jR^{-2}\longrightarrow0.
\end{equation}
Lemma~\ref{Lem:6.8} yields
\begin{equation}\label{eq:6.135}
 \mathscr P_{V_{R,\tau}}(r)\longrightarrow \frac{n-2k}{k+1}MR^bd
 \qquad(1\leq r\leq2)\quad\hbox{as }\tau\downarrow0.
\end{equation}
Equations \eqref{eq:6.134},
\eqref{eq:6.135},
\eqref{eq:6.127}, \eqref{eq:6.132}, and
\eqref{eq:6.133}, first with \(R\) fixed, then with
\(j\to\infty\), and then with \(\tau\downarrow0\), yield that
\begin{equation}\label{eq:6.136}
 \frac{n-2k}{k+1}MR^b|d|\leq CR^b\omega(8R).
\end{equation}
The constant in \eqref{eq:6.136} is independent of \(R\).
Dividing by \(R^b\) and letting \(R\downarrow0\) gives \(d=0\).
\end{proof}

\begin{lemma}\label{Lem:6.11}
Let \(M>0\) and let \(q_0,q_1\in\Phi^k(B_r)\). Assume their boundary
values are continuous.  If
\[
 \mu_k[q_0]=\mu_k[q_1]=M\delta_0,
 \qquad q_0\leq q_1\quad\hbox{on }\partial B_r,
\]
then \(q_0\leq q_1\) in \(B_r\).
\end{lemma}

\begin{proof}
The homogeneity \(\mu_k[cq]=c^k\mu_k[q]\) reduces the proof to
\(M=1\).  The boundary inequality and the measure identity in the
statement are the subsolution hypotheses of Lemma~4.3 in
\cite{TW2002} for boundary value \(q_1|_{\partial B_r}\).  That lemma
gives a solution
\(\widetilde q\geq q_0\) with this boundary value.  Since
\(\mu_k[\widetilde q]=\mu_k[q_1]=\delta_0\), Theorem~4.5 of
\cite{TW2002} implies \(\widetilde q=q_1\).  Hence \(q_0\leq q_1\)
in \(B_r\).
\end{proof}

\begin{corollary}\label{Cor:6.12}
Under the hypotheses of Proposition~\ref{Prop:6.10},
assume in addition that \(W\geq\Phi_\kappa\) in
{\(B_{r_0}\setminus\{0\}\)} and that
\(\mu_k[-W]=\mu_k[-\Phi_\kappa]=M\delta_0\) in
{\(B_{r_0}\)}.
Then \(W\equiv\Phi_\kappa\) in
{\(B_{r_0}\setminus\{0\}\)}.
\end{corollary}

\begin{proof}
Put \(q=-W\), \(\psi=-\Phi_\kappa\), and
\(e=q-\psi\leq0\).  Inequality
\eqref{eq:6.109} and the strong maximum principle
give \(e\equiv0\) or \(e<0\).  In the second case choose a small
{sphere \(\partial B_r\), \(0<r<r_0\),} and let
\(c_0=\min\limits_{\partial B_r}(W-\Phi_\kappa)>0\).  Then
\(-W\leq-\Phi_\kappa-c_0\) on \(\partial B_r\), and
\(\mu_k[-W]=\mu_k[-\Phi_\kappa-c_0]\) in \(B_r\).
Lemma~\ref{Lem:6.11} gives
\(W\geq\Phi_\kappa+c_0\) in \(B_r\).  This contradicts the limit $W(x)-\Phi_\kappa(x)\longrightarrow0$ as $x\to0$, proved  in Proposition~\ref{Prop:6.10}.
Hence \(e\equiv0\) in {\(B_{r_0}\setminus\{0\}\)}.
\end{proof}

\begin{proposition}\label{Prop:6.13}
Let \(\Gamma_j\to\infty\), and suppose that \(v_j\) solve
\eqref{eq:6.78}, satisfy \eqref{eq:6.79}, and converge locally
uniformly to the normalized bubble \(U\) in \eqref{eq:6.81}.
Then, for every fixed \(\vartheta>1\) and every
\(S_j=o(\Gamma_j)\),
\begin{equation}\label{eq:6.137}
 v_j<\vartheta U\quad\hbox{in }B_{S_j}
\end{equation}
for all sufficiently large \(j\).
\end{proposition}

\begin{proof}
Assume on the contrary that \eqref{eq:6.137} does not hold, and let
\begin{equation}
 \rho_j=\inf\left\{r\leq S_j:
 \max_{{\overline B_r}}\frac{v_j}{U}\geq\vartheta\right\}.
\end{equation}
Local convergence gives \(\rho_j\to\infty\), while
\(\rho_j=o(\Gamma_j)\).  For some
\(e_j\in\mathbb S^{n-1}\),
\begin{equation}\label{eq:6.138}
 v_j<\vartheta U\quad\hbox{in }B_{\rho_j},
 \qquad
 v_j(\rho_je_j)=\vartheta U(\rho_j).
\end{equation}
Set $W_j(x)=\rho_j^bv_j(\rho_jx)$, $\varepsilon_j=\rho_j^{-2}$ and $\kappa=h_0^{-1/\gamma}$. Then
\begin{equation}\label{eq:6.139}
 \sigma_k(-D^2W_j)=\varepsilon_jW_j^{p_*},
 \qquad
 W_j(e_j)\longrightarrow\vartheta\kappa.
\end{equation}
Moreover, one has
\begin{equation}\label{eq:6.140}
 U(r)=\kappa r^{-b}+o(r^{-b}).
\end{equation}
{Direct differentiation of the
explicit formula \eqref{eq:6.27} also gives
\(-U'(r)=\kappa b r^{-b-1}+O(r^{-b-3})\).}
Equations \eqref{eq:6.138} and
\eqref{eq:6.140} imply
\(W_j(x)\leq C|x|^{-b}\) for \(0<|x|<1\).
For every \(K\Subset B_1\setminus\{0\}\), equations
\eqref{eq:6.138},
\eqref{eq:6.140}, and \eqref{eq:6.6} imply that $\{W_j\}$ is precompact in $C^{0}(K)$, while
\begin{equation}\label{eq:6.141}
 \|\varepsilon_jW_j^{p_*}\|_{L^\infty(K)}\longrightarrow0.
\end{equation}
Equation \eqref{eq:6.139},
\eqref{eq:6.141}, and weak continuity of
Hessian measures \cite[Theorem~1.1]{TW1999} imply that, by selecting a diagonal
subsequence,
\begin{equation}\label{eq:6.142}
 W_j\longrightarrow W
 \quad\hbox{locally uniformly in }B_1\setminus\{0\},
 \qquad
 \mu_k[-W]=0\quad\hbox{in} \,\, B_1\setminus\{0\}.
\end{equation}
Let \(d\mu_j=\varepsilon_jW_j^{p_*}\dd x\).  For
\(\varphi\in C_c(B_1)\), scaling gives
\begin{equation}\label{eq:6.143}
 \int\varphi\dd\mu_j
 =\int_{B_{\rho_j}}\varphi(y/\rho_j)v_j(y)^{p_*}\dd y.
\end{equation}
Equations \eqref{eq:6.138} and
\eqref{eq:6.143} {provide} the integrable dominating function \(\|\varphi\|_\infty\vartheta^{p_*}U^{p_*}\); hence applying
the dominated convergence theorem to \eqref{eq:6.143} yields
\begin{equation}\label{eq:6.144}
 \mu_j\rightharpoonup M\delta_0,
 \qquad
 M=\int_{\R^n}U^{p_*}\dd x
 =\frac{|\mathbb S^{n-1}|}{k}
 \binom{n-1}{k-1}(\kappa b)^k.
\end{equation}
{The last equality in
\eqref{eq:6.144} follows from the radial Hessian formula and
\(-U'(r)=\kappa b r^{-b-1}+O(r^{-b-3})\), since}
\[
 \int_{\R^n}U^{p_*}\dd x
 =\lim_{r\to\infty}\frac{|\mathbb S^{n-1}|}{k}
 \binom{n-1}{k-1}r^{n-k}(-U'(r))^k
 =\frac{|\mathbb S^{n-1}|}{k}
 \binom{n-1}{k-1}(\kappa b)^k.
\]
Since \(b<n\), the estimate \(W_j(x)\leq C|x|^{-b}\) for \(0<|x|<1\) and local uniform
convergence in $B_1\setminus \{0\}$ imply \(-W_j\to-W\) in
\(L^1_{\mathrm{loc}}(B_1)\). Let the same symbol \(-W\) denote the
upper-semicontinuous \(k\)-convex representative of $-W$ on
\(B_1\). Lemma~2.2 and Theorem~1.1 in \cite{TW1999}, together with
\eqref{eq:6.144}, imply $\mu_k[-W]=M\delta_0$.
Fix \(r<1\), and denote the minimum and maximum of \(W\) on
\(\partial B_r\) by \(m_r\) and \(M_r\).  Applying Lemma~\ref{Lem:6.11} to \(-W\), $-\Phi_\kappa+\Phi_\kappa(r)-m_r$ and $-\Phi_\kappa+\Phi_\kappa(r)-M_r$, we get
\begin{equation}
 \Phi_\kappa(x)-\Phi_\kappa(r)+m_r
 \leq W(x)\leq
 \Phi_\kappa(x)-\Phi_\kappa(r)+M_r
 \qquad(0<|x|<r).
\end{equation}
Thus
\begin{equation}\label{eq:6.145}
 W=\Phi_\kappa+O(1)\qquad(x\to0).
\end{equation}
Apply Theorem~\ref{Thm:6.7} to \(q=-W\), there is
\(d\in\R\) such that
\begin{equation}\label{eq:6.146}
 W=\Phi_\kappa+d+o(1).
\end{equation}

Fix
\(x\in B_1\setminus\{0\}\) and \(\delta\in(0,1/2)\).  By
\eqref{eq:6.145}, choose
\(0<r_\delta<|x|\)
so that \(W\geq(1-\delta/2)\Phi_\kappa\) on
\(\partial B_{r_\delta}\).
Local convergence yields
\(W_j\geq(1-\delta)\Phi_\kappa\) on
\(\partial B_{r_\delta}\).  Set
\begin{equation}\label{eq:6.147}
 R_j=\frac{2\Gamma_j}{\rho_j}\longrightarrow\infty,
 \qquad
 H_{j,\delta}(x)=(1-\delta)\kappa
 (|x|^{-b}-R_j^{-b}).
\end{equation}
The function \(-H_{j,\delta}\) is \(k\)-convex and \(k\)-harmonic
in \(B_{R_j}\setminus B_{r_\delta}\).  Equation
\eqref{eq:6.142} and the choice of
\(r_\delta\) give
\(H_{j,\delta}\leq(1-\delta)\Phi_\kappa\leq W_j\) on
\(\partial B_{r_\delta}\).  Equation
\eqref{eq:6.147} {yields}
\(H_{j,\delta}=0<W_j\) on \(\partial B_{R_j}\).
Hessian comparison implies
\(W_j\geq H_{j,\delta}\).  Evaluating this inequality at \(x\),
letting \(j\to\infty\), and then letting \(\delta\downarrow0\) yields \(W(x)\geq\Phi_\kappa(x)\).  Since \(x\) is arbitrary, we have
\begin{equation}\label{eq:6.148}
 W\geq\Phi_\kappa\quad\hbox{in }B_1\setminus\{0\}.
\end{equation}
Equations \eqref{eq:6.139},
\eqref{eq:6.142}, and
\eqref{eq:6.146} verify the hypotheses
of Proposition~\ref{Prop:6.10}.  Hence \(d=0\).
Corollary~\ref{Cor:6.12}
and \eqref{eq:6.148} imply
\begin{equation}\label{eq:6.149}
 W\equiv\Phi_\kappa.
\end{equation}
Fix
\(0<R_0<1/12\).  For all large \(j\), the ball
\(B_{3R_0}(e_j)\) lies in $B_{4\Gamma_j/\rho_j}$ and does not meet the
origin.  Applying \eqref{eq:6.5} with \(x_0=e_j\) and \(R=R_0\), \eqref{eq:6.139} implies
\begin{equation}\label{eq:6.150}
 \mu_j(B_r(e_j))\leq CW_j(e_j)^kr^{n-2k}\leq Cr^{n-2k},
 \qquad0<r<R_0/16.
\end{equation}
Put
\(\Theta_j=\varepsilon_j^{1/(2k)}W_j^{(k+1)/(n-2k)}\).
If \(\sigma_k(-D^2V)=V^{p_*}\) in \(B_2\), \(V(0)=1\), and
\(0<V\leq2^{(n-2k)/(k+1)}\), then \eqref{eq:6.6} with \(c=r=1\) implies
\(|\nabla\log V|\leq C\) in \(B_1\). Choose a fixed
\(\varrho_0>0\) so that \(e^{-C\varrho_0}\geq1/2\). Then
\begin{equation}\label{eq:6.151}
 V(y)\geq e^{-C|y|}\geq\frac12
 \quad\hbox{in }B_{\varrho_0},
 \qquad
 \int_{B_{\varrho_0}}V^{p_*}\dd y
 \geq2^{-p_*}|B_{\varrho_0}|=:c_*.
\end{equation}
Choose \(0<r_*<R_0/32\) so that
\begin{equation}\label{eq:6.152}
 Cr_*^{n-2k}<\frac{c_*}{2C_0}.
\end{equation}
Suppose that \(\Theta_j\) is unbounded in \(B_{r_*/2}(e_j)\).
Choose \(y_j\in B_{r_*/2}(e_j)\) with
\(\Theta_j(y_j)\to\infty\), and choose \(z_j\in B_{r_*}(e_j)\)
maximizing \((r_*-|z-e_j|)\Theta_j(z)\).  Put
\(d_j=r_*-|z_j-e_j|\).  Then
\begin{equation}\label{eq:6.153}
 d_j\Theta_j(z_j)
 \geq\frac{r_*}{2}\Theta_j(y_j)\longrightarrow\infty,
 \qquad
 \Theta_j(x)\leq2\Theta_j(z_j)
 \quad\hbox{if }|x-z_j|\leq d_j/2.
\end{equation}
Put
\begin{equation}\label{eq:6.154}
 H_j=W_j(z_j),\qquad
 \widehat{r_j}=\varepsilon_j^{-1/(2k)}H_j^{-(k+1)/(n-2k)},
 \qquad
 V_j(y)=H_j^{-1}W_j(z_j+\widehat{r_j}y).
\end{equation}
Equations \eqref{eq:6.153} and
\eqref{eq:6.154} imply
\begin{equation}\label{eq:6.155}
 \frac{\widehat{r_j}}{d_j}
 =\frac1{d_j\Theta_j(z_j)}\longrightarrow0,\qquad
 V_j(0)=1,\qquad
 0<V_j\leq2^{(n-2k)/(k+1)}
 \quad\hbox{in }B_{d_j/(2\widehat{r_j})}.
\end{equation}
The identities \(\frac{n-2k}{k+1}(p_*-k)=2k\) and
\eqref{eq:6.154} yield that
\begin{equation}\label{eq:6.156}
 \sigma_k(-D^2V_j)=V_j^{p_*}.
\end{equation}
Equations \eqref{eq:6.6},
\eqref{eq:6.155}, and
\eqref{eq:6.156}--\eqref{eq:6.151}
give
\begin{equation}\label{eq:6.157}
 \int_{B_{\varrho_0}}V_j^{p_*}\dd y\geq c_*.
\end{equation}
Changing variables by $x=z_j+\widehat{r_j}y$ and using $W_j(x)=H_jV_j(y)$, we have
\begin{equation}\label{eq:6.158}
 \mu_j(B_{\varrho_0\widehat{r_j}}(z_j))
 =\varepsilon_jH_j^{p_*}(\widehat{r_j})^n
 \int_{B_{\varrho_0}}V_j^{p_*}\dd y
 =\frac1{\varepsilon_j^{(n-2k)/(2k)}H_j}
 \int_{B_{\varrho_0}}V_j^{p_*}\dd y.
\end{equation}
Equations \eqref{eq:6.157} and
\eqref{eq:6.158} yield
\begin{equation}\label{eq:6.159}
 \mu_j(B_{\varrho_0\widehat{r_j}}(z_j))
 \geq
 \frac{c_*}
 {\varepsilon_j^{(n-2k)/(2k)}H_j}.
\end{equation}
The denominator equals
\(\varepsilon_j^{(n-2k)/(2k)}H_j
=v_j(\rho_jz_j)\leq C_0\).
Here \(|z_j|\leq1+r_*\) and \(\rho_j=o(\Gamma_j)\), so
\eqref{eq:6.79} applies to \(\rho_jz_j\).  Equation
\eqref{eq:6.155} gives
\(B_{\varrho_0\widehat{r_j}}(z_j)\subset B_{r_*}(e_j)\) for all
large \(j\).
Equations \eqref{eq:6.150},
\eqref{eq:6.159}, and
\eqref{eq:6.152} imply
\begin{equation}
 \frac{c_*}{C_0}
 \leq\mu_j(B_{\varrho_0\widehat{r_j}}(z_j))
 \leq\mu_j(B_{r_*}(e_j))
 \leq Cr_*^{n-2k}<\frac{c_*}{2C_0},
\end{equation}
which is a contradiction.
Hence
\begin{equation}\label{eq:6.160}
 \Theta_j\leq C\quad\hbox{in }B_{r_*/2}(e_j).
\end{equation}

The normalized function
\(\widehat W_j=\varepsilon_j^{(n-2k)/(2k(k+1))}W_j
=\Theta_j^{(n-2k)/(k+1)}\)
solves \(\sigma_k(-D^2\widehat W_j)=\widehat W_j^{p_*}\).
For a fixed \(0<s_0<r_*/8\), equation
\eqref{eq:6.160} and
\(\frac{n-2k}{k+1}(p_*-k)=2k\) imply that
\begin{equation}\label{eq:6.161}
 s_0^{2k}\widehat W_j^{p_*-k}
 =s_0^{2k}\Theta_j^{2k}\leq C
 \quad\hbox{in }B_{2s_0}(e_j).
\end{equation}
Here \(B_{2s_0}(e_j)\subset B_{r_*/2}(e_j)\), and
\([(1-s)e_j,e_j]\subset B_{s_0}(e_j)\) for \(0<s<s_0\).
Applying \eqref{eq:6.6} on \(B_{2s_0}(e_j)\), with the uniform
coefficient bound \eqref{eq:6.161}, yields
\begin{equation}\label{eq:6.162}
 |\nabla\log W_j|\leq C
 \quad\hbox{in }B_{s_0}(e_j)
\end{equation}
for some fixed \(s_0>0\).
By passing to a subsequence, let \(e_j\to e\) in
\(\mathbb S^{n-1}\).  Fix \(0<s<s_0\) and put
\(y_{j,s}=(1-s)e_j\).  Equations
\eqref{eq:6.142} and
\eqref{eq:6.149} give
\begin{equation}\label{eq:6.163}
 W_j(y_{j,s})\longrightarrow
 \Phi_\kappa((1-s)e)=\kappa(1-s)^{-b}.
\end{equation}
Equation
\eqref{eq:6.162} implies
\begin{equation}\label{eq:6.164}
 |\log W_j(e_j)-\log W_j(y_{j,s})|\leq Cs.
\end{equation}
We deduce from \eqref{eq:6.139},
\eqref{eq:6.163}, and
\eqref{eq:6.164} that
\begin{equation}\label{eq:6.165}
 |\log\vartheta+b\log(1-s)|\leq Cs.
\end{equation}
Letting \(s\downarrow0\) in \eqref{eq:6.165} gives
\(\vartheta=1\).  This contradicts \(\vartheta>1\) and
proves \eqref{eq:6.137}.
\end{proof}

\begin{proposition}\label{Prop:6.14}
Every entire solution $u\not\equiv0$ of \eqref{eq:1.8} is bounded and satisfies
\begin{equation}\label{eq:6.166}
 \liminf_{R\to\infty}
 R^{\frac{n-2k}{k}}\inf_{B_{2R}}u<\infty.
\end{equation}
\end{proposition}

\begin{proof}
Lemma~\ref{Lem:2.2} gives \(u>0\). Let \(R_j\to\infty\) be arbitrary, put
\(m_j=\inf\limits_{B_{2R_j}}u\), and choose \(x_j\in B_{2R_j}\) maximizing
\(\operatorname{dist}(x,\partial B_{2R_j})^{n-2k}u(x)\).  Define
\begin{equation}\label{eq:6.167}
 \begin{gathered}
 A_j=u(x_j),\qquad
 \sigma_j=\frac18\operatorname{dist}(x_j,\partial B_{2R_j}),
 \qquad
 \Gamma_j=\sigma_jA_j^{(k+1)/(n-2k)},\\
 v_j(y)=A_j^{-1}u(x_j+A_j^{-(k+1)/(n-2k)}y),
 \qquad
 \Psi_j=A_j\sigma_j^{n-2k}m_j^k.
 \end{gathered}
\end{equation}
We will show that
\begin{equation}\label{eq:6.168}
 \liminf_{j\to\infty}\Psi_j<\infty.
\end{equation}
Suppose on the contrary that \(\Psi_j\to\infty\).  Maximality gives
\begin{equation}
 v_j(y)\leq
 \left(1-\frac{|y|}{8\Gamma_j}\right)^{-(n-2k)}
 \qquad(|y|<8\Gamma_j).
\end{equation}
Thus \(v_j\) satisfies \eqref{eq:6.78},
\(v_j(0)=1\), \(v_j\leq2^{n-2k}\) in \(B_{4\Gamma_j}\), and
\eqref{eq:6.80} holds.  Since \(A_j\geq m_j\), we have
\begin{equation}\label{eq:6.169}
 \Gamma_j^{n-2k}
 =\Psi_j\left(\frac{A_j}{m_j}\right)^k
 \geq\Psi_j.
\end{equation}
In particular, \(\Gamma_j\to\infty\).  Put
\begin{equation}\label{eq:6.170}
 \widehat\Gamma_j
 =\Psi_j^{-\frac1{2(n-2k)}}\Gamma_j.
\end{equation}
Equations \eqref{eq:6.169} and
\eqref{eq:6.170} imply that
\(\widehat\Gamma_j\to\infty\) and
\(\widehat\Gamma_j=o(\Gamma_j)\).

By \eqref{eq:6.6}, Arzel\`a--Ascoli, and Proposition~\ref{Prop:6.6},
\(v_j\to U\) locally uniformly along a subsequence. Proposition~\ref{Prop:6.13}
with \(S_j=2\widehat\Gamma_j\) yields
\begin{equation}\label{eq:6.171}
 v_j\leq C\widehat\Gamma_j^{-b}
 \quad\hbox{on }\partial B_{\widehat\Gamma_j}.
\end{equation}
Moreover, we have
\begin{equation}\label{eq:6.172}
 A_j^{-(k+1)/(n-2k)}\widehat\Gamma_j
 =\sigma_j\frac{\widehat\Gamma_j}{\Gamma_j}
 =o(\sigma_j),
 \qquad
 x_j+A_j^{-(k+1)/(n-2k)}\partial B_{\widehat\Gamma_j}
 \subset B_{2R_j}.
\end{equation}
Equations \eqref{eq:6.171} and
\eqref{eq:6.172} give
\(m_j\leq CA_j\widehat\Gamma_j^{-b}\).  Since \(bk=n-2k\), one has
\begin{equation}\label{eq:6.173}
 \Psi_j
 \leq C\sigma_j^{n-2k}A_j^{k+1}
 \widehat\Gamma_j^{-(n-2k)}
 =C\Gamma_j^{n-2k}\widehat\Gamma_j^{-(n-2k)}
 =C\Psi_j^{1/2}.
\end{equation}
Equation~\eqref{eq:6.173} contradicts
\(\Psi_j\to\infty\) and proves
\eqref{eq:6.168}.  Since the sequence \(R_j\to\infty\) was arbitrary, for every such
sequence and for the quantities defined by
\eqref{eq:6.167}, one has $\liminf\limits_{j\to\infty}\Psi_j<\infty$.

Suppose now that \(u\) is unbounded.  Choose \(R_j\to\infty\) so that
\(M_j=\sup\limits_{B_{R_j}}u\to\infty\), put
\(m_j=\inf\limits_{B_{2R_j}}u\), choose \(x_j\in B_{2R_j}\) maximizing
\(\operatorname{dist}(x,\partial B_{2R_j})^{n-2k}u(x)\), and define
\(A_j,\sigma_j,\Gamma_j,v_j\), and \(\Psi_j\) by
\eqref{eq:6.167}.
Equation~\eqref{eq:4.18}, used
with \(R_0=1\), and the choice of \(x_j\) yield
\begin{equation}\label{eq:6.174}
 \Psi_j
 \geq8^{-(n-2k)}R_j^{n-2k}M_jm_j^k
 \geq cM_j\longrightarrow\infty.
\end{equation}
Equation~\eqref{eq:6.174} contradicts
\eqref{eq:6.168}.  Hence \(u\) is bounded.

Finally, suppose on the contrary that \eqref{eq:6.166} fails. By passing
to a subsequence, choose \(R_j\uparrow\infty\)
such that $R_j^{\frac{n-2k}{k}}m_j\longrightarrow\infty$ with \(m_j=\inf\limits_{B_{2R_j}}u\).
Choose \(x_j\in B_{2R_j}\) maximizing
\(\operatorname{dist}(x,\partial B_{2R_j})^{n-2k}u(x)\), and define
\(A_j,\sigma_j,\Gamma_j,v_j\), and \(\Psi_j\) by
\eqref{eq:6.167}.
The choice of \(x_j\) and \(\sup_{B_{R_j}}u\geq u(0)\) imply that
\begin{equation}\label{eq:6.175}
 \Psi_j
 \geq8^{-(n-2k)}u(0)
 \left(R_j^{\frac{n-2k}{k}}m_j\right)^k
 \longrightarrow\infty.
\end{equation}
Equation~\eqref{eq:6.175} contradicts
\eqref{eq:6.168}.  This proves
\eqref{eq:6.166}.
\end{proof}

\begin{proposition}\label{Prop:6.15}
Every entire solution $u\not\equiv0$ of \eqref{eq:1.8} satisfies
\begin{equation}\label{eq:6.176}
 \int_{\R^n}
 u^{\frac{n(k+1)-2k}{n-2k}}\dd x<\infty.
\end{equation}
\end{proposition}

\begin{proof}
By Proposition~\ref{Prop:6.14}, there are
\(R_j\uparrow\infty\) and a constant \(C\) such that
\begin{equation}\label{eq:6.177}
 \inf_{B_{2R_j}}u
 \leq CR_j^{-\frac{n-2k}{k}}.
\end{equation}
Choose \(x_j\in\overline B_{2R_j}\) so that
\(u(x_j)=\min\limits_{\overline B_{2R_j}}u
=\inf\limits_{B_{2R_j}}u\).  Apply \eqref{eq:6.5} with
\(Z=u\), \(r=4R_j\), and \(R=65R_j\).  The inequality
\(4R_j<65R_j/16\) verifies the radius condition in
\eqref{eq:6.5}.  Since
\(B_{R_j}\subset B_{4R_j}(x_j)\),
\eqref{eq:6.177} implies
\begin{equation}\label{eq:6.178}
 \int_{B_{R_j}}u^{p_*}\dd x
 \leq\mu_k[-u](B_{4R_j}(x_j))
 \leq Cu(x_j)^kR_j^{n-2k}\leq C.
\end{equation}
Letting \(R_j\uparrow\infty\), \eqref{eq:6.178} and the monotone convergence theorem prove
$\int_{\R^n}u^{p_*}\dd x<+\infty$.

Combining this with Proposition~\ref{Prop:6.14} implies that
\begin{equation}
 \int_{\R^n}
 u^{\frac{n(k+1)-2k}{n-2k}}\dd x
 \leq
 \|u\|_{L^\infty(\R^n)}^{\frac{n-4k}{n-2k}}
 \int_{\R^n}u^{p_*}\dd x<\infty.
\end{equation}
This proves the integral growth condition \eqref{eq:6.176}.
\end{proof}

\begin{proof}[Completion of the proof of
Theorem~\ref{Thm:1.6}]
Let \(u\not\equiv0\). Then Lemma \ref{Lem:2.2} implies that $u>0$.
By \eqref{eq:6.176}, we have
\(\int_{B_R}u^{\frac{n(k+1)-2k}{n-2k}}\dd x
\leq C\leq CR^2\log R\)
for every sufficiently large \(R\).  Thus \eqref{eq:6.3} holds. By Theorem~\ref{Thm:6.2}, we derived the unconditional classification results in Theorem \ref{Thm:1.6} for $n>4k$. This, combining with the proof of Theorem \ref{Thm:1.6} for $2k<n\leq 4k$ in Section 5, concludes our proof of Theorem \ref{Thm:1.6}.
\end{proof}

\section{Classification results for
  \texorpdfstring{\(n=2k\)}{n=2k}: Theorem~\ref{Thm:1.7}}
\label{Sec:7}

At \(n=2k\), the Hessian--Sobolev power is replaced by the
Moser--Trudinger scale of Tian and Wang~\cite{TianWang2010}.  In {this} section, we prove
Theorem~\ref{Thm:1.7}.  The power transformation \(w=u^{-\frac{2k}{n-2k}}\) used in
Section~\ref{Sec:5} degenerates at this dimension, so our
proof uses a different transformation adapted directly to the exponential
equation.  We first construct a nonnegative divergence and identify
its equality case.  Then, we estimate its boundary term by a finite
iteration through the Newton tensors.  Finally, the integral growth condition \eqref{eq:1.11} supplies disjoint annuli on which a cutoff with vanishing
energy can be built.

Throughout this section, let \(k\geq2\), let
\(u\in C^2(\R^{n})\) solve \eqref{eq:1.10} and satisfy \eqref{eq:1.11}, and set
\begin{equation}
    v=-\frac{u}{k+1},
    \qquad
    A=D^2v.
\end{equation}
Thus
\begin{equation}\label{eq:7.1}
    A\in\Gamma_k,
    \qquad
    \sigma_k(A)=\frac1{(k+1)^k}e^{-(k+1)v}.
\end{equation}

We do not need the finite mass assumption  \(\int_{\R^{n}}e^u\,\dd x<\infty\) in this section.

\subsection{\texorpdfstring{The divergence of the key vector field
\(\widetilde{\mathcal J}\) and its equality case}
{The divergence of the key vector field and its equality case}}

We now turn to the proof of Theorem~\ref{Thm:1.7}. As in \eqref{eq:3.1} and \eqref{eq:3.36}, we construct
the following vector field
\begin{equation}\label{eq:7.2}
    \widetilde{\mathcal{J}}
    =e^v\left(
       L_k(A)\nabla v
       +\frac{|\nabla v|^2}{2}T_{k-1}(A)\nabla v
    \right),
\end{equation}
whose divergence is nonnegative and boundary flux is controlled by the
same nonnegative density.  Once a sequence of cutoffs makes that flux
vanish, the divergence must vanish identically; the equality case will
then determine the solution.

\begin{proposition}\label{Prop:7.1}
The vector field \eqref{eq:7.2} satisfies
\begin{equation}
    \operatorname{div}\widetilde{\mathcal{J}}
    =e^v\biggl[
       \tr\bigl(L_k(A)A\bigr)
       +2\nabla v^TL_k(A)\nabla v
       +\frac{|\nabla v|^2}{2}
          \nabla v^TT_{k-1}(A)\nabla v
    \biggr]\geq0.
\label{eq:7.3}
\end{equation}
For every \(Y\in\R^{n}\),
\begin{equation}\label{eq:7.4}
    |\widetilde{\mathcal{J}}\cdot Y|^2
    \leq
    \frac12e^v|\nabla v|^2
    \operatorname{div}\big(\widetilde{\mathcal{J}})
    Y^TT_{k-1}(A)Y.
\end{equation}
Consequently, if \(0\leq\eta\leq1\) is compactly supported and
Lipschitz, then
\begin{align}
    \int_{\R^{n}}\eta^2\operatorname{div}\widetilde{\mathcal{J}}\dd x
    \leq (n+2)\int_{\R^{n}} e^v\sigma_{k-1}(A)|\nabla v|^2
    |\nabla\eta|^2\dd x.
\label{eq:7.5}
\end{align}
\end{proposition}

\begin{proof}
{By \eqref{eq:2.5}, the Newton tensors of \(D^2v\)
are divergence-free.}  Moreover,  \eqref{eq:2.11} and \eqref{eq:7.1}
give
\[
    \nabla\sigma_k(A)=-(k+1)\sigma_k(A)\nabla v,
    \qquad
    T_{k-1}(A)A=L_k(A)+\frac12\sigma_k(A)\Id.
\]
Therefore,
\begin{align*}
\operatorname{div}\bigl(L_k(A)\nabla v\bigr)
&=
\tr\bigl(L_k(A)A\bigr)
-\frac{k+1}{2}\sigma_k(A)|\nabla v|^2,
\\
\operatorname{div}\bigl(
 |\nabla v|^2T_{k-1}(A)\nabla v
\bigr)
&=
2\nabla v^TL_k(A)\nabla v
+(k+1)\sigma_k(A)|\nabla v|^2.
\end{align*}
Thus we have
\begin{align*}
    \operatorname{div}\widetilde{\mathcal{J}}
    ={}&e^v\left[\nabla v^T\left(
       L_k(A)\nabla v
       +\frac{|\nabla v|^2}{2}T_{k-1}(A)\nabla v
    \right)\right.\\
    &+\left.\left(\tr\bigl(L_k(A)A\bigr)
-\frac{k+1}{2}\sigma_k(A)|\nabla v|^2 +\nabla v^TL_k(A)\nabla v
+\frac{k+1}{2}\sigma_k(A)|\nabla v|^2\right)\right]\\
={}&e^v\biggl[
       \tr\bigl(L_k(A)A\bigr)
       +2\nabla v^TL_k(A)\nabla v
       +\frac{|\nabla v|^2}{2}
          \nabla v^TT_{k-1}(A)\nabla v
    \biggr].
\end{align*}
Since \(n=2k\), Lemma~\ref{Lem:2.4} gives
\begin{equation}\label{eq:7.6}
    L_k(A)^2
    \preceq
    \frac12\tr\bigl(L_k(A)A\bigr)T_{k-1}(A).
\end{equation}
Hence
\[
\begin{split}
\nabla v^TL_k(A)\nabla v
\geq-&\left(
 \frac12|\nabla v|^2\tr\bigl(L_k(A)A\bigr)
 \nabla v^TT_{k-1}(A)\nabla v
\right)^{1/2}.
\end{split}
\]
Consequently,
\[
\operatorname{div}\widetilde{\mathcal{J}}\geq e^v\left(
 \sqrt{\tr\bigl(L_k(A)A\bigr)}
 -\sqrt{
   \frac{|\nabla v|^2}{2}
   \nabla v^TT_{k-1}(A)\nabla v
 }
\right)^2\geq0.
\]
This proves \eqref{eq:7.3}.

By Corollary~\ref{Cor:2.3}, \(T_{k-1}(A)>0\).  It follows from
\eqref{eq:7.6} that
\[
L_k(A)T_{k-1}(A)^{-1}L_k(A)
\preceq
\frac12\tr\bigl(L_k(A)A\bigr)\Id.
\]
Therefore, we obtain
\begin{align*}
&\quad\left(
       L_k(A)\nabla v
       +\frac{|\nabla v|^2}{2}T_{k-1}(A)\nabla v
    \right)^TT_{k-1}(A)^{-1}\left(
       L_k(A)\nabla v
       +\frac{|\nabla v|^2}{2}T_{k-1}(A)\nabla v
    \right)\\
&\leq
\frac12|\nabla v|^2\tr\bigl(L_k(A)A\bigr)
+|\nabla v|^2\nabla v^TL_k(A)\nabla v
+\frac{|\nabla v|^4}{4}
  \nabla v^TT_{k-1}(A)\nabla v\\
&=\frac{|\nabla v|^2}{2e^v}
  \operatorname{div}\widetilde{\mathcal{J}}.
\end{align*}
Then
\eqref{eq:7.4} follows immediately.

Finally, we prove \eqref{eq:7.5}.
Taking \(Y=\nabla\eta\) in
\eqref{eq:7.4}, we obtain
\[
 \bigl|\widetilde{\mathcal J}\cdot\nabla\eta\bigr|^2
 \leq
 \frac12 e^v|\nabla v|^2
 \operatorname{div}\widetilde{\mathcal J}\,
 \nabla\eta^TT_{k-1}(A)\nabla\eta.
\]
Consequently,
\begin{align*}
 \int_{\R^n}\eta^2
 \operatorname{div}\widetilde{\mathcal J}\dd x
 &\leq
 \sqrt{2}\int_{\R^n}
 \Bigl(
   \eta^2\operatorname{div}\widetilde{\mathcal J}
 \Bigr)^{1/2}
 \Bigl(
   e^v|\nabla v|^2
   \nabla\eta^TT_{k-1}(A)\nabla\eta
 \Bigr)^{1/2}
 \dd x\\
 &\leq\frac12 \int_{\R^n}\eta^2
 \operatorname{div}\widetilde{\mathcal J}\dd x
 +\int_{\R^n}
 e^v|\nabla v|^2
 \nabla\eta^TT_{k-1}(A)\nabla\eta
 \dd x.
\end{align*}
Thus
\begin{equation}
 \label{eq:7.7}
 \int_{\R^n}\eta^2
 \operatorname{div}\widetilde{\mathcal J}\dd x
 \leq
 2\int_{\R^n}
 e^v|\nabla v|^2
 \nabla\eta^TT_{k-1}(A)\nabla\eta
 \dd x.
\end{equation}
Since \(T_{k-1}(A)>0\), one has
\(\nabla\eta^TT_{k-1}(A)\nabla\eta \leq \tr\bigl(T_{k-1}(A)\bigr) |\nabla\eta|^2\).
By
\[
 \tr\bigl(T_{k-1}(A)\bigr)
 =
 (n-k+1)\sigma_{k-1}(A)
 =
 (k+1)\sigma_{k-1}(A)
\]
and \eqref{eq:7.7}, we obtain
\[
 \int_{\R^n}\eta^2
 \operatorname{div}\widetilde{\mathcal J}\dd x
 \leq
 2(k+1)\int_{\R^n}
 e^v\sigma_{k-1}(A)|\nabla v|^2
 |\nabla\eta|^2\dd x.
\]
This proves \eqref{eq:7.5}.  The calculation is first justified for smooth
compactly supported \(\eta\); the Lipschitz case follows by the
standard compactly supported smooth approximation.  {For
\(u\in C^2\), the regularity discussion at the end of the Introduction
makes \(u\) locally smooth, and \eqref{eq:2.5} then justifies the
Newton-tensor integrations by parts.}
\end{proof}

The following Proposition \ref{Prop:7.2} indicates that the {classification} result in
Theorem~\ref{Thm:1.7} can be deduced from
\(\operatorname{div}\widetilde{\mathcal{J}}\equiv0\).

\begin{proposition}\label{Prop:7.2}
If \(\operatorname{div}\widetilde{\mathcal{J}}\equiv0\), then \(u\) has the form
\eqref{eq:1.12}.
\end{proposition}

\begin{proof}
Fix a point \(x\in\R^{n}\) arbitrarily.  If \(\nabla v(x)=0\), then
\eqref{eq:7.3} gives
\(\tr\bigl(L_k(A)A\bigr)=0\),
so \eqref{eq:2.18} shows that \(A=\lambda \Id\) for some \(\lambda>0\).

If \(\nabla v(x)\ne0\), then \(\operatorname{div}\widetilde{\mathcal{J}}=0\) and
\eqref{eq:7.4} {imply} \(\widetilde{\mathcal{J}}=0\), hence
\eqref{eq:7.2} {yields}
\begin{equation}\label{eq:7.8}
    L_k(A)\nabla v
    +\frac{|\nabla v|^2}{2}T_{k-1}(A)\nabla v=0.
\end{equation}
Together with \eqref{eq:7.3}, this yields
\begin{equation}\label{eq:7.9}
\nabla v^TL_k(A)\nabla v
 =-\frac{|\nabla v|^2}{2}\nabla v^TT_{k-1}(A)\nabla v,
 \qquad
\tr\bigl(L_k(A)A\bigr)
 =\frac{|\nabla v|^2}{2}
  \nabla v^TT_{k-1}(A)\nabla v.
\end{equation}

{At the point under consideration, choose orthonormal
coordinates in which \(A\) is diagonal,
and denote the corresponding eigenvalues by
\(\lambda_1,\ldots,\lambda_{2k}\).}  Since \(L_k(A)\) and
\(T_{k-1}(A)\) are polynomials in \(A\), they are diagonal in these
coordinates.  The \(i\)-th component of
\eqref{eq:7.8} is
\[
 \left(
   \bigl(L_k(A)\bigr)_{ii}
   +\frac{|\nabla v|^2}{2}
    \bigl(T_{k-1}(A)\bigr)_{ii}
 \right){\partial_i v}=0.
\]
Hence, whenever \({\partial_i v}\ne0\),
\begin{equation}\label{eq:7.10}
 \bigl(L_k(A)\bigr)_{ii}
 =
 -\frac{|\nabla v|^2}{2}
  \bigl(T_{k-1}(A)\bigr)_{ii}.
\end{equation}

Taking the \(i\)-th diagonal component of
\eqref{eq:7.6}, we obtain
\(\bigl(L_k(A)\bigr)_{ii}^2 \leq \frac12\tr\bigl(L_k(A)A\bigr) \bigl(T_{k-1}(A)\bigr)_{ii}\).
Combining the second identity in
\eqref{eq:7.9} with
\eqref{eq:7.10} gives, for \({\partial_i v}\ne0\),
\[
 \frac{|\nabla v|^4}{4}
 \bigl(T_{k-1}(A)\bigr)_{ii}^2
 \leq
 \frac{|\nabla v|^2}{4}
 \nabla v^TT_{k-1}(A)\nabla v\,
 \bigl(T_{k-1}(A)\bigr)_{ii}.
\]
Since \(\nabla v\ne0\) and \(T_{k-1}(A)>0\), we have
\[
 \bigl(T_{k-1}(A)\bigr)_{ii}
 \leq
 \frac{\nabla v^TT_{k-1}(A)\nabla v}
      {|\nabla v|^2}
 \qquad\text{for } i \text{ such that }
 {\partial_i v}\ne0.
\]
Multiplying this inequality by \({(\partial_i v)^2}\)
and summing over all indices for which
\({\partial_i v}\ne0\), we get
\begin{align}\label{eq:7.11}
 \sum_{\{i:\,{\partial_i v}\ne0\}}
 \bigl(T_{k-1}(A)\bigr)_{ii}{(\partial_i v)^2}
 &\leq
 \frac{\nabla v^TT_{k-1}(A)\nabla v}
      {|\nabla v|^2}
 \sum_{\{i:\,{\partial_i v}\ne0\}}
 {(\partial_i v)^2}
 =\nabla v^TT_{k-1}(A)\nabla v.
\end{align}
On the other hand, because \(T_{k-1}(A)\) is diagonal in the chosen
basis, one has
\[
 \sum_{\{i:\,{\partial_i v}\ne0\}}
 \bigl(T_{k-1}(A)\bigr)_{ii}{(\partial_i v)^2}
 =
 \nabla v^TT_{k-1}(A)\nabla v.
\]
Thus equality holds in \eqref{eq:7.11}.  Since every
individual inequality is multiplied by the strictly positive number
\({(\partial_i v)^2}\), equality must hold for every
index with \({\partial_i v}\ne0\).
Therefore,
\[
 \bigl(T_{k-1}(A)\bigr)_{ii}
 =
 \frac{\nabla v^TT_{k-1}(A)\nabla v}
      {|\nabla v|^2}
 \qquad\text{for } i \text{ such that }
 {\partial_i v}\ne0.
\]
For every \(i\) such that \({\partial_i v}\ne0\),
\eqref{eq:7.10} and the equality proved above
give
\begin{equation*}
 \bigl(L_k(A)\bigr)_{ii}^2
 =
 \frac{|\nabla v|^4}{4}
 \bigl(T_{k-1}(A)\bigr)_{ii}^2
 =
 \frac{|\nabla v|^2}{4}
 \nabla v^TT_{k-1}(A)\nabla v\,
 \bigl(T_{k-1}(A)\bigr)_{ii}
 =
 \frac12\tr\bigl(L_k(A)A\bigr)
 \bigl(T_{k-1}(A)\bigr)_{ii},
\end{equation*}
where the last equality follows from
\eqref{eq:7.9}.  Thus equality holds in
\eqref{eq:7.6} in the \(i\)-th eigendirection.
Hence equality holds in
\eqref{eq:7.6} in every eigendirection for which
\({\partial_i v}\ne0\).  Lemma~\ref{Lem:2.5} then implies
\begin{equation}\label{eq:7.12}
 \lambda_j=\lambda_\ell>0
 \qquad(j,\ell\ne i).
\end{equation}
There cannot be two nonzero components
\({\partial_i v},{\partial_j v}\).
Indeed,
\eqref{eq:7.12} applied to \(i\) gives
\(\lambda_j=\lambda_\ell\) for every \(\ell\ne i,j\), while its
application to \(j\) {yields} \(\lambda_i=\lambda_\ell\).  Since
\(2k\geq4\), such an index \(\ell\) exists, and hence all eigenvalues
of \(A\) are equal.  Thus \(A\) is scalar and \(L_k(A)=0\). Then
\eqref{eq:7.8} contradicts
\(T_{k-1}(A)>0\) and \(\nabla v\ne0\).  Thus \(\nabla v\) has one
nonzero component, say \({\partial_i v}\).  For any
\(j\ne i\), by \eqref{eq:2.6} and
\(\sigma_k(A)=
\lambda_i\sigma_{k-1}(\lambda(A)|i)+
\sigma_k(\lambda(A)|i)\), we have
\[
\bigl(T_{k-1}(A)\bigr)_{ii}
=\binom{2k-1}{k-1}\lambda_j^{k-1},\qquad\sigma_k(A)=\bigl(T_{k-1}(A)\bigr)_{ii}(\lambda_i+\lambda_j),
\]
and
\(\bigl(T_k(A)\bigr)_{ii} =\bigl(T_{k-1}(A)\bigr)_{ii}\lambda_j\).
Therefore,
\(\frac{\bigl(L_k(A)\bigr)_{ii}} {\bigl(T_{k-1}(A)\bigr)_{ii}} =\frac12(\lambda_i-\lambda_j)\),
and \eqref{eq:7.10} becomes
\begin{equation}\label{eq:7.13}
    |\nabla v|^2=\lambda_j-\lambda_i,
\end{equation}
so we have \(\lambda_j=\lambda_{\ell}\) for any \(j,\ell\neq i\).

{We have}
\begin{equation}
    D^2(e^v)=e^v\bigl(A+\nabla v\otimes\nabla v\bigr), \qquad\forall\,\,x\in\R^n,
\end{equation}
{Therefore,} \eqref{eq:7.13} implies that
\begin{equation}\label{eq:7.14}
    D^2(e^v)=\frac{\Delta(e^v)}{2k}\Id,
    \qquad
    \Delta(e^v)>0, \qquad\forall\,\,x\in\R^n.
\end{equation}
For any \(\ell\) and \(j\ne \ell\), it follows from \eqref{eq:7.14} that
\(\frac1{2k}\partial_\ell\Delta(e^v) ={\partial_\ell\partial_j^2(e^v)} ={\partial_j^2\partial_\ell(e^v)}=0\)
in the distributional sense.
Thus \(\Delta(e^v)\) is a positive constant.
Since \(e^v>0\) on
\(\R^{n}\), we deduce from \eqref{eq:7.14} that
\(e^{v(x)}=a_0\bigl(1+\lambda|x-x_0|^2\bigr), \qquad a_0,\lambda>0\).
At \(x=x_0\), equation \eqref{eq:7.1} gives
\(a_0^{-(k+1)} =2^k(k+1)^k\binom{2k}{k}\lambda^k\).
Using \(u=-(k+1)v\) proves \eqref{eq:1.12}.
\end{proof}

\subsection{The local estimate}

The following lemma allows us to control the right-hand side of
\eqref{eq:7.5} by using the integral growth condition \eqref{eq:1.11}.

\begin{lemma}\label{Lem:7.3}
Let \(0\leq\chi\leq1\) be smooth and compactly supported, and suppose
that
\begin{equation}\label{eq:7.15}
    |\nabla\chi|\leq\frac{C_0}{R}
\end{equation}
for some \(R>0\).  If \(\theta>2k\), then
\begin{equation}\label{eq:7.16}
    \int_{\R^{n}}
    \sigma_{k-1}(A)|\nabla v|^2e^v\chi^\theta\dd x
    \leq
    CR^{-2k}\int_{\supp\chi}e^v\dd x,
\end{equation}
where \(C\) depends only on \(k,\theta,C_0\).
\end{lemma}

\begin{proof}
 For \(1\leq s\leq k\), define
\begin{align*}
\Lambda_s&:=\int_{\R^{n}}\sigma_{k-s}(A)|\nabla v|^{2s}e^v\chi^\theta\dd x,\\
M_s&:=\int_{\R^{n}}
 (T_{k-s}(A))_{ij}{\partial_i v}\,
 {\partial_j v}|\nabla v|^{2(s-1)}e^v\chi^\theta\dd x,\\
E_s&:=\int_{\R^{n}}
 (T_{k-s}(A))_{ij}{\partial_i v}\,
 {\partial_j\chi}|\nabla v|^{2(s-1)}
 e^v\chi^{\theta-1}\dd x.
\end{align*}
Arguing as in the calculation in
\eqref{eq:3.7}--\eqref{eq:3.10} and using
\(A=D^2v\) and
\(\partial_i(e^v)=e^v{\partial_i v}\), we obtain
\begin{equation}\label{eq:7.17}
    M_s
    =\frac{k+s}{2s}\Lambda_s
     +\frac1{2s}M_{s+1}
     +\frac{\theta}{2s}E_{s+1},
    \qquad1\leq s\leq k-1.
\end{equation}
Then, as in \eqref{eq:3.13} and \eqref{eq:3.14}, we can derive
\[
    k\int_{\R^{n}}\sigma_k(A)e^v\chi^\theta\dd x
    =-M_1-\theta E_1,
    \qquad M_k=\Lambda_k.
\]
Multiplying \eqref{eq:7.17} by
\(1/(2^{s-1}(s-1)!)\), summing over \(s\) and using
\(
    \frac{2k}{2^kk!}=\frac1{2^{k-1}(k-1)!},
\)
we can obtain
\begin{equation}\label{eq:7.18}
\begin{split}
k\int_{\R^{n}}\sigma_k(A)e^v\chi^\theta\dd x
+\sum_{s=1}^k\frac{k+s}{2^ss!}\Lambda_s
=-\sum_{s=1}^k
  \frac{\theta}{2^{s-1}(s-1)!}E_s.
\end{split}
\end{equation}
Next, we estimate \(E_s\).
Since \(T_{k-s}(A)>0\), we have
\begin{align*}
 \left|
 (T_{k-s}(A))_{ij}{\partial_i v}\,
 {\partial_j\chi}
 \right|
 &\leq
 \left(
   (T_{k-s}(A))_{ij}{\partial_i v}\,
   {\partial_j v}
 \right)^{1/2}
 \left(
   (T_{k-s}(A))_{ij}{\partial_i\chi}\,
   {\partial_j\chi}
 \right)^{1/2}.
\end{align*}
Moreover, one has
\begin{align*}
 (T_{k-s}(A))_{ij}{\partial_i v}\,
 {\partial_j v}
 \leq
 \tr\bigl(T_{k-s}(A)\bigr)|\nabla v|^2,\qquad
 (T_{k-s}(A))_{ij}{\partial_i\chi}\,
 {\partial_j\chi}
 \leq
 \tr\bigl(T_{k-s}(A)\bigr)|\nabla\chi|^2.
\end{align*}
Consequently, by \eqref{eq:2.3} and \eqref{eq:7.15}, we have
\(\left| (T_{k-s}(A))_{ij}{\partial_i v}\, {\partial_j\chi} \right| \leq \frac{C}{R}\, \sigma_{k-s}(A)|\nabla v|\).
Therefore, \(|E_s|\leq \frac{C}{R}\int_{\R^n}\sigma_{k-s}(A)|\nabla v|^{2s-1}e^v\chi^{\theta-1}\dd x\).
By using Young's inequality and integrating the above inequality, we obtain
\begin{equation}\label{eq:7.19}
 |E_s|
 \leq
 \varepsilon\Lambda_s+C_\varepsilon J_s
 \qquad\text{with}\qquad
 J_s:=
 R^{-2s}\int_{\R^n}
 \sigma_{k-s}(A)e^v\chi^{\theta-2s}\dd x.
\end{equation}

Then, we will estimate \(J_s\) for \(1\leq s\leq k-1\).
By \eqref{eq:2.4}, we get
\[
 (k-s)J_s
 =
 R^{-2s}\int_{\R^n}
 (T_{k-s-1}(A))_{ij}{\partial_i\partial_j v}
 e^v\chi^{\theta-2s}\dd x.
\]
Using
\(
 \partial_j(T_{k-s-1}(A))_{ij}=0
\)
and integrating by parts, we obtain
\begin{align*}
 (k-s)J_s
 ={}&
 -R^{-2s}\int_{\R^n}
 (T_{k-s-1}(A))_{ij}{\partial_i v}\,
 {\partial_j v}
 e^v\chi^{\theta-2s}\dd x\\
 &-(\theta-2s)R^{-2s}
 \int_{\R^n}
 (T_{k-s-1}(A))_{ij}{\partial_i v}\,
 {\partial_j\chi}
 e^v\chi^{\theta-2s-1}\dd x.
\end{align*}
Because \(T_{k-s-1}(A)>0\), the first integral on the right-hand
side is nonpositive.  Hence
\begin{align*}
 (k-s)J_s
 \leq{}&
 C R^{-2s}\int_{\R^n}
 \left|
 (T_{k-s-1}(A))_{ij}{\partial_i v}\,
 {\partial_j\chi}
 \right|
 e^v\chi^{\theta-2s-1}\dd x.
\end{align*}
Noting that
\begin{equation*}
 \left|
 (T_{k-s-1}(A))_{ij}{\partial_i v}\,
 {\partial_j\chi}
 \right|
 \leq
 \left(
 (T_{k-s-1}(A))_{ij}{\partial_i v}\,
 {\partial_j v}
 \right)^{1/2}
 \left(
 (T_{k-s-1}(A))_{ij}{\partial_i\chi}\,
 {\partial_j\chi}
 \right)^{1/2}
 \leq
 \tr\bigl(T_{k-s-1}(A)\bigr)
 |\nabla v|\,|\nabla\chi|,
\end{equation*}
by using \eqref{eq:2.3} and
\eqref{eq:7.15}, we obtain
\[
 J_s
 \leq
 C R^{-2s-1}\int_{\R^n}
 \sigma_{k-s-1}(A)|\nabla v|
 e^v\chi^{\theta-2s-1}\dd x.
\]
Using Young's inequality and integrating, we conclude that
\begin{equation}\label{eq:7.20}
 J_s\leq
 \varepsilon\Lambda_{s+1}+C_\varepsilon J_{s+1}.
\end{equation}
Since
\(J_k=R^{-2k}\int_{\R^{n}} e^v\chi^{\theta-2k}\dd x\),
by \eqref{eq:7.19} and \eqref{eq:7.20}, we deduce from iteration and the fact \(0\leq\chi\leq1\) that
\[
    \sum_{s=1}^k|E_s|
    \leq
    \varepsilon\sum_{s=1}^k\Lambda_s
    +C_\varepsilon R^{-2k}\int_{\supp\chi}e^v\dd x.
\]
Inserting the above inequality into \eqref{eq:7.18} and choosing \(\varepsilon\)
small, we arrive at
\eqref{eq:7.16}.
{By the regularity stated at the end of the Introduction and
\eqref{eq:2.5}, every integration by parts above is valid in the
classical sense.}
\end{proof}

\subsection{Completion of our proof for Theorem~\ref{Thm:1.7}}

\begin{proof}[Proof of Theorem~\ref{Thm:1.7}]
Let \(\{R_j\}\) be the sequence of radii given by \eqref{eq:1.11}.
Pass to a subsequence, still denoted by \(R_j\), such that
\begin{equation}
    R_{j+1}\geq8R_j.
\end{equation}
For each \(j\), choose \(\chi_j\in C_c^\infty(B_{R_j})\) such that
\[
    \chi_j=1
    \quad\text{on }
    \left\{\frac{R_j}{4}<|x|<\frac{R_j}{2}\right\},
    \qquad
    |\nabla\chi_j|\leq\frac{C}{R_j}.
\]
Lemma~\ref{Lem:7.3} and
\eqref{eq:1.11} imply that
\begin{equation}\label{eq:7.21}
\int_{\{R_j/4<|x|<R_j/2\}}
 e^v\sigma_{k-1}(A)|\nabla v|^2\dd x\leq
CR_j^{-2k}\int_{B_{R_j}}e^v\dd x
\leq CR_j^2.
\end{equation}
For each integer \(m\geq1\), define
\[
 \eta_m(x)
 :=
 1-\frac1m\sum_{j=m}^{2m-1}
 \min\left\{
  1,\,
  \max\left\{
   0,\,
   \frac{4|x|-R_j}{R_j}
  \right\}
 \right\}.
\]
Then \(0\leq\eta_m\leq1\), \(\eta_m=1\) on \(B_{R_m/4}\), and
\(\eta_m=0\) outside \(B_{R_{2m-1}/2}\).  Moreover,
\[
 |\nabla\eta_m|
 =
 \frac{4}{mR_j}
 \quad\text{a.e. on }
 \left\{\frac{R_j}{4}<|x|<\frac{R_j}{2}\right\},
 \qquad m\leq j\leq2m-1,
\]
and \(\nabla\eta_m=0\) almost everywhere outside these annuli.
From \eqref{eq:7.21}, we infer that
\begin{equation}
\int_{\R^{n}} e^v\sigma_{k-1}(A)|\nabla v|^2
 |\nabla\eta_m|^2\dd x
\leq
\frac{C}{m^2}\sum_{j=m}^{2m-1}R_j^{-2}
\int_{\{R_j/4<|x|<R_j/2\}}
e^v\sigma_{k-1}(A)|\nabla v|^2\dd x
\leq\frac{C}{m}.
\end{equation}
By Proposition~\ref{Prop:7.1}, one has
\(\int_{\R^{n}}\eta_m^2\operatorname{div}\widetilde{\mathcal{J}}\dd x \longrightarrow0\).
Since \(R_m\to\infty\) and \(\eta_m=1\) on \(B_{R_m/4}\), for every
fixed \(R>0\), we have
\[
    0\leq\int_{B_R}\operatorname{div}\widetilde{\mathcal{J}}\dd x
    \leq\int_{\R^{n}}\eta_m^2\operatorname{div}\widetilde{\mathcal{J}}\dd x
    \longrightarrow0.
\]
Thus \(\operatorname{div}\widetilde{\mathcal{J}}\equiv0\), and
Proposition~\ref{Prop:7.2} proves
\eqref{eq:1.12}.

It remains to verify admissibility and the mass.  For
\eqref{eq:1.12}, direct differentiation gives, for
\(1\leq j\leq k\),
\begin{align}
\sigma_j(-D^2u)
={}&\binom{2k}{j}
\left(
 \frac{2(k+1)\lambda}{1+\lambda|x-x_0|^2}
\right)^j
\left(
 1-\frac{j}{k}
 \frac{\lambda|x-x_0|^2}{1+\lambda|x-x_0|^2}
\right)>0.
\end{align}
For \(j=k\), this is exactly
\(\sigma_k(-D^2u)=e^u\).  Finally, with
\(y=\sqrt\lambda(x-x_0)\),
\begin{align*}
\int_{\R^{n}}e^u\dd x
=2^k(k+1)^k\binom{2k}{k}
  \int_{\R^{n}}(1+|y|^2)^{-(k+1)}\dd y=\frac{2^k(k+1)^k\binom{2k}{k}\pi^k}{k!}.
\end{align*}
This proves \eqref{eq:1.13}.
\end{proof}

\section{Classification results for \texorpdfstring{\(n=2k\)}{n=2k}: Theorem~\ref{Thm:1.8}}
\label{Sec:8}

\subsection{\texorpdfstring{Global and asymptotic estimates for the
\(n/2\)-Hessian}{Global and asymptotic estimates for the $\frac{n}{2}$-Hessian}}

Let \(k\geq2\), \(n=2k\), and let \(v\in C^3(\R^n)\) satisfy
\begin{equation}\label{eq:8.1}
 D^2v\in\Gamma_k,
 \qquad
 \lim_{|x|\to\infty}v(x)=+\infty,
 \qquad
 \int_{\R^n}\sigma_k(D^2v)\dd x<\infty.
\end{equation}
For a regular value \(t\), set \(\Sigma_t=\{v=t\}\) and
\begin{equation}\label{eq:8.2}
 P_j(t)=\int_{\Sigma_t}|\nabla v|^{n-2j+1}
 \sigma_{j-1}\bigl(D^2v|_{T\Sigma_t}\bigr)\dd S,
 \qquad 1\leq j\leq k.
\end{equation}

Neither of the next two theorems assumes an equation for \(v\).

\begin{theorem}\label{Thm:8.1}
Under \eqref{eq:8.1},
\begin{equation}\label{eq:8.3}
 \sup_{t\ \mathrm{regular}}P_j(t)
 \leq C_{k,j}
 \left(\int_{\R^n}\sigma_k(D^2v)\dd x\right)^{\frac{n-j}{k}},
 \qquad 1\leq j\leq k,
\end{equation}
and
\begin{equation}\label{eq:8.4}
 \sup_{t\ \mathrm{regular}}P_k(t)
 =k\int_{\R^n}\sigma_k(D^2v)\dd x.
\end{equation}
Moreover,
\begin{equation}\label{eq:8.5}
 \Delta_n v={\operatorname{div}}
 \bigl(|\nabla v|^{n-2}\nabla v\bigr)\geq0,
 \qquad
 \int_{\R^n}\Delta_n v\dd x
 \leq C_k\left(\int_{\R^n}\sigma_k(D^2v)\dd x
 \right)^{2-\frac1k}.
\end{equation}
\end{theorem}

Thus finite \(k\)-Hessian mass controls all the level-set integrals in
\eqref{eq:8.2} and, in particular, the total \(n\)-Laplacian mass,
without assuming any equation for \(v\).

\begin{theorem}\label{Thm:8.2}
Under \eqref{eq:8.1}, there exists
\(\beta\in(0,\infty)\) such that
\begin{equation}\label{eq:8.6}
 \frac{\sup_{B_R}v}{\log R}\longrightarrow\beta,
 \qquad
 \int_{\R^n}\Delta_n v\dd x
 =|\mathbb S^{n-1}|\beta^{n-1},
\end{equation}
and, for every sufficiently large \(R\),
\begin{equation}\label{eq:8.7}
 \int_{B_R}e^v\dd x\leq CR^{n+\beta}.
\end{equation}
There are \(R_i\to\infty\) and \(c_i\to\infty\) such that
\begin{equation}
 v(R_i x)-c_i\longrightarrow\beta\log|x|
 \quad\hbox{strongly in }
 W^{1,n}_{\mathrm{loc}}(\R^n\setminus\{0\}).
\end{equation}
\end{theorem}

Thus properness and finite \(k\)-Hessian mass alone force a logarithmic
blow-down whose slope is determined exactly by the total
\(n\)-Laplacian mass.

For Theorem~\ref{Thm:8.3}, assume, for some \(a>0\), that
\begin{equation}\label{eq:8.8}
 \sigma_k(D^2v)=a e^{-(k+1)v}\quad\hbox{in }\R^n.
\end{equation}
For the sequences in Theorem~\ref{Thm:8.2}, write
\(v_i(x)=v(R_i x)-c_i\) and
\(\varepsilon_i=aR_i^n e^{-(k+1)c_i}\).

\begin{theorem}\label{Thm:8.3}
Under \eqref{eq:8.1} and
\eqref{eq:8.8}, the sequences in
Theorem~\ref{Thm:8.2} may be chosen so that
\begin{equation}\label{eq:8.9}
 v_i\longrightarrow\beta\log|x|
 \quad\hbox{in }C^1_{\mathrm{loc}}(\R^n\setminus\{0\}),
 \qquad
 \varepsilon_i\longrightarrow0.
\end{equation}
Moreover,
\begin{equation}\label{eq:8.10}
 \int_{\R^n}\sigma_k(D^2v)\dd x
 =\frac{|\mathbb S^{n-1}|}{k}
 \binom{n-1}{k-1}\beta^k,
\end{equation}
and the total \(n\)-Laplacian mass is given explicitly by
\begin{equation}\label{eq:8.11}
 \int_{\R^n}\Delta_n v\dd x
 =|\mathbb S^{n-1}|\beta^{n-1}
 =|\mathbb S^{n-1}|
 \left(
 \frac{k\displaystyle\int_{\R^n}\sigma_k(D^2v)\dd x}
 {|\mathbb S^{n-1}|\binom{n-1}{k-1}}
 \right)^{2-\frac1k}.
\end{equation}
For every \(1\leq j\leq k\),
\begin{align}\label{eq:8.12}
 \sup_{t\ \mathrm{regular}}P_j(t)
 =\lim_{t\to\infty}P_j(t)
 & =|\mathbb S^{n-1}|\binom{n-1}{j-1}\beta^{n-j}=|\mathbb S^{n-1}|\binom{n-1}{j-1}
 \left(
 \frac{k\displaystyle\int_{\R^n}\sigma_k(D^2v)\dd x}
 {|\mathbb S^{n-1}|\binom{n-1}{k-1}}
 \right)^{\frac{n-j}{k}}.
\end{align}
\end{theorem}

Thus \eqref{eq:8.8} improves the equation-free blow-down
to \(C^1\) convergence, yields the exact power law
\eqref{eq:8.11}, and makes every estimate in
\eqref{eq:8.3} an equality with the explicit constant in
\eqref{eq:8.12}.

For Theorem~\ref{Thm:8.1}, we combine the Newton-tensor formulas in
Reilly~\cite[Section~2]{Reilly1973}, the algebraic inequalities in
Wang~\cite[\S2.5]{Wang2009}, and the Michael--Simon inequality of Wu
and Yi~\cite[Corollary~1.3]{WuYi}, and iterate the resulting estimates
from \(P_k\) to \(P_1\).  The logarithmic upper estimate of
Kilpel\"ainen and Zhong~\cite[Theorem~1.2]{KilpelainenZhong} is combined
with these level-set estimates to prove Theorem~\ref{Thm:8.2}.
Finally, the weak continuity of Hessian measures of Trudinger and
Wang~\cite{TW1999} and the interior gradient calculation of Chou and
Wang~\cite{ChouWang2001} give the \(C^1\) convergence and the exact
formulas in Theorem~\ref{Thm:8.3}.

We will prove Theorems \ref{Thm:8.1}--\ref{Thm:8.3} in Subsection \ref{Subsec:8.3}.

The following identity is used on a bounded connected component of a sublevel set.

\begin{lemma}\label{Lem:8.4}
Let \(D\) be a bounded connected component of \(\{v<t\}\), where \(t\)
is a regular value, and suppose in a neighborhood of \(\overline D\) that
\(D^2v\in\Gamma_k\) and \(v\) solves \eqref{eq:8.8}.
On \(\partial D\), define
\begin{equation}\label{eq:8.13}
 \dd\nu=\frac1k|\nabla v|
 \sigma_{k-1}(D^2v|_{T\partial D})\dd S.
\end{equation}
Then
\begin{equation}\label{eq:8.14}
 \nu(\partial D)=\int_Da e^{-(k+1)v}\dd x
\end{equation}
and
\begin{equation}\label{eq:8.15}
 \int_{\partial D}(x\cdot\nabla v)\dd\nu
 =2a\left(\int_D e^{-(k+1)v}\dd x
              -e^{-(k+1)t}|D|\right).
\end{equation}
\end{lemma}

\begin{proof}
Apply Corollary~\ref{Cor:2.10} with
\(g(s)=ae^{-(k+1)s}\) and
\(G(s)=-a e^{-(k+1)s}/(k+1)\).
Equation~\eqref{eq:2.40} gives \eqref{eq:8.14}.  Since \(n=2k\),
substitution of \(g\) and \(G\) in \eqref{eq:2.41} {yields}
\eqref{eq:8.15}.
\end{proof}

\subsection{Proof of Theorem~\ref{Thm:1.8}}

Let \(u\) satisfy \eqref{eq:1.10} and
\eqref{eq:1.14}, and set
\(v=-u/(k+1)\).  Then
\begin{equation}\label{eq:8.16}
 D^2v\in\Gamma_k,
 \qquad
 \sigma_k(D^2v)=\frac1{(k+1)^k}e^{-(k+1)v},
 \qquad
 \int_{\R^n}\sigma_k(D^2v)\dd x
 =\frac1{(k+1)^k}\int_{\R^n}e^u\dd x<\infty,
\end{equation}
and \(v(x)\to+\infty\) as \(|x|\to\infty\).
{Thus
Theorems~\ref{Thm:8.1}--\ref{Thm:8.3} and Lemma \ref{Lem:8.4} can be applied.}

\begin{proof}[Proof of Theorem~\ref{Thm:1.8}]
{
Let \(R_i,c_i,v_i\), and \(\varepsilon_i\) be given by
Theorems~\ref{Thm:8.2} and~\ref{Thm:8.3}.  By \eqref{eq:8.16},
\begin{equation}\label{eq:8.17}
 \sigma_k(D^2v_i)=\varepsilon_i e^{-(k+1)v_i},
 \qquad
 \int_E\varepsilon_i e^{-(k+1)v_i}\dd x
 =\int_{R_iE}\sigma_k(D^2v)\dd x
\end{equation}
for every measurable set \(E\subset\R^n\).
}
Since the functions \(v_i\) are smooth, Sard's theorem shows that the
critical values of each \(v_i\) form a null set.  Their union over the
countable sequence is still null, so we can pick a real number \(s\) such that it is a regular
value for every \(v_i\).  Choose
\(0<r_-<e^{s/\beta}<r_+\).
{Equation~\eqref{eq:8.9} gives, for all sufficiently large \(i\),}
\(v_i<s\quad\hbox{on }\partial B_{r_-}, \qquad v_i>s\quad\hbox{on }\partial B_{r_+}\).
The inequalities are strict on small fixed annular neighborhoods of $\partial B_{r_-}$ and $\partial B_{r_+}$.
Since \(\Delta v_i>0\), the maximum principle yields
\(B_{r_-}\subset\{v_i<s\}\).  Let \(C_i\subseteq\{v_i<s\}\) be the connected component which contains $B_{r_-}$.  A path from \(B_{r_-}\) to the complement of
\(B_{r_+}\) would cross \(\partial B_{r_+}\), thus \(\overline C_i\Subset B_{r_+}\).  Therefore,
\begin{equation}\label{eq:8.18}
 B_{r_-}\subset C_i\subset B_{r_+},
 \qquad
 \partial C_i\subset\{r_-<|x|<r_+\}.
\end{equation}

By \eqref{eq:8.17}, we have
\[
 \int_{\R^n}\sigma_k(D^2v)\dd x
 \geq\int_{C_i}\varepsilon_i e^{-(k+1)v_i}\dd x
 \geq\int_{B_{r_-}}\varepsilon_i e^{-(k+1)v_i}\dd x
 =\int_{B_{R_ir_-}}\sigma_k(D^2v)\dd x
 \longrightarrow\int_{\R^n}\sigma_k(D^2v)\dd x.
\]
Let \(\nu_i\) be the measure in \eqref{eq:8.13} on \(\partial C_i\).
Equation \eqref{eq:8.14} gives
\[
    \nu_i(\partial C_i)
    =\int_{C_i}\varepsilon_i e^{-(k+1)v_i}\dd x
    \longrightarrow\int_{\R^n}\sigma_k(D^2v)\dd x.
\]
By \eqref{eq:8.9} and \eqref{eq:8.18}, one has $\sup_{\partial C_i}|x\cdot\nabla v_i-\beta|\longrightarrow0$. It follows that
\begin{equation}\label{eq:8.19}
 \int_{\partial C_i}(x\cdot\nabla v_i)\dd\nu_i
 \longrightarrow\beta\int_{\R^n}\sigma_k(D^2v)\dd x.
\end{equation}
On the other hand, \eqref{eq:8.15} gives
\[
 \int_{\partial C_i}(x\cdot\nabla v_i)\dd\nu_i
 =2\left(\int_{C_i}\varepsilon_i e^{-(k+1)v_i}\dd x
 -\varepsilon_i e^{-(k+1)s}|C_i|\right).
\]
{By \eqref{eq:8.9} and \eqref{eq:8.18},} the right-hand side tends to
\(2\int_{\R^n}\sigma_k(D^2v)\dd x\).  Since this integral is positive,
comparison with \eqref{eq:8.19} implies $\beta=2$.

Equation~\eqref{eq:8.7} implies that, for every sufficiently large \(R\),
\(\int_{B_R}e^{-u/(k+1)}\dd x =\int_{B_R}e^v\dd x \leq CR^{2k+2}\).
Thus condition \eqref{eq:1.11} holds, so
Theorem~\ref{Thm:1.7} yields \eqref{eq:1.12}.  Moreover,
\eqref{eq:8.10}, \eqref{eq:8.16}, and \(\beta=2\) give
\begin{align*}
 \int_{\R^n}e^u\dd x=(k+1)^k\frac{|\mathbb S^{2k-1}|}{k}
   \binom{2k-1}{k-1}2^k=\frac{2^k(k+1)^k\binom{2k}{k}\pi^k}{k!},
\end{align*}
which is \eqref{eq:1.13}. This concludes our proof of Theorem \ref{Thm:1.8}.
\end{proof}

\subsection{Proofs of
Theorems~\ref{Thm:8.1}--\ref{Thm:8.3}}
\label{Subsec:8.3}

\subsubsection{Proof of Theorem~\ref{Thm:8.1}}

Let \(v\) satisfy \eqref{eq:8.1}.  For a regular value
\(t\), let \(\Sigma_t=\{v=t\}\) and
\(B=D^2v|_{T\Sigma_t}\).  The outward unit normal is
\(\nabla v/|\nabla v|\).  The identity
\(\sigma_j(B) =T_j(D^2v)\left[ \frac{\nabla v}{|\nabla v|}, \frac{\nabla v}{|\nabla v|} \right]\)
and Corollary~\ref{Cor:2.3} give \(B\in\Gamma_{k-1}\).
Since \(n=2k\), equation \eqref{eq:2.21} {yields}
\(m_j=n-2j\).  Thus \eqref{eq:8.2} agrees with
\eqref{eq:2.25}.  Equation \eqref{eq:2.27} gives
\begin{equation}\label{eq:8.20}
 P_k(t)
 =\int_{\Sigma_t}T_{k-1}(D^2v)\nabla v\cdot
       \frac{\nabla v}{|\nabla v|}\dd S
 =k\int_{\{v<t\}}\sigma_k(D^2v)\dd x
 \leq k\int_{\R^n}\sigma_k(D^2v)\dd x.
\end{equation}
We next descend from \(P_k\) to \(P_1\).

\begin{lemma}\label{Lem:8.5}
For \(1\leq j\leq k-1\), on \(\Sigma_t\),
\begin{equation}\label{eq:8.21}
 (n-j)\sigma_j(D^2v)-(n-2j)\sigma_j(B)\geq0.
\end{equation}
Moreover, \(P_j\) has a nondecreasing locally absolutely continuous
extension such that, for almost every \(t\),
\begin{equation}
 P_j'(t)=\int_{\Sigma_t}|\nabla v|^{n-2j-1}
 \bigl((n-j)\sigma_j(D^2v)-(n-2j)\sigma_j(B)\bigr)\dd S.
\end{equation}
On \(\Sigma_t\),
\begin{equation}\label{eq:8.22}
 \left|T_{j-1}(B)\nabla_{\Sigma_t}|\nabla v|\right|^2
 \leq\frac{\sigma_j(B)}{n-2j}
 \bigl((n-j)\sigma_j(D^2v)-(n-2j)\sigma_j(B)\bigr),
\end{equation}
where \(\nabla_{\Sigma_t}=\nabla|_{T\Sigma_t}\).
\end{lemma}

\begin{proof}
Since \(n=2k\), equation~\eqref{eq:2.21} gives
\(m_j=n-2j\) and \(j+m_j=n-j\).  Equations
\eqref{eq:2.22}, \eqref{eq:2.24},
\eqref{eq:2.25}, and \eqref{eq:2.26} yield
\eqref{eq:8.21}--\eqref{eq:8.22}.

For later use, equation \eqref{eq:2.23} reads
\begin{equation}\label{eq:8.23}
 \operatorname{div}\left(
 |\nabla v|^{n-2j}T_{j-1}(D^2v)\nabla v\right)
 =|\nabla v|^{n-2j}
 \bigl((n-j)\sigma_j(D^2v)-(n-2j)\sigma_j(B)\bigr).
\end{equation}
On \(\{|\nabla v|>0\}\), its right-hand side is
\begin{equation}\label{eq:8.24}
 (n-j)|\nabla v|^{n-2j}\sigma_j(D^2v)
 -(n-2j)|\nabla v|^{n-2j-2}
 T_j(D^2v)[\nabla v,\nabla v].
\end{equation}
Both terms extend continuously with value zero on
\(\{|\nabla v|=0\}\).  The divergence theorem therefore gives
\begin{equation}\label{eq:8.25}
 P_j(t)=\int_{\{v<t\}}|\nabla v|^{n-2j}
 \bigl((n-j)\sigma_j(D^2v)-(n-2j)\sigma_j(B)\bigr)\dd x.
\end{equation}
At a critical point, the integrand in
\eqref{eq:8.25} is understood through the continuous
expression \eqref{eq:8.24}. This is the extension in
\eqref{eq:2.25}.
\end{proof}

The next algebraic estimate will be used in the Michael--Simon inequality.

{
\begin{lemma}\label{Lem:8.6}
Let \(1\leq j\leq n-3\), and let \(B\in\Gamma_{j+1}\)
be an \((n-1)\times(n-1)\) symmetric matrix.  Then
\begin{equation}\label{eq:8.26}
 \frac{1}{j!}\frac{d^{j}}{dz^{j}}\left[\det\bigl(T_j(B)+zT_{j-1}(B)\bigr)\right]\Big|_{z=0}
 \geq c_{n,j}\sigma_j(B)^{n-2}.
\end{equation}
\end{lemma}

\begin{proof}
Let \(\lambda_1\geq\cdots\geq\lambda_{n-1}\) be the eigenvalues of
\(B\).
The estimates in \cite[\S2.5(ii)--(iv), (viii)]{Wang2009} give
\(\lambda_1\cdots\lambda_j\geq c_{n,j}\sigma_j(\lambda)\) and
\(\sigma_j(\lambda|i)\geq c_{n,j}\sigma_j(\lambda)\) for \(i>j\).
For the remaining estimate, \(j=1\) is immediate; for \(j\geq2\),
induction in
\(\sigma_{j-1}(\alpha)=\alpha_1\sigma_{j-2}(\alpha|1)
+\sigma_{j-1}(\alpha|1)\) for the ordered
\(\alpha=\lambda|i\in\Gamma_j\), using the deletion property, gives
\[
 \sigma_{j-1}(\lambda|i)
 \geq
 \prod_{\substack{1\leq r\leq j\\r\ne i}}\lambda_r
 \quad\hbox{for }i\leq j,
\]

Therefore, we have
\begin{equation}\label{eq:8.27}
\begin{aligned}
 &\quad \frac{1}{j!}\frac{d^{j}}{dz^{j}}\left[\det\bigl(T_j(B)+zT_{j-1}(B)\bigr)\right]\Big|_{z=0}\\
 &=\frac{1}{j!}\frac{d^{j}}{dz^{j}}\left[\prod_{i=1}^{n-1}
 \bigl(\sigma_j(\lambda|i)
       +z\sigma_{j-1}(\lambda|i)\bigr)\right]\Big|_{z=0}\\
 &=
 \sum_{\substack{I\subset\{1,\ldots,n-1\}\\ |I|=j}}
 \prod_{i\in I}\sigma_{j-1}(\lambda|i)
 \prod_{i\notin I}\sigma_j(\lambda|i) \\
 &\geq \prod_{i=1}^j\sigma_{j-1}(\lambda|i)
 \prod_{i=j+1}^{n-1}\sigma_j(\lambda|i)\\
 &\geq
 c_{n,j}(\lambda_1\cdots\lambda_j)^{j-1}
 \sigma_j(\lambda)^{n-1-j}
 \geq c_{n,j}\sigma_j(\lambda)^{n-2},
\end{aligned}
\end{equation}
This proves \eqref{eq:8.26}.
\end{proof}
}

We use the positive-tensor Michael--Simon inequality of Wu and Yi
\cite[Corollary~1.3]{WuYi}.  If \(\Sigma\) is a closed smooth
\((n-1)\)-dimensional hypersurface and \(T\) is a smooth positive definite
symmetric tensor on \(\Sigma\), then
\begin{equation}\label{eq:8.28}
 c_n\left(\int_\Sigma(\det T)^{\frac{1}{n-2}}\dd S\right)^{\frac{n-2}{n-1}}
 \leq\int_\Sigma
 \left(|\operatorname{div}_\Sigma T|^2
       +|\tr(Th)|^2\right)^{1/2}\dd S,
\end{equation}
where \(h\) is the second fundamental form. A compact regular level has
finitely many connected components. Apply \eqref{eq:8.28} on each
component. The inequality
\((\sum_\alpha a_\alpha)^{(n-2)/(n-1)}
\leq\sum_\alpha a_\alpha^{(n-2)/(n-1)}\)
then gives \eqref{eq:8.28} on the whole level. On a regular level
\(\Sigma_t\), we use
\(h(X,Y)=\left\langle D_X\frac{\nabla v}{|\nabla v|},Y\right\rangle\).

\begin{proposition}\label{Prop:8.7}
For every \(1\leq j\leq k-1\),
\begin{equation}\label{eq:8.29}
    \sup_tP_j(t)<\infty.
\end{equation}
\end{proposition}

\begin{proof}
Fix a regular value \(t\).  For \(2\leq j\leq k-1\) and \(\eta>0\),
the tangential Hessian
\(|\nabla v|h=D^2v|_{T\Sigma_t}\) belongs to
\(\Gamma_{k-1}\subset\Gamma_j\). Hence
\(T_{j-1}(|\nabla v|h)\) and
\(T_{j-2}(|\nabla v|h)\) are positive definite. The tensor below is
therefore smooth and uniformly positive definite on each component of
\(\Sigma_t\), as required in \eqref{eq:8.28}. We
apply \eqref{eq:8.28} to
\[
 |\nabla v|^{n-2j}T_{j-1}(|\nabla v|h)
 +\eta |\nabla v|^{n-2j+2}T_{j-2}(|\nabla v|h).
\]
The Codazzi identities
{\cite[Section~2]{ReillyMeanCurvature1973}} give
\begin{align*}
 \operatorname{div}_{\Sigma_t}
 \bigl(|\nabla v|^{n-2j}T_{j-1}(|\nabla v|h)\bigr)
 &=(n-1-j)|\nabla v|^{n-2j-1}
   T_{j-1}(|\nabla v|h)\nabla_{\Sigma_t}|\nabla v|,\\
 \tr\left(|\nabla v|^{n-2j}T_{j-1}(|\nabla v|h)h\right)
 &=j |\nabla v|^{n-2j-1}\sigma_j(|\nabla v|h).
\end{align*}
Lemma~\ref{Lem:8.5}, \eqref{eq:8.2} and Cauchy's inequality imply
\begin{equation}
 \int_{\Sigma_t}
 \left|\operatorname{div}_{\Sigma_t}
  \bigl(|\nabla v|^{n-2j}T_{j-1}(|\nabla v|h)\bigr)\right|\dd S
 \leq C\sqrt{P_{j+1}P_j'}. \label{eq:8.30}
\end{equation}
By \eqref{eq:8.2} and integrating, we get
\begin{equation}\label{eq:8.31}
  \int_{\Sigma_t}\tr\left(|\nabla v|^{n-2j}T_{j-1}(|\nabla v|h)h\right)\dd S=jP_{j+1}(t).
\end{equation}

{Apply Lemma~\ref{Lem:8.6} with its \(j\) replaced by \(j-1\).}  The coefficient
containing \(j-1\) factors of \(T_{j-2}(|\nabla v|h)\) gives
\[
 \det\bigl(|\nabla v|^{n-2j}T_{j-1}(|\nabla v|h)
 +\eta |\nabla v|^{n-2j+2}T_{j-2}(|\nabla v|h)\bigr)
 \geq c\eta^{j-1}
 |\nabla v|^{(n-1)(n-2j)+2(j-1)}
 \sigma_{j-1}(|\nabla v|h)^{n-2}.
\]
Since \((n-1)(n-2j)+2(j-1)=(n-2)(n-2j+1)\),
inequalities \eqref{eq:8.28}, \eqref{eq:8.30}, \eqref{eq:8.31} yield
\begin{align}
 c\eta^{\frac{j-1}{n-1}}P_j^{\frac{n-2}{n-1}}
 \leq{}&P_{j+1}+\eta P_j
 +C\sqrt{P_{j+1}P_j'}
 +C\eta\sqrt{P_jP_{j-1}'}.
\label{eq:8.32}
\end{align}
For \(j=1\), apply \eqref{eq:8.28} to
\(|\nabla v|^{n-2}\Id\).  The same
calculation gives
\begin{equation}\label{eq:8.33}
cP_1^{\frac{n-2}{n-1}}
 \leq P_2+C\sqrt{P_2P_1'}.
\end{equation}

Taking the intersection of the full measure sets for positive
rational \(\eta\), and then using continuity in \(\eta\), shows that
\eqref{eq:8.32} and \eqref{eq:8.33} hold on one common full measure
set for every \(\eta>0\). Every limit \(\eta\downarrow0\) below is
taken at a point of this common set, where all derivatives in
\eqref{eq:8.32} are finite.

Lemma~\ref{Lem:8.8} below, applied to
\eqref{eq:8.20}, \eqref{eq:8.32}, and \eqref{eq:8.33}, proves
\eqref{eq:8.29}.
\end{proof}

\begin{lemma}\label{Lem:8.8}
Let \(I=(t_0,\infty)\), and suppose that
\(P_1,\ldots,P_k:I\to[0,\infty)\).  Assume that
\(P_1,\ldots,P_{k-1}\) are nondecreasing and locally absolutely
continuous, \(P_k\) is bounded, and there is a set of full measure
in \(I\) on which \eqref{eq:8.32} and \eqref{eq:8.33} hold for every
\(\eta>0\).  Put \(K=\sup_{t\in I}P_k(t)\).  Then
\[
 \sup_{t\in I}P_j(t)\leq C_kK^{(n-j)/k},
 \qquad 1\leq j\leq k-1.
\]
\end{lemma}

\begin{proof}
If \(K=0\), then \(P_k=0\).  For
\(j=k-1,\ldots,2\), assume inductively that \(P_{j+1}=0\), divide
\eqref{eq:8.32} by \(\eta^{(j-1)/(n-1)}\), and let
\(\eta\downarrow0\).  Since
\(1-(j-1)/(n-1)>0\), the limit gives \(P_j=0\) on the full measure
set, and continuity {extends this equality to} \(I\).  This index set is
empty when \(k=2\).  Thus \(P_2=\cdots=P_k=0\), and
\eqref{eq:8.33} and continuity give
\(P_1=0\).  We may assume that \(K>0\).

Set
\(\lambda=K^{-1/k}, \qquad \widehat P_j(s)=\lambda^{n-j}P_j(s/\lambda)\).
Choosing \(\eta=\lambda\widehat\eta\) in \eqref{eq:8.32} shows that
the functions \(\widehat P_j\) satisfy \eqref{eq:8.32} and
\eqref{eq:8.33} with the same constants.  Moreover,
\(\sup_{s\in\lambda I}\widehat P_k(s)=1\).  Relabel
\(\lambda I\) as \(I\) and \(\widehat P_j\) as \(P_j\).  It is therefore enough to obtain a
bound depending only on \(k\) under the normalization \(K=1\).

For \(1\leq j\leq k\), put \(d_j=3(k-j)+1\).  Define
\begin{equation}
 Y_j=(1+P_j)^{1/d_j},
 \qquad
 Y=\sum_{j=1}^{k-1}Y_j.
\end{equation}
The functions \(Y_j\) are nondecreasing and locally absolutely
continuous.  At almost every point where their derivatives exist,
choose \(\ell\) such that
\(Z=Y_\ell=\max_{1\leq j\leq k-1}Y_j\).
We consider points where \(Z\) is larger than a fixed constant
depending only on \(k\).

Suppose first that \(\ell\geq2\).  Take \(\eta=Z^{-3}\) in
\eqref{eq:8.32}.  Since \(d_{\ell+1}=d_\ell-3\), maximality gives
\(P_{\ell+1}+\eta P_\ell\leq CZ^{d_\ell-3}\).
For \(\ell=k-1\), the same estimate follows from \(P_k\leq1\).
The left-hand side of \eqref{eq:8.32} is bounded below by
\(cZ^{L_\ell}\), where
\begin{equation}\label{eq:8.34}
 L_\ell=\frac{d_\ell(n-2)-3(\ell-1)}{n-1},
 \qquad
 L_\ell-(d_\ell-3)=\frac{3k-1}{2k-1},
 \qquad
 2L_\ell-2d_\ell+4=\frac{2k}{2k-1}.
\end{equation}
The positive difference in \eqref{eq:8.34} allows the nonderivative
terms to be absorbed when \(Z\) is large.  The two derivative terms satisfy
\(\sqrt{P_{\ell+1}P_\ell'} \leq CZ^{(d_\ell-3)/2}(P_\ell')^{1/2}\),
\(\eta\sqrt{P_\ell P_{\ell-1}'} \leq CZ^{(d_\ell-6)/2}(P_{\ell-1}')^{1/2}\).
It follows that at least one of
\(P_\ell'\geq cZ^{2L_\ell-(d_\ell-3)}, \qquad P_{\ell-1}'\geq cZ^{2L_\ell+6-d_\ell}\)
holds.  In the first case,
\(Y_\ell' =\frac1{d_\ell}Y_\ell^{1-d_\ell}P_\ell' \geq cZ^{2k/(2k-1)}\).
In the second case, \(d_{\ell-1}=d_\ell+3\) and
\(Y_{\ell-1}\leq Z\).  Hence
\(Y_{\ell-1}' =\frac1{d_{\ell-1}}Y_{\ell-1}^{1-d_{\ell-1}}P_{\ell-1}' \geq cZ^{2k/(2k-1)}\).

Suppose next that \(\ell=1\).  Set
\(L_1=\frac{d_1(n-2)}{n-1}\).
Maximality and \(d_2=d_1-3\) give
\(P_2\leq CZ^{d_1-3}\).  This also holds when \(k=2\), because then
\(P_2=P_k\leq1\).  Since
\(L_1-(d_1-3)=\frac{3k-1}{2k-1}\),
equation \eqref{eq:8.33} gives
\(P_1'\geq cZ^{2L_1-(d_1-3)}\).
Consequently,
\(Y_1'\geq cZ^{2L_1-2d_1+4} =cZ^{2k/(2k-1)}\).

In every case, since \(Y\leq(k-1)Z\),
\(Y'\geq c_kY^\rho, \qquad \rho=\frac{2k}{2k-1}>1\),
whenever \(Y\) is sufficiently large.  Since \(Y\) is nondecreasing,
integration gives
\(\bigl(Y^{1-\rho}\bigr)' \leq-c_k(\rho-1)\).
If \(Y(t_1)\) exceeds this fixed threshold, then for every
\(t\geq t_1\) in \(I\),
\(Y(t)^{1-\rho} \leq Y(t_1)^{1-\rho}-c_k(\rho-1)(t-t_1)\).
Choosing
\(t-t_1>Y(t_1)^{1-\rho}/[c_k(\rho-1)]\) makes this upper bound
negative, which is impossible.  Hence
\(\sup_{t\in I}P_j(t)\leq C_k\) when \(K=1\).
Returning to the original functions gives
\(P_j(t)\leq C_k\lambda^{-(n-j)} =C_kK^{(n-j)/k}\).
\end{proof}

\begin{proof}[Completion of the proof of Theorem~\ref{Thm:8.1}]
Put
\(M=\int_{\R^n}\sigma_k(D^2v)\dd x\).
Equation \eqref{eq:8.20} gives \(\sup_tP_k(t)\leq kM\).
Lemma~\ref{Lem:8.8} and the estimates
\eqref{eq:8.32} and \eqref{eq:8.33} yield
\(\sup_tP_j(t)\leq C_{k,j}M^{(n-j)/k}\) for \(1\leq j\leq k-1\).
Together with \eqref{eq:8.20}, this proves \eqref{eq:8.3}.
Since the proper sublevel sets exhaust \(\R^n\), letting
\(t\) tend to infinity in \eqref{eq:8.20} proves
\eqref{eq:8.4}.

Since \(n=2k\), the exponent in Corollary~\ref{Cor:2.8} is \(p_0=n\).
Apply Corollary~\ref{Cor:2.8} to \(u=-v\).  We obtain
\begin{align}\label{eq:8.35}
 \Delta_n v
 =|\nabla v|^{n-2}\left(\Delta v+(n-2)
 D^2v\left[
 \frac{\nabla v}{|\nabla v|},
 \frac{\nabla v}{|\nabla v|}
 \right]\right)
 \geq0.
\end{align}
The identity holds in the distributional sense across
\(\{\nabla v=0\}\).  The \(j=1\) cases of
\eqref{eq:2.23} and \eqref{eq:2.25}, followed by exhaustion, give
\begin{equation}\label{eq:8.36}
 \int_{\R^n}\Delta_n v\dd x
 =\lim_{t\to\infty}P_1(t)<\infty.
\end{equation}
Since \(P_1\) is nondecreasing,
\(\int_{\R^n}\Delta_n v\dd x=\sup_tP_1(t)\).
The \(j=1\) case of \eqref{eq:8.3} now proves
\eqref{eq:8.5}.
\end{proof}

\subsubsection{Proof of Theorem~\ref{Thm:8.2}}

We first study the growth at infinity.

After a translation, assume that \(v(0)=\min_{\R^n}v\).
Equations \eqref{eq:8.35} and \eqref{eq:8.36} give
\(v-v(0)\geq0\),
\(\operatorname{div}(|\nabla v|^{n-2}\nabla v)=\Delta_n v\geq0\), and
\(\int_{\R^n}\Delta_n v\dd x<\infty\). The first two statements match
the sign convention in \cite[Theorem~1.2]{KilpelainenZhong} with
\(p=n\). That theorem gives
\[
 \sup_{B_R}(v-v(0))
 \leq C+C\int_0^{cR}
 \left(\int_{B_t}\Delta_n v\dd x\right)^{1/(n-1)}
 \frac{\dd t}{t}.
\]
For large \(R\), the integral over \(0<t<1\) is finite because
\(v\in C^3\). Equation
\eqref{eq:8.36} bounds the integral over \(1<t<cR\) by
\((\int_{\R^n}\Delta_n v\dd x)^{1/(n-1)}\log(cR)\). Hence
\begin{equation}\label{eq:8.37}
 \sup_{B_R}v
 \leq C+C\left(\int_{\R^n}\Delta_n v\dd x\right)^{1/(n-1)}
 \log(2+R)\leq C+C\log(2+R).
\end{equation}

In logarithmic polar coordinates, write
\begin{equation}
 V(\tau,\theta)=v(e^\tau\theta),
 \qquad
 m(\tau)=\frac1{|\mathbb S^{n-1}|}
       \int_{\mathbb S^{n-1}}V(\tau,\theta)\dd\theta.
\end{equation}
Since \(n=2k\), equation \eqref{eq:2.31} has no first order term.
It gives
\begin{equation}\label{eq:8.38}
 \partial_\tau^2V
 +\frac1{n-1}\Delta_{\mathbb S^{n-1}}V
 =\frac{e^{2\tau}}{n-1}
 \bigl(\Delta v+(n-2)D^2v[\theta,\theta]\bigr)\geq0.
\end{equation}

Thus \(m\) is convex. The proper asymptotic assumption $\lim\limits_{|x|\rightarrow+\infty}v(x)=+\infty$ gives \(m(\tau)\to+\infty\), and
\eqref{eq:8.37} provides an upper linear bound.  Hence
\begin{equation}\label{eq:8.39}
    \beta:=\lim_{\tau\to\infty}m'_+(\tau)\in(0,\infty),
    \qquad m(\tau)=\beta \tau+o(\tau).
\end{equation}

\begin{proposition}\label{Prop:8.9}
As \(R\to\infty\),
\begin{equation}\label{eq:8.40}
 \frac{\sup_{B_R}v}{\log R}\longrightarrow\beta,
 \qquad
 \int_{\R^n}\Delta_n v\dd x
 =|\mathbb S^{n-1}|\beta^{n-1}.
\end{equation}
For every fixed \(R_0>1\),
\begin{equation}\label{eq:8.41}
 \frac1{\log R}
 \int_{B_R\setminus B_{R_0}}
 \left|\nabla v-\beta\nabla\log|x|\right|^n\dd x
 \longrightarrow0.
\end{equation}
\end{proposition}

\begin{proof}
We first prove the first limit in \eqref{eq:8.40}. Given any sequence
\(R_i\to\infty\) and
\(K\Subset\{a\leq|x|\leq b\}\), where \(0<a<b<\infty\), equation
\eqref{eq:8.37} gives
\[
 0\leq\frac{v(R_ix)-v(0)}{\log R_i}
 \leq\frac{C+C\log(2+bR_i)}{\log R_i}\leq C_K
 \quad\hbox{on }K.
\]
{Since
\(D^2v\in\Gamma_k\), the functions
\((v(R_i x)-v(0))/\log R_i\) are continuous, subharmonic, and locally
uniformly bounded on \(\R^n\setminus\{0\}\). Thus
\cite[Theorem~4.1.9(a)]{H} applies locally there.} Its alternative of
convergence to \(-\infty\) is excluded by
the first inequality. After a subsequence,
\((v(R_i x)-v(0))/\log R_i\) converges in
\(L^1_{\mathrm{loc}}(\R^{n}\setminus\{0\})\). In cylinder
coordinates, write
\(W_i(s,\theta)= \frac{V(\log R_i+s,\theta)-v(0)}{\log R_i}\).
By \eqref{eq:8.39}, spherical averages $\frac1{|\mathbb S^{n-1}|}\int_{\mathbb S^{n-1}}W_{i}(s,\theta)\dd\theta$ converge locally uniformly
in \(s\) to \(\beta\).  It follows from \eqref{eq:8.38} that the
distributional limit of
\(\left(\partial_{ss}+\frac1{n-1}\Delta_{\mathbb S^{n-1}}\right)W_i\)
is a nonnegative measure \(\mu\). For every nonnegative
\(\phi\in C_c^\infty(\R)\), convergence of the spherical
averages and \eqref{eq:8.38} give
\[
 \int_{\R\times\mathbb S^{n-1}}\phi(s)\dd\mu
 =\lim_{i\to\infty}\int_{\R\times\mathbb S^{n-1}}
 W_i(s,\theta)\phi''(s)\dd s\dd\theta
 =|\mathbb S^{n-1}|\beta\int_{\R}\phi''(s)\dd s=0.
\]
Hence \(\mu=0\). The limit \(W\geq0\) is smooth and satisfies
\(\partial_s^2W +\frac1{n-1}\Delta_{\mathbb S^{n-1}}W=0 \quad\hbox{in }\R\times\mathbb S^{n-1}\),
and \(\frac1{|\mathbb S^{n-1}|}\int_{\mathbb S^{n-1}}W(s,\theta)\dd\theta=\beta\).
Since \(W\geq0\) and its spherical mean is \(\beta\), every
spherical-harmonic coefficient is bounded on \(\R\). Each nonconstant
coefficient \(c\) solves
\(c''-\frac{\lambda}{n-1}c=0 \qquad(\lambda>0)\),
so it vanishes.  Hence \(W\equiv\beta\).

Suppose that, after
passing to a sequence, for some \(\varepsilon>0\),
\(\sup_{B_{R_i}}v\geq(\beta+2\varepsilon)\log R_i\).
Choose \(x_i\in \overline{B_{R_i}}\) where this supremum is attained.  Then
\(|x_i|\to\infty\); otherwise the left-hand side would stay bounded.  Since
\(|x_i|\leq R_i\),
\[
 \frac{v(x_i)-v(0)}{\log|x_i|}
 \geq(\beta+2\varepsilon)\frac{\log R_i}{\log|x_i|}
       -\frac{v(0)}{\log|x_i|}
 \geq\beta+\varepsilon
\]
for all sufficiently large \(i\). {
Repeating the compactness argument with the scales \(|x_i|\), after a
subsequence the functions \((v(|x_i|x)-v(0))/\log|x_i|\) converge in
\(L^1_{\mathrm{loc}}(\R^n\setminus\{0\})\) to the constant \(\beta\).
After a further subsequence, \(x_i/|x_i|\) converges in
\(\mathbb S^{n-1}\). The moving-point conclusion
\cite[Theorem~4.1.9(b)]{H}, applied at \(x_i/|x_i|\), gives}
$\limsup\limits_{i\rightarrow\infty}\frac{v(x_i)-v(0)}{\log |x_i|}\leq\beta$, which is a contradiction. Moreover, it follows from \eqref{eq:8.39} that $\liminf\limits_{R\rightarrow+\infty}\frac{\sup_{B_R}v}{\log R}\geq\beta$.  This proves the first limit in \eqref{eq:8.40}.

Noting that \(\int_{\partial B_r}|\nabla v|^{n-2}\nabla v\cdot
\frac{x}{|x|}\dd S=\int_{B_r}\Delta_n v\dd x
\to\int_{\R^n}\Delta_n v\dd x\) as \(r\to\infty\),
{logarithmic averaging gives
\(\int_{R_0}^R(\int_{B_r}\Delta_n v\dd x)\frac{\dd r}{r}
=(\int_{\R^n}\Delta_n v\dd x)\log(R/R_0)+o(\log R)\). This asymptotic and
\(m(\log R)-m(\log R_0)=\beta\log(R/R_0)+o(\log R)\) from
\eqref{eq:8.39} imply}
\begin{align}
 \int_{B_R\setminus B_{R_0}}
 |\nabla v|^{n-2}\nabla v\cdot\nabla\log|x|\dd x
 &=\left(\int_{\R^n}\Delta_n v\dd x\right)\log(R/R_0)
   +o(\log R),\label{eq:8.42}\\
 \int_{B_R\setminus B_{R_0}}
 |\nabla\log|x||^{n-2}\nabla\log|x|\cdot\nabla v\dd x
 &=|\mathbb S^{n-1}|\beta\log(R/R_0)+o(\log R),\label{eq:8.43}\\
 \int_{B_R\setminus B_{R_0}}|\nabla\log|x||^n\dd x
 &=|\mathbb S^{n-1}|\log(R/R_0).\notag
\end{align}
Since \(P_1(t)\to\int_{\R^n}\Delta_n v\dd x\) as $t\rightarrow+\infty$, the coarea formula gives
\[
 \int_{B_R}|\nabla v|^n\dd x
 \leq\int_{\{v<\sup_{B_R}v\}}|\nabla v|^n\dd x
 =\int_{v(0)}^{\sup_{B_R}v}P_1(t)\dd t.
\]
Together with the first limit in \eqref{eq:8.40}, this implies
\begin{align}
 \limsup_{R\to\infty}\frac1{\log R}
 \int_{B_R\setminus B_{R_0}}|\nabla v|^n\dd x
 \leq\beta\int_{\R^n}\Delta_n v\dd x.
\end{align}
H\"older's inequality in \eqref{eq:8.42} gives
\(\int_{\R^n}\Delta_n v\dd x
\leq|\mathbb S^{n-1}|\beta^{n-1}\). By H\"{o}lder's inequality, \eqref{eq:8.43} {implies}
\begin{align*}
 \int_{B_R\setminus B_{R_0}}|\nabla v|^n\dd x
 \geq
 \frac{
 \left|\displaystyle\int_{B_R\setminus B_{R_0}}
 |\nabla\log|x||^{n-2}\nabla\log|x|\cdot\nabla v\dd x\right|^n}
 {\left(\displaystyle\int_{B_R\setminus B_{R_0}}
 |\nabla\log|x||^n\dd x\right)^{n-1}},
\end{align*}
which yields the reverse inequalities. Therefore,
\begin{equation}\label{eq:8.44}
 \int_{\R^n}\Delta_n v\dd x
 =|\mathbb S^{n-1}|\beta^{n-1},
 \qquad
 \frac1{\log R}\int_{B_R\setminus B_{R_0}}|\nabla v|^n\dd x
 \longrightarrow|\mathbb S^{n-1}|\beta^n.
\end{equation}
Finally, apply
\[
 c_n|\xi-\eta|^n
 \leq\bigl(|\xi|^{n-2}\xi-|\eta|^{n-2}\eta\bigr)
       \cdot(\xi-\eta)=|\xi|^{n}-|\xi|^{n-2}\xi\cdot\eta-|\eta|^{n-2}\eta\cdot\xi+|\eta|^{n}
\]
with \(\xi=\nabla v\) and
\(\eta=\beta\nabla\log|x|\). Integrating this inequality, dividing by
$\log R$, and using that the four limits on the right-hand side sum to
zero by \eqref{eq:8.42}--\eqref{eq:8.44}, we obtain \eqref{eq:8.41}.
\end{proof}

\begin{proposition}\label{Prop:8.10}
There are \(R_i\to\infty\) and \(c_i=\beta\log R_{i}+o(\log R_{i})\to\infty\) such that
\begin{equation}\label{eq:8.45}
 v_i(x):=v(R_i x)-c_i
 \longrightarrow\beta\log|x|
 \quad\hbox{strongly in }
 W^{1,n}_{\mathrm{loc}}(\R^n\setminus\{0\}).
\end{equation}
Moreover, for every sufficiently large \(R\),
\begin{equation}\label{eq:8.46}
    \int_{B_R}e^v\dd x\leq CR^{n+\beta}.
\end{equation}
\end{proposition}

\begin{proof}
Equation \eqref{eq:8.41} is equivalent in cylinder coordinates to
\[
 \int_0^{t}\int_{\mathbb S^{n-1}}
 \left|\nabla_{\tau,\theta}
       (V(\tau,\theta)-\beta\tau)\right|^n
 \dd\theta\dd\tau=o(t).
\]
Choose \(T_i\geq16i\) so large that
\[
 \int_0^{T_i}\int_{\mathbb S^{n-1}}
 \left|\nabla_{\tau,\theta}(V(\tau,\theta)-\beta\tau)\right|^n
 \dd\theta\dd\tau\leq i^{-4}T_i.
\]
Partition \([T_i/2,T_i]\) into disjoint intervals of length \(2i\).
One of their centers \(t_i\geq8i\) satisfies
\[
 \int_{t_i-i}^{t_i+i}\int_{\mathbb S^{n-1}}
 \left|\nabla_{\tau,\theta}(V(\tau,\theta)-\beta\tau)\right|^n
 \dd\theta\dd\tau\leq8i^{-3}\leq4i^{-2}.
\]
Set \(R_i=e^{t_i}\). Then
\begin{equation}\label{eq:8.47}
  \int_{e^{-i}\leq |x|\leq e^i}|\nabla\bigl(v(R_i x)-\beta\log|x|\bigr)|^{n}\dd x\leq 4i^{-2}.
\end{equation}
Let \(A_*=\{1<|x|<2\}\) and
\(f_i(x)=v(R_i x)-\beta\log|x|\).
Define
\begin{equation}\label{eq:8.48}
 c_i=\frac1{|A_*|}\int_{A_*}f_i(x)\dd x.
\end{equation}
Let \(A\Subset\R^n\setminus\{0\}\) be a compact annulus.  Choose a
bounded connected Lipschitz domain
\(\Omega\Subset\R^n\setminus\{0\}\) which contains \(A_*\cup A\).
For all sufficiently large \(i\), equation \eqref{eq:8.47} and the
Poincar\'e inequality with average over \(A_*\) give
\(\|f_i-c_i\|_{L^n(\Omega)} \leq C_\Omega\|\nabla f_i\|_{L^n(\Omega)} \longrightarrow0\).
Equation \eqref{eq:8.47} also {implies}
\(\|\nabla f_i\|_{L^n(\Omega)}\to0\).  Thus
\(f_i-c_i\longrightarrow0 \quad\hbox{strongly in }W^{1,n}(A)\).
Since \(A\) is arbitrary, this proves \eqref{eq:8.45} with the same
constants \(c_i\) on every compact annulus.  Polar coordinates,
\eqref{eq:8.39}, and the definition \eqref{eq:8.48} give
\(\frac{c_i}{t_i}\longrightarrow\beta\).
Hence
\(c_i=\beta\log R_i+o(\log R_i)\longrightarrow\infty\).

To prove \eqref{eq:8.46}, {multiply \eqref{eq:8.38} by
\(e^V\), integrate by parts on \(\mathbb S^{n-1}\), and apply
Cauchy's inequality to obtain}
\begin{align*}
 \frac{\dd^2}{\dd {\tau}^2}\left(
 \frac1{|\mathbb S^{n-1}|}
 \int_{\mathbb S^{n-1}}e^{V({\tau},\theta)}\dd\theta\right)
 &\geq\frac1{|\mathbb S^{n-1}|}
 \int_{\mathbb S^{n-1}}e^V
 \left((\partial_{{\tau}}V)^2
 +\frac{|\nabla_\theta V|^2}{n-1}\right)\dd\theta
 \\[1mm]
 &\geq
 \frac{\left(\dfrac{\dd}{\dd {\tau}}
   \dfrac1{|\mathbb S^{n-1}|}
   \int_{\mathbb S^{n-1}}e^{V({\tau},\theta)}\dd\theta\right)^2}
 {\dfrac1{|\mathbb S^{n-1}|}
   \int_{\mathbb S^{n-1}}e^{V({\tau},\theta)}\dd\theta},
\end{align*}
which indicates that $\log\left(\frac1{|\mathbb S^{n-1}|}
       \int_{\mathbb S^{n-1}}e^{V({\tau},\theta)}\dd\theta\right)$ is convex. By Jensen's inequality, then using \eqref{eq:8.39} and the first limit in \eqref{eq:8.40}, we derive
\[
 m({\tau})\leq
 \log\left(\frac1{|\mathbb S^{n-1}|}
       \int_{\mathbb S^{n-1}}e^{V({\tau},\theta)}\dd\theta\right)
 \leq\sup_\theta V({\tau},\theta),
 \qquad
 \frac1{{\tau}}\log\left(\frac1{|\mathbb S^{n-1}|}
       \int_{\mathbb S^{n-1}}e^{V({\tau},\theta)}\dd\theta\right)
 \longrightarrow\beta.
\]
The convexity and the last limit show that the right derivative of
\(\log(\frac1{|\mathbb S^{n-1}|}
\int_{\mathbb S^{n-1}}e^{V({\tau},\theta)}\dd\theta)\) is at most \(\beta\).
Hence its difference with \(\beta {\tau}\) is nonincreasing. Consequently,
for large \({\tau}\),
\(\log\left(\frac1{|\mathbb S^{n-1}|} \int_{\mathbb S^{n-1}}e^{V({\tau},\theta)}\dd\theta\right) \leq\beta {\tau}+C\).
Since
\(\dd x=e^{n{\tau}}\dd {\tau}\dd\theta\),
integration in \({\tau}\) proves
\eqref{eq:8.46}.
\end{proof}

{Propositions~\ref{Prop:8.9} and~\ref{Prop:8.10} prove
Theorem~\ref{Thm:8.2}.}

\subsubsection{Proof of Theorem~\ref{Thm:8.3}}

We now assume that $v$ solves equation \eqref{eq:8.8}.  The functions in
\eqref{eq:8.45} satisfy
\begin{equation}
 \sigma_k(D^2v_i)=\varepsilon_i e^{-(k+1)v_i},
 \qquad
 \varepsilon_i=aR_i^n e^{-(k+1)c_i}.
\end{equation}
Since \(n=2k\), the scaling is critical and
\begin{equation}\label{eq:8.49}
 \int_E\varepsilon_i e^{-(k+1)v_i(x)}\dd x
 =\int_{R_iE}\sigma_k(D^2v(y))\dd y,
 \qquad
 \int_{\R^n}\varepsilon_i e^{-(k+1)v_i}\dd x
 =\int_{\R^n}\sigma_k(D^2v)\dd x.
\end{equation}

\begin{lemma}\label{Lem:8.11}
On every compact subset of \(\R^n\setminus\{0\}\),
\begin{equation}\label{eq:8.50}
    v_i\longrightarrow\beta\log|x|
    \quad\hbox{uniformly}.
\end{equation}
Moreover,
\begin{equation}\label{eq:8.51}
 \varepsilon_i\longrightarrow0,
 \qquad
 \varepsilon_i e^{-(k+1)v_i}\longrightarrow0
 \quad\hbox{in }L^q_{\mathrm{loc}}(\R^n\setminus\{0\})
 \quad\hbox{for every }1\leq q<\infty.
\end{equation}
\end{lemma}

\begin{proof}
Fix two punctured annuli
\(\Omega_0\Subset\Omega_1\Subset\R^n\setminus\{0\}\).  The strong
\(W^{1,n}(\Omega_1)\) convergence in \eqref{eq:8.45} and the
Moser--Trudinger inequality~\cite{Moser} imply, for every finite \(s\),
\begin{equation}\label{eq:8.52}
    \sup_i\int_{\Omega_1}e^{s|v_i|}\dd x<\infty.
\end{equation}
Indeed, it follows from \eqref{eq:8.45} that the integral average $\overline{v_{i}}:=\frac{1}{|\Omega_{1}|}\int_{\Omega_1}v_i \dd x$ and $\|\nabla v_i\|_{L^n(\Omega_1)}$ are uniformly bounded. Young's inequality $s|v_i-\overline{v_{i}}|\leq \alpha|v_i-\overline{v_{i}}|^{\frac{n}{n-1}}+C_{s,\alpha}$, together with the Moser--Trudinger inequality, implies \eqref{eq:8.52}.

The sets \(R_i\Omega_1\) escape from every compact subset of \(\R^n\), as $i\rightarrow+\infty$.
Thus the finite mass assumption in \eqref{eq:8.1} and
\eqref{eq:8.49} give
\(\int_{\Omega_1}\varepsilon_i e^{-(k+1)v_i}\dd x \longrightarrow0\).
If \(\eta\geq0\) is a nonzero smooth function compactly supported in
\(\Omega_1\), then $\varepsilon_i\int_{\Omega_1}\eta e^{-(k+1)v_i}\dd x
 \longrightarrow0$. By \eqref{eq:8.52} and Vitali's theorem, $\int_{\Omega_1}\eta e^{-(k+1)v_i}\dd x$ converges to
\(\int\eta|x|^{-(k+1)\beta}\dd x>0\).  Hence
\(\varepsilon_i\to0\), and \eqref{eq:8.52} proves the second assertion
in \eqref{eq:8.51}.

Next, we prove \eqref{eq:8.50}. Choose \(q>1=\frac{n}{2k}\), put
\(\gamma=n(1-1/q)>0\), and choose annuli
\(\Omega_0\Subset\Omega'\Subset\Omega''\Subset\Omega_1\).
For every ball \(B_{2r}(x)\Subset\Omega_1\), H\"older's inequality and
the already proved \(L^q(\Omega_1)\) bound give
\begin{equation}\label{eq:8.53}
 \mu_k[v_i](B_r(x))
 =\int_{B_r(x)}\varepsilon_i e^{-(k+1)v_i}\dd x
 \leq C r^{n-2k+\gamma}=C r^\gamma,
\end{equation}
where \(C\) is independent of \(i,x,r\).  We also have
\(\sup\limits_i\|v_i\|_{L^1(\Omega_1)}<\infty\) by \eqref{eq:8.45}.

Since every
\(k\)-admissible function is subharmonic, the mean-value inequality and
the preceding \(L^1\) bound give
\begin{equation}
  \sup_{\Omega''}v_i\leq C
\end{equation}
with \(C\) depending only on
\(n,k,\Omega'',\Omega_1\) and the uniform \(L^1\) bound.  Thus
\(v_i-C\leq0\) on \(\Omega''\).

Choose \(d>0\), depending only on
\(\operatorname{dist}(\Omega',\partial\Omega'')\), so that
\(B_{8d}(x)\Subset\Omega''\) for every \(x\in\Omega'\). The uniform \(L^1\) bound yields, for every such
\(x\), a radius \(\rho_i\in[d,2d]\) satisfying
\(\frac1{|\partial B_{\rho_i}|} \int_{\partial B_{\rho_i}(x)}|v_i-C|\dd S\leq C\).
Since \(v_i-C\leq0\) on \(\Omega''\), one has
\[
\begin{aligned}
0
&\leq
\Big|
 \sup_{\partial B_{\rho_i}(x)}(v_i-C)
\Big|
=
-\sup_{\partial B_{\rho_i}(x)}(v_i-C)
\leq
\frac{1}{|\partial B_{\rho_i}(x)|}\int_{\partial B_{\rho_i}(x)}|v_i-C|\dd S
\leq C.
\end{aligned}
\]
Moreover, \eqref{eq:8.53} and \(n=2k\) give
\begin{align}\label{eq:8.54}
&\int_0^r
 \left(
  \frac{\mu_k[v_i](B_t(x))}{t^{n-2k}}
 \right)^{1/k}
 \frac{\dd t}{t}\leq
 C^{1/k}\int_0^r
 t^{\gamma/k}\frac{\dd t}{t}
 =
 \frac{kC^{1/k}}{\gamma}r^{\gamma/k}
 \leq Cr^{\gamma/k}.
\end{align}
Since \(2\rho_i\leq4d\) and
\(B_{8d}(x)\Subset\Omega''\), \eqref{eq:8.54} and
\cite[Theorem~9.1]{Wang2009} yield
\[
\begin{aligned}
0\leq C-v_i(x)
&\leq
C\left\{
 \int_0^{2\rho_i}
 \left(
  \frac{\mu_k[v_i](B_t(x))}{t^{n-2k}}
 \right)^{1/k}
 \frac{\dd t}{t}
 +
 \left|
  \sup_{\partial B_{\rho_i}(x)}(v_i-C)
 \right|
\right\}\\
&\leq
C\left((2\rho_i)^{\gamma/k}+1\right)
\leq C,
\qquad x\in\Omega'.
\end{aligned}
\]
Consequently, \(-C\leq v_i\leq C\) in \(\Omega'\), and
\(\sup_i\|v_i\|_{L^\infty(\Omega')}<\infty\).

Let \(x,y\in\Omega_0\), \(0<|x-y|<R\leq1\), and
\(B_{4R}(x)\Subset\Omega'\).
{By
\cite[Theorem~1]{TrudingerWeak1997}, the Dirichlet problem on
\(B_{4R}(x)\) has a unique
\(w_i\in\Phi^k(B_{4R}(x))\cap C(\overline{B_{4R}(x)})\) satisfying}
\[
 \mu_k[w_i]=0\quad\hbox{in }B_{4R}(x),
 \qquad w_i=v_i\quad\hbox{on }\partial B_{4R}(x).
\]
The comparison principle, \eqref{eq:8.54}, and
\cite[Theorem~9.2]{Wang2009} give
\begin{equation}\label{eq:8.55}
 0\leq w_i-v_i\leq CR^{\gamma/k}
 \qquad\hbox{in }B_R(x).
\end{equation}
Lemma~\ref{Lem:4.1}, applied to the two constant functions with the
boundary extrema of \(v_i\), {yields}
\(\inf_{\partial B_{4R}(x)}v_i\leq w_i\leq
\sup_{\partial B_{4R}(x)}v_i\).  Hence
\(\|w_i\|_{L^\infty(B_{4R}(x))}\leq C\).

Choose \(\varphi_{i,\ell}\in C^\infty(\partial B_{4R}(x))\) such that
\(\|\varphi_{i,\ell}-v_i\|_{L^\infty(\partial B_{4R}(x))}\leq1/\ell\), and
choose \(0<\delta_\ell\leq1\) with \(\delta_\ell\downarrow0\).
A smooth extension of \(\varphi_{i,\ell}\), plus a sufficiently large
multiple of \(|y-x|^2-16R^2\), is a strict admissible subsolution.
Thus \cite{CNS1985} gives a smooth admissible solution of
\[
 \sigma_k(D^2w_{i,\ell})=\delta_\ell
 \quad\hbox{in }B_{4R}(x),
 \qquad w_{i,\ell}=\varphi_{i,\ell}
 \quad\hbox{on }\partial B_{4R}(x).
\]
Set
\[
 e_{i,\ell}=\|\varphi_{i,\ell}-v_i\|_{L^\infty(\partial B_{4R}(x))},
 \qquad
 a_\ell=\frac{\delta_\ell^{1/k}}{2\binom nk^{1/k}}.
\]
The mixed Hessian expansion and its positivity
\cite[Section~2]{TW2002} yield
\[
 \mu_k\!\left[w_i+a_\ell(|y-x|^2-16R^2)\right]
 \geq\binom nk(2a_\ell)^k\dd y=\delta_\ell\dd y.
\]
Lemma~\ref{Lem:4.1} {gives}
\begin{equation}\label{eq:8.56}
 w_i-e_{i,\ell}+a_\ell(|y-x|^2-16R^2)
 \leq w_{i,\ell}\leq w_i+e_{i,\ell}
 \quad\hbox{in }B_{4R}(x).
\end{equation}
Since \(\|w_i\|_{L^\infty(B_{4R}(x))}\leq C\), equation
\eqref{eq:8.56} {implies} \(\|w_{i,\ell}\|_{L^\infty(B_{4R}(x))}\leq C\).
For \(p\in B_R(x)\), Lemma~\ref{Lem:8.12} on \(B_{2R}(p)\) {yields}
\(|\nabla w_{i,\ell}(p)|\leq C/R\).  Moreover,
\eqref{eq:8.56} {implies}
\(\|w_{i,\ell}-w_i\|_{L^\infty(B_{4R}(x))}
\leq e_{i,\ell}+16a_\ell R^2\to0\).
Consequently,
\begin{equation}\label{eq:8.57}
 |w_i(x)-w_i(y)|\leq C\frac{|x-y|}{R}.
\end{equation}
Equations \eqref{eq:8.55} and \eqref{eq:8.57} give
\begin{equation}\label{eq:8.58}
 |v_i(x)-v_i(y)|
 \leq C\left(\frac{|x-y|}{R}+R^{\gamma/k}\right).
\end{equation}
If
\(0<|x-y|<\min\{1,\operatorname{dist}(\Omega_0,\partial\Omega')^2/16\}\),
take \(R=|x-y|^{1/2}\) in \eqref{eq:8.58}.  Together with the
uniform \(L^\infty(\Omega')\) bound, this gives
\begin{equation}\label{eq:8.59}
 \sup_i\|v_i\|_{C^{\min\{1/2,\gamma/(2k)\}}(\overline{\Omega_0})}
 <\infty.
\end{equation}

Let \((v_{i_j})\) be any subsequence.  Equation \eqref{eq:8.59} {implies
that a further subsequence converges uniformly on \(\overline{\Omega_0}\)}
to some \(w\).  Equation \eqref{eq:8.45} gives
\[
 \|w-\beta\log|x|\|_{L^1(\Omega_0)}
 \leq |\Omega_0|\|w-v_{i_{j_\ell}}\|_{L^\infty(\Omega_0)}
 +\|v_{i_{j_\ell}}-\beta\log|x|\|_{L^1(\Omega_0)}\longrightarrow0.
\]
Thus every subsequence has a further subsequence converging uniformly
to \(\beta\log|x|\).  If \eqref{eq:8.50} failed, a subsequence would
satisfy
\(\|v_{i_j}-\beta\log|\mathord\cdot|\|_{L^\infty(\Omega_0)}\geq\varepsilon\)
for some \(\varepsilon>0\), contradicting the preceding conclusion.
Since \(\Omega_0\Subset\R^n\setminus\{0\}\) is arbitrary,
\eqref{eq:8.50} follows.
\end{proof}

We shall use the following form of the interior gradient estimate.
\begin{lemma}\label{Lem:8.12}
Let \(z\in C^3(B_r)\) be \(k\)-admissible and satisfy
\begin{equation}\label{eq:8.60}
    \sigma_k(D^2z)=F(x,z)\geq0.
\end{equation}
Then
\begin{equation}\label{eq:8.61}
 |\nabla z(0)|
 \leq C\frac{\|z\|_{L^\infty(B_r)}}r
 +C\left(\|z\|_{L^\infty(B_r)}^{k+1}\delta\right)^{1/(2k+1)}
 +C\left(\|z\|_{L^\infty(B_r)}^{k+1}\delta\right)^{1/(2k)},
\end{equation}
where
\(\delta=\operatorname{Lip}_xF+\operatorname{Lip}_sF\) on
\(B_r\times[-4\|z\|_{L^\infty(B_r)},4\|z\|_{L^\infty(B_r)}]\),
and \(C\) depends only on \(n\) and \(k\).
\end{lemma}
\begin{proof}
Put \(M=\|z\|_{L^\infty(B_r)}\).  The result is immediate if
\(M=0\).  We first take \(r=1\).  Let \((x,\xi)\) maximize
\begin{equation}\label{eq:8.62}
 (1-|y|^2)(4M-z(y))^{-1/2}\partial_\eta z(y)
 \quad\hbox{on }\overline B_1\times\mathbb S^{n-1}.
\end{equation}
If the maximum is nonpositive, then \(\nabla z(0)=0\).  Otherwise
\(x\in B_1\), and a rotation gives
\(\xi=e_1=\nabla z(x)/|\nabla z(x)|\).  Set
\(\rho=1-|x|^2\), \(H=4M-z(x)\), \(P=|\nabla z(x)|\), and
\(S=T_{k-1}(D^2z(x))\).  Equation \eqref{eq:8.62} gives
\begin{equation}\label{eq:8.63}
 |\nabla z(0)|\leq C\rho P.
\end{equation}
If \(\rho P\leq C_0M\), equation \eqref{eq:8.61} follows from
\eqref{eq:8.63}.  We henceforth assume \(\rho P>C_0M\).

At \(x\), \(z_i=P\delta_{1i}\).  Differentiating the logarithm of
\eqref{eq:8.62} gives
\begin{equation}\label{eq:8.64}
 \frac{z_{1i}}P=\frac{2x_i}{\rho}-\frac{z_i}{2H}.
\end{equation}
Since \(3M\leq H\leq5M\), equation \eqref{eq:8.64} and
\(C_0\geq40\) {yield}
\begin{equation}\label{eq:8.65}
 z_{11}=P\left(\frac{2x_1}{\rho}-\frac{P}{2H}\right)
 \leq\frac{2P}{\rho}-\frac{P^2}{10M}
 \leq-\frac{P^2}{20M}.
\end{equation}

Write \(B=D^2z(x)|_{e_1^\perp}\).  Then \(B\in\Gamma_{k-1}\).  If
\(k=2\), then \(\sigma_1(D^2z)=\sigma_1(B)+z_{11}\leq\sigma_1(B)\).
If \(k\geq3\), then \(T_{k-3}(B)\succeq0\) and the block identity gives
\[
 \sigma_{k-1}(D^2z)=\sigma_{k-1}(B)+z_{11}\sigma_{k-2}(B)
 -z_{1\alpha}(T_{k-3}(B))_{\alpha\beta}z_{1\beta}
 \leq\sigma_{k-1}(B),
\]
where Greek indices range from \(2\) to \(n\).  Thus
\begin{equation}\label{eq:8.66}
 \sigma_{k-1}(D^2z)\leq\sigma_{k-1}(B).
\end{equation}
The second block identity gives
\begin{equation}\label{eq:8.67}
 0\leq\sigma_k(D^2z)=\sigma_k(B)+z_{11}\sigma_{k-1}(B)
 -z_{1\alpha}(T_{k-2}(B))_{\alpha\beta}z_{1\beta}.
\end{equation}
Since \(T_{k-2}(B)\succeq0\), equations \eqref{eq:8.65} and
\eqref{eq:8.67} {yield}
\begin{equation}\label{eq:8.68}
 \sigma_k(B)\geq-z_{11}\sigma_{k-1}(B)
 \geq\frac{P^2}{20M}\sigma_{k-1}(B)>0.
\end{equation}
Thus \(B\in\Gamma_{k-1}\) and \eqref{eq:8.68} imply
\(B\in\Gamma_k\).
The Newton--Maclaurin inequality in dimension \(n-1\) gives
\(\frac{P^2}{M}\leq C\frac{\sigma_k(B)}{\sigma_{k-1}(B)} \leq C\sigma_{k-1}(B)^{1/(k-1)}\).
Using
\(S^{11}=\sigma_{k-1}(B)\) and
\(\tr S=(n-k+1)\sigma_{k-1}(D^2z)\) in the preceding
inequality and \eqref{eq:8.66}, we obtain
\begin{equation}\label{eq:8.69}
 S^{11}\geq\frac1{n-k+1}\tr S,
 \qquad
 \tr S\geq c\left(\frac{P^2}{M}\right)^{k-1}.
\end{equation}

Equation \eqref{eq:8.60} and \(z\in C^3\) give
\(F(\mathord\cdot,z(\mathord\cdot))=\sigma_k(D^2z)\in C^1(B_1)\).
Twice differentiating the logarithm of \eqref{eq:8.62} {yields}
\begin{equation}\label{eq:8.70}
 0\succeq-\frac{2\delta_{ij}}\rho-\frac{4x_ix_j}{\rho^2}
 +\frac{z_{ij}}{2H}+\frac{z_iz_j}{2H^2}+\frac{z_{1ij}}P
 +\frac1{P^2}\sum_{\alpha=2}^nz_{\alpha i}z_{\alpha j}
 -\frac{z_{1i}z_{1j}}{P^2}.
\end{equation}
Contracting \eqref{eq:8.70} with \(S^{ij}\), and using
\(S^{ij}z_{ij}=kF(x,z(x))\) and
\(S^{ij}z_{ij1}=\partial_1(F(\mathord\cdot,z(\mathord\cdot)))(x)\),
gives
\begin{align}
0\geq{}&-\frac{2\tr S}{\rho}
-\frac{4S^{ij}x_ix_j}{\rho^2}+\frac{kF(x,z(x))}{2H}
+\frac{P^2S^{11}}{2H^2}
+\frac{\partial_1(F(\mathord\cdot,z(\mathord\cdot)))(x)}P \notag\\
&+\frac{S^{ij}}{P^2}\sum_{\alpha=2}^nz_{\alpha i}z_{\alpha j}
-\frac{S^{ij}z_{1i}z_{1j}}{P^2}.                 \label{eq:8.71}
\end{align}
Equation \eqref{eq:8.64} and \(S\succeq0\) {imply}
\(\frac{S^{ij}z_{1i}z_{1j}}{P^2} \leq\frac{12S^{ij}x_ix_j}{\rho^2}+\frac{3P^2S^{11}}{8H^2}\).
Since \(F\geq0\), one has \(kF(x,z(x))/(2H)\geq0\).
Since \(S\succeq0\), one also has
\(S^{ij}\sum_{\alpha=2}^nz_{\alpha i}z_{\alpha j}/P^2\geq0\).
These two inequalities and
\(S^{ij}x_ix_j\leq\tr S\) turn
\eqref{eq:8.71} into
\begin{equation}\label{eq:8.72}
 \frac{P^2S^{11}}{8H^2}
 \leq\frac{18\tr S}{\rho^2}
 +\frac{|\nabla(F(\mathord\cdot,z(\mathord\cdot)))(x)|}{P}.
\end{equation}
Equations \eqref{eq:8.69} and \eqref{eq:8.72}, together with
\(H\leq5M\), give
\[
 c\tr S\,P^2
 \leq\frac{CM^2\tr S}{\rho^2}
 +\frac{CM^2|\nabla(F(\mathord\cdot,z(\mathord\cdot)))(x)|}{P}.
\]
Since \(\rho P>C_0M\),
\(CM^2\tr S/\rho^2
\leq(C/C_0^2)\tr S\,P^2\).
Increasing \(C_0\) absorbs this term.  Equation \eqref{eq:8.69}
then {implies}
\begin{equation}\label{eq:8.73}
 \tr S\,P^3
 \leq CM^2|\nabla(F(\mathord\cdot,z(\mathord\cdot)))(x)|,
 \qquad
 P^{2k+1}\leq CM^{k+1}|\nabla(F(\mathord\cdot,z(\mathord\cdot)))(x)|.
\end{equation}

For every \(\eta\in\mathbb S^{n-1}\), the difference quotient gives
\[
 \left|\partial_\eta(F(\mathord\cdot,z(\mathord\cdot)))(x)\right|
 \leq\limsup_{h\to0}
 \frac{|F(x+h\eta,z(x+h\eta))-F(x,z(x))|}{|h|}
 \leq\operatorname{Lip}_xF+\operatorname{Lip}_sF\,P
 \leq\delta(1+P).
\]
Equation \eqref{eq:8.73}, Young's inequality, and \eqref{eq:8.63}
{give}
\[
 |\nabla z(0)|\leq CM
 +C(M^{k+1}\delta)^{1/(2k+1)}
 +C(M^{k+1}\delta)^{1/(2k)}.
\]
This proves \eqref{eq:8.61} for \(r=1\).

For general \(r\), set \(\widetilde z(y)=z(ry)\) and
\(\widetilde F(y,s)=r^{2k}F(ry,s)\).  The two Lipschitz seminorms are
\[
 \operatorname{Lip}_y\widetilde F=r^{2k+1}\operatorname{Lip}_xF,
 \qquad
 \operatorname{Lip}_s\widetilde F=r^{2k}\operatorname{Lip}_sF.
\]
Applying the estimate for \(r=1\) before replacing the two seminorms
by \(\delta\) gives
\[
 |\nabla_y\widetilde z(0)|\leq CM
 +C(M^{k+1}r^{2k+1}\operatorname{Lip}_xF)^{1/(2k+1)}
 +C(M^{k+1}r^{2k}\operatorname{Lip}_sF)^{1/(2k)}.
\]
Since \(|\nabla_y\widetilde z(0)|=r|\nabla z(0)|\), division by \(r\)
proves \eqref{eq:8.61}.
\end{proof}

\begin{proposition}\label{Prop:8.13}
The convergence in \eqref{eq:8.45} improves to
\begin{equation}\label{eq:8.74}
 v_i\longrightarrow\beta\log|x|
 \quad\hbox{in }C^1_{\mathrm{loc}}(\R^n\setminus\{0\}).
\end{equation}
\end{proposition}

\begin{proof}
Fix compact annuli $A_0\Subset\widetilde A\Subset\R^n\setminus\{0\}$. Lemma~\ref{Lem:8.11} gives
\(\|v_i-\beta\log|\mathord\cdot|\|_{L^\infty(\widetilde A)} \longrightarrow0\).
Fix
\(0<r<\frac12\operatorname{dist}(A_0,\partial\widetilde A)\).
All constants below are independent of \(r\).
For \(x_0\in A_0\), subtract the affine tangent function
\(\ell_{x_0}(x)=\beta\log|x_0| +\beta\frac{x_0}{|x_0|^2}\cdot(x-x_0)\)
and set \(z_i=v_i-\ell_{x_0}\).  Then
\(D^2z_i=D^2v_i,\qquad D^2z_i\in\Gamma_k\).
Taylor's formula gives, uniformly for \(x_0\in A_0\),
\begin{equation}\label{eq:8.75}
 \|z_i\|_{L^\infty(B_r(x_0))}
 \leq
 \|v_i-\beta\log|\mathord\cdot|\|_{L^\infty(\widetilde A)}+
 \sup_{x\in B_r(x_0)}
 \left|\beta\log|x|-\ell_{x_0}(x)\right|
 \leq
 \|v_i-\beta\log|\mathord\cdot|\|_{L^\infty(\widetilde A)}
 +Cr^2.
\end{equation}

On \(B_r(x_0)\), \(z_i\) solves
\(\sigma_k(D^2z_i) =\varepsilon_i e^{-(k+1)(z_i+\ell_{x_0}(x))}\).
For all sufficiently large \(i\), \eqref{eq:8.75} gives
\[
 \sup_{x_0\in A_0}
 \left(
 \|\ell_{x_0}\|_{C^1(B_r(x_0))}
 +4\|z_i\|_{L^\infty(B_r(x_0))}
 \right)\leq C.
\]
Hence, whenever
\(|s|\leq4\|z_i\|_{L^\infty(B_r(x_0))}\),
\[
\begin{aligned}
&\quad \varepsilon_i e^{-(k+1)(s+\ell_{x_0}(x))}
 +\left|
 \nabla_x\left(\varepsilon_i
 e^{-(k+1)(s+\ell_{x_0}(x))}\right)
 \right|+
 \left|
 \partial_s\left(\varepsilon_i
 e^{-(k+1)(s+\ell_{x_0}(x))}\right)
 \right|\\
&\leq
 \varepsilon_i e^{(k+1)(|s|+|\ell_{x_0}(x)|)}
 \left(1+(k+1)|\nabla\ell_{x_0}(x)|+(k+1)\right)
 \leq C\varepsilon_i.
\end{aligned}
\]
Therefore
\begin{equation}\label{eq:8.76}
 \sup_{x_0\in A_0}
 \left\|\varepsilon_i e^{-(k+1)(s+\ell_{x_0}(x))}\right\|_{C^{0,1}
 (B_r(x_0)\times
 [-4\|z_i\|_{L^\infty(B_r(x_0))},
   4\|z_i\|_{L^\infty(B_r(x_0))}])}
 \leq C\varepsilon_i\longrightarrow0.
\end{equation}
Lemma~\ref{Lem:8.12}, applied after translating \(x_0\) to the origin,
and \eqref{eq:8.75}--\eqref{eq:8.76} give
\begin{align*}
 &\quad \sup_{x_0\in A_0}
 \left|\nabla v_i(x_0)-\beta\frac{x_0}{|x_0|^2}\right|\\
 &\leq
 \frac Cr\left(
 Cr^2+
 \|v_i-\beta\log|\mathord\cdot|\|_{L^\infty(\widetilde A)}
 \right)+
 C\left[
 \left(
 Cr^2+
 \|v_i-\beta\log|\mathord\cdot|\|_{L^\infty(\widetilde A)}
 \right)^{k+1}
 C\varepsilon_i
 \right]^{1/(2k+1)}\\
 &\quad
 +C\left[
 \left(
 Cr^2+
 \|v_i-\beta\log|\mathord\cdot|\|_{L^\infty(\widetilde A)}
 \right)^{k+1}
 C\varepsilon_i
 \right]^{1/(2k)}.
\end{align*}
For every fixed
\(0<r<\frac12\operatorname{dist}(A_0,\partial\widetilde A)\), one has
\(\limsup_{i\to\infty}\sup_{x_0\in A_0} \left|\nabla v_i(x_0)-\beta\frac{x_0}{|x_0|^2}\right| \leq Cr\).
Consequently,
\[
 0\leq
 \limsup_{i\to\infty}
 \left\|
 \nabla v_i-\nabla(\beta\log|\mathord\cdot|)
 \right\|_{L^\infty(A_0)}
 \leq
 \inf_{0<r<\frac12\operatorname{dist}(A_0,\partial\widetilde A)}
 Cr=0.
\]
Together with Lemma~\ref{Lem:8.11}, this gives $\|v_i-\beta\log|\mathord\cdot|\|_{C^1(A_0)}
 \longrightarrow0$. Since \(A_0\Subset\R^n\setminus\{0\}\) is arbitrary,
\eqref{eq:8.74} follows.
\end{proof}

Lemma~\ref{Lem:8.11} and Proposition~\ref{Prop:8.13} prove \eqref{eq:8.9} in Theorem~\ref{Thm:8.3}. It remains to prove \eqref{eq:8.10}--\eqref{eq:8.12} in Theorem~\ref{Thm:8.3}.

\begin{proof}[Completion of the proof of Theorem~\ref{Thm:8.3}]
Equations \eqref{eq:8.20}, \eqref{eq:8.23} and \eqref{eq:8.25} imply that, for
\(1\leq j\leq k\),
\begin{equation}
 \lim_{t\to\infty}P_j(t)
 =\int_{\R^n}\operatorname{div}\left(
 |\nabla v|^{n-2j}T_{j-1}(D^2v)\nabla v\right)\dd x<\infty,
\end{equation}
where for \(j=k\) the right-hand side equals
\(k\int_{\R^n}\sigma_k(D^2v)\dd x\).

Let
\(v_i\) be the functions in Proposition~\ref{Prop:8.13}, and choose
\(\chi\in C_c^\infty(B_2)\) such that \(0\leq\chi\leq1\) and
\(\chi=1\) on \(B_1\).  Since \(n=2k\), we have
\[
\begin{aligned}
&|\nabla v_i(x)|^{n-2j}T_{j-1}(D^2v_i(x))\nabla v_i(x)=
 R_i^{n-1}
 \left(
 |\nabla v|^{n-2j}T_{j-1}(D^2v)\nabla v
 \right)(R_ix).
\end{aligned}
\]
Consequently, with \(y=R_ix\),
\begin{align}\label{eq:8.77}
&\int_{\R^n}\chi(x)
 \operatorname{div}\left(
 |\nabla v_i|^{n-2j}T_{j-1}(D^2v_i)\nabla v_i\right)\dd x=
 \int_{\R^n}\chi(y/R_i)
 \operatorname{div}\left(
 |\nabla v|^{n-2j}T_{j-1}(D^2v)\nabla v\right)\dd y\notag\\
&\longrightarrow
 \int_{\R^n}
 \operatorname{div}\left(
 |\nabla v|^{n-2j}T_{j-1}(D^2v)\nabla v\right)\dd y
 =\lim_{t\to\infty}P_j(t).
\end{align}

We next pass to the limit in the left-hand side of \eqref{eq:8.77}.  On the annulus \(B_{2}\setminus B_1\supseteq\operatorname{supp}\nabla\chi\),
Proposition~\ref{Prop:8.13} gives
\(v_i\to\beta\log|x|\) in \(C^1\).  The matrix-valued weak continuity in
\cite[Theorem~2.4 and the statement following (2.23)]{TW2002}, applied
at order \(j\), implies
\begin{equation}\label{eq:8.78}
 T_{j-1}(D^2v_i)\dd x
 \rightharpoonup T_{j-1}\bigl(D^2(\beta\log|x|)\bigr)\dd x
 \quad\hbox{on} \,\, B_{2}\setminus B_1.
\end{equation}
The tensors \(T_{j-1}(D^2v_i)\) are positive semidefinite.
For \(K=\operatorname{supp}\nabla\chi\), equation \eqref{eq:8.78}
and a nonnegative cutoff equal to one on \(K\) give
\(\sup_i\int_K\tr T_{j-1}(D^2v_i)\dd x<\infty\).
Positive semidefiniteness {implies}
\(|(T_{j-1})_{ab}|
\leq\frac12((T_{j-1})_{aa}+(T_{j-1})_{bb})\).
Thus every matrix component has uniformly bounded total variation on
\(K\). Proposition~\ref{Prop:8.13} gives uniform convergence of the
remaining multiplier.

The \(C^1\) convergence and \eqref{eq:8.78} therefore {yield}
\begin{align}\label{eq:8.79}
 &\lim_{i\to\infty}\int_{\R^n}\chi
 \operatorname{div}\left(
 |\nabla v_i|^{n-2j}T_{j-1}(D^2v_i)\nabla v_i\right)\dd x\notag\\
 &\quad=
 -\lim_{i\to\infty}
 \int_{\R^n}|\nabla v_i|^{n-2j}
 T_{j-1}(D^2v_i)\nabla v_i\cdot\nabla\chi\dd x\notag\\
 &\quad=
 -\int_{\R^n}|\nabla(\beta\log|x|)|^{n-2j}
 T_{j-1}\bigl(D^2(\beta\log|x|)\bigr)
 \nabla(\beta\log|x|)\cdot\nabla\chi\dd x.
\end{align}

For \(r=|x|\) and \(e_r=x/r\), one has
\[
\begin{aligned}
 D^2(\beta\log|x|)=\frac{\beta}{r^2}(\Id-2e_r\otimes e_r),\qquad
 T_{j-1}\bigl(D^2(\beta\log|x|)\bigr)[e_r,e_r]
 &=\binom{n-1}{j-1}\left(\frac{\beta}{r^2}\right)^{j-1}.
\end{aligned}
\]
Therefore,
\begin{equation}\label{eq:8.80}
 |\nabla(\beta\log|x|)|^{n-2j}
 T_{j-1}\bigl(D^2(\beta\log|x|)\bigr)
 \nabla(\beta\log|x|)\cdot e_r
 =\binom{n-1}{j-1}\beta^{n-j}r^{1-n}.
\end{equation}
Equations \eqref{eq:8.77}, \eqref{eq:8.79}, and \eqref{eq:8.80}
give
\begin{align}\label{eq:8.81}
 \lim_{t\to\infty}P_j(t)
 &=-\binom{n-1}{j-1}\beta^{n-j}
 \int_{\mathbb S^{n-1}}\int_0^\infty
 \partial_r\chi(r\theta)\dd r\dd\theta=|\mathbb S^{n-1}|\binom{n-1}{j-1}\beta^{n-j}.
\end{align}
For \(j=k\), \eqref{eq:8.81} and \(\lim\limits_{t\to\infty}P_k(t)=k\int_{\R^n}\sigma_k(D^2v)\dd x\) prove \eqref{eq:8.10}.
Solving \eqref{eq:8.10} for \(\beta\) gives
\(\beta=\left( \frac{k\displaystyle\int_{\R^n}\sigma_k(D^2v)\dd x} {|\mathbb S^{n-1}|\binom{n-1}{k-1}} \right)^{1/k}\).
By \eqref{eq:8.6} and \(n=2k\), we have
\[
\begin{aligned}
 \int_{\R^n}\Delta_n v\dd x
 =|\mathbb S^{n-1}|\beta^{n-1}=|\mathbb S^{n-1}|
 \left(
 \frac{k\displaystyle\int_{\R^n}\sigma_k(D^2v)\dd x}
 {|\mathbb S^{n-1}|\binom{n-1}{k-1}}
 \right)^{2-1/k},
\end{aligned}
\]
which is \eqref{eq:8.11}. From \eqref{eq:8.81}, we get
\begin{equation}\label{eq:8.82}
\begin{aligned}
 \lim_{t\to\infty}P_j(t)
 =|\mathbb S^{n-1}|\binom{n-1}{j-1}
 \left(
 \frac{k\displaystyle\int_{\R^n}\sigma_k(D^2v)\dd x}
 {|\mathbb S^{n-1}|\binom{n-1}{k-1}}
 \right)^{(n-j)/k}.
\end{aligned}\end{equation}
Finally, Lemma~\ref{Lem:8.5} implies $\sup\limits_{t\ \mathrm{regular}}P_j(t)
 =\lim\limits_{t\to\infty}P_j(t)$, combining this with \eqref{eq:8.82} proves \eqref{eq:8.12} and completes our proof.
\end{proof}

\end{document}